\documentclass[11pt]{article}

\usepackage[utf8]{inputenc} % allow utf-8 input
\usepackage[T1]{fontenc}    % use 8-bit T1 fonts
\usepackage{hyperref}       % hyperlinks
\usepackage{url,authblk}            % simple URL typesetting
\usepackage{booktabs}       % professional-quality tables
\usepackage{amsfonts}       % blackboard math symbols
\usepackage{nicefrac}       % compact symbols for 1/2, etc.
\usepackage{microtype}      % microtypography
\usepackage{xcolor}         % colors

\usepackage[round]{natbib}
\usepackage{graphicx}
\usepackage{amsmath,amssymb,mathtools,amsthm}
\usepackage{mathrsfs,bm}
\usepackage{dirtytalk}
\usepackage{tikz}
\usepackage{placeins}
\usepackage{csvsimple}
\usepackage{siunitx}
\usepackage{algorithm}
\usepackage{algorithmic}
\newcommand*\circled[1]{\tikz[baseline=(char.base)]{
            \node[shape=circle,draw,inner sep=2pt] (char) {#1};}}

\newcommand{\R}{\mathbb{R}}
\newcommand{\E}{\mathbb{E}}

\newcommand{\norm}[1]{\left\|#1\right\|} 
\newcommand{\RFFAssumptionone}{A_{0}}
\newcommand{\RFFAssumptionfour}{A_{1}}
\newcommand{\SpectralAssumptionone}{A_{2}}
\newcommand{\SpectralAssumptiontwo}{A_{3}}
\newcommand{\SpectralAssumptionthree}{A_{4}}
\newcommand{\SpectralAssumptionfour}{A_{5}}
\newcommand{\Samplesizeassumption}{B}

\newcounter{remarkno}

\newtheorem{theorem}{Theorem}

\newtheorem{corollary}[theorem]{Corollary}

\newtheorem{remark}[remarkno]{Remark}

\newenvironment{theorem*}{{\bf Lemma:}}

\title{Minimax Optimal Kernel Two-Sample Tests with Random Features}
\author{Soumya Mukherjee}
\author{Bharath K.~Sriperumbudur}%\thanks{bks18@psu.edu}}
\affil{Department of Statistics\\  
Pennsylvania State University, 
University Park, PA 16802, USA.\\
\texttt{\{szm6510,bks18\}@psu.edu}}

\begin{document}

\maketitle

\begin{abstract}
  % Reproducing Kernel Hilbert Space (RKHS) embedding of probability distributions has proved to be an effective approach, via MMD (maximum mean discrepancy) for nonparametric hypothesis testing problems involving distributions defined over general (non-Euclidean) domains. While a lot of work has been done on this topic, only recently, minimax optimal two-sample tests have been constructed which incorporate, unlike MMD, both the mean element and a regularized version of the covariance operator. However, as with most kernel algorithms, the optimal test scales cubically in the sample size, thereby limiting its applicability. In this paper, we propose a spectral regularized two-sample test based on random Fourier feature (RFF) approximation and investigate the trade-offs between statistical optimality and computational efficiency. We show the proposed test to be minimax optimal if the approximation order of RFF (which depends on the smoothness of the likelihood ratio and the decay rate of the eigenvalues of the integral operator) is sufficiently large. We develop a practically implementable permutation-based version of the proposed test with a data-adaptive strategy for selecting the regularization parameter. We finally illustrate the computational superiority of the RFF-based test over the closest competing statistical test available through experiments on simulated and benchmark datasets, which is obtained at a reasonably small loss of statistical power.

  Reproducing Kernel Hilbert Space (RKHS) embedding of probability distributions has proved to be an effective approach, via MMD (maximum mean discrepancy), for nonparametric hypothesis testing problems involving distributions defined over general (non-Euclidean) domains. While a substantial amount of work has been done on this topic, only recently have minimax optimal two-sample tests been constructed that incorporate, unlike MMD, both the mean element and a regularized version of the covariance operator. However, as with most kernel algorithms, the optimal test scales cubically in the sample size, limiting its applicability. In this paper, we propose a spectral-regularized two-sample test based on random Fourier feature (RFF) approximation and investigate the trade-offs between statistical optimality and computational efficiency. We show the proposed test to be minimax optimal if the approximation order of RFF (which depends on the smoothness of the likelihood ratio and the decay rate of the eigenvalues of the integral operator) is sufficiently large. We develop a practically implementable permutation-based version of the proposed test with a data-adaptive strategy for selecting the regularization parameter. Finally, through numerical experiments on simulated and benchmark datasets, we demonstrate that the proposed RFF-based test is computationally efficient and performs almost similarly (with a small drop in power) to the exact test.
  
  % Finally, we demonstrate that the proposed RFF-based test achieves high computational efficiency in experiments performed on simulated and benchmark datasets.
  
  % We finally illustrate the computational superiority of the RFF-based test over the closest competing statistical test available through experiments on simulated and benchmark datasets, which is obtained at a reasonably small loss of statistical power.
\end{abstract}

\section{Introduction}

% Two-sample or homogeneity testing is a crucial problem in statistics which involves determining whether two distributions are equal by analyzing random samples drawn from each distribution. While this problem has been extensively studied both under parametric (Student's t-test, Hotelling $T^2$ test) and non-parametric (Mann-Whitney U-test, Kolmogorov-Smirnov test, Cramer-von Mises test) settings, classical tests are often limited to applications to low-dimensional Euclidean data domains and their performance do not scale to large data dimensions or sample sizes. A fundamental line of work that focuses on extending the notion of two-sample tests to general domains is based on kernel embedding of probability distributions into reproducing kernel Hilbert spaces and has led to the advent of nonparametric two-sample tests such as the MMD (Maximum Mean Discrepancy) test (\cite{gretton2006kernel,gretton2012kernel}). The lack of minimax optimality of the vanilla MMD test has been understood only very recently and has been analyzed extensively in a series of recent papers, together with proposals for improving the MMD test to obtain minimax optimal tests. (\cite{li2019optimality,schrab2023mmd,SpectralTwoSampleTest}). 

Two-sample, or homogeneity, testing is a fundamental problem in statistics, which seeks to determine whether two probability distributions are equal by analyzing random samples drawn from each of them. This problem has been extensively studied in both parametric (e.g., Student's $t$-test, Hotelling's $T^{2}$ test) and non-parametric (e.g., Mann-Whitney U-test, Kolmogorov-Smirnov test, Cramer-von Mises test) settings. However, classical tests are often restricted to low-dimensional Euclidean data domains and face significant limitations in scalability when applied to high-dimensional data or large sample sizes.

To address these challenges, an important line of research has extended two-sample testing to more general domains through the use of kernel embeddings of probability distributions into reproducing kernel Hilbert spaces (RKHS). This approach has led to the development of nonparametric tests such as the Maximum Mean Discrepancy (MMD) test \citep{gretton2006kernel, gretton2012kernel}. Despite its broad applicability, the vanilla MMD test lacks minimax optimality - a deficiency that has only recently been rigorously analyzed. A series of recent works \citep{li2019optimality, schrab2023mmd, SpectralTwoSampleTest} have addressed this limitation, proposing refined versions of the MMD test that achieve minimax optimality.

% The analysis presented in \cite{li2019optimality} and \cite{schrab2023mmd} is restricted to translation invariant kernels on $\R^d$. Further, the vanilla MMD test does not take into account the covariance operator and is not minimax optimal. Both these issues have been comprehensively addressed in \cite{SpectralTwoSampleTest}, since their analysis caters to kernels on general domains and they regularize the spectrum of the covariance operator in order to achieve minimax optimality with respect to a suitable class of alternatives, which we will discuss later. In this spectral regularized test, instead of looking only at the difference between the mean embeddings of the two probability distributions, one looks at the regularized covariance operator weighted mean embeddings, which leads to a generalization of the Hotelling $T^2$ test to the infinite dimensional RKHS setting.

The analysis in \cite{li2019optimality} and \cite{schrab2023mmd} primarily uses translation-invariant kernels defined on $\R^d$. On the other hand, the vanilla MMD test, while effective in many non-Euclidean settings, does not account for the covariance operator of the distributions under comparison, thereby failing to achieve minimax optimality. These limitations have been comprehensively addressed in \cite{SpectralTwoSampleTest}, which extends the analysis to kernels on general domains and introduces spectral regularization of the covariance operator to achieve minimax optimality with respect to an appropriately defined class of alternatives. The spectral-regularized approach, instead of relying solely on the difference between the mean embeddings of the two distributions, incorporates the regularized covariance operator-weighted mean embeddings. This refinement effectively generalizes the classical Hotelling's $T^2$ test to the infinite-dimensional setting of reproducing kernel Hilbert spaces (RKHS), enabling more robust and theoretically optimal tests for complex data distributions. 
Despite the theoretical advantages and its ability to handle non-Euclidean data, the spectral regularized test 
%the tests proposed in \cite{li2019optimality}, \cite{schrab2023mmd}, and 
\citep{SpectralTwoSampleTest} is computationally expensive compared to the vanilla MMD test since it scales cubically with the number of samples compared to the quadratic scaling of the MMD test, making it less practical for large-scale applications. 
%In particular, the spectral-regularized MMD two-sample test scales cubically with the number of samples, posing a significant challenge in handling large datasets. 
Consequently, the spectral-regularized test statistic, while minimax optimal, is computationally demanding.

The current work explores a specific approximation technique, Random Fourier features (RFF), to mitigate the computational burden associated with the spectral-regularized two-sample test in \cite{SpectralTwoSampleTest}. Random Fourier features~\citep{Rahimi-08a}, widely studied in statistical learning, provide an efficient approximation for kernel functions and offer a trade-off between statistical performance and computational efficiency. For kernels %on $\R^{d} \times \R^{d}$ 
of the form 
\begin{equation*}\label{Kernel in terms of spectral distribution}
K(x,y) = \int_{\Theta} \varphi(x, \theta) \varphi(y, \theta)\, d \Xi(\theta)\,,
\end{equation*}
where $\varphi$ is a feature function and $\Xi$ is a probability distribution on $\R^{d}$ (referred to as the spectral distribution or inverse Fourier transform of $K$), the kernel can be approximated via Monte Carlo sampling. Specifically, given $l$ random samples $\theta^{1:l} :=(\theta_{i})_{i=1}^{l}$ drawn from $\Xi$, an approximate kernel $K_{l}$ is constructed as:
\begin{equation}
\label{Random feature approximation of kernel}
K_{l}(x, y)=\frac{1}{l} \sum_{i=1}^l \varphi\left(x, \theta_i\right) \varphi\left(y, \theta_i\right)=\sum_{i=1}^l \varphi_{i}(x)\varphi_{i}(y)=\left\langle\Phi_l(x), \Phi_l(y)\right\rangle_2,
\end{equation}
 where
$\varphi_{i}(\cdot) := \frac{1}{\sqrt{l}} \varphi(\cdot,\theta_{i})$ for $i=1,2,\ldots,l$,
and the random feature map is given by:
\[\Phi_l(x)=\frac{1}{\sqrt{l}}\left(\varphi\left(x, \theta_1\right), \ldots, \varphi\left(x, \theta_l\right)\right)^{\top} =: \left(\varphi_{1}(x), \ldots, \varphi_{l}(x)\right)^{\top}.
\]

% In addition to Random Fourier Features, alternative approximation techniques such as the Nyström method - another sampling-based approach - could be investigated for similar purposes. Exploring these methods and their computational versus statistical tradeoffs remains an intriguing avenue for future research.

The primary objective of this paper is to understand the trade-off between the number of random features $l$, which governs computational complexity, and the statistical optimality of the resulting approximate hypothesis test based on the kernel approximation $K_l$. This analysis aims to bridge the gap between computational efficiency and statistical optimality in kernel-based two-sample testing for large-scale problems. In particular, we make the following key contributions.
\subsection{Contributions}
The main contributions are:
\begin{itemize}
    \item[(i)] \textit{Computationally efficient and statistically optimal test.} We propose a random Fourier feature (RFF)-based approximation to the spectral-regularized two-sample test statistic, significantly reducing the computational complexity while retaining statistical optimality (Section~\ref{subsec: Construction of the test statistic and the test (Approx)}). 
    We provide a comprehensive theoretical analysis of the tradeoff between computational efficiency and statistical power by deriving sufficient conditions on the number of random features required to ensure that the hypothesis test based on the approximate kernel retains minimax optimality, under the polynomial and exponential decay rates of the eigenvalues of the integral operator (Section~\ref{subection: Oracle test}). %, we \textcolor{red}{maybe expand this a bit after editing the whole paper.}
    \item[(ii)] \textit{Permutation test and adaptive regularization.} We develop a permutation-based implementation of the proposed test (Section~\ref{subsection: Permutation test}), incorporating a fully data-adaptive strategy for selecting the regularization parameter (Sections~\ref{subsection: Adaptation over regularization parameter} and \ref{subsection: Adaptation over kernel and regularization parameter}), thereby enhancing its practical applicability. We also investigate the tradeoff between computational efficiency and statistical optimality for the permutation-based adaptive test, showing that the separation rates are minimax optimal (up to logarithmic factors) while being computationally efficient (Section~\ref{Section: Computational complexity of test statistics}).
    \item[(iii)]  We validate the effectiveness of the proposed test through extensive experiments on synthetic and real-world datasets, demonstrating both its computational advantages and statistical performance (Section~\ref{sec:expts}).
\end{itemize}
\subsection{Related work}
% Our contributions are as follows:

% \begin{itemize}
%     \item We construct a random fourier feature-based approximation of the spectral regularized two-sample test statistic which is computationally efficient.
%     \item We present a theoretical analysis of the tradeoff between computational efficiency and statistical power, deriving sufficient conditions on the number of random features, to retain minimax optimality of the hypothesis test based on the approximate kernel.
%     \item We develop a permutation-based implementation of the hypothesis test with a completely data-adaptive strategy for regularization parameter selection.
%     \item We perform empirical validation of the proposed hypothesis test using synthetic and real-world datasets.
% \end{itemize}
RFF was first employed by \cite{zhao2015fastmmd} in MMD two-sample testing to improve its computational efficiency, but with no theoretical guarantees. Recently, \cite{choi2024computational} investigated the trade-offs between computational efficiency and the statistical power of the RFF-MMD test. While \cite{choi2024computational} focuses on accelerating the classical MMD test, our approach is centered on the more general and efficient spectral-regularized MMD test of \cite{SpectralTwoSampleTest}, which integrates the regularized covariance operator alongside the mean embeddings. This refinement ensures minimax optimality over a broader class of alternatives and enhances sensitivity to distributional differences.

A key distinction between the two approaches lies in the underlying assumptions on the kernel. The analysis in \cite{choi2024computational} relies on the translation invariance of the kernel on \( \mathbb{R}^d \) and further imposes a product structure, requiring each component to be translation invariant on \( \mathbb{R} \). These structural constraints are fundamental to their theoretical results. In contrast, our analysis imposes no such restrictions, allowing for a broader class of kernels, including those defined on more general domains. Consequently, our framework offers greater flexibility and applicability beyond Euclidean settings.

Another fundamental difference emerges in the characterization of alternatives. \cite{choi2024computational} considers a Sobolev smoothness assumption, where the difference in densities belongs to a Sobolev ball of a given order. Our work instead formulates the regularity condition in terms of the range of fractional power of an integral operator, which offers a more general perspective grounded in functional analysis. This formulation naturally aligns with the properties of kernel integral operators and accommodates a richer class of distributional differences. Finally, while their result on the minimax separation rate assumes that the two distributions have bounded support, our analysis does not require such an assumption, further extending its applicability. We refer the reader to Section~\ref{Sec:setup} for details. These distinctions highlight that our approach is not merely a computational improvement but also a theoretically grounded extension that provides a broader and more flexible framework for efficient two-sample testing. \vspace{1mm}\\
%By working within a spectral-regularized setting with weaker assumptions on the kernel and alternatives, our method advances both the statistical rigor and practical applicability of kernel-based hypothesis testing.

The paper is organized as follows. Definitions, notations, and technical preliminaries are captured in Section~\ref{sec:def}. A summary of minimax testing, MMD test, and spectral regularized MMD test is provided in Section~\ref{Sec:setup}. The proposed approximate spectral MMD test, along with its permuted adaptive version, is presented in Section~\ref{sec:approx}. Section~\ref{sec:approx} also discusses the statistical optimality of the proposed tests. Section~\ref{Section: Computational complexity of test statistics} discusses the theoretical tradeoff between computational complexity and statistical optimality of the proposed approximate test, while the empirical tradeoff is demonstrated in Section~\ref{sec:expts} through simulation studies. All the proofs of results are provided in Section \ref{Proofs of main theorems, corollaries and propositions}, while supplementary results are relegated to appendices.

\section{Definitions,  notations, and preliminaries}\label{sec:def}

For constants $a$ and $b$, $a \lesssim b$ (\emph{resp.} $a \gtrsim b$) denotes that there exists a positive constant $c$ (\emph{resp.} $c'$) such that $a\leq cb$ (\emph{resp.} $a \geq c' b)$. $a \asymp b$ denotes that there exists positive constants $c$ and $c'$ such that  $cb \leq a \leq c' b$. $[\ell]$ is used to denote $\{1,\ldots,\ell\}$.

Given a topological space $\mathcal{X}$, let $M^{b}_{+}(\mathcal{X})$ denote the space of all finite non-negative Borel measures on $\mathcal{X}$. We denote the space of bounded continuous functions defined on $\mathcal{X}$ by $C_{b}(\mathcal{X})$. For any $\mu \in M^{b}_{+}(\mathcal{X})$, let $L^r(\mathcal{X},\mu)$ denote the Banach space of $r$-power $(r\geq 1)$ $\mu$-integrable functions. For $f \in L^r(\mathcal{X},\mu)\eqcolon L^r(\mu)$, we denote $L^r$-norm of $f$ as $\norm{f}_{L^r(\mu)}\coloneq (\int_{\mathcal{X}}|f|^r\,d\mu)^{1/r}$. $\mu^n := \mu \times \stackrel{n}{...} \times \mu$ denotes the $n$-fold product measure. The equivalence class of the function $f$ is defined as $[f]_{\sim}$ and consists of functions $g \in L^r(\mathcal{X},\mu)$ such that $\norm{f-g}_{L^r(\mu)}=0$. 

For any Hilbert space $H$, we denote the corresponding inner product and norm using $\langle \cdot,\cdot\rangle_{H}$ and $\norm{\cdot}_{H}$, respectively. For any two abstract Hilbert spaces $H_1$ and $H_2$, let $\mathcal{L}(H_1, H_2)$ denote the space of bounded linear operators from $H_1$ to $H_2$. For $S \in \mathcal{L}(H_1,H_2)$, its adjoint is denoted by $S^*$. $S \in \mathcal{L}(H) := \mathcal{L}(H,H)$ is called self-adjoint if $S^*=S$. For $S \in \mathcal{L}(H)$, $\operatorname{Tr}(S)$, $\norm{S}_{\mathcal{L}^2(H)}$, and $\norm{S}_{\mathcal{L}^{\infty}(H)}$ denote the trace, Hilbert-Schmidt and operator norms of $S$, respectively. For $x,y \in H$, $x \otimes_{H} y$ is an element of the tensor product space of $H \otimes H$ which can also be seen as an operator from $H \to H$ as $(x \otimes_{H}y)z=x\langle y,z \rangle_{H} $ for any $z \in H$.

\subsection{Mean element, covariance operator, and integral operator}
Let us denote the reproducing kernel Hilbert spaces corresponding to reproducing kernels $K: \mathcal{X} \times \mathcal{X} \to \mathbb{R}$ and its random feature approximation $K_{l}$ (as defined in \eqref{Random feature approximation of kernel}) as $\mathcal{H}$ and $\mathcal{H}_{l}$, respectively. We will now define some relevant functions and operators for defining the MMD test, the spectral regularized MMD test, and the proposed computationally efficient random Fourier features-based modification of the spectral regularized MMD test, which we refer to as the RFF test. 

Given an RKHS $\mathcal{H}$ associated with the reproducing kernel $K$, the RKHS embedding of probability measure $P$ is given by
$$
\mu_P(\cdot)=\int_{\mathcal{X}} K(\cdot, x)\, d P(x),
$$
which is also referred to as the mean embedding/element of $P$. The defining characteristic of the mean element is that it satisfies the relation $\E_{X \sim P}\left[f(X)\right] = \int_{\mathcal{X}} f(x) d P(x) = \langle f, \mu_P \rangle_{\mathcal{H}}$ for any $f \in \mathcal{H}$. The covariance operator for the probability measure $P$ maps from $\mathcal{H}$ to $\mathcal{H}$ and is given by 
$$
\Sigma_{P} = \int_{\mathcal{X}} \left(K(\cdot,x) - \mu_{P}\right) \otimes_{\mathcal{H}} \left(K(\cdot,x) - \mu_{P}\right) d P(x),
$$
with its action on a function $f\in \mathcal{H}$ being defined by
$$
\Sigma_{P}f = \int_{\mathcal{X}} K(\cdot,x) f(x) d P(x) - \mu_{P} \int_{\mathcal{X}} f(x) dP(x).
$$
The defining property of the covariance operator is that it satisfies the relation $$
\begin{aligned}
\operatorname{Cov}_{X\sim P}\left[f(X),g(X)\right] 
&= \E_{X \sim P}\left[f(X)g(X)\right] - \E_{X \sim P}\left[f(X)\right] \E_{X \sim P}\left[g(X)\right]\\
&= \int_{\mathcal{X}} f(x) g(x) d P(x) - \int_{\mathcal{X}} f(x) d P(x) \int_{\mathcal{X}} g(x) d P(x)\\
&= \langle f, \Sigma_{P}g \rangle_{\mathcal{H}}
\end{aligned}$$ for any $f,g \in \mathcal{H}$. One can also express the covariance operator as
$$
\begin{aligned}
&\Sigma_{P} = \frac{1}{2} \int_{\mathcal{X} \times \mathcal{X}} \left(K(\cdot,x) - K(\cdot,y)\right) \otimes_{\mathcal{H}} \left(K(\cdot,x) - K(\cdot,y)\right) dP(x) dP(y).
\end{aligned}$$

The integral operator for the probability measure $P$ maps from $L^{2}(P)$ to $L^{2}(P)$ and is defined by its action on any $f \in L^{2}(P)$, given by 
$$
\mathcal{T}_{P}f = \int_{\mathcal{X}} K(\cdot,x) f(x) d P(x) - \mu_{P} \int_{\mathcal{X}} f(x) dP(x).
$$
We define the (centered) inclusion operator for the probability measure $P$ as $$\mathfrak{I}_{P}: \mathcal{H} \rightarrow L^{2}(P), f \mapsto\left[f-\int_{\mathcal{X}} f(x) dP(x)\right]_{\sim}.$$ The adjoint of the (centered) inclusion operator 
 is given by $$\mathfrak{I}_{P}^{*}: L^{2}(P) \rightarrow \mathcal{H}, f \mapsto \int K(\cdot, x) f(x) d P(x)-\mu_{P} \int_{\mathcal{X}} f(x) dP(x).$$ Moreover, $$\Sigma_{P}=\mathfrak{I}_{P}^{*}\mathfrak{I}_{P}= \frac{1}{2} \int_{\mathcal{X} \times \mathcal{X}}(K(\cdot, x)-K(\cdot, y)) \otimes_{\mathcal{H}}(K(\cdot, x)-K(\cdot, y)) d P(x) d P(y)$$ and $$\mathcal{T}_{P}=\mathfrak{I}_{P} \mathfrak{I}_{P}^{*}=\Upsilon_{P}-\left(1 \otimes_{L^{2}(P)} 1\right) \Upsilon_{P}-\Upsilon_{P}\left(1 \otimes_{L^{2}(P)}1\right)+\left(1 \otimes_{L^{2}(P)} 1\right) \Upsilon_{P}\left(1 \otimes_{L^{2}(P)} 1\right),$$ where $$\Upsilon_{P}: L^{2}(P) \rightarrow L^{2}(P), f \mapsto \int K(\cdot, x) f(x) d P(x).$$ We refer the reader to \citep[Proposition C.2]{ApproximateKernelPCARandomFeaturesStergeSriperumbudur} for details.

% We define the average distribution corresponding to distributions $P$ and $Q$ as $R = \frac{P+Q}{2}$. 
The mean embedding, covariance operator, integral operator, and inclusion operator of the distribution $R:=\frac{P+Q}{2}$ corresponding to the RKHS $\mathcal{H}$ are denoted by $\mu_{PQ}$, $\Sigma_{PQ}$, $\mathcal{T}_{PQ}$ and $\mathfrak{I}_{PQ}$, respectively. Similarly, we denote the mean embedding, covariance operator, integral operator, and inclusion operator of the distribution $R$ corresponding to the $l$-dimensional RKHS $\mathcal{H}_{l}$ as $\mu_{PQ,l}$, $\Sigma_{PQ,l}$, $\mathcal{T}_{PQ,l}$ and $\mathfrak{A}_{PQ,l}$, respectively. When it is clear from context, we drop the first subscript in $\mu_{PQ,l}$, $\Sigma_{PQ,l}$, $\mathcal{T}_{PQ,l}$ and $\mathfrak{A}_{PQ,l}$ and use the notations $\mu_{l}$, $\Sigma_{l}$, $\mathcal{T}_{l}$ and $\mathfrak{A}_{l}$ instead.
\subsection{Spectral regularization}
Consider any function $s: [0,\infty) \to [0,\infty)$. We will refer to such a function as a regularizer or a spectral function. If the domain of $s$ does not contain $0$, it is referred to as a positive regularizer/spectral function. Given any regularizer/spectral function $s$ and a compact, self-adjoint operator $\mathcal{M}$ defined on a separable Hilbert space $H$, we invoke functional calculus to define the operator $s(\mathcal{M})$ as 
$$
s(\mathcal{M}):=\sum_{i \geq 1} s\left(\tau_i\right)\left(\psi_i \otimes_H \psi_i\right)+s(0)\left(I-\sum_{i \geq 1} \psi_i \otimes_H \psi_i\right),
$$
where $\mathcal{M}$ has the spectral representation, $\mathcal{M}=\sum_i \tau_i \psi_i \otimes_H \psi_i$ with $\left(\tau_i, \psi_i\right)_i$ being the eigenvalues and eigenfunctions of $\mathcal{M}$. Often, $s$ is chosen to regularize/modify the spectrum of $\mathcal{M}$ in a certain way. For any $\lambda>0$, choosing $s(x) = g_{\lambda}(x) = (x + \lambda I)^{-1}$ and with $I$ representing the identity operator, we define the regularized covariance operator $\Sigma_{PQ,\lambda}$ as $\Sigma_{PQ,\lambda} \coloneq g_{\lambda}(\Sigma_{PQ}) = (\Sigma_{PQ} + \lambda I)^{-1}$. The operators $\mathcal{T}_{PQ,\lambda}$, $\Sigma_{PQ,\lambda,l}$ and $\mathcal{T}_{PQ,\lambda,l }$ are defined analogously, with $I_{l}$ playing the role of the identity operator while defining $\Sigma_{PQ,\lambda,l}$.

$\mathcal{N}_{1}(\lambda)$ and $\mathcal{N}_{2}(\lambda)$ characterize the intrinsic dimensionality of the RKHS $\mathcal{H}$, where
$$\mathcal{N}_{1}(\lambda)\coloneq\operatorname{Tr}\left(\Sigma_{P Q, \lambda}^{-1 / 2} \Sigma_{P Q} \Sigma_{P Q, \lambda}^{-1 / 2}\right),\,\,\text{and}\,\,
\mathcal{N}_{2}(\lambda)\coloneq\left\|\Sigma_{P Q, \lambda}^{-1 / 2} \Sigma_{P Q} \Sigma_{P Q, \lambda}^{-1 / 2}\right\| _{\mathcal{L}^{2}(\mathcal{H})},$$ $$\mathcal{N}_{1,l}(\lambda)\coloneq\operatorname{Tr}\left(\Sigma_{P Q, \lambda,l}^{-1 / 2} \Sigma_{P Q,l} \Sigma_{P Q, \lambda,l}^{-1 / 2}\right),\,\,\text{and}\,\, \mathcal{N}_{2,l}(\lambda)\coloneq\left\|\Sigma_{P Q, \lambda,l}^{-1 / 2} \Sigma_{P Q,l} \Sigma_{P Q, \lambda,l}^{-1 / 2}\right\| _{\mathcal{L}^{2}(\mathcal{H}_{l})}$$ play analogous roles with respect to the RKHS $\mathcal{H}_{l}$. For any operator $\mathcal{M}: S_{1} \to S_{2}$, we define $\operatorname{Ran}(\mathcal{M})$ as the range space of the operator $\mathcal{M}$, given by $\operatorname{Ran}(\mathcal{M}) \coloneq \left\{\mathcal{M}f: f \in S_{1}\right\}$.

\section{Problem setup}\label{Sec:setup}
In this section, we introduce the problem and formalism of minimax testing in Section~\ref{subsec:minimax}. We then recall the MMD and spectral regularized MMD tests, along with their statistical optimality results, in Sections~\ref{subsec:mmd} and \ref{subsec : Spectral Regularized MMD Test}, respectively.
\subsection{Minimax testing}\label{subsec:minimax}
The problem of interest in the current paper is the canonical problem of two-sample testing, which involves analyzing mutually independent random samples $\mathbb{X}^{1:N} \coloneq\left(X_i\right)_{i=1}^N \overset{i.i.d}{\sim}P$ and $\mathbb{Y}^{1:M} \coloneq\left(Y_j\right)_{j=1}^M  \overset{i.i.d}{\sim}Q$ drawn from two probability distributions $P$ and $Q$ defined on a topological space $\mathcal{X}$ to test $H_{0}: P=Q$ against $H_{1}:P \neq Q$. Let us denote a test function based on $\mathbb{X}^{1:N}$ and $\mathbb{Y}^{1:M}$ as $\phi(\mathbb{X}^{1:N},\mathbb{Y}^{1:M})\coloneq \phi_{N,M}$ that takes the value $\phi_{N,M}=1$ when $H_{0}$ is rejected, while taking the value $\phi_{N,M}=0$ when $H_{0}$ is not rejected. Further, let us denote the collection of exact level-$\alpha$ (i.e., Type-I error less than or equal to $\alpha$) tests for any given finite $N,\,M$ to be $\Phi_{N,M,\alpha}$. For some choice of probability metric $\rho$ defined over the space of probability distributions on $\mathcal{X}$, consider the class of alternatives $\mathcal{P}_{\Delta} = \left\{(P,Q): \rho^{2}(P,Q) \geq \Delta\right\}$, where $\Delta$ is the separation boundary (also referred to as contiguity radius). Then, the Type II error of a test $\phi_{N,M} \in \Phi_{N,M,\alpha}$ with respect to $\mathcal{P}_{\Delta}$ is given by $$R_{\Delta}(\phi_{N,M}) = \underset{(P,Q) \in \mathcal{P}_{\Delta}}{\sup} \E(1-\phi_{N,M}),$$ where the expectation is jointly over the distribution of $\mathbb{X}^{1:N}$ and $\mathbb{Y}^{1:M}$. In this paper, we consider the minimax framework of shrinking alternatives in the non-asymptotic setting, where 
%where $\lim_{N,M \to \infty} \Delta = 0$ and the interest is in analyzing whether the Type-II error is sufficiently small for some $\phi_{N,M} \in \Phi_{N,M,\alpha}$. Equivalently, in the non-asymptotic setting, 
for any given $0<\delta<1-\alpha$, the minimax separation $\Delta^{*}$ is the smallest possible separation boundary such that $\inf\{  R_{\Delta}(\phi_{N,M}):\phi_{N,M} \in \Phi_{N,M,\alpha}\} \leq \delta$ and a test $\phi_{N,M} \in \Phi_{N,M,\alpha}$ is said to achieve the minimax optimal rate if $R_{\Delta}(\phi_{N,M}) \leq \delta$ for some $\Delta \asymp \Delta^{*}$. 
%The interest, in this equivalent formulation, is in analyzing how fast the separation boundary $\Delta$ decays to zero as $N,M \to \infty$, with a faster decay being indicative of a better test.

% The authors of \cite{SpectralTwoSampleTest} propose a spectral-regularized kernel two-sample test that achieves minimax optimality with respect to an appropriately defined class of alternatives over very general choices of the data domain $\mathcal{X}$, as compared to previous works in the literature such as \cite{li2019optimality} and \cite{schrab2023mmd}. We now proceed to introduce the MMD test and the Spectral Regularized MMD test.

% We now proceed to introduce the MMD test, the Spectral Regularized MMD test and the Approximate Spectral Regularized MMD test based on Random Fourier Feature sampling.

\subsection{Maximum mean discrepancy (MMD) test}\label{subsec:mmd}
Given samples $\mathbb{X}^{1:N}$ and $\mathbb{Y}^{1:M}$, the MMD test \citep{gretton2006kernel,gretton2012kernel} involves constructing a test statistic based on 
\begin{equation*}\label{Squared MMD in terms of kernel}
\begin{aligned}
    \operatorname{MMD}^2(P,Q) &= \norm{\mu_{P} - \mu_{Q}}^2_{\mathcal{H}}\\
    &=\E_{X,X^{\prime}\sim P}K(X,X^{\prime}) + \E_{Y,Y^{\prime} \sim Q}K(Y,Y^{\prime}) - 2\E_{X\sim P,Y \sim Q} K(X,Y)
\end{aligned}
\end{equation*}
as
\begin{equation}\label{U-statistic estimator of Squared MMD}
\begin{aligned}
    &\widehat{\operatorname{MMD}}^{2}(P,Q)\\ &= \frac{1}{N(N-1)} \frac{1}{M(M-1)}\sum_{1\leq i \neq j \leq N} \sum_{1\leq i^{\prime} \neq j^{\prime} \leq M} \left\langle K(\cdot,X_{i}) - K(\cdot,Y_{i^{\prime}}), K(\cdot,X_{j}) - K(\cdot,Y_{j^{\prime}}) \right\rangle_{\mathcal{H}}\\
    &=\frac{1}{N(N-1)}\sum_{1 \leq i \neq j \leq N}K(X_{i},X_{j}) + \frac{1}{M(M-1)}\sum_{1 \leq i \neq j \leq M}K(Y_{i},Y_{j}) - \frac{2}{NM}\sum_{1\leq i \leq N,1\leq j\leq M} K(X_{i},Y_{j}),
\end{aligned}
\end{equation}
which is a U-statistic estimator of $\operatorname{MMD}^{2}(P,Q)$. The MMD test rejects the null hypothesis $H_{0}:P=Q$ if $\widehat{\operatorname{MMD}}^{2}(P,Q)$ is larger than a certain critical threshold that depends on the level $\alpha$, where the threshold is obtained as the $(1-\alpha)$-quantile of the asymptotic distribution of $\widehat{\operatorname{MMD}}^{2}(P,Q)$ under $H_0$ or as the empirical $(1-\alpha)$-quantile of the permuted version of $\widehat{\operatorname{MMD}}^{2}(P,Q)$.
% (if the test is based on the asymptotic distribution of $\widehat{\operatorname{MMD}}^{2}(P,Q)$ under the null) or on the level of the test $\alpha$ together with the sample sizes $N$ and $M$ (if the the test is based on the permutation-based estimation of the null distribution of $\widehat{\operatorname{MMD}}^{2}(P,Q)$).
% \textcolor{red}{Soumya: Can you improve this section? Basucally, show the RFF versio of this test as per kim's paper, and discuss the computational-statistical optimality wrt certian alternatives as per his paper. Then we move to spectral regularized test.}
The MMD test statistic given by \eqref{U-statistic estimator of Squared MMD} has a computational complexity of $O((N+M)^{2}d)$, assuming that a single kernel evaluation $K(\cdot,\cdot)$ requires $O(d)$ operations, which is typically the case. 

\cite{zhao2015fastmmd} were the first to propose using RFF to reduce the computational complexity of the classical MMD test statistic, which was recently investigated from a theoretical perspective by \cite{choi2024computational}. They employed translation invariant kernel $K$ on $\R^{d}$, i.e., \begin{equation}\label{Translation invariant Kernel in terms of spectral distribution}
   K(x,y) = \upsilon(x-y)= \int_{\Theta} \exp \left\{i\theta^{\top}(x-y)\right\}\, d \Xi(\theta)\,,\,x, y \in \R^d
\end{equation} for some continuous positive definite function $\upsilon$. The second equality follows from Bochner's theorem (\citet{wendland2004scattered}, Theorem 6.6),  
%provides an equivalent representation of the kernel in terms of its spectral distribution and allows one to express the kernel $K$ in the form \eqref{Kernel in terms of spectral distribution}. More specifically, we have that,
where $\Xi$ is a finite non-negative Borel measure on $\Theta=\R^d$, determined by the inverse Fourier transform of $\upsilon$. Since $K$ is real-valued and symmetric, \eqref{Translation invariant Kernel in terms of spectral distribution}  reduces to:
\begin{equation*}\label{Translation invariant real-valued symmetric Kernel in terms of spectral distribution}
\begin{aligned}
     K(x,y) =& \int_{\Theta} \cos\left(\theta^{\top}(x-y)\right)\, d \Xi(\theta)
     = \,\upsilon(0) \int_{\Theta} \cos\left(\theta^{\top}(x-y)\right)\, d \frac{\Xi}{\upsilon(0)}(\theta) \\
     =& \int_{\Theta} \varphi_{\theta}(x)^\top\varphi_{\theta}(y)\, d \frac{\Xi}{\upsilon(0)}(\theta)\,,
\end{aligned}
\end{equation*}
where $\varphi_{\theta}(\cdot) = [\sqrt{\upsilon(0)}\cos(\theta^\top \cdot),\sqrt{\upsilon(0)}\sin(\theta^\top \cdot)]^\top$. 
% \begin{bmatrix}
%     \sqrt{\upsilon(0)}\cos(\theta^\top \cdot)\\
%     \sqrt{\upsilon(0)}\sin(\theta^\top \cdot)
% \end{bmatrix}$. %is the feature map associated with the kernel $K$. 
%The final equality follows from the trigonometric identity: $\cos(a-b) = \cos(a)\cos(b) + \sin(a)\sin(b)$. 
Since $\upsilon(0) = \int_{\Theta} d\Xi(\theta)$, we can assume, without loss of generality, that $\Xi$ is a probability measure on $\Theta=\R^d$.

Using $l$ random samples $\left(\theta_{i}\right)_{i=1}^{l}$ drawn from $\Xi$, one can construct an approximate Monte Carlo kernel estimator:
\begin{equation*}\label{Random feature approximation of kernel special case of translation onvariant symmetric real-valued kernel}
    K_{l}(x,y) = \frac{1}{l}\sum_{i=1}^{l} \left\langle \varphi_{\theta_{i}}(x),  \varphi_{\theta_{i}}(y)\right\rangle_2 = \left\langle \Phi_{l}(x), \Phi_{l}(y) \right\rangle_{2},
\end{equation*}
where $\varphi_{\theta_{i}}(x) = \left[\cos(\theta_{i}^{\top}x),\sin(\theta_{i}^{\top}x)\right]^{\top}$ and $\Phi_{l}(x) = \frac{1}{\sqrt{l}} \left[\varphi_{\theta_{1}}(x)^{\top} ,\cdots ,\varphi_{\theta_{l}}(x)^{\top}  \right]^{\top}$. Based on this approximation, an approximate RFF-based V-statistic estimator of $\operatorname{MMD}^{2}(P,Q)$ is obtained as 
\begin{equation}\label{V-statistic RFF-MMD squared estimate for special case of translation onvariant symmetric real-valued kernel}
    \widehat{\operatorname{MMD}}_{V,l}^{2}(P,Q) = \norm{\frac{1}{N}\sum_{i=1}^{N}\Phi_{l}(X_{i})-\frac{1}{M}\sum_{i=1}^{M}\Phi_{l}(Y_{j})}_{2}^{2},
\end{equation}
while an approximate RFF-based U-statistic estimator of $\operatorname{MMD}^{2}(P,Q)$ is given by
\begin{equation}\label{U-statistic RFF-MMD squared estimate for special case of translation onvariant symmetric real-valued kernel}
\begin{aligned}
    &\widehat{\operatorname{MMD}}_{U,l}^{2}(P,Q)\\
    =& \frac{1}{N(N-1)}\sum_{1 \leq i \neq j \leq N}\left\langle\Phi_{l}(X_{i}),  \Phi_{l}(X_{j})\right\rangle_{2} + \frac{1}{M(M-1)}\sum_{1 \leq i \neq j \leq M}\left\langle\Phi_{l}(Y_{i}),  \Phi_{l}(Y_{j})\right\rangle_{2}) \\
    &- \frac{2}{NM}\sum_{1\leq i \leq N,1\leq j\leq M} \left\langle\Phi_{l}(X_{i}),  \Phi_{l}(Y_{j})\right\rangle_{2}\,.
\end{aligned}
\end{equation}
Both the RFF-based V-statistic and U-statistic estimators, given by \eqref{V-statistic RFF-MMD squared estimate for special case of translation onvariant symmetric real-valued kernel} and \eqref{U-statistic RFF-MMD squared estimate for special case of translation onvariant symmetric real-valued kernel}, significantly reduce the computational cost of the classical MMD test from quadratic complexity $O((N+M)^{2}d)$ to linear complexity  $O((N+M)ld)$. This reduction is particularly useful for large-scale data applications, where the classical MMD test is computationally prohibitive.

\cite{choi2024computational} investigated the theoretical properties of permutation tests based on $\widehat{\operatorname{MMD}}_{V,l}^{2}(P,Q)$ and $\widehat{\operatorname{MMD}}_{U,l}^{2}(P,Q)$, providing both negative and positive results regarding the feasibility of achieving a favorable computational-statistical tradeoff within their problem setup. Specifically, they established that permutation tests based on $\widehat{\operatorname{MMD}}_{V,l}^{2}(P,Q)$ and $\widehat{\operatorname{MMD}}_{U,l}^{2}(P,Q)$ fail to achieve pointwise consistency when the number of random Fourier features $l$ remains fixed, even as the sample sizes $N$ and $M$ tend to $+\infty$, and achieves pointwise consistency if \( l \) is allowed to diverge to infinity, even at an arbitrarily slow rate, as $N$ and $M$ grow. Moreover, they showed that the permutation tests achieve the minimax separation boundary of $\min\{N,M\}^{-\frac{2s}{4s+d}}$ (as enjoyed by the MMD test \citep{schrab2023mmd}) for the class of alternatives consisting of densities with bounded support and separated in the $L^2$ metric, where the difference of densities belongs to the Sobolev ball of order $s$ and fixed radius in $\R^d$, as long as $l \geq \min\{N,M\}^{\frac{4d}{4s+d}}$, resulting in a computational complexity of $O((N+M)\min\{N,M\}^{\frac{4d}{4s+d}}d)$. This means, for $s>\frac{3d}{4}$, the complexity is sub-linear and tends to linear as $s\rightarrow\infty$, while for $s<\frac{3d}{4}$, it is computationally beneficial to use the MMD test without the RFF approximation, though both are statistically minimax optimal.

%From the perspective of uniform consistency over a sufficiently large class of alternatives, their most significant positive result shows that these permutation tests achieve the minimax separation boundary of $\min\{N,M\}^{-\frac{2s}{4s+d}}$ for a class of alternatives consisting of densities with bounded support and separated in the $L^2$) metric, where the difference of densities belongs to the Sobolev ball of order $s$ and fixed radius in $\R^d$. To attain minimax optimality, the tests require at least $l \geq \min\{N,M\}^{\frac{4d}{4s+d}}$ random Fourier features, resulting in a computational complexity of $O((N+M)\min\{N,M\}^{\frac{4d}{4s+d}}d)$.

\subsection{Spectral regularized MMD test}\label{subsec : Spectral Regularized MMD Test}
Despite the widespread popularity and elegant theoretical properties of the classical version of the MMD test, it is not sensitive enough to capture all potential discrepancies between the distributions $P$ and $Q$ for finite sample sizes. This leads the classical MMD test to not be minimax optimal with respect to a natural class of alternatives $\mathcal{P}$, which we will define shortly. More specifically, \citep{SpectralTwoSampleTest} expressed the squared MMD in terms of the integral operator $\mathcal{T}_{PQ}$ and the \say{likelihood ratio deviation} $u \coloneq \frac{dP}{dR} - 1$ as 
\begin{equation}\label{Squared MMD in terms of integral operator and likelihood}
    \operatorname{MMD}^{2}(P,Q) = 4 \left\langle \mathcal{T}_{PQ}u,u\right\rangle_{L^{2}(R)}.
\end{equation}
Discrepancies between $P$ and $Q$ are captured by how far the function $u$ deviates from the $0$ function. Provided the kernel $K$ is bounded, the operator $\mathcal{T}_{PQ}: L^{2}(R) \mapsto L^{2}(R)$ is a positive self-adjoint trace-class operator, with its eigenvalue-eigenfunction pairs being denoted by $(\lambda_{i},\tilde{\phi}_{i})_i$. As a consequence of \eqref{Squared MMD in terms of integral operator and likelihood}, we can express the squared MMD as
\begin{equation*}\label{Squared MMD in terms of eigenvalue eigenfunctions of integral operator}
    \operatorname{MMD}^{2}(P,Q) = 4 \sum_{i \in I}\lambda_{i}\left\langle u,\tilde{\phi}_{i}\right\rangle_{\mathcal{H}}^{2},
\end{equation*}
where $I$ is the index set corresponding to the eigenvalues of $\mathcal{T}_{PQ}$. Since $\mathcal{T}_{PQ}$ is trace-class, $\lim_{i \to \infty} \lambda_{i} = 0$. Consequently, the Fourier coefficients $\langle u,\tilde{\phi}_{i}\rangle_{\mathcal{H}}^{2}$ of the likelihood ratio deviation $u$ corresponding to the larger $i$'s (i.e., higher frequencies) are given lesser weightage and therefore, $\operatorname{MMD}^{2}(P,Q)$ is less sensitive to deviations of $u$ from $0$ in the higher-frequency components. On the other hand, one can consider a uniform weighting of all the frequency components, as in 
\begin{equation*}\label{Norm of u squared}
    \norm{u}_{L^2(R)}^2=\sum_i\left\langle u, \widetilde{\phi}_i\right\rangle_{L^2(R)}^2=\chi^2\left(P \| R\right)=\frac{1}{2} \int_\mathcal{X} \frac{(d P-d Q)^2}{d(P+Q)}=: \underline{\rho}^2(P, Q),
\end{equation*}
where $\underline{\rho}^2(P,Q)\coloneq \chi^2\left(P\left\Vert\right.R\right) = \frac{1}{2}\int_\mathcal{X}\frac{(dP-dQ)^2}{d(P+Q)}=\norm{\frac{dP}{dR}-1}_{L^{2}(R)}^{2}$ is a metric over probability measures that induces the same topology as the Hellinger distance \citep[Lemma F.18]{SpectralTwoSampleTest}. Since such a uniform weighting mitigates the issue of reduced sensitivity to the high-frequency components of $u$, \cite{SpectralTwoSampleTest} proposed a regularization of the spectrum of the integral operator $\mathcal{T}_{PQ}$ to arrive at an analog $\eta_{\lambda}(P,Q)$ of $\operatorname{MMD}^{2}(P,Q)$, referred to as the spectral regularized discrepancy and defined as
\begin{equation*}
    \eta_\lambda(P, Q)=4\left\langle\mathcal{T} g_\lambda(\mathcal{T}) u, u\right\rangle_{L^2(R)},
\end{equation*}
where $g_{\lambda} : (0,\infty) \to (0,\infty)$ is a positive regularizer/spectral function satisfying $\lim_{\lambda \to 0} x g_{\lambda}(x) \asymp 1$. The salient feature of $\eta_\lambda(P, Q)$ is that it satisfies $\eta_\lambda(P, Q) \asymp \norm{u}_{L^2(R)}^2$ if $u \in \operatorname{Ran}(\mathcal{T}^{\theta})$, $\theta>0$ and $\lambda>0$ is chosen such that $\norm{u}_{L^2(R)}^2 \gtrsim \lambda^{2\theta}$, which shows that it is better equipped to detect discrepancies between $P$ and $Q$ under mild conditions. Therefore, following \cite{SpectralTwoSampleTest}, the natural class of alternatives to consider for studying minimax optimality in the current setting is 
\begin{equation}\label{Class of alternatives}
    \mathcal{P}:=\mathcal{P}_{\theta, \Delta}:=\left\{(P, Q): \frac{d P}{d R}-1 \in \operatorname{Ran}(\mathcal{T}^\theta), \,\underline{\rho}^2(P, Q) \geq \Delta\right\}.
\end{equation}

One should note that, for $\theta \in (0,\frac{1}{2}]$, $\operatorname{Ran}(\mathcal{T}^\theta)$ is an interpolation space between $\mathcal{H}$ and $L^{2}(R)$, containing functions which are less than smooth than those belonging to the RKHS $\mathcal{H}$, with the degree of smoothness decreasing as $\theta$ approaches 0. On the other hand, for $\theta>\frac{1}{2}$, $\operatorname{Ran}\left(\mathcal{T}^\theta\right)$ is a subspace of the RKHS $\mathcal{H}$ and contains progressively smoother functions as $\theta$ increases beyond $\frac{1}{2}$.

\citep{SpectralTwoSampleTest} provided an alternate expression for the spectral regularized discrepancy $\eta_\lambda(P, Q)$ as 
\begin{equation*}\label{eta as norm sqaured}
    \eta_\lambda(P, Q) = \norm{g_{\lambda}^{1/2}(\Sigma_{PQ})\left(\mu_{P} - \mu_{Q}\right)}_{\mathcal{H}}^{2},
\end{equation*}
which shows that spectral regularized discrepancy takes into account the covariance operator $\Sigma_{PQ}$ in addition to the discrepancy between the mean embeddings $\mu_{P}$ and $\mu_{Q}$. Another expression for $\eta_\lambda(P, Q)$, which will be useful for constructing a statistical estimator, is given by
\begin{equation*}\label{eta as double integral}
\begin{aligned}
    &\eta_\lambda(P, Q)\\
    &=\!\!\! \int_{\mathcal{X}^{4}} \left\langle g_{\lambda}^{1/2}\left(\Sigma_{PQ})(K(\cdot,x) - K(\cdot,y)\right), g_{\lambda}^{1/2}\left(\Sigma_{PQ})(K(\cdot,x^{\prime}) - K(\cdot,y^{\prime})\right)\right\rangle_{\mathcal{H}} dP(x) dP(x^{\prime}) dQ(y) dQ(y^{\prime}).
\end{aligned}
\end{equation*}
%However, note that the information from the samples $\mathbb{X}^{1:N}$ and $\mathbb{Y}^{1:M}$ must now be used to estimate both the covariance operator as well as the mean embeddings.
%, which will require the use of sample splitting. 
To estimate $\eta_\lambda(P,Q)$ based on samples $\mathbb{X}^{1:N}$ and $\mathbb{Y}^{1:M}$, \citep{SpectralTwoSampleTest} proposed to split the samples and use part of the samples to estimate the mean elements and the rest to estimate the covariance operator. Formally, 
we split the samples $\left(X_i\right)_{i=1}^N$ into $\left(X_i\right)_{i=1}^{N-s}$ and $\left(X_i^1\right)_{i=1}^s=\left(X_i\right)_{i=N-s+1}^N$, and $\left(Y_j\right)_{j=1}^M$ into $\left(Y_j\right)_{j=1}^{M-s}$ and $(Y_j^1)_{j=1}^s=\left(Y_j\right)_{j=M-s+1}^M$. Define $n=N-s$ and $m=M-s$. Define $Z_i=\alpha_i X_i^1+\left(1-\alpha_i\right) Y_i^1$, for $1 \leq i \leq s$, where $\left(\alpha_i\right)_{i=1}^s \stackrel{i . i . d}{\sim} \operatorname{Bernoulli}(1 / 2)$. It can be shown that $\left(Z_i\right)_{i=1}^s \overset{i . i . d}{\sim} R=\frac{P+Q}{2}$. A U-statistic estimator of $\Sigma_{PQ}$ is then constructed based on $\mathbb{Z}^{1:s} \coloneq \left(Z_i\right)_{i=1}^s$, given by
\[
\hat{\Sigma}_{P Q} \coloneq \frac{1}{2 s(s-1)} \sum_{i \neq j}^s\left(K\left(\cdot, Z_i\right)-K\left(\cdot, Z_j\right)\right) \otimes_{\mathcal{H}}\left(K\left(\cdot, Z_i\right)-K\left(\cdot, Z_j\right)\right).
\]
Using this estimate of $\Sigma_{PQ}$, the sample-based estimate of the spectral regularized discrepancy is constructed, referred to as the spectral regularized test statistic $\hat{\eta}_{\lambda}$, and is given by 
\begin{equation}\label{Spectral Regularized Kernel Test statistic}
\hat{\eta}_{\lambda} \coloneq \frac{1}{n(n-1)} \frac{1}{m(m-1)} \sum_{1 \leq i \neq j \leq n} \sum_{1 \leq i^{\prime} \neq j^{\prime} \leq m} u\left(X_i, X_j, Y_{i^{\prime}}, Y_{j^{\prime}}\right),
\end{equation}
where \[u\left(X_i, X_j, Y_{i^{\prime}}, Y_{j^{\prime}}\right) \coloneq \left\langle g_{\lambda}^{1/2}\left(\hat{\Sigma}_{P Q}\right) \left(K\left(\cdot, X_i\right)-K\left(\cdot, Y_{i^{\prime}}\right)\right), g_{\lambda}^{1/2}\left(\hat{\Sigma}_{P Q}\right) \left(K\left(\cdot, X_j\right)-K\left(\cdot, Y_{j^{\prime}}\right)\right)\right\rangle_{\mathcal{H}}.\] Conditional on $\left(Z_i\right)_{i=1}^s$, \eqref{Spectral Regularized Kernel Test statistic} is a two-sample U-statistic and is therefore a natural estimator of the spectral regularized discrepancy $\eta_{\lambda}$. \citep{SpectralTwoSampleTest} proposed a permutation based test involving $\hat{\eta}_\lambda$ and showed it to be minimax optimal w.r.t.~$\mathcal{P}$. Concretely, if $\lambda_i\asymp i^{-\beta},\,\beta>1$, i.e., polynomial decay of the eigenvalues of $\mathcal{T}$, then the permutation test enjoys the minimax separation radius of $(N+M)^{-\frac{4\theta\beta}{4\theta\beta+1}}$ w.r.t.~$\mathcal{P}$ if $\theta>\frac{1}{2}-\frac{1}{4\beta}$ and if $\lambda_i\asymp e^{-i}$, i.e., exponential decay of eigenvalues of $\mathcal{T}$, then the permutation test has a minimax separation rate of $\sqrt{\log(N+M)}(N+M)^{-1}$ w.r.t.~$\mathcal{P}$ if $\theta>\frac{1}{2}$. However, computationally, the test scales as $O(s^3+n^2+m^2+ms^2+ns^2)$, which means for $s=O(N+M)$, the test scales cubically in the number of samples, unlike the MMD test, which scales quadratically in the sample size. In the following, we propose a random feature approximation to the spectral regularized MMD test and demonstrate an improved computational behavior for $s=O(N+M)$ while retaining the minimax optimality.

%As for the MMD test, we reject the null hypothesis of equality of the two distributions $P$ and $Q$ if $\hat{\eta}_{\lambda}$ exceeds a certain critical threshold. 

\section{Approximate spectral regularized MMD test}\label{sec:approx}

It is shown in \cite{SpectralTwoSampleTest} that, unlike the vanilla MMD test, the spectral regularized MMD test is minimax optimal with respect to the class of alternatives $\mathcal{P}$ defined in \eqref{Class of alternatives}. However, as we later show in detail, the computational complexity of the spectral regularized MMD test statistic is cubic in the number of samples in the worst-case scenario, as compared to the quadratic complexity of the classical MMD test. In the present work, we develop a computationally efficient approximation to the U-statistic estimator $\hat{\eta}_{\lambda}$ of the spectral regularized discrepancy $\eta_{\lambda}$, which we will denote as $\hat{\eta}_{\lambda,l}$.

\subsection{Construction of the test statistic and the test}\label{subsec: Construction of the test statistic and the test (Approx)}

To construct the approximate spectral regularized test statistic, we first consider an approximation to the kernel $K$ based on random sampling of features from the spectral distribution $\Xi$ (inverse Fourier transform) corresponding to the kernel $K$. If the kernel $K$ associated with the RKHS $\mathcal{H}$ is of the form 
$$
K(x,y) = \int_{\Theta} \varphi(x, \theta) \varphi(y, \theta) d \Xi(\theta),
$$
where $\varphi$ is a feature function and $\Xi$ is a probability distribution on $\R^{d}$ (referred to as the spectral distribution or inverse Fourier transform of $K$), the kernel can be approximated via Monte Carlo sampling. Specifically, given $l$ random samples $\theta^{1:l} =(\theta_{i})_{i=1}^{l}$ drawn from $\Xi$, an approximate kernel $K_{l}$ is constructed as:
\begin{equation*}
\label{Random feature approximation of kernel repeated}
K_{l}(x, y)=\frac{1}{l} \sum_{i=1}^l \varphi\left(x, \theta_i\right) \varphi\left(y, \theta_i\right)=\sum_{i=1}^l \varphi_{i}(x)\varphi_{i}(y)=\left\langle\Phi_l(x), \Phi_l(y)\right\rangle_2,
\end{equation*}
 where
$\varphi_{i}(\cdot) = \frac{1}{\sqrt{l}} \varphi(\cdot,\theta_{i})$  for $i=1,2,\dots,l$, 
and the random feature map is given by:
\[\Phi_l(x)=\frac{1}{\sqrt{l}}\left(\varphi\left(x, \theta_1\right), \ldots, \varphi\left(x, \theta_l\right)\right)^{\top} = \left(\varphi_{1}(x), \ldots, \varphi_{l}(x)\right)^{\top}.
\]

Analogous to the spectral regularized discrepancy $\eta_{\lambda}$ defined with respect to the kernel $K$, one can define the approximate spectral regularized discrepancy $\eta_{\lambda,l}$ with respect to the approximate kernel $K_{l}$ as 
\begin{equation*}\label{Approximate spectral regularized discrepancy}
    \eta_{\lambda,l}=\left\|g_{\lambda}^{1/2}(\Sigma_{P Q,l})\left(\mu_{Q,l}-\mu_{P,l}\right)\right\|_{\mathcal{H}_{l}}^{2}
\end{equation*}
and our primary goal is to construct a test of equality of $P$ and $Q$ based on a statistical estimator of $\eta_{\lambda,l}$, which is $\hat{\eta}_{\lambda,l}$. Thus, $\hat{\eta}_{\lambda,l}$ can be viewed as a RFF-based approximation to $\hat{\eta}_{\lambda}$ as well as a statistical estimator of $\eta_{\lambda,l}$.

Let $\Sigma_{PQ,l}$ be the (centered) covariance operator corresponding to the approximate kernel $K_{l}$, given by
\[\Sigma_{PQ,l}= \frac{1}{2} \int_{\mathcal{X} \times \mathcal{X}}(K_{l}(\cdot, x)-K_{l}(\cdot, y)) \otimes_{\mathcal{H}}(K_{l}(\cdot, x)-K_{l}(\cdot, y)) d R(x) d R(y),\] 
where $R=\frac{P+Q}{2}$. Analogous to $\hat{\Sigma}_{PQ}$, we can construct a U-statistic estimate of $\Sigma_{PQ,l}$ based on $\mathbb{Z}^{1:s}$, given by
\[
\hat{\Sigma}_{P Q,l} \coloneq \frac{1}{2 s(s-1)} \sum_{i \neq j}^s\left(K_{l}\left(\cdot, Z_i\right)-K_{l}\left(\cdot, Z_j\right)\right) \otimes_{\mathcal{H}_{l}}\left(K_{l}\left(\cdot, Z_i\right)-K_{l}\left(\cdot, Z_j\right)\right).
\]

Finally, using the above estimate of $\Sigma_{PQ,l}$, we can construct an RFF-based approximation to the spectral regularized test statistic $\hat{\eta}_{\lambda}$. We denote this approximate spectral regularized MMD test statistic as $\hat{\eta}_{\lambda,l}$, which is defined as
\begin{equation}\label{Approximate Kernel Test statistic}
\hat{\eta}_{\lambda,l}:=\frac{1}{n(n-1)} \frac{1}{m(m-1)} \sum_{1 \leq i \neq j \leq n} \sum_{1 \leq i^{\prime} \neq j^{\prime} \leq m} t\left(X_i, X_j, Y_{i^{\prime}}, Y_{j^{\prime}}\right),
\end{equation}
where
\[\begin{aligned}&t\left(X_i, X_j, Y_{i^{\prime}}, Y_{j^{\prime}}\right)\\
\coloneqq&\left\langle g_{\lambda}^{1/2}(\hat{\Sigma}_{P Q,l}) \left(K_{l}\left(\cdot, X_i\right)-K_{l}\left(\cdot, Y_{i^{\prime}}\right)\right), g_{\lambda}^{1/2}(\hat{\Sigma}_{P Q,l}) \left(K_{l}\left(\cdot, X_j\right)-K_{l}\left(\cdot, Y_{j^{\prime}}\right)\right)\right\rangle_{\mathcal{H}_{l}}.
\end{aligned}\]

Conditioned on $\left(Z_i\right)_{i=1}^s$ and $\theta^{1:l}$, \eqref{Approximate Kernel Test statistic} is a two-sample U-statistic and is therefore a natural estimator of the approximate spectral regularized discrepancy $\eta_{\lambda,l}$. As with the MMD and spectral regularized MMD tests, we reject the null hypothesis of equality of $P$ and $Q$ if $\hat{\eta}_{\lambda,l}$ exceeds a certain critical threshold. In the following, we first propose a test based on $\eta_{\lambda,l}$ and demonstrate its minimax optimality in Section~\ref{subection: Oracle test}. Since this test's threshold depends on the unknown distributions and regularization parameter, in Sections~\ref{subsection: Permutation test}, \ref{subsection: Adaptation over regularization parameter}, and \ref{subsection: Adaptation over kernel and regularization parameter}, we present a practical version of the test whose threshold is completely data-dependent, and demonstrate its minimax optimality w.r.t.~$\mathcal{P}$. The computational considerations and computational-statistical trade-off discussion are provided in Section~\ref{Section: Computational complexity of test statistics}.

\subsection{Assumptions}
% (might get shifted to the Appendix)

Before proceeding further, we explicitly state the assumptions regarding the underlying data domain $\mathcal{X}$, the reproducing kernel $K$, its associated RKHS $\mathcal{H}$, its associated functional operators, and the spectral function $g_{\lambda}$. Most of the assumptions are the same as in \cite{SpectralTwoSampleTest}, with some minor changes. These assumptions ensure the existence and well-definedness of functional representations of the distributions $P$, $Q$, and $R = \frac{P+Q}{2}$ together with their associated functional operators. The assumptions regarding the specific form of the kernel are, in fact, quite general (they are satisfied by popularly used kernels like Gaussian and Laplace kernels), and they allow the use of the RFF machinery to develop a computationally efficient statistical test.  Further, the assumptions regarding the spectral function $g_{\lambda}$ ensure that $\hat{\eta}_{\lambda,l} \asymp \norm{u}_{L^{2}(R)}^{2}$ under mild conditions on the likelihood ratio deviation $u:=\frac{dP}{dR}-1$, the regularization parameter $\lambda$ and the number of random (spectral) features $l$ (see Proposition~\ref{Proposition: Upper and lower bound of eta}). 

We make the following assumptions regarding the underlying data domain $\mathcal{X}$, the reproducing kernel $K$, and its associated RKHS $\mathcal{H}$.\vspace{1mm}

$\boldsymbol{(\RFFAssumptionone)}$  $(\mathcal{X},\mathcal{B})$ is a 
second countable (i.e., completely separable) space endowed with Borel $\sigma$-algebra $\mathcal{B}$. $(\mathcal{H},K)$ is an RKHS of real-valued functions on $\mathcal{X}$ with a continuous reproducing kernel $K$ such that $\sup_{x} K(x,x) \leq \kappa.$ \vspace{1mm}

$\boldsymbol{(\RFFAssumptionfour)}$ The reproducing kernel $K$ corresponding to the Hilbert space $\mathcal{H}$ is of the form $$K(x, y)=\int_{\Theta} \varphi(x, \theta) \varphi(y, \theta) d \Xi(\theta)=\langle\varphi(x, \cdot), \varphi(y, \cdot)\rangle_{L^2(\Xi)},$$
where $\varphi: \mathcal{X} \times \Theta \rightarrow \mathbb{R}$ is continuous, $\sup _{\theta \in \Theta, x \in \mathcal{X}}|\varphi(x, \theta)| \leq \sqrt{\kappa}$ and $\Xi$ is (without loss of generality) a probability measure on a second countable space $(\Theta, \mathcal{A})$ endowed with Borel $\sigma$-algebra $\mathcal{A}$.

\begin{remark}
(i) $\boldsymbol{(\RFFAssumptionone)}$ ensures the separability of $L^{r}(\mathcal{X},\mu)$ for any $\sigma$-finite measure defined on $\mathcal{B}$ and Bochner-measurability of $K(\cdot,x)$. This leads to the well-definedness of the mean embeddings $\mu_{P}$ and $\mu_{Q}$. Let $\Sigma_{PQ}$ be the (centered) covariance operator corresponding to kernel $K$ and distribution $R=\frac{P+Q}{2}$. Under $\boldsymbol{(\RFFAssumptionone)}$, $\Sigma_{PQ}$ is a self-adjoint positive trace-class operator and therefore, using Theorem VI.16 and VI.17 of \citet{reed1980methods}, $\Sigma_{PQ}$ has a spectral representation given by \begin{equation}\label{Spectral representation of SigmaPQ}
\Sigma_{PQ} = \sum_{i \in I} \lambda_{i} \phi_{i} \otimes_{\mathcal{H}} \phi_{i},
\end{equation}
where $(\lambda_{i})_{i \in I} \subset \mathbb{R}^{+}$ and $(\phi_{i})_{i \in I}$ are respectively the eigenvalues and orthonormal system of eigenfunctions of $\Sigma_{PQ}$ spanning $\overline{\operatorname{Ran}(\Sigma_{PQ})}$, with the eigenvalues and eigenfunctions being indexed in the decreasing order of magnitude of the eigenvalues. We assume in this paper that the index set $I$ is countable, which implies that $\lim_{i \to \infty}\lambda_{i} = 0$.\vspace{1mm}\\
(ii) $\boldsymbol{(\RFFAssumptionone)}$ and $\boldsymbol{(\RFFAssumptionfour)}$ are essential for the validity of the results in \cite{ApproximateKernelPCARandomFeaturesStergeSriperumbudur}, which we utilize for providing theoretical guarantees concerning the RFF approximation error. 
\end{remark}

The following are the assumptions on the regularizer $g_\lambda$ (corresponding to Assumptions $A_1$, $A_2$, and $A_4$ in \citet{SpectralTwoSampleTest}), which are common in the inverse problem literature.
\begin{itemize}
    \item[] $\boldsymbol{(\SpectralAssumptionone)}$ $\sup _{x \in \Gamma}\left|x g_\lambda(x)\right| \leq C_1$; \label{Assumption Spectral Regularizer A1}
    \item[] $\boldsymbol{(\SpectralAssumptiontwo)}$ $\sup _{x \in \Gamma}\left|\lambda g_\lambda(x)\right| \leq C_2$; \label{Assumption Spectral Regularizer A2}
    \item[] $\boldsymbol{(\SpectralAssumptionthree)}$ $ \inf _{x \in \Gamma} g_\lambda(x)(x+\lambda) \geq C_4$, \label{Assumption Spectral Regularizer A4}
\end{itemize}
where $\Gamma:=[0, \kappa]$ and $C_1$, $C_2$ and $C_4$ are finite positive constants (all independent of $\lambda$). We also assume for the convenience of reporting our results that the sample sizes $N$ and $M$ satisfy the following general condition:
\begin{itemize}
    \item[] $\boldsymbol{(\Samplesizeassumption)}$ $M\leq N \leq DM$ for some constant $D\geq1$.
\end{itemize}

\subsection{Oracle test}
\label{subection: Oracle test}
We now proceed to provide a level-$\alpha$ test for testing $H_{0}:P=Q$ against $H_{1}:P \neq Q$ for a fixed choice of the regularization parameter $\lambda>0$ satisfying certain mild conditions.

\begin{theorem}[RFF-based Oracle Test]\label{Type-I error bound of Oracle Test in terms of N2} Suppose $\boldsymbol{(\RFFAssumptionone)}$--$\boldsymbol{(\SpectralAssumptiontwo)}$ hold. Let $n,m \geq 2$ and $\hat{\eta}_{\lambda,l}$ be the random feature approximation of the test statistic as defined in \eqref{Approximate Kernel Test statistic}. Given any $\alpha>0$, 
suppose
$$ l\ge %L\left(\frac{\alpha}{2},\frac{1}{2}\right)= 
\max\left\{2\log\frac{2}{1-\sqrt{1-\frac{\alpha}{4}}},\frac{128\kappa^{2}\log\frac{2}{1-\sqrt{1-\frac{\alpha}{4}}}}{\left\|\Sigma_{P Q}\right\|_{\mathcal{L}^{\infty}(\mathcal{H})}^{2}}\right\}$$
% Suppose the number of random features l and the regularization parameter $\lambda$ satisfy $l\geq L(\frac{\alpha}{2},\frac{1}{2})$ 
and $$\operatorname{max}\left\{\frac{140 \kappa}{s} \log \frac{32 \kappa s}{1-\sqrt{1-\frac{\alpha}{4}}},\frac{86 \kappa}{l}\log \frac{64 \kappa l}{\alpha} \right\} \leq \lambda \leq \frac{1}{2}\left\|\Sigma_{P Q}\right\|_{\mathcal{L}^{\infty}(\mathcal{H})}.$$ Then the level-$\alpha$ critical region for testing $H_0 :P=Q$ vs. $H_1:P\neq Q$ is given by $\left\{\hat{\eta}_{\lambda,l} \geq \gamma\right\}$, i.e., 
\[P_{H_{0}}\left\{\hat{\eta}_{\lambda,l} \geq \gamma\right\} \leq \alpha,\]
where $\gamma:=\frac{4 \sqrt{3}(C_{1}+C_{2}) A(\lambda,\alpha,l)}{\sqrt{\alpha}}\left(\frac{1}{n}+\frac{1}{m}\right)$
and $A(\lambda,\alpha,l):=\frac{4\sqrt{2\kappa \mathcal{N}_{1}(\lambda) \log \frac{8}{\alpha}}}{\sqrt{\lambda l}} + \frac{16\kappa \log \frac{8}{\alpha}}{\lambda l}  + 2\sqrt{2}\mathcal{N}_{2}(\lambda)$. %\textcolor{red}{the notation $L(\alpha/2,1/2)$ is bit wierd. Revisit later. No assumptions in this result? we have to mention the assumptions in the statement.}
 % Then, for any $\alpha>0$ and $\operatorname{max}\left\{\frac{140 \kappa}{s} \log \frac{32 \kappa s}{1-\sqrt{1-\frac{\alpha}{4}}},\frac{86 \kappa}{l}\log \frac{64 \kappa l}{\alpha} \right\} \leq \lambda \leq \frac{1}{2}\left\|\Sigma_{P Q}\right\|_{\mathcal{L}^{\infty}(\mathcal{H})}$ and $l\geq \max\left\{2\log\frac{2}{1-\sqrt{1-\frac{\alpha}{4}}},\frac{128\kappa^{2}\log\frac{2}{1-\sqrt{1-\frac{\alpha}{4}}}}{\left\|\Sigma_{P Q}\right\|_{\mathcal{L}^{\infty}(\mathcal{H})}^{2}}\right\}= L(\frac{\alpha}{2},\frac{1}{2})$, we have that
% \[
% P_{H_{0}}\left\{\hat{\eta}_{\lambda,l} \geq \gamma_{1}\right\} \leq \alpha,
% \]
% where $\gamma_{1}=\frac{4 \sqrt{3}(C_{1}+C_{2}) A(\lambda,\alpha,l)}{\sqrt{\alpha}}\left(\frac{1}{n}+\frac{1}{m}\right)$
% and $A(\lambda,\alpha,l)=\frac{4\sqrt{2\kappa N_{1}(\lambda) \log \frac{8}{\alpha}}}{\sqrt{\lambda l}} + \frac{16\kappa \log \frac{8}{\alpha}}{\lambda l}  + 2\sqrt{2}N_{2}(\lambda)$.
\end{theorem}

Based on Theorem \ref{Type-I error bound of Oracle Test in terms of N2} (proved in Section~\ref{subsec:thm1}), we obtain a valid two-sample test of equality of $P$ and $Q$ that rejects the null hypothesis when $\hat{\eta}_{\lambda,l} \geq \gamma$ and $l$ is larger than $L(\frac{\alpha}{2},\frac{1}{2})$. However, the critical threshold $\gamma$ depends on the knowledge of $P$ and $Q$ through the quantities $\mathcal{N}_{1}(\lambda)=\operatorname{Tr}(\Sigma_{P Q, \lambda}^{-1 / 2} \Sigma_{P Q} \Sigma_{P Q, \lambda}^{-1 / 2})$ and $\mathcal{N}_{2}(\lambda)=\|\Sigma_{P Q, \lambda}^{-1 / 2} \Sigma_{P Q} \Sigma_{P Q, \lambda}^{-1 / 2}\| _{\mathcal{L}^{2}(\mathcal{H})}$, which characterize the degrees of freedom of $\mathcal{H}$. Further, the lower bound on the number of random Fourier features $l$ depends not only on the level $\alpha$ but also on the knowledge of $P$ and $Q$ through $\Sigma_{PQ}$. Since $P$ and $Q$ are unknown and we only have access to samples $\mathbb{X}^{1:N} =\left(X_i\right)_{i=1}^N \overset{i.i.d}{\sim}P$ and $\mathbb{Y}^{1:M} =\left(Y_j\right)_{j=1}^M  \overset{i.i.d}{\sim}Q$, this test cannot be implemented in practice. Hence, we refer to this test as the \emph{RFF-based Oracle Test}. We develop completely data-driven two-sample tests in the later sections of this paper based on a permutation testing approach that yields a critical region that utilizes only the sample information and therefore can be implemented in practice. Further, we will show that these latter tests match the statistical efficiency of the RFF-based Oracle Test.

The following result (proved in Section~\ref{subsec:thm2}) provides Type-II error analysis of the RFF-based Oracle Test by characterizing the behavior of the separation boundary $\Delta_{N,M}$ between $P$ and $Q$, the number of random Fourier features $l$, and the regularization parameter $\lambda>0$ that ensures that the test achieves a given Type-II error bound.% $4\delta>0$.

\begin{theorem}[Separation boundary of RFF-based Oracle Test]
\label{Power analysis of Oracle test}
Suppose %that Assumptions 
$\boldsymbol{(\RFFAssumptionone)}$--$\boldsymbol{(\SpectralAssumptionthree)}$, and $\boldsymbol{(\Samplesizeassumption)}$ hold. Let the number of samples $s$ split from $\mathbb{X}^{1:N}$ and $\mathbb{Y}^{1:M}$ for estimating $\Sigma_{PQ,l}$ be chosen as $s=d_{1}N=d_{2}M$ for $0\leq d_{1} \leq d_{2} \leq 1$, while the number of samples $n=N-s$ and $m=M-s$ for estimating $\mu_{P,l}$ and $\mu_{Q,l}$ respectively satisfy $n,m\geq 2$. For any $0\leq \alpha \leq 1$, consider the level-$\alpha$ test proposed in Theorem \ref{Type-I error bound of Oracle Test in terms of N2} for testing $H_{0}:P = Q$ against $H_{1}: P \neq Q$.
% with $\hat{\eta}_{\lambda,l}$ as the test statistic and  $\gamma=\frac{ 4\sqrt{3}(C_{1}+C_{2})}{\sqrt{\alpha}}\left(\frac{1}{n}+\frac{1}{m}\right) \left[\frac{4\sqrt{2\kappa \log \frac{8}{\alpha}}}{\sqrt{\lambda l}} + \frac{16\kappa \log \frac{8}{\alpha}}{\lambda l}  + 2\sqrt{2}N_{2}(\lambda)\right]$ as the critical threshold. 
Further, assume that  $\underset{\theta>0}{\sup}\underset{(P,Q) \in \mathcal{P}}{\sup} \Vert\mathcal{T}_{PQ}^{-\theta}u\Vert_{L^{2}(R)} < \infty$ and the regularization parameter $\lambda$ satisfies $\lambda = d_{\theta}\Delta_{N,M}^{\frac{1}{2\theta}} \leq \frac{1}{2}\|\Sigma_{PQ}\|_{\mathcal{L}^{\infty}(\mathcal{H})}$ for some constant $d_{\theta}>0$ that depends on $\theta$. Then, for any $0<\delta\leq 1$, provided $(N+M) \geq \frac{32\kappa d_{2}}{\delta}$, $\mathcal{N}_{2}(d_{\theta}\Delta_{N,M}^{\frac{1}{2\theta}}) \geq 1$, and
$\Delta_{N,M}$ and number of random features $l$ satisfy the following conditions:
\begin{enumerate}

\item $\Delta_{N,M}^{\frac{1}{2\theta}} \gtrsim d_{\theta}^{-1}\max\left\{\frac{\log(N+M)}{(N+M)},\frac{\log (\frac{2}{\delta})}{l},\frac{1}{l}\log \frac{32 \kappa l}{\delta}\right\}$ \label{Theorem Power Analysis Oracle Test Condition 1}

%\item $\Delta_{N,M}^{\frac{1}{2\theta}} \gtrsim d_{\theta}^{-1}\times\frac{\log (\frac{2}{\delta})}{l}$ \label{Theorem Power Analysis Oracle Test Condition 2}

\item $\frac{\Delta_{N,M}^{\frac{1}{2\theta}} }{\mathcal{N}_{1}\left(d_{\theta}\Delta_{N,M}^{1/2\theta}\right)} \gtrsim d_{\theta}^{-1}\frac{\log (\frac{2}{\delta})}{l}$ \label{Theorem Power Analysis Oracle Test Condition 3}

% \item $ \Delta_{N,M}^{\frac{1}{2\theta}} \gtrsim d_{\theta}^{-1}\times\frac{1}{l}\log \frac{32 \kappa l}{\delta}$ \label{Theorem Power Analysis Oracle Test Condition 4}

%% $\frac{\Delta_{N,M}}{N_{2}\left(d_{\theta}\Delta_{N,M}^{\frac{1}{2\theta}}\right)} \geq \frac{r_{1}}{\sqrt{\alpha}(N+M)}$, $\Delta_{N,M}^{\frac{1+2\theta}{2\theta}} \geq \frac{r_{2,\theta}\log(\frac{8}{\alpha})}{l\sqrt{\alpha}(N+M)}$ and $\frac{\Delta_{N,M}^{\frac{1+4\theta}{2\theta}}}{N_{1}\left(d_{\theta}\Delta_{N,M}^{\frac{1}{2\theta}}\right)} \geq \frac{r_{3,\theta} \log(\frac{8}{\alpha})}{l\alpha(N+M)^{2}}$.

%% \item $\Delta_{N,M}^{\frac{1+4\theta}{4\theta}} \gtrsim \left(\sqrt{\frac{\log(\frac{8}{\alpha})}{\alpha}} + \sqrt{\frac{\log(\frac{4}{\delta})}{\delta}}\right)\frac{1}{\sqrt{l}(N+M)}$ \label{Theorem Power Analysis Oracle Test Condition 5}

\item $\Delta_{N,M}^{\frac{1+4\theta}{4\theta}} \gtrsim \max\left\{d_{\theta}^{-1/2},d_{\theta}^{-2}\right\}\left(\sqrt{\frac{\log(\frac{8}{\alpha})}{\alpha}} + \frac{\sqrt{\log(\frac{4}{\delta})}}{\delta^{2}}\right) \frac{1}{\sqrt{l}(N+M)}$ \label{Theorem Power Analysis Oracle Test Condition 5}

% \item $\Delta_{N,M}^{\frac{1+2\theta}{2\theta}} \gtrsim  \left(\frac{\log(\frac{8}{\alpha})}{\sqrt{\alpha}} + \frac{\log(\frac{4}{\delta})}{\sqrt{\delta}}\right)\frac{1}{l(N+M)}$ \label{Theorem Power Analysis Oracle Test Condition 6}

\item $\Delta_{N,M}^{\frac{1+2\theta}{2\theta}} \gtrsim d_{\theta}^{-1}\left(\frac{\log(\frac{8}{\alpha})}{\sqrt{\alpha}} + \frac{\log(\frac{4}{\delta})}{\delta^{2}}\right)\frac{1}{l(N+M)}$ \label{Theorem Power Analysis Oracle Test Condition 6}

% \item $\frac{\Delta_{N,M}}{N_{2}\left(d_{\theta}\Delta_{N,M}^{\frac{1}{2\theta}}\right)} \gtrsim \frac{\alpha^{-1/2} + \delta^{-1}}{(N+M)}$ \label{Theorem Power Analysis Oracle Test Condition 7}

\item $\frac{\Delta_{N,M}}{\mathcal{N}_{2}\left(d_{\theta}\Delta_{N,M}^{\frac{1}{2\theta}}\right)} \gtrsim \frac{\alpha^{-1/2} + \delta^{-2}}{(N+M)}$ \label{Theorem Power Analysis Oracle Test Condition 7}

% \item $\Delta_{N,M}^{\frac{3+4\theta}{4\theta}}\gtrsim \frac{\sqrt{\log(\frac{4}{\delta})}}{\delta^{2}}\times \frac{1}{\sqrt{l}(N+M)^{2}}$ \label{Theorem Power Analysis Oracle Test Condition 8}

\item $\Delta_{N,M}^{\frac{3+4\theta}{4\theta}}\gtrsim d_{\theta}^{-\frac{3}{2}}  \frac{\sqrt{\log(\frac{4}{\delta})}}{\delta}\frac{1}{\sqrt{l}(N+M)^{2}}$ \label{Theorem Power Analysis Oracle Test Condition 8}

% \item $\Delta_{N,M}^{\frac{1+\theta}{\theta}} \gtrsim \frac{\log(\frac{4}{\delta})}{\delta^{2}}\times \frac{1}{l(N+M)^{2}}$ \label{Theorem Power Analysis Oracle Test Condition 9}

\item $\Delta_{N,M}^{\frac{1+\theta}{\theta}} \gtrsim d_{\theta}^{-2} \frac{\log(\frac{4}{\delta})}{\delta}\frac{1}{l(N+M)^{2}}$ \label{Theorem Power Analysis Oracle Test Condition 9}

% \item $\frac{\Delta_{N,M}^{\frac{1+2\theta}{2\theta}}}{N_{2}(d_{\theta}\Delta_{N,M}^{\frac{1}{2\theta}})} \gtrsim \frac{1}{\delta^2(N+M)^2}$ \label{Theorem Power Analysis Oracle Test Condition 10}

\item $\frac{\Delta_{N,M}^{\frac{1+2\theta}{2\theta}}}{\mathcal{N}_{2}(d_{\theta}\Delta_{N,M}^{\frac{1}{2\theta}})} \gtrsim d_{\theta}^{-1} \frac{1}{\delta(N+M)^2}$ \label{Theorem Power Analysis Oracle Test Condition 10}

\item $l\geq \max\left\{2\log\frac{2}{1-\sqrt{1-\delta}},\frac{128\kappa^{2}\log\frac{2}{1-\sqrt{1-\delta}}}{\left\|\Sigma_{P Q}\right\|_{\mathcal{L}^{\infty}(\mathcal{H})}^{2}}\right\}$\label{Theorem Power Analysis Oracle Test Condition 11},

\end{enumerate}
we have that the power of the test for the class of $\Delta_{NM}$-separated alternatives $\mathcal{P}_{\theta,\Delta_{NM}}$ as defined in \eqref{Class of alternatives} is at least $1-4\delta$, i.e.,
\[
\underset{(P,Q) \in \mathcal{P}_{\theta,\Delta_{NM}}}{\inf} P_{H_{1}}\left(\hat{\eta}_{\lambda,l} \geq \gamma\right) \geq 1-4\delta.
\]
\end{theorem}

It is natural to compare the RFF-based Oracle test to the \say{exact} Oracle Test based on $\hat{\eta}_{\lambda}$ as proposed in Theorem 4.2 of \cite{SpectralTwoSampleTest}, since the \say{exact} Oracle Test satisfies minimax optimality with respect to $\mathcal{P}$ (see Theorem 3.1, Theorem 3.2, Corollary 3.3 and Corollary 3.4 of \cite{SpectralTwoSampleTest}). However, Theorem \ref{Power analysis of Oracle test} is stated in a general form, which obscures the statistical performance of the RFF-based Oracle Test and the conditions under which it matches the statistical efficiency of the \say{exact} Oracle Test (achieving minimax optimality). To that end, we delineate, in particular, the performance of the RFF-based Oracle test by characterizing the behavior of $\Delta_{N,M}$, $l$ and $\lambda>0$ under polynomial and exponential decay of the eigenvalues of the covariance operator $\Sigma_{PQ}$ and develop Corollaries \ref{Power Analysis of Oracle test polynomial decay} and \ref{Power Analysis of Oracle test exponential decay}, which are proved in Sections~\ref{subsec:cor3} and \ref{subsec:cor4}, respectively.

\begin{corollary}[RFF Oracle Test under polynomial decay]
\label{Power Analysis of Oracle test polynomial decay} Suppose the eigenvalues $(\lambda_{i})_{i \in I}$ of $\Sigma_{PQ}$ decay at a polynomial rate, i.e., $\lambda_{i} \asymp i^{-\beta}$ for $\beta > 1$. Then, for any 
% Suppose that Assumptions $(\RFFAssumptionone)$,$(\RFFAssumptionfour)$,$(\SpectralAssumptionone)$,$(\SpectralAssumptiontwo)$, $(\SpectralAssumptionfour)$ and $(\Samplesizeassumption)$ hold true. Let the number of samples $s$ split from $\mathbb{X}^{1:N}$ and $\mathbb{Y}^{1:M}$ for estimating $\Sigma_{PQ,l}$ be chosen to be $O(N+M)$, specifically $s=d_{1}N=d_{2}M$ for $0\leq d_{1} \leq d_{2} \leq 1$, while the number of samples $n=N-s$ and $m=M-s$ for estimating $\mu_{P,l}$ and $\mu_{Q,l}$ respectively satisfy $n,m\geq 2$. For any $0\leq \alpha \leq 1$, consider the level-$\alpha$ test of $H_{0}:P = Q$ against $H_{1}: P \neq Q$ proposed in Theorem \ref{Type-I error bound of Oracle Test in terms of N2} with $\hat{\eta}_{\lambda,l}$ as the test statistic and  $\gamma=\frac{ 4\sqrt{3}(C_{1}+C_{2})}{\sqrt{\alpha}}\left(\frac{1}{n}+\frac{1}{m}\right) \left[\frac{4\sqrt{2\kappa \log \frac{8}{\alpha}}}{\sqrt{\lambda l}} + \frac{16\kappa \log \frac{8}{\alpha}}{\lambda l}  + 2\sqrt{2}N_{2}(\lambda)\right]$ as the critical threshold. Further, assume that  $\underset{\theta>0}{\sup}\underset{(P,Q) \in \mathcal{P}}{\sup} \norm{\mathcal{T}_{PQ}^{-\theta}u}_{L^{2}(R)} < \infty$ and the regularization parameter $\lambda$ satisfies $\lambda = d_{\theta}\Delta_{N,M}^{\frac{1}{2\theta}} \leq \frac{1}{2}\|\Sigma_{PQ}\|_{\mathcal{L}^{\infty}(\mathcal{H})}$ for some constant $d_{\theta}>0$ that depends on $\theta$.
%Then, for any 
$0<\delta\leq 1$, 
there exists constants $c(\alpha,\delta,\theta,\beta)>0$ and $k(\alpha,\delta,\theta,\beta) \in \mathbb{N}$ such that, for any choice of $N+M \geq k(\alpha,\delta,\theta,\beta)$, 
%the power of the level-$\alpha$ test of $H_{0}:P = Q$ against $H_{1}: P \neq Q$ proposed in Theorem \ref{Type-I error bound of Oracle Test in terms of N2} with $\hat{\eta}_{\lambda,l}$ as the test statistic and \\ $\gamma=\frac{ 4\sqrt{3}(C_{1}+C_{2})}{\sqrt{\alpha}}\left(\frac{1}{n}+\frac{1}{m}\right) \left[\frac{4\sqrt{2\kappa \log \frac{8}{\alpha}}}{\sqrt{\lambda l}} + \frac{16\kappa \log \frac{8}{\alpha}}{\lambda l}  + 2\sqrt{2}N_{2}(\lambda)\right]$ as the critical threshold is at least $1-4\delta$ over the class of $\Delta_{NM}$-separated alternatives $\mathcal{P}_{\theta,\Delta_{NM}}$ as defined in \eqref{Class of alternatives} i.e. 
\[
\underset{(P,Q) \in \mathcal{P}_{\theta,\Delta_{NM}}}{\inf} P_{H_{1}}\left(\hat{\eta}_{\lambda,l} \geq \gamma\right) \geq 1-4\delta,
\]
when 
%the separation boundary achieves the following rate of decay
\[
\Delta_{N,M} = 
\begin{cases}
c(\alpha,\delta,\theta,\beta)\left(N+M\right)^{\frac{-4\beta\theta}{1+4\beta\theta}},& \theta > \frac{1}{2} - \frac{1}{4\beta}\\
c(\alpha,\delta,\theta,\beta)\left[\frac{\log(N+M)}{N+M}\right]^{2\theta}, &  \theta \leq \frac{1}{2} - \frac{1}{4\beta}
\end{cases},
\]
provided the number of random features  
is large enough, i.e., 
$$l\gtrsim \begin{cases}
(N+M)^{\frac{2(\beta+1)}{1+4\theta\beta}},&  \theta > \frac{1}{2} - \frac{1}{4\beta}\\
\left[\frac{N+M}{\log (N+M)}\right]^{\frac{\beta+1}{\beta}},& \theta \leq \frac{1}{2} - \frac{1}{4\beta}
\end{cases}.
$$
\end{corollary}

\begin{corollary}[RFF Oracle Test under exponential decay]
\label{Power Analysis of Oracle test exponential decay} Suppose the eigenvalues $(\lambda_{i})_{i \in I}$ of $\Sigma_{PQ}$ decay at an exponential rate i.e. $\lambda_{i} \asymp e^{-\tau i}$ for $\tau >0$. 
% Suppose that Assumptions $(\RFFAssumptionone)$,$(\RFFAssumptionfour)$,$(\SpectralAssumptionone)$,$(\SpectralAssumptiontwo)$, $(\SpectralAssumptionfour)$ and $(\Samplesizeassumption)$ hold true. Let the number of samples $s$ split from $\mathbb{X}^{1:N}$ and $\mathbb{Y}^{1:M}$ for estimating $\Sigma_{PQ,l}$ be chosen to be $O(N+M)$, specifically $s=d_{1}N=d_{2}M$ for $0\leq d_{1} \leq d_{2} \leq 1$, while the number of samples $n=N-s$ and $m=M-s$ for estimating $\mu_{P,l}$ and $\mu_{Q,l}$ respectively satisfy $n,m\geq 2$. For any $0\leq \alpha \leq 1$, consider the level-$\alpha$ test of $H_{0}:P = Q$ against $H_{1}: P \neq Q$ proposed in Theorem \ref{Type-I error bound of Oracle Test in terms of N2} with $\hat{\eta}_{\lambda,l}$ as the test statistic and  $\gamma=\frac{ 4\sqrt{3}(C_{1}+C_{2})}{\sqrt{\alpha}}\left(\frac{1}{n}+\frac{1}{m}\right) \left[\frac{4\sqrt{2\kappa \log \frac{8}{\alpha}}}{\sqrt{\lambda l}} + \frac{16\kappa \log \frac{8}{\alpha}}{\lambda l}  + 2\sqrt{2}N_{2}(\lambda)\right]$ as the critical threshold. Further, assume that  $\underset{\theta>0}{\sup}\underset{(P,Q) \in \mathcal{P}}{\sup} \norm{\mathcal{T}_{PQ}^{-\theta}u}_{L^{2}(R)} < \infty$.
Then, for any $0<\delta\leq 1$, there exists constants 
% provided the number of random features $l$ is large enough such that $ \frac{\left[\log(\frac{4}{\delta})+\log(\frac{8}{\alpha})\right]\log l}{l} \lesssim \lambda = d_{\theta}\Delta_{N,M}^{\frac{1}{2\theta}} \leq \min\left\{\frac{1}{2}\|\Sigma_{PQ}\|_{\mathcal{L}^{\infty}(\mathcal{H})},e^{-1}\right\}$ and \\ $l \geq 2\bar{C}\left[\log(\frac{4}{\delta})+\log(\frac{8}{\alpha})\right]e \left[\log\left[\log(\frac{4}{\delta})+\log(\frac{8}{\alpha})\right]+1\right]$ where $\lambda$ is the regularization parameter, $d_{\theta}>0$ is a constant that depends only on $\theta$ and $\bar{C} \geq 1$ is a universal constant , there exists constants 
$c(\alpha,\delta,\theta)>0$ and $k(\alpha,\delta,\theta) \in \mathbb{N}$ such that, for any choice of $N+M \geq k(\alpha,\delta,\theta)$, %the power of the level-$\alpha$ test of $H_{0}:P = Q$ against $H_{1}: P \neq Q$ proposed in Theorem \ref{Type-I error bound of Oracle Test in terms of N2} with $\hat{\eta}_{\lambda,l}$ as the test statistic and  $\gamma=\frac{ 4\sqrt{3}(C_{1}+C_{2})}{\sqrt{\alpha}}\left(\frac{1}{n}+\frac{1}{m}\right) \left[\frac{4\sqrt{2\kappa \log \frac{8}{\alpha}}}{\sqrt{\lambda l}} + \frac{16\kappa \log \frac{8}{\alpha}}{\lambda l}  + 2\sqrt{2}N_{2}(\lambda)\right]$ as the critical threshold is at least $1-4\delta$ over the class of $\Delta_{NM}$-separated alternatives $\mathcal{P}_{\theta,\Delta_{NM}}$ as defined in \eqref{Class of alternatives} i.e. 
\[
\underset{(P,Q) \in \mathcal{P}_{\theta,\Delta_{NM}}}{\inf} P_{H_{1}}\left(\hat{\eta}_{\lambda,l} \geq \gamma\right) \geq 1-4\delta,
\]
when 
%the separation boundary achieves the following rate of decay 
\[
\Delta_{N,M} = 
\begin{cases}
c(\alpha,\delta,\theta)\frac{\sqrt{\log(N+M)}}{N+M},& \theta > \frac{1}{2}\\
c(\alpha,\delta,\theta)\left[\frac{\log(N+M)}{N+M}\right]^{2\theta},& \theta \leq \frac{1}{2}
\end{cases},
\]
provided the number of random features 
is large enough, i.e.,
$$l\gtrsim
\begin{cases}
\left(N+M\right)^{\frac{1}{2\theta}}\log(N+M)^{1-\frac{1}{4\theta}}, &  \theta > \frac{1}{2}\\
N+M,& \theta \leq \frac{1}{2}
\end{cases}.
$$
% subject to the conditions\\ $l \gtrsim \left(N+M\right)^{\frac{1}{2\theta}}\log(N+M)^{1-\frac{1}{4\theta}}$ in the former and $l \gtrsim N+M$ in the latter case. 

% When the smoothness index is lower i.e. $\theta \leq \frac{1}{2}$, the number of random features required to achieve minimax separation rate for $\Delta_{NM}$ is larger than $N+M$, which makes the RFF-based hypothesis test infeasible in practice under this smoothness condition for $u = \frac{dP}{dR}-1$.

\end{corollary}

\subsection{Permutation test}
\label{subsection: Permutation test}

As discussed in Section \ref{subection: Oracle test}, the Oracle test cannot be practically implemented due to its dependence on the unknown distributions $P$ and $Q$. In this section, we propose a statistical test that matches the statistical performance of the RFF-based Oracle test (and therefore, matches the statistical performance of the \say{exact} Oracle test proposed in Theorem 4.2 of \cite{SpectralTwoSampleTest}) without requiring the knowledge of $P$ and $Q$ other than the information contained in the samples $\mathbb{X}^{1:N} =\left(X_i\right)_{i=1}^N \overset{i.i.d}{\sim}P$ and $\mathbb{Y}^{1:M} =\left(Y_j\right)_{j=1}^M  \overset{i.i.d}{\sim}Q$. This statistical test of hypothesis is based on permutation testing (\cite{lehmann,kim2022minimax}) using the same test statistic as the RFF-based Oracle test, i.e., $\hat{\eta}_{\lambda,l}$, but the critical region is now fully data-driven.

We begin by describing the concept underlying the permutation test. The RFF-based test statistic defined in \eqref{Approximate Kernel Test statistic}, just like its exact kernel-based counterpart in \cite{SpectralTwoSampleTest}, involves sample splitting resulting in three sets of independent samples,  $(X_i)_{i=1}^n \stackrel{i.i.d.}{\sim} P$, $(Y_j)_{j=1}^m \stackrel{i.i.d.}{\sim} Q$, $(Z_i)_{i=1}^s \stackrel{i.i.d.}{\sim} \frac{P+Q}{2}$. Define $(U_i)_{i=1}^n:=(X_i)_{i=1}^n$, and $(U_{n+j})_{j=1}^m:=(Y_j)_{j=1}^m$. Let $\Pi_{n+m}$ be the set of all possible permutations of $\{1,\ldots,n+m\}$ with $\pi \in \Pi_{n+m}$ be a randomly selected permutation from the $D$ possible permutations, where $D :=|\Pi_{n+m}|= (n+m)!$. Define $(X^{\pi}_i)_{i=1}^n := (U_{\pi(i)})_{i=1}^n$ and $(Y^{\pi}_j)_{j=1}^m := (U_{\pi(n+j)})_{j=1}^m$. Let $\hat{\eta}_{\lambda,l}^{\pi}:=\hat{\eta}_{\lambda,l}(X^{\pi},Y^{\pi},Z)$ be the statistic based on the permuted samples and random features and $\hat{\eta}_{\lambda}^{\pi}:=\hat{\eta}_{\lambda}(X^{\pi},Y^{\pi},Z)$ be the statistic based on the permuted samples using the exact kernel. Let $(\pi^i)_{i=1}^B$ be $B$ randomly selected permutations from $\Pi_{n+m}$. For simplicity, define $\hat{\eta}^i_{\lambda,l}:= \hat{\eta}_{\lambda,l}^{\pi^i}$ to represent the statistic based on permuted samples w.r.t.~the random permutation $\pi^i$. Similarly, define $\hat{\eta}^i_{\lambda}:= \hat{\eta}_{\lambda}^{\pi^i}$. Given the samples $(X_i)_{i=1}^n$, $(Y_j)_{j=1}^m$ and $(Z_i)_{i=1}^s$, define
\begin{equation*}
\label{All Permutation CDF for RFF based test}
F_{\lambda,l}(x)\coloneq \frac{1}{D}\sum_{\pi \in \Pi_{n+m}}\mathbf{1}(\hat{\eta}^{\pi}_{\lambda,l} \leq x)
\end{equation*}
to be the permutation distribution function for $\hat{\eta}_{\lambda,l}$. Similarly, define  \begin{equation*}
\label{All Permutation CDF for exact test}
F_{\lambda}(x)\coloneq \frac{1}{D}\sum_{\pi \in \Pi_{n+m}}\mathbf{1}(\hat{\eta}^{\pi}_{\lambda} \leq x)
\end{equation*}
to be the permutation distribution function for $\hat{\eta}_{\lambda}$. Define 
\begin{equation}
\label{All Permutation quantile for RFF based test}
q_{1-\alpha}^{\lambda,l}:= \inf\{q \in \mathbb{R}: F_{\lambda,l}(q) \geq 1-\alpha\}
\end{equation}
and 
\begin{equation*}
\label{All Permutation quantile for exact test}
q_{1-\alpha}^{\lambda}:= \inf\{q \in \mathbb{R}: F_{\lambda}(q) \geq 1-\alpha\}.
\end{equation*}
Furthermore, we define the empirical permutation distribution functions for $\hat{\eta}_{\lambda,l}$ and $\hat{\eta}_{\lambda}$ based on $B$ random permutations as
\begin{equation}
\label{B Permutation CDF for RFF based test}
\hat{F}^{B}_{\lambda,l}(x):= \frac{1}{B}\sum_{i=1}^{B}\mathbf{1}(\hat{\eta}^{i}_{\lambda,l} \leq x),
\end{equation}
and 
\begin{equation*}
\label{B Permutation CDF for exact test}
\hat{F}^{B}_{\lambda}(x):= \frac{1}{B}\sum_{i=1}^{B}\mathbf{1}(\hat{\eta}^{i}_{\lambda} \leq x).
\end{equation*}
Further, define 
\begin{equation}
\label{B Permutation quantile for RFF based test}
\hat{q}_{1-\alpha}^{B,\lambda,l}\coloneq\inf\{q \in \mathbb{R}: \hat{F}^B_{\lambda,l}(q) \geq 1-\alpha\}
\end{equation}
and 
\begin{equation*}
\label{B Permutation quantile for exact test}\hat{q}_{1-\alpha}^{B,\lambda}\coloneq \inf\{q \in \mathbb{R}: \hat{F}^B_{\lambda}(q) \geq 1-\alpha\}.\end{equation*}

We now proceed to provide a level-$\alpha$ permutation test for testing $H_{0}:P=Q$ against $H_{1}:P \neq Q$ for a fixed choice of the regularization parameter $\lambda>0$ satisfying certain mild conditions. We refer to this test as the RFF-based Permutation Test. The following result is proved in Section~\ref{subsec:thm5}.

\begin{theorem}[RFF-based Permutation Test]
\label{Type I-error of RFF-based permutation test} 
% Let $0<\alpha\leq 1$ and $w,\tilde{w}$ be such that $0 < \tilde{w} < w  < \frac{1}{2}$. 
% Define the RFF-based test statistic $\hat{\eta}_{\lambda,l}$ and the $(1-w\alpha)$-th quantile corresponding to the empirical distribution function for the RFF-based test statistic using $B$ randomly selected permutations of $(U_{i})_{i=1}^{n+m}$ as defined in Section \ref{subsection: Permutation test}. Provided the number of permutations 
Suppose $\boldsymbol{(\RFFAssumptionone)}$--$\boldsymbol{(\SpectralAssumptiontwo)}$ hold. Let $B \geq \frac{1}{2\tilde{w}^{2}\alpha^{2}}\log\frac{2}{\alpha(1-w-\tilde{w})}$, where $0<\alpha\leq 1$, and $0 < \tilde{w} < w  < \frac{1}{2}$. Then,
% the level-$\alpha$ critical region for testing $H_0 :P=Q$ vs $H_1:P\neq Q$ is given by $\left\{\hat{\eta}_{\lambda,l} \geq \hat{q}_{1-w\alpha}^{B,\lambda,l}\right\}$ i.e, 
\[P_{H_{0}}\left\{\hat{\eta}_{\lambda,l} \geq \hat{q}_{1-w\alpha}^{B,\lambda,l}\right\} \leq \alpha\,,\]
i.e., $\left\{\hat{\eta}_{\lambda,l} \geq \hat{q}_{1-w\alpha}^{B,\lambda,l}\right\}$ is the level-$\alpha$ critical region for testing $H_0 :P=Q$ vs. $H_1:P\neq Q$.
\end{theorem}
The following result (proved in Section~\ref{subsec:thm6}) characterizes the separation boundary $\Delta_{N,M}$ of the RFF-based Permutation test, for a given Type-II error, and obtains sufficient conditions on the number of random Fourier features $\l$.
% Next, we perform a power analysis (Type-II error analysis) of the RFF-based Permutation Test by characterizing the behaviour of the separation boundary $\Delta_{N,M}$ between $P$ and $Q$, the number of random Fourier features $l$ and the regularization parameter $\lambda>0$ that ensures that the test achieves a given Type-II error bound $7\delta>0$.

\begin{theorem}[Separation boundary of RFF-based Permutation Test]
\label{Power Analysis for RFF based test based on permutations} Suppose %that Assumptions 
$\boldsymbol{(\RFFAssumptionone)}$--$\boldsymbol{(\SpectralAssumptionthree)}$, and $\boldsymbol{(\Samplesizeassumption)}$ hold. Let the number of samples $s$ split from $\mathbb{X}^{1:N}$ and $\mathbb{Y}^{1:M}$ for estimating $\Sigma_{PQ,l}$ be chosen as $s=d_{1}N=d_{2}M$ for $0\leq d_{1} \leq d_{2} \leq 1$, while the number of samples $n=N-s$ and $m=M-s$ for estimating $\mu_{P,l}$ and $\mu_{Q,l}$ respectively satisfy $n,m\geq 2$. For any $ \alpha$,$w$ and $\tilde{w}$ such that $0\leq \alpha \leq 1$, $0 < \tilde{w} < w  < \frac{1}{2}$ and $0\leq (w-\tilde{w})\alpha < e^{-1}$, consider the level-$\alpha$ test of $H_{0}:P = Q$ against $H_{1}: P \neq Q$ proposed in Theorem \ref{Type I-error of RFF-based permutation test} with $\hat{\eta}_{\lambda,l}$ as the test statistic and $\hat{q}_{1-w\alpha}^{B,\lambda,l}$ as the critical threshold. 
% with $\hat{\eta}_{\lambda,l}$ as the test statistic and  $\gamma=\frac{ 4\sqrt{3}(C_{1}+C_{2})}{\sqrt{\alpha}}\left(\frac{1}{n}+\frac{1}{m}\right) \left[\frac{4\sqrt{2\kappa \log \frac{8}{\alpha}}}{\sqrt{\lambda l}} + \frac{16\kappa \log \frac{8}{\alpha}}{\lambda l}  + 2\sqrt{2}N_{2}(\lambda)\right]$ as the critical threshold. 
Further, assume that  $\underset{\theta>0}{\sup}\underset{(P,Q) \in \mathcal{P}}{\sup} \Vert\mathcal{T}_{PQ}^{-\theta}u\Vert_{L^{2}(R)} < \infty$ and the regularization parameter $\lambda$ satisfies $\lambda = d_{\theta}\Delta_{N,M}^{\frac{1}{2\theta}} \leq \frac{1}{2}\|\Sigma_{PQ}\|_{\mathcal{L}^{\infty}(\mathcal{H})}$ for some constant $d_{\theta}>0$ that depends on $\theta$. 
%Then, for any $0<\delta\leq 1$, provided $(N+M) \geq \frac{32\kappa d_{2}}{\delta}$, $\mathcal{N}_{2}(d_{\theta}\Delta_{N,M}^{\frac{1}{2\theta}}) \geq 1$, and $\Delta_{N,M}$ and number of random features $l$ satisfy the following conditions:
Define $\tilde{\alpha}=(w-\tilde{w})\alpha$ and let $C^{*}$
be an absolute positive constant as defined in Lemma \ref{Lemma 16}. Then, for any $0<\delta\leq 1$, provided $(N+M) \geq \max\left\{\frac{32\kappa d_{2}}{\delta},\frac{2C^{*}\log\frac{2}{\tilde{\alpha}}}{(1-d_{2})\sqrt{\delta}}\right\}$, $\mathcal{N}_{2}(d_{\theta}\Delta_{N,M}^{\frac{1}{2\theta}}) \geq 1$, $B \geq \frac{\log\frac{2}{\min\left\{{\delta},\alpha(1-w-\tilde{w})\right\}}}{2\tilde{w}^{2}\alpha^{2}}$ and
$\Delta_{N,M}$ and number of random features $l$ satisfy the following conditions:
\begin{enumerate}
\item $\Delta_{N,M}^{\frac{1}{2\theta}} \gtrsim d_{\theta}^{-1}\max\left\{ \frac{\log(N+M)}{(N+M)},\frac{\log \frac{2}{\delta}}{l},\frac{1}{l}\log \frac{32 \kappa l}{\delta}\right\}$ \label{Theorem Power Analysis RFF based permutation test Condition 1}
%\item $\Delta_{N,M}^{\frac{1}{2\theta}} \gtrsim d_{\theta}^{-1}\times\frac{\log (\frac{2}{\delta})}{l}$ \label{Theorem Power Analysis RFF based permutation test Condition 2}
\item $\frac{\Delta_{N,M}^{\frac{1}{2\theta}} }{\mathcal{N}_{1}\left(d_{\theta}\Delta_{N,M}^{1/2\theta}\right)} \gtrsim d_{\theta}^{-1}\frac{\log (\frac{2}{\delta})}{l}$ \label{Theorem Power Analysis RFF based permutation test Condition 3}
% \item $ \Delta_{N,M}^{\frac{1}{2\theta}} \gtrsim d_{\theta}^{-1}\times\frac{1}{l}\log \frac{32 \kappa l}{\delta}$ \label{Theorem Power Analysis RFF based permutation test Condition 4}
% \item $\Delta_{N,M}^{\frac{1+4\theta}{4\theta}} \gtrsim \log(\frac{1}{\tilde{\alpha}})\times\sqrt{\frac{\log(\frac{4}{\delta})}{\delta}}\times\frac{1}{\sqrt{l}(N+M)}$ \label{Theorem Power Analysis RFF based permutation test Condition 5}
\item$\Delta_{N,M}^{\frac{1+4\theta}{4\theta}} \gtrsim d_{\theta}^{-1/2}\left[\log(\frac{1}{\tilde{\alpha}})\right]^{4}\frac{\sqrt{\log(\frac{4}{\delta})}}{\delta^{2}\sqrt{l}(N+M)}$\label{Theorem Power Analysis RFF based permutation test Condition 5}
% \item $\Delta_{N,M}^{\frac{1+2\theta}{2\theta}} \gtrsim \log(\frac{1}{\tilde{\alpha}}) \times \frac{\log(\frac{4}{\delta})}{\sqrt{\delta}}\times\frac{1}{l(N+M)}$\label{Theorem Power Analysis RFF based permutation test Condition 6}
\item$\Delta_{N,M}^{\frac{1+2\theta}{2\theta}} \gtrsim d_{\theta}^{-1}\left[\log(\frac{1}{\tilde{\alpha}})\right]^{4} \frac{\log(\frac{4}{\delta})}{\delta^{2}l(N+M)}$\label{Theorem Power Analysis RFF based permutation test Condition 6}
% \item $\frac{\Delta_{N,M}}{N_{2}\left(d_{\theta}\Delta_{N,M}^{\frac{1}{2\theta}}\right)} \gtrsim \frac{\log(\frac{1}{\tilde{\alpha}})}{\sqrt{\delta}}\times\frac{1}{(N+M)}$\label{Theorem Power Analysis RFF based permutation test Condition 7}
\item$\frac{\Delta_{N,M}}{\mathcal{N}_{2}\left(d_{\theta}\Delta_{N,M}^{\frac{1}{2\theta}}\right)} \gtrsim \frac{\left[\log(\frac{1}{\tilde{\alpha}})\right]^{4}}{\delta^{2}(N+M)}$\label{Theorem Power Analysis RFF based permutation test Condition 7}
% \item $\Delta_{N,M}^{\frac{3+4\theta}{4\theta}}\gtrsim  \frac{\sqrt{\log(\frac{4}{\delta})}}{\delta^{2}}\times \left(\log(\frac{1}{\tilde{\alpha}})\right)^{2} \times\frac{1}{\sqrt{l}(N+M)^{2}}$\label{Theorem Power Analysis RFF based permutation test Condition 8}
\item$\Delta_{N,M}^{\frac{3+4\theta}{4\theta}}\gtrsim d_{\theta}^{-\frac{3}{2}} \frac{\sqrt{\log(\frac{4}{\delta})}}{\delta}\left(\log(\frac{1}{\tilde{\alpha}})\right)^{2} \frac{1}{\sqrt{l}(N+M)^{2}}$\label{Theorem Power Analysis RFF based permutation test Condition 8}
% \item $\Delta_{N,M}^{\frac{1+\theta}{\theta}} \gtrsim \frac{\log(\frac{4}{\delta})}{\delta^{2}}\times \left(\log(\frac{1}{\tilde{\alpha}})\right)^{2}\times\frac{1}{l(N+M)^{2}}$\label{Theorem Power Analysis RFF based permutation test Condition 9}
\item 
$\Delta_{N,M}^{\frac{1+\theta}{\theta}} \gtrsim d_{\theta}^{-2}\frac{\log(\frac{4}{\delta})}{\delta}\left(\log(\frac{1}{\tilde{\alpha}})\right)^{2}\frac{1}{l(N+M)^{2}}$\label{Theorem Power Analysis RFF based permutation test Condition 9}
% \item $\frac{\Delta_{N,M}^{\frac{1+2\theta}{2\theta}}}{N_{2}(d_{\theta}\Delta_{N,M}^{\frac{1}{2\theta}})} \gtrsim \frac{\left(\log(\frac{1}{\tilde{\alpha}})\right)^{2}}{\delta^2(N+M)^2}$  \label{Theorem Power Analysis RFF based permutation test Condition 10}
\item 
$\frac{\Delta_{N,M}^{\frac{1+2\theta}{2\theta}}}{\mathcal{N}_{2}(d_{\theta}\Delta_{N,M}^{\frac{1}{2\theta}})} \gtrsim d_{\theta}^{-1}\frac{\left(\log(\frac{1}{\tilde{\alpha}})\right)^{2}}{\delta(N+M)^2}$ \label{Theorem Power Analysis RFF based permutation test Condition 10}
\item $l\geq \max\left\{2\log\frac{2}{1-\sqrt{1-\delta}},\frac{128\kappa^{2}\log\frac{2}{1-\sqrt{1-\delta}}}{\left\|\Sigma_{P Q}\right\|_{\mathcal{L}^{\infty}(\mathcal{H})}^{2}}\right\}$\label{Theorem Power Analysis RFF based permutation test Condition 11},
\end{enumerate}
we have that the power of the test for the class of $\Delta_{NM}$-separated alternatives $\mathcal{P}_{\theta,\Delta_{NM}}$ as defined in \eqref{Class of alternatives} is at least $1-7\delta$, i.e.,
\[
\underset{(P,Q) \in \mathcal{P}_{\theta,\Delta_{NM}}}{\inf} P_{H_{1}}\left(\hat{\eta}_{\lambda,l} \geq \hat{q}_{1-w\alpha}^{B,\lambda,l}\right) \geq 1-7\delta.
\]
% For any $0<\alpha\leq e^{-1}$ and $\delta>0$ (from $E_{4}$)  

% $(n+m) \geq \frac{2C^{*}\log(\frac{2}{\alpha})}{\sqrt{\delta}}$

% Further, let $m \leq n \leq D m$ for some constant $D \geq 1$.

% $(N+M) \geq \max\left\{\frac{32\kappa d_{1}}{\delta},\frac{2C^{*}\log(\frac{2}{\alpha})}{(1-d_{2})\sqrt{\delta}}\right\}$

% the number of randomly selected permutations $B \geq \frac{\log(\frac{2}{\delta})}{2\tilde{\alpha}^{2}}$
\end{theorem}

Theorem~\ref{Power Analysis for RFF based test based on permutations} is stated under very general conditions, making it difficult to appreciate its significance. To better understand its significance, in the following, we derive Corollaries \ref{Power Analysis of RFF based permutation test polynomial decay} and \ref{Power Analysis of RFF based permutation test exponential decay} (proved in Sections~\ref{subsec:cor7} and \ref{subsec:cor8}, respectively) to characterize the behavior of $\Delta_{N,M}$, $l$ and $\lambda>0$ under polynomial and exponential decay of the eigenvalues of the covariance operator $\Sigma_{PQ}$.
\begin{corollary}[RFF Permutation Test under polynomial decay]
\label{Power Analysis of RFF based permutation test polynomial decay} Suppose the eigenvalues $(\lambda_{i})_{i \in I}$ of $\Sigma_{PQ}$ decay at a polynomial rate, i.e., $\lambda_{i} \asymp i^{-\beta}$ for $\beta > 1$. Let $0\leq \alpha \leq 1$, $0 < \tilde{w} < w  < \frac{1}{2}$, $0\leq (w-\tilde{w})\alpha < e^{-1}$, and $\tilde{\alpha}:=(w-\tilde{w})\alpha$. 
% Suppose that Assumptions $(\RFFAssumptionone)$,$(\RFFAssumptionfour)$,$(\SpectralAssumptionone)$,$(\SpectralAssumptiontwo)$, $(\SpectralAssumptionfour)$ and $(\Samplesizeassumption)$ hold true. Let the number of samples $s$ split from $\mathbb{X}^{1:N}$ and $\mathbb{Y}^{1:M}$ for estimating $\Sigma_{PQ,l}$ be chosen to be $O(N+M)$, specifically $s=d_{1}N=d_{2}M$ for $0\leq d_{1} \leq d_{2} \leq 1$, while the number of samples $n=N-s$ and $m=M-s$ for estimating $\mu_{P,l}$ and $\mu_{Q,l}$ respectively satisfy $n,m\geq 2$. For any $ \alpha$,$w$ and $\tilde{w}$ such that $0\leq \alpha \leq 1$, $0 < \tilde{w} < w  < \frac{1}{2}$ and $0\leq (w-\tilde{w})\alpha < e^{-1}$, consider the level-$\alpha$ test of $H_{0}:P = Q$ against $H_{1}: P \neq Q$ proposed in Theorem \ref{Type I-error of RFF-based permutation test} with $\hat{\eta}_{\lambda,l}$ as the test statistic and $\hat{q}_{1-w\alpha}^{B,\lambda,l}$ as the critical threshold. Further, assume that  $\underset{\theta>0}{\sup}\underset{(P,Q) \in \mathcal{P}}{\sup} \norm{\mathcal{T}_{PQ}^{-\theta}u}_{L^{2}(R)} < \infty$ and the regularization parameter $\lambda$ satisfies $\lambda = d_{\theta}\Delta_{N,M}^{\frac{1}{2\theta}} \leq \frac{1}{2}\|\Sigma_{PQ}\|_{\mathcal{L}^{\infty}(\mathcal{H})}$ for some constant $d_{\theta}>0$ that depends on $\theta$. Finally, define $\tilde{\alpha}=(w-\tilde{w})\alpha$.
Then, for any $0<\delta\leq 1$, there exists constants $c(\tilde{\alpha},\delta,\theta,\beta)>0$ and $k(\tilde{\alpha},\delta,\theta,\beta) \in \mathbb{N}$ such that, for any choice of $N+M \geq k(\tilde{\alpha},\delta,\theta,\beta)$, 
% the power of the level-$\alpha$ test of $H_{0}:P = Q$ against $H_{1}: P \neq Q$ proposed in Theorem \ref{Type I-error of RFF-based permutation test} with $\hat{\eta}_{\lambda,l}$ as the test statistic and $\hat{q}_{1-w\alpha}^{B,\lambda,l}$ as the critical threshold is at least $1-7\delta$ over the class of $\Delta_{NM}$-separated alternatives $\mathcal{P}_{\theta,\Delta_{NM}}$ as defined in \eqref{Class of alternatives} i.e. 
\[
\underset{(P,Q) \in \mathcal{P}_{\theta,\Delta_{NM}}}{\inf} P_{H_{1}}\left(\hat{\eta}_{\lambda,l} \geq \hat{q}_{1-w\alpha}^{B,\lambda,l}\right) \geq 1-7\delta,
\]
when 
%the separation boundary achieves the following rate of decay
\[
\Delta_{N,M} = 
\begin{cases}
c(\tilde{\alpha},\delta,\theta,\beta)\left(N+M\right)^{\frac{-4\beta\theta}{1+4\beta\theta}},& \theta > \frac{1}{2} - \frac{1}{4\beta}\\
c(\tilde{\alpha},\delta,\theta,\beta)\left[\frac{\log(N+M)}{N+M}\right]^{2\theta},&\theta \leq \frac{1}{2} - \frac{1}{4\beta}
\end{cases},
\]
provided $B \geq \frac{1}{2\tilde{w}^{2}\alpha^{2}}\log\frac{2}{\min\left\{{\delta},\alpha(1-w-\tilde{w})\right\}}$ and
% the number of randomly selected permutations $B$ and the number of random features $l$ 
% are large enough, subject to the conditions  and 
$$l\gtrsim \begin{cases}
(N+M)^{\frac{2(\beta+1)}{1+4\theta\beta}},&  \theta > \frac{1}{2} - \frac{1}{4\beta}\\
\left[\frac{N+M}{\log (N+M)}\right]^{\frac{\beta+1}{\beta}},& \theta \leq \frac{1}{2} - \frac{1}{4\beta}
\end{cases}.
$$
% $$l\gtrsim
% \begin{cases}
% \left(N+M\right)^{\frac{2(\beta+1)}{1+4\beta\theta}}, &  \left\{\theta > \frac{1}{2} - \frac{1}{4\beta}\cap 1 < \beta \leq 2\right\}\cup\left\{ 
% \frac{1}{2} - \frac{1}{4\beta}< \theta \leq \frac{2\beta^2 + \beta + 2}{4\beta(\beta - 2)}\cap \beta>2\right\} \\
% \left(N+M\right)^{\frac{2(2\beta-1)}{1+2\beta+4\beta\theta}}, &   \left\{\theta > \frac{2\beta^2 + \beta + 2}{4\beta(\beta - 2)}\cap \beta>2\right\}\\
% \left[\frac{N+M}{\log(N+M)}\right]^{\frac{\beta+1}{\beta}},&  \left\{\theta \leq \frac{1}{2} - \frac{1}{4\beta}\cap 1 < \beta < \frac{1+\sqrt{3}}{2}\right\}\cup\left\{
% \frac{2\beta^2 - 2\beta - 1}{4\beta(\beta + 1)} \leq \theta \leq \frac{1}{2} - \frac{1}{4\beta}\cap\beta \geq \frac{1+\sqrt{3}}{2}\right\}\\
% \left(N+M\right)^{\frac{2\beta-1}{1+4\beta\theta}}, & \left\{\theta < \frac{2\beta^2 - 2\beta - 1}{4\beta(\beta + 1)}\cap \beta \geq \frac{1+\sqrt{3}}{2}\right\} .
% \end{cases}
% $$
\end{corollary}

\begin{corollary}[RFF Permutation Test under exponential decay]
\label{Power Analysis of RFF based permutation test exponential decay} Suppose the eigenvalues $(\lambda_{i})_{i \in I}$ of $\Sigma_{PQ}$ decay at an exponential rate, i.e., $\lambda_{i} \asymp e^{-\tau i}$ for $\tau >0$. Let $0\leq \alpha \leq 1$, $0 < \tilde{w} < w  < \frac{1}{2}$, $0\leq (w-\tilde{w})\alpha < e^{-1}$, and $\tilde{\alpha}:=(w-\tilde{w})\alpha$. Then, for any $0<\delta\leq 1$, there exists constants $c(\tilde{\alpha},\delta,\theta)>0$ and $k(\tilde{\alpha},\delta,\theta) \in \mathbb{N}$ such that, for any choice of $N+M \geq k(\tilde{\alpha},\delta,\theta)$, 
\[
\underset{(P,Q) \in \mathcal{P}_{\theta,\Delta_{NM}}}{\inf} P_{H_{1}}\left(\hat{\eta}_{\lambda,l} \geq \hat{q}_{1-w\alpha}^{B,\lambda,l}\right) \geq 1-7\delta,
\]
when
\[
\Delta_{N,M} = 
\begin{cases}
c(\tilde{\alpha},\delta,\theta)\frac{\sqrt{\log(N+M)}}{N+M},&\theta > \frac{1}{2}\\
c(\tilde{\alpha},\delta,\theta)\left[\frac{\log(N+M)}{N+M}\right]^{2\theta},&\theta \leq \frac{1}{2}
\end{cases},
\]
provided 
% the number of randomly selected permutations $B$ and the number of random features $l$ 
% are large enough, subject to the conditions 
$B \geq \frac{1}{2\tilde{w}^{2}\alpha^{2}}\log\frac{2}{\min\left\{{\delta},\alpha(1-w-\tilde{w})\right\}}$ and $$l \gtrsim \begin{cases}\left(N+M\right)^{\frac{1}{2\theta}}\log(N+M)^{1-\frac{1}{4\theta}},& \theta > \frac{1}{2}\\ N+M,& \theta \leq \frac{1}{2}
\end{cases}.$$
\end{corollary}

The above results demonstrate that even though the critical region of the RFF-based Permutation test is fully data-driven and does not require the knowledge of $P$ and $Q$, it matches the statistical performance of the RFF-based Oracle test when the number of random Fourier features $l$ is sufficiently large, and moreover the test is minimax optimal with respect to $\mathcal{P}$. However, it still suffers from a drawback. The efficacy of the test is dependent on choosing the regularization parameter $\lambda$ correctly. The \say{optimal} choice of $\lambda$ that yields a minimax separation boundary (as given in Theorem \ref{Power Analysis for RFF based test based on permutations}, and Corollaries \ref{Power Analysis of RFF based permutation test polynomial decay}, \ref{Power Analysis of RFF based permutation test exponential decay}) depends on the unknown smoothness index $\theta$ of the likelihood ratio deviation $u=\frac{dP}{dR}-1$ (and on the eigenvalue decay rate $\beta$ when the eigenvalues of $\Sigma_{PQ}$ decay at a polynomial rate). We would like to point out that the computation of the test statistic $\hat{\eta}_{\lambda,l}^{i}$ for each of the $B$ permutations can be essentially parallelized, so the computational complexity is not adversely affected. We explicitly calculate the computational complexity of the RFF-based Permutation Test and compare it with that of the \say{exact} Permutation Test in Section \ref{Section: Computational complexity of test statistics}.

\subsection{Adaptation for the choice of regularization parameter} 
\label{subsection: Adaptation over regularization parameter} To mitigate the issue of the dependence of the statistical performance of the RFF-based Permutation Test on the optimal choice of the regularization parameter $\lambda$, we develop a union test, i.e., an aggregation of multiple tests for a range of values of $\lambda$ belonging to a completely data-driven finite set $\Lambda$. In the current section, we prove that the aggregation over multiple choices of the regularization parameter preserves the minimax optimality of the test (upto log log factors) for a wide range of values of the unknown smoothness index $\theta$ of the likelihood ratio deviation $u=\frac{dP}{dR}-1$ (and the eigenvalue decay rate $\beta$ when the eigenvalues of $\Sigma_{PQ}$ decay at a polynomial rate).

Denote the optimal choice of $\lambda$ that leads to a minimax optimal RFF-based Oracle Test or Permutation Test (as defined in Theorems~\ref{Type-I error bound of Oracle Test in terms of N2} and \ref{Type I-error of RFF-based permutation test}, respectively) as $\lambda^{*}$. Assume that there exists a positive constant $\lambda_{L}$  and $b\in \mathbb{N}$ such that $\lambda_{L} \leq \lambda^{*} \leq \lambda_{U}$, where $\lambda_{U} = 2^{b}\lambda_{L}$. Let us define $\Lambda \coloneq \left\{\lambda \in \mathbb{R} : \lambda = 2^{i}\lambda_{L}, i = 0,1,\dots,b\right\} = \left\{\lambda_{L},2\lambda_{L},\dots,\lambda_{U}\right\}$ and let $|\Lambda|$ be its cardinality, given by $|\Lambda| = 1 + \log_{2}\frac{\lambda_{U}}{\lambda_{L}} = 1 + b$. Further, define $s^{*} = \sup\left\{\lambda \in \Lambda : \lambda \leq \lambda^{*}\right\}$. Then, clearly, we have that $\frac{\lambda^{*}}{2} \leq s^{*} \leq \lambda^{*}$ and consequently, $s^{*} \asymp \lambda^{*}$.

We now proceed to provide a level-$\alpha$ permutation-based union test for testing $H_{0}:P=Q$ against $H_{1}:P \neq Q$, where the null hypothesis is rejected if and only if at least one of the permutation tests is rejected for some $\lambda \in \Lambda$. We refer to this test as the RFF-based Adaptive Test, and the following result is proved in Section~\ref{subsec:thm9}.

\begin{theorem}[RFF-based Adaptive Test]
\label{Type I-error of RFF-based permutation test with adaptation over lambda} Suppose $\boldsymbol{(\RFFAssumptionone)}$--$\boldsymbol{(\SpectralAssumptiontwo)}$ hold. Let $B \geq \frac{|\Lambda|^{2}}{2\tilde{w}^{2}\alpha^{2}}\log\frac{2|\Lambda|}{\alpha(1-w-\tilde{w})}$, where $0<\alpha\leq 1$, and $0 < \tilde{w} < w  < \frac{1}{2}$. Then, 
% Define the RFF-based test statistic $\hat{\eta}_{\lambda,l}$ and the $(1-w\alpha)$-th quantile corresponding to the empirical distribution function for the RFF-based test statistic using $B$ randomly selected permutations of $(U_{i})_{i=1}^{n+m}$ as defined in Section \ref{subsection: Permutation test}. Further, define the potential choices of the regularization parameter $\Lambda$ and its cardinality $|\Lambda|$ as in Section \ref{subsection: Adaptation over regularization parameter}. Provided the number of permutations $B \geq \frac{|\Lambda|^{2}}{2\tilde{w}^{2}\alpha^{2}}\log(\frac{2|\Lambda|}{\alpha(1-w-\tilde{w})})$, 
the level-$\alpha$ critical region for testing $H_0 :P=Q$ vs. $H_1:P\neq Q$ is given by $\underset{\lambda \in \Lambda}{\bigcup}\left\{\hat{\eta}_{\lambda,l} \geq \hat{q}_{1-\frac{w\alpha}{|\Lambda|}}^{B,\lambda,l}\right\}$, i.e., 
\[P_{H_{0}}\left(\underset{\lambda \in \Lambda}{\bigcup}\left\{\hat{\eta}_{\lambda,l} \geq \hat{q}_{1-\frac{w\alpha}{|\Lambda|}}^{B,\lambda,l}\right\}\right) \leq \alpha\,.\]
\end{theorem} 
The following result (proved in Section~\ref{subsec:thm10}) performs the power analysis of the RFF-based Adaptive Test by characterizing the behaviour of the separation boundary $\Delta_{N,M}$ between $P$ and $Q$ and the number of random Fourier features $l$ under polynomial and exponential decay of the eigenvalues of the covariance operator $\Sigma_{PQ}$.
\begin{theorem}[Separation boundary of RFF-based Adaptive Test]
\label{Power Analysis for RFF based test based on permutations and adaptation} Suppose %that Assumptions 
$\boldsymbol{(\RFFAssumptionone)}$--$\boldsymbol{(\SpectralAssumptionthree)}$, and $\boldsymbol{(\Samplesizeassumption)}$ hold. Let the number of samples $s$ split from $\mathbb{X}^{1:N}$ and $\mathbb{Y}^{1:M}$ for estimating $\Sigma_{PQ,l}$ be chosen as $s=d_{1}N=d_{2}M$ for $0\leq d_{1} \leq d_{2} \leq 1$, while the number of samples $n=N-s$ and $m=M-s$ for estimating $\mu_{P,l}$ and $\mu_{Q,l}$ respectively satisfy $n,m\geq 2$. Let $ \alpha$,$w$ and $\tilde{w}$ be such that $0\leq \alpha \leq 1$, $0 < \tilde{w} < w  < \frac{1}{2}$ and $0\leq (w-\tilde{w})\alpha < e^{-1}$.  
% consider the level-$\alpha$ test of $H_{0}:P = Q$ against $H_{1}: P \neq Q$ proposed in Theorem \ref{Type I-error of RFF-based permutation test} with $\hat{\eta}_{\lambda,l}$ as the test statistic and $\hat{q}_{1-w\alpha}^{B,\lambda,l}$ as the critical threshold. 
% with $\hat{\eta}_{\lambda,l}$ as the test statistic and  $\gamma=\frac{ 4\sqrt{3}(C_{1}+C_{2})}{\sqrt{\alpha}}\left(\frac{1}{n}+\frac{1}{m}\right) \left[\frac{4\sqrt{2\kappa \log \frac{8}{\alpha}}}{\sqrt{\lambda l}} + \frac{16\kappa \log \frac{8}{\alpha}}{\lambda l}  + 2\sqrt{2}N_{2}(\lambda)\right]$ as the critical threshold. 
Further, assume that  $c_1:=\underset{\theta>0}{\sup}\underset{(P,Q) \in \mathcal{P}}{\sup} \Vert\mathcal{T}_{PQ}^{-\theta}u\Vert_{L^{2}(R)} < \infty$. 
Then, for any $0<\delta\le 1$, and any choice of $\theta^{*} \in (0,\frac{1}{4}]$, if $\theta\geq\theta^{*}$ and $B \geq \frac{|\Lambda|^{2}}{2\tilde{w}^{2}\alpha^{2}}\max\left\{\log\frac{2|\Lambda|}{\alpha(1-w-\tilde{w})},\log\frac{2}{\delta}\right\}$,
there exists $\tilde{k}\in\mathbb{N}$ such that for all $N+M \geq \tilde{k}$, such that
the power of the level-$\alpha$ test of $H_{0}:P = Q$ against $H_{1}: P \neq Q$ proposed in Theorem \ref{Type I-error of RFF-based permutation test with adaptation over lambda} is at least $1-7\delta$ over the class of $\Delta_{NM}$-separated alternatives $\mathcal{P}_{\theta,\Delta_{NM}}$ as defined in \eqref{Class of alternatives}, i.e., 
\[
\underset{\theta\geq\theta^{*}}{\inf}
\underset{(P,Q) \in \mathcal{P}_{\theta,\Delta_{NM}}}{\inf} P_{H_{1}}\left(\underset{\lambda \in \Lambda}{\bigcup}\left\{\hat{\eta}_{\lambda,l} \geq \hat{q}_{1-\frac{w\alpha}{|\Lambda|}}^{B,\lambda,l}\right\}\right) \geq 1-7\delta,
\]
provided one of the following scenarios is true:
\begin{itemize}
    \item[(i)](Polynomial decay of eigenvalues) The eigenvalues $(\lambda_{i})_{i \in I}$ of $\Sigma_{PQ}$ decay at a polynomial rate, i.e., $\lambda_{i} \asymp i^{-\beta}$ for $\beta > 1$, $\lambda_{L} = r_{1}\frac{\log(N+M)}{N+M}$, $\lambda_{U}=\min\left\{r_{2},\frac{1}{2}\norm{\Sigma_{PQ}}_{\mathcal{L}^{\infty}(\mathcal{H})})\right\}$ for some constants $r_{1},r_{2}>0$, the separation boundary has decay rate of
    \[
    \Delta_{N,M} = 
    c(\tilde{\alpha},\delta)\max\left\{\left[\frac{\log(N+M)}{N+M}\right]^{2\theta},\left[\frac{\log\log (N+M)}{N+M}\right]^{\frac{4\beta\theta}{1+4\beta\theta}} \right\},
    \]
    and the number of random features satisfies $l \gtrsim \max\left\{(N+M)^{\frac{1}{2\theta^{*}}},(N+M)^{2}\right\}$, where $c(\tilde{\alpha},\delta)>0$ is a constant that depends only on $\tilde{\alpha}$ and $\delta$.
    %$c(\tilde{\alpha},\delta) = \max\left\{16c_{1}^{2},(16c_{1}^{2})^{\frac{5}{3}}\right\} \times \left[\log(\frac{1}{\tilde{\alpha}})\right]^{4} \times \frac{\log(\frac{4}{\delta})}{\delta^{2}}$.
    
    \item[(ii)](Exponential decay of eigenvalues) The eigenvalues $(\lambda_{i})_{i \in I}$ of $\Sigma_{PQ}$ decay at an exponential rate, i.e., $\lambda_{i} \asymp e^{-\tau i}$ for $\tau >0$, $\lambda_{L} = r_{3}\frac{\log(N+M)}{N+M}$, $\lambda_{U}=\min\left\{r_{4},e^{-1},\frac{1}{2}\norm{\Sigma_{PQ}}_{\mathcal{L}^{\infty}(\mathcal{H})}\right\}$ for some constants $r_{3},r_{4}>0$, the separation boundary has decay rate of
    \[
    \Delta_{N,M} = 
    c(\tilde{\alpha},\delta,\theta)\max\left\{\left[\frac{\log(N+M)}{N+M}\right]^{2\theta},\frac{\sqrt{\log(N+M)}\log\log(N+M)}{N+M} \right\},
    \]
    with the number of random features satisfying $l \gtrsim \max\left\{(N+M)^{\frac{1}{2\theta^{*}}}\left[\log(N+M)\right],N+M\right\}$, where $c(\tilde{\alpha},\delta,\theta)>0$ is a constant that depends only on $\tilde{\alpha}$, $\delta$ and $\theta$.
    %$ = \max\left\{(\frac{1}{16c_{1}^{2}})^{2\theta-1},16c_{1}^{2}\right\} \times \left[\max\left\{(\frac{1}{2\theta^{*}})^{\frac{1}{4\theta^{*}}},\sqrt{2}\right\}\right]^{2\theta}\times \frac{\left[\log(\frac{8}{\tilde{\alpha}})+\log(\frac{4}{\delta})\right]^{\frac{4\theta}{\theta^{*}}}}{\delta^{2\theta}}$.
\end{itemize}
    
\end{theorem}

We therefore demonstrate that the full data-driven RFF-based Adaptive Test matches the statistical performance of the RFF-based Oracle Test and the \say{exact} Oracle Test when the number of random Fourier features $l$ is sufficiently large. Further, even without the knowledge of the optimal regularization parameter $\lambda^{*}$, the RFF-based Adaptive Test achieves minimax optimality with respect to $\mathcal{P}$ up to a log log factor over a wide range of $\theta$ ($\theta \geq \frac{1}{2}$).

\subsection{Adaptation for choice of kernel and regularization parameter} 
\label{subsection: Adaptation over kernel and regularization parameter}

Despite adapting over the choice of the regularization parameter $\lambda$, the effectiveness of the RFF-based Adaptive Test proposed in Theorem \ref{Type I-error of RFF-based permutation test with adaptation over lambda} is still determined by the choice of the kernel $K$. Further, the class of alternatives with respect to which the RFF-based Adaptive Test is guaranteed to have high power implicitly depends on the choice of the kernel $K$ through the integral operator $\mathcal{T}$ corresponding to $K$. In order to extend the class of alternatives and further improve the efficacy of the RFF-based Adaptive test, we can perform an additional level of adaptation over the choice of the kernel $K$.

More specifically, we consider a collection of kernels $\mathcal{K}$ and define the class of alternatives to be \begin{equation}\label{Class of alternatives aggregated over kernels}
\widetilde{\mathcal{P}}_{K}\eqqcolon\widetilde{\mathcal{P}}_{\theta, \Delta,K}:=\left\{(P, Q): \frac{d P}{d R}-1 \in \operatorname{Ran}\left(\mathcal{T}_{K}^\theta\right), \underline{\rho}^2(P, Q) \geq \Delta\right\}\,.
\end{equation}
$\mathcal{T}_{K}$'s are defined analogous to $\mathcal{T}$ as the integral operator corresponding to every kernel $K \in \mathcal{K}$. Further, $\hat{q}_{1-\alpha,K}^{B,\lambda,l}$'s are defined as the analogue of the quantiles $\hat{q}_{1-\alpha}^{B,\lambda,l}$ defined in \eqref{B Permutation quantile for RFF based test} corresponding to every kernel $K \in \mathcal{K}$.

The most common scenario involves considering $\mathcal{K}$ as a parametrized family of kernels, all of which have the same functional form. For instance, one can consider $\mathcal{K}$ to be the collection of Gaussian kernels indexed by the bandwidth parameter $h$ which belongs to a collection $W$, i.e., $\mathcal{K} = \left\{K_{h}: K_{h}(x,y) = \exp\left(-\frac{\norm{x-y}_{2}^{2}}{2h}\right), x,y \in \R^d\textrm{ and }h \in W\subset(0,\infty)\right\}$. In such scenarios, we denote the dependence of any quantity on the kernel $K \in \mathcal{K}$ using the corresponding $h \in W$, i.e., we replace $K$ by $h$ and $\mathcal{K}$ by $W$ for notational convenience.

In the following theorem, we propose a level-$\alpha$ permutation-based union test for testing $H_{0}:P=Q$ against $H_{1}: (P, Q) \in \cup_{k \in \mathcal{K} }\cup_{\theta>0}\widetilde{\mathcal{P}}_{K}$ in the case when $|\mathcal{K}|<\infty$, where the null hypothesis is rejected if and only if atleast one of the permutation tests is rejected for some $(\lambda,K) \in \Lambda \times \mathcal{K}$. We refer to this test as the RFF-based Kernel Adaptive Test. The following result is proved in Section~\ref{subsec:thm11}.

% Choice of kernel $\mathcal{K}.$

% Define $\hat{\eta}_{\lambda,l}^{(K)}$, $q_{1-\frac{w\alpha,}{|\Lambda|},K}^{B,\lambda,l}$

% $H_{0}:P=Q$ vs. $H_{1}: (P, Q) \in \cup_{k \in \mathcal{K} }\cup_{\theta>0}\widetilde{\mathcal{P}}_{K}$ where

% $\mathcal{T}_{K}$ and $\Sigma_{K}$ are defined analogous to $\mathcal{T}$ and $\Sigma$ as the integral and covariance operators, respectively, corresponding to every kernel $K \in \mathcal{K}.$

% Define $\mathcal{H}_{K}$

\begin{theorem}[RFF-based Kernel Adaptive Test]
\label{Type I-error of RFF-based permutation test with adaptation over kernel and lambda} Suppose $\boldsymbol{(\RFFAssumptionone)}$--$\boldsymbol{(\SpectralAssumptiontwo)}$ hold. Let $|\mathcal{K}|<\infty$ and $B \geq \frac{|\Lambda|^{2}|\mathcal{K}|^2}{2\tilde{w}^{2}\alpha^{2}}\log\frac{2|\Lambda||\mathcal{K}|}{\alpha(1-w-\tilde{w})}$, where $0<\alpha\leq 1$, and $0 < \tilde{w} < w  < \frac{1}{2}$. Then, 
% Let $0<\alpha\leq 1$ be the chosen level of significance and $w,\tilde{w}$ be such that $0 < \tilde{w} < w  < \frac{1}{2}$. Define the RFF-based test statistic $\hat{\eta}_{\lambda,l,K}$ and the $(1-w\alpha)$-th quantile corresponding to the empirical distribution function for the RFF-based test statistic using $B$ randomly selected permutations of $(U_{i})_{i=1}^{n+m}$ as defined in Section \ref{subsection: Adaptation over kernel and regularization parameter}. Further, define the potential choices of the regularization parameter $\Lambda$ and its cardinality $|\Lambda|$ as in Section \ref{subsection: Adaptation over regularization parameter}. Similarly, define the potential choices of the kernel $\mathcal{K}$ and its cardinality $|\mathcal{K}|$. Provided the number of permutations $B \geq \frac{|\Lambda|^{2}|\mathcal{K}|^{2}}{2\tilde{w}^{2}\alpha^{2}}\log(\frac{2|\Lambda||\mathcal{K}|}{\alpha(1-w-\tilde{w})})$ and $|\mathcal{K}|<\infty$, 
the level-$\alpha$ critical region for testing $H_0 :P=Q$ vs. $H_1:P\neq Q$ is given by $\underset{(\lambda,K) \in \Lambda \times \mathcal{K}}{\bigcup}\left\{\hat{\eta}_{\lambda,l}^{(K)} \geq q_{1-\frac{w\alpha}{|\Lambda||\mathcal{K}|},K}^{B,\lambda,l}\right\}$, i.e., 
\[P_{H_{0}}\left(\underset{(\lambda,K) \in \Lambda \times \mathcal{K}}{\bigcup}\left\{\hat{\eta}_{\lambda,l}^{(K)} \geq q_{1-\frac{w\alpha}{|\Lambda||\mathcal{K}|},K}^{B,\lambda,l}\right\}\right) \leq \alpha\,,\]
where $\hat{\eta}_{\lambda,l}^{(K)}$ is the permuted test statistic in Theorem~\ref{Type I-error of RFF-based permutation test with adaptation over lambda} but for a given kernel $K$.
\end{theorem}

The following result provides separation rates for the RFF-based Kernel Adaptive Test, which match the minimax rates as long as $l$ is large enough. 
% Finally, we conduct a power analysis(Type-II error analysis) of the RFF-based Kernel Adaptive Test by examining the behavior of the separation boundary $\Delta_{N,M}$ between distributions $P$ and $Q$ in relation to the number of random Fourier features $l$. We consider scenarios where the eigenvalues of the covariance operator $\Sigma_{PQ}$ exhibit polynomial and exponential decay. 
The proof is almost similar to that of Theorem \ref{Power Analysis for RFF based test based on permutations and adaptation} upon replacing $|\Lambda|$ by $|\Lambda||\mathcal{K}|$, and is therefore omitted.

\begin{theorem}[Separation boundary of RFF-based Kernel Adaptive Test]
\label{Power Analysis for RFF based test based on permutations and adaptation over kernel and lambda} Suppose %that Assumptions 
$\boldsymbol{(\RFFAssumptionone)}$--$\boldsymbol{(\SpectralAssumptionthree)}$, and $\boldsymbol{(\Samplesizeassumption)}$ hold. Let the number of samples $s$ split from $\mathbb{X}^{1:N}$ and $\mathbb{Y}^{1:M}$ for estimating $\Sigma_{PQ,l}$ be chosen as $s=d_{1}N=d_{2}M$ for $0\leq d_{1} \leq d_{2} \leq 1$, while the number of samples $n=N-s$ and $m=M-s$ for estimating $\mu_{P,l}$ and $\mu_{Q,l}$ respectively satisfy $n,m\geq 2$. Let $ \alpha$,$w$ and $\tilde{w}$ be such that $0\leq \alpha \leq 1$, $0 < \tilde{w} < w  < \frac{1}{2}$ and $0\leq (w-\tilde{w})\alpha < e^{-1}$.  
% Suppose that Assumptions $(\RFFAssumptionone)$,$(\RFFAssumptionfour)$,$(\SpectralAssumptionone)$,$(\SpectralAssumptiontwo)$, $(\SpectralAssumptionfour)$ and $(\Samplesizeassumption)$ hold true. Let the number of samples $s$ split from $\mathbb{X}^{1:N}$ and $\mathbb{Y}^{1:M}$ for estimating $\Sigma_{PQ,l}$ be chosen to be $O(N+M)$, specifically $s=d_{1}N=d_{2}M$ for $0\leq d_{1} \leq d_{2} \leq 1$, while the number of samples $n=N-s$ and $m=M-s$ for estimating $\mu_{P,l}$ and $\mu_{Q,l}$ respectively satisfy $n,m\geq 2$. Further, define the potential choices of the regularization parameter $\Lambda$ and its cardinality $|\Lambda|$ as in Section \ref{subsection: Adaptation over regularization parameter}. Similarly, define the potential choices of the kernel $\mathcal{K}$ and its cardinality $|\mathcal{K}|$. Let $ \alpha$,$w$ and $\tilde{w}$ be such that $0\leq \alpha \leq 1$, $0 < \tilde{w} < w  < \frac{1}{2}$ and $0\leq (w-\tilde{w})\alpha < e^{-1}$. 
Assume $c_{1} = \underset{K \in \mathcal{K}}{\sup}\underset{\theta>0}{\sup}\underset{(P,Q) \in \widetilde{\mathcal{P}}_{K}}{\sup} \Vert\mathcal{T}_{K}^{-\theta}u\Vert_{L^{2}(R)} < \infty$. Then, for any $0<\delta\le 1$, and any choice of $\theta^{*} \in (0,\frac{1}{4}]$, if $\theta\geq\theta^{*}$, and $B \geq \frac{|\Lambda|^{2}|\mathcal{K}|^{2}}{2\tilde{w}^{2}\alpha^{2}}\max\left\{\log\frac{2|\Lambda||\mathcal{K}|}{\alpha(1-w-\tilde{w})},\log\frac{2}{\delta}\right\}$,
there exists $k\in\mathbb{N}$ such that for all $N+M \geq k$, 
the power of the level-$\alpha$ test of $H_{0}:P = Q$ against $H_{1}: P \neq Q$ proposed in Theorem \ref{Type I-error of RFF-based permutation test with adaptation over kernel and lambda} is at least $1-7\delta$ over the class of $\Delta_{NM}$-separated alternatives $\mathcal{P}_{\theta,\Delta_{NM},K}$ as defined in \eqref{Class of alternatives aggregated over kernels}, i.e., 
\[
\underset{K \in \mathcal{K}}{\inf}\underset{\theta\geq\theta^{*}}{\inf}
\underset{(P,Q) \in \widetilde{\mathcal{P}}_{K}}{\inf} P_{H_{1}}\left(\underset{(\lambda,K) \in \Lambda \times \mathcal{K}}{\bigcup}\left\{\hat{\eta}^{(K)}_{\lambda,l} \geq q_{1-\frac{w\alpha}{|\Lambda||\mathcal{K}|},K}^{B,\lambda,l}\right\}\right) \geq 1-7\delta,
\]
provided one of the following scenarios is true: For any $K \in \mathcal{K}$ and $(P,Q) \in \tilde{\mathcal{P}}_{K}$,
\begin{itemize}
    \item[(i)](Polynomial decay of eigenvalues) The eigenvalues $(\lambda_{i})_{i \in I}$ of $\Sigma_{PQ}$ decay at a polynomial rate, i.e., $\lambda_{i} \asymp i^{-\beta}$ for $\beta > 1$, $\lambda_{L} = r_{1}\frac{\log(N+M)}{N+M}$, $\lambda_{U}=\min\left\{r_{2},\frac{1}{2}\norm{\Sigma_{PQ}}_{\mathcal{L}^{\infty}(\mathcal{H})}\right\}$ for some constants $r_{1},r_{2}>0$, the separation boundary satisfies
    \[
    \Delta_{N,M} = 
    c(\tilde{\alpha},\delta)\max\left\{\left[\frac{\log(N+M)}{N+M}\right]^{2\theta},\left[\frac{\log\left\{|\mathcal{K}|\log (N+M)\right\}}{N+M}\right]^{\frac{4\beta\theta}{1+4\beta\theta}} \right\},
    \]
    and the number of random features satisfies $l \gtrsim \max\left\{(N+M)^{\frac{1}{2\theta^{*}}},(N+M)^{2}\right\}$, where $c(\tilde{\alpha},\delta)>0$ is a constant that depends only on $\tilde{\alpha}$ and $\delta$.
    %with $c(\tilde{\alpha},\delta) = \max\left\{16c_{1}^{2},(16c_{1}^{2})^{\frac{5}{3}}\right\} \times \left[\log(\frac{1}{\tilde{\alpha}})\right]^{4} \times \frac{\log(\frac{4}{\delta})}{\delta^{2}}$.
    
    \item[(ii)](Exponential decay of eigenvalues) Let the eigenvalues $(\lambda_{i})_{i \in I}$ of $\Sigma_{PQ}$ decay at an exponential rate, i.e., $\lambda_{i} \asymp e^{-\tau i}$ for $\tau >0$, $\lambda_{L} = r_{3}\frac{\log(N+M)}{N+M}$, $\lambda_{U}=\min\left\{r_{4},e^{-1},\frac{1}{2}\norm{\Sigma_{PQ}}_{\mathcal{L}^{\infty}(\mathcal{H})})\right\}$ for some constants $r_{3},r_{4}>0$, the separation boundary achieves the following rate of decay
    \[
    \Delta_{N,M} = 
    c(\tilde{\alpha},\delta,\theta)\max\left\{\left[\frac{\log(N+M)}{N+M}\right]^{2\theta},\frac{\sqrt{\log(N+M)}\log\left\{|\mathcal{K}|\log (N+M)\right\}}{N+M} \right\}
    \]
    and the number of random features satisfies $l \gtrsim \max\left\{(N+M)^{\frac{1}{2\theta^{*}}}\left[\log(N+M)\right],N+M\right\}$, where $c(\tilde{\alpha},\delta,\theta)>0$ is a constant that depends only on $\tilde{\alpha}$, $\delta$ and $\theta$.
    %with $c(\tilde{\alpha},\delta,\theta) = \max\left\{(\frac{1}{16c_{1}^{2}})^{2\theta-1},16c_{1}^{2}\right\} \times \left[\max\left\{(\frac{1}{2\theta^{*}})^{\frac{1}{4\theta^{*}}},\sqrt{2}\right\}\right]^{2\theta}\times \frac{\left[\log(\frac{8}{\tilde{\alpha}})+\log(\frac{4}{\delta})\right]^{\frac{4\theta}{\theta^{*}}}}{\delta^{2\theta}}$.
\end{itemize}
    
\end{theorem}

We thus show that the fully data-driven RFF-based Kernel Adaptive Test attains the same statistical performance as both the RFF-based Oracle Test and the \say{exact} Oracle Test when the number of random Fourier features $l$ is sufficiently large. Moreover, even without prior knowledge of the true kernel $K$ or the optimal regularization parameter $\lambda^{*}$, the RFF-based Kernel Adaptive Test remains minimax optimal with respect to $\tilde{\mathcal{P}}_{K}$ up to a log log factor across a broad range of  $\theta$ ($\theta \geq \frac{1}{2}$). %\textcolor{red}{discuss a bit either here or in dicsussions or right after permutation how the recent work of lester can be used in our framework to further speeden up the testing as he uses fast permutation testing.}

\section{Computational complexity of test statistics}
\label{Section: Computational complexity of test statistics}

The primary focus of the current paper is to show that the use of RFF sampling reduces the computational complexity of the spectral regularized MMD test without compromising the statistical efficiency of the test, provided the number of random features $l$ is sufficiently large. The permutation-based tests proposed in Theorems \ref{Type I-error of RFF-based permutation test} and \ref{Type I-error of RFF-based permutation test with adaptation over lambda} can be implemented in practice, so we focus on the computational complexity of these tests only. Further, the fully data-adaptive tests (the approximate RFF-based test proposed in Theorem \ref{Type I-error of RFF-based permutation test with adaptation over lambda} and the \say{exact} test) primarily involves the computation of the approximate RFF-based test statistic for each of the $B$ randomly chosen permutations of the samples $(\mathbb{X}^{1:N},\mathbb{Y}^{1:M})$ and each $\lambda$ in the chosen range between $\lambda_{L}$ and $\lambda_{U}$. However, since each permutation test is essentially implemented in a parallel manner, the computational efficiency of the entire adaptive test is captured by the computational efficiency analysis for a single $\lambda$.

% For both the exact and RFF-based permutation tests, we choose $s\asymp N+M$ in our numerical experiments and each permutation test is essentially implemented in a parallel manner.

To give a complete picture, for a fixed regularization parameter $\lambda>0$ and the number of random Fourier features sampled $l$, we explicitly calculate and compare the number of mathematical operations (addition, subtraction, multiplication, division) required to compute the \say{exact} spectral regularized MMD test statistic $\hat{\eta}_{\lambda}$ and the RFF-based approximate spectral regularized MMD test statistic $\hat{\eta}_{\lambda,l}$. We consider the data domain $\mathcal{X}$ to be embedded in $\R^d$, i.e., the data dimension is $d$ and $s\asymp N+M$. We also consider norm-based translational invariant kernels (such as the Gaussian kernel and the Laplace kernel) for obtaining concrete estimates of the computational complexity of the test-statistics.

\subsection{Computational complexity of ``exact'' spectral regularized MMD test statistic $\hat{\eta}_{\lambda}$}

For the sake of completeness, we provide the computational algorithm stated in Theorem 4.1 of \cite{SpectralTwoSampleTest} used for computing the \say{exact} spectral regularized MMD test statistic $\hat{\eta}_{\lambda}$.

%\textbf{(Theorem 4.1 of \cite{SpectralTwoSampleTest})} 
\begin{theorem}[Theorem 4.1 of \cite{SpectralTwoSampleTest}]
Let $\left(\hat{\lambda}_i, \hat{\alpha}_i\right)_i$ be the eigensystem of $\frac{1}{s} \tilde{H}_s^{1 / 2} K_s \tilde{H}_s^{1 / 2}$ where $K_s:=\left[K\left(Z_i, Z_j\right)\right]_{i, j \in[s]}$, $H_s=I_s-\frac{1}{s} \mathbf{1}_s \mathbf{1}_s^{\top}$, and $\tilde{H}_s=\frac{s}{s-1} H_s$. Define \\$G:=\sum_i\left(\frac{g_\lambda\left(\hat{\lambda}_i\right)-g_\lambda(0)}{\hat{\lambda}_i}\right) \hat{\alpha}_i \hat{\alpha}_i^{\top}$. Then
$$
\hat{\eta}_\lambda=\frac{1}{n(n-1)}\left(\circled{\emph{\small{1}}} - \circled{\emph{\small{2}}}\right)+\frac{1}{m(m-1)}\left(\circled{\emph{\small{3}}} - \circled{\emph{\small{4}}}\right)-\frac{2}{n m}\circled{\emph{\small{5}}}
$$
where $\circled{\emph{\small{1}}}=\mathbf{1}_n^{\top} A_1 \mathbf{1}_n$, $\circled{\emph{\small{2}}}=\operatorname{Tr}\left(A_1\right)$, $\circled{\emph{\small{3}}}=\mathbf{1}_n^{\top} A_2 \mathbf{1}_n$, $\circled{\emph{\small{4}}}=\operatorname{Tr}\left(A_2\right)$, and
$$
\circled{\emph{\small{5}}}=\mathbf{1}_m^{\top}\left(g_\lambda(0) K_{m n}+\frac{1}{s} K_{m s} \tilde{H}_s^{1 / 2} G \tilde{H}_s^{1 / 2} K_{n s}^{\top}\right) \mathbf{1}_n
$$
with $A_1:=g_\lambda(0) K_n+\frac{1}{s} K_{n s} \tilde{H}_s^{1 / 2} G \tilde{H}_s^{1 / 2} K_{n s}^{\top}$ and $A_2:=g_\lambda(0) K_m+\frac{1}{s} K_{m s} \tilde{H}_s^{1 / 2} G \tilde{H}_s^{1 / 2} K_{m s}^{\top}$. Here $K_n:=\left[K\left(X_i, X_j\right)\right]_{i, j \in[n]}, K_m:=\left[K\left(Y_i, Y_j\right)\right]_{i, j \in[m]},\left[K\left(X_i, Z_j\right)\right]_{i \in[n], j \in[s]}=: K_{n s}$,\\
$K_{m s}:=\left[K\left(Y_i, Z_j\right)\right]_{i \in[m], j \in[s]}$, and $K_{m n}:=\left[K\left(Y_i, X_j\right)\right]_{i \in[m], j \in[n]}$.
\end{theorem}

We calculate in detail the computational complexity of each step involved in computing $\hat{\eta}_\lambda$ based on Theorem 4.1 of \cite{SpectralTwoSampleTest}in Section \ref{Appendix: Computational complexity details of exact test statistic} of the Appendix. Based on this calculation, the total computational complexity of the \say{exact} spectral regularized MMD test statistic $\hat{\eta}_{\lambda}$ in terms of number of mathematical operations is \[O(s^3+ns^2+ms^2+s^2d+n^2d+m^2d+nsd+msd+mnd).\]

\subsection{Computational complexity of RFF-based approximate spectral regularized MMD test statistic}

The following theorem, whose proof is shown as Algorithm~\ref{alg:approx-test-statistic} provides a computational form for the RFF-based approximate spectral regularized MMD test statistic, when the kernel $K$ is symmetric, real-valued, and translation invariant on $\R^d$.

\begin{theorem}[Computational complexity of RFF-based approximate spectral regularized MMD test statistic]\label{Computational complexity of RFF-based approximate spectral regularized MMD test statistic} Let \(\{X_i\}_{i=1}^{n}\), \(\{Y_j\}_{j=1}^{m}\), \(\{Z_i\}_{i=1}^{s}\), and \(\{\theta_i\}_{i=1}^{l}\) be as described in 
Sections~\ref{subsec : Spectral Regularized MMD Test} and~\ref{subsec: Construction of the test statistic and the test (Approx)}.
Form the matrices $X = \left[X_{1}  \dots X_{n}\right]$, $Y = \left[Y_{1}  \dots Y_{m}\right]$, $Z = \left[Z_{1}  \dots Z_{s}\right]$ and $\Theta = \begin{bmatrix}
    \theta_{1}^{T} \\ \vdots \\\theta_{l}^{T}
\end{bmatrix}$. Define  $M_X = X^{T} \Theta ^{T}$, $M_Y = Y^{T} \Theta ^{T}$ and $M_{Z} = Z^{T} \Theta ^{T}$.
Let
$\Phi(X) = \frac{K(0,0)}{\sqrt{l}}P_{l}^{T}\begin{bmatrix} \cos(M_X) \;\big|\;\sin(M_X) \end{bmatrix}^{T}$, $\Phi(Y) = \frac{K(0,0)}{\sqrt{l}}P_{l}^{T}\begin{bmatrix} \cos(M_Y) \;\big|\; \sin(M_Y) \end{bmatrix}^{T}$ and $\Phi(Z) = \frac{K(0,0)}{\sqrt{l}}P_{l}^{T}\begin{bmatrix} \cos(M_Z) \;\big|\; \sin(M_Z) \end{bmatrix}^{T}$
where $P_{l}$ is the $2l\times 2l$ column-interleaving permutation matrix
$P_{l} = 
\begin{bmatrix}
e_{1,2l} \quad e_{l+1,2l} \quad e_{2,2l} \quad e_{l+2,2l}\;\cdots \;e_{l,2l}\quad e_{2l,2l}
\end{bmatrix}$,
and $e_{i,2l}$ is the standard basis vector in $\mathbb{R}^{2l}$ with $1$ in the $i$-th coordinate and $0$ elsewhere. 

Let $(\hat{\lambda}_{i},\hat{\alpha}_{i})_{i=1}^{2l}$ be the eigenvalue-eigenvector pairs of $\frac{1}{s(s-1)}\Phi(Z)\left[sI_{s} - \mathbf{1}_{s}\mathbf{1}_{s}^{T}\right]\Phi(Z)^{T}$. Define $G = \sum_{i=1}^{2l} \sqrt{g_{\lambda}(\hat{\lambda}_{i})}\hat{\alpha}_{i} \hat{\alpha}_{i}^{T}$. Then the RFF-based approximate spectral regularized MMD test statistic, as in 
\eqref{Approximate Kernel Test statistic}, is given by
\[
\hat{\eta}_{\lambda,l}
\;=\;
\frac{A}{n(n-1)}
\;+\;
\frac{B}{m(m-1)}
\;-\;
\frac{2\,C}{n\,m},
\]
where
$A = \mathbf{1}_{n}^{T} \Phi(X)^{T}G^{T}G\Phi(X) \mathbf{1}_{n} - \sum_{i=1}^{n} e_{i,n}^{T} \Phi(X)^{T}G^{T}G\Phi(X) e_{i,n}$, $B = \mathbf{1}_{m}^{T} \Phi(Y)^{T}G^{T}G\Phi(Y) \mathbf{1}_{m} - \sum_{j=1}^{m} e_{j,m}^{T} \Phi(Y)^{T}G^{T}G\Phi(Y) e_{j,m}$, and $C= \mathbf{1}_{n}^{T} \Phi(X)^{T}G^{T}G\Phi(Y) \mathbf{1}_{m}$. 

\end{theorem}

\renewcommand{\algorithmicensure}{\textbf{Output:}}
% %%% Algorithm form
\begin{algorithm}[!htbp]
\caption{Computation of RFF-based approximate spectral regularized MMD test statistic $\hat{\eta}_{\lambda,l}$}
\label{alg:approx-test-statistic}
\begin{algorithmic}[1]
\REQUIRE $\{X_i\}_{i=1}^N  \overset{i.i.d}{\sim}P$ ; $\{Y_j\}_{j=1}^M  \overset{i.i.d}{\sim}Q$; Number of random features $l$, Number of sample points $s \in \mathbb{N}$ for covariance operator estimation; Kernel $K$ with spectral (Fourier) distribution $\Xi$; spectral regularizer $g_{\lambda}$
\ENSURE RFF-based approximate spectral regularized MMD test statistic $\hat{\eta}_{\lambda,l}$
\STATE Split $\{X_i\}_{i=1}^N$ into $\left(X_i\right)_{i=1}^{n}\coloneqq\left(X_i\right)_{i=1}^{N-s}$ and $\left(X_i^1\right)_{i=1}^s\coloneqq \left(X_i\right)_{i=N-s+1}^N$. Similarly, split $\{Y_j\}_{j=1}^M$ into $\left(Y_j\right)_{j=1}^{m} \coloneqq \left(Y_j\right)_{j=1}^{M-s}$ and $\left(Y_j^1\right)_{j=1}^s\coloneqq \left(Y_j\right)_{j=M-s+1}^M$.
\STATE Construct the matrices $X = \left[X_{1}  \dots X_{n}\right]$ and $Y = \left[Y_{1}  \dots Y_{m}\right]$.
\STATE Sample $\left\{\theta_{i}\right\}_{i=1}^{l} \overset{i.i.d}{\sim} \Xi$ and form the matrix $\Theta = \begin{bmatrix}
    \theta_{1}^{T} \\ \vdots \\\theta_{l}^{T}
\end{bmatrix}$.
\STATE Sample $\left(\alpha_i\right)_{i=1}^s \stackrel{i . i . d}{\sim} \operatorname{Bernoulli}(1 / 2)$ and compute $Z_i=\alpha_i X_i^1+\left(1-\alpha_i\right) Y_i^1$, for $1 \leq i \leq s$. Then, form the matrix $Z = \left[Z_{1}  \dots Z_{s}\right]$.
\STATE Compute the matrices $M_X = X^{T} \Theta ^{T} = (\Theta X)^T$, $M_Y = Y^{T} \Theta ^{T} = (\Theta Y)^T$ and $M_Z = Z^{T} \Theta ^{T} = (\Theta Z)^T$.
\STATE Compute matrix of random features corresponding to $X_{i}$'s ($i=1,\dots,n$) as $\Phi(X) = \frac{1}{\sqrt{l}}P_{l}^{T}\begin{bmatrix} \cos(M_X) \,|\,\sin(M_X) \end{bmatrix}^{T}$, the  matrix of random features corresponding to $Y_{j}$'s ($j=1,\dots,m$) as $\Phi(Y) = \frac{1}{\sqrt{l}}P_{l}^{T}\begin{bmatrix} \cos(M_Y) \,|\, \sin(M_Y) \end{bmatrix}^{T}$ and the matrix of random features corresponding to $Z_{i}$'s as $\Phi(Z) = \frac{1}{\sqrt{l}}P_{l}^{T}\begin{bmatrix} \cos(M_Z) \,|\, \sin(M_Z) \end{bmatrix}^{T}$, where $P_{l} = 
\begin{bmatrix}
e_{1,2l} \quad e_{l+1,2l} \quad e_{2,2l} \quad e_{l+2,2l}\;\cdots \;e_{l,2l}\quad e_{2l,2l}
\end{bmatrix}$.
\STATE Compute the matrix $K_s = \Phi(Z)\Phi(Z)^{T} $ and the vector $v_Z = \Phi(Z) \mathbf{1}_{s}$.
\STATE Compute the matrix $\hat{\Sigma}_{PQ,l} = \frac{1}{s(s-1)} (s K_{s} - v_{Z}v_{Z}^{T})$.
\STATE Compute the eigenvalue-eigenvector pairs $(\hat{\lambda}_{i},\hat{\alpha}_{i})$ corresponding to $\hat{\Sigma}_{PQ,l}$. Construct the diagonal matrix $D = \begin{bmatrix}
\hat{\lambda}_{1} & & \\
& \ddots & \\
& & \hat{\lambda}_{2l}
\end{bmatrix}$ and the matrix $V = \left[ \hat{\alpha}_{1} \dots \hat{\alpha}_{2l} \right]$.
\STATE Construct the matrix $G = V L^{1/2} V^{T}$, where $L^{1/2} = \begin{bmatrix}
\sqrt{g_{\lambda}(\hat{\lambda}_{1})} & & \\
& \ddots & \\
& & \sqrt{g_{\lambda}(\hat{\lambda}_{2l})}
\end{bmatrix}$.
\STATE Compute the matrices $\Psi(X) = G \Phi(X)$ and $\Psi(Y) = G \Phi(Y)$.
\STATE Compute the vectors $v_{X,i} = \Psi(X) e_{i,n}$ for $i=1,\dots,n$ and $v_{Y,j} = \Psi(Y) e_{j.m}$ for $j=1,\dots,m$, where $\left\{e_{i,n}\right\}_{i=1}^{n}$ and $\left\{e_{j,m}\right\}_{j=1}^{m}$ are standard basis vectors $\R^n$ and $\R^m$, respectively.
\STATE Compute $v_{X} = \sum_{i=1}^{n} v_{X,i}$ and $v_{Y} = \sum_{j=1}^{m} v_{Y,j}$.
\STATE Compute $A = v_{X}^{T} v_{X} - \sum_{i=1}^{n}v_{X,i}^{T} v_{X,i}$.
\STATE Compute $B = v_{Y}^{T} v_{Y} - \sum_{j=1}^{m}v_{X,j}^{T} v_{X,j}$.
\STATE Compute $C = v_{X}^{T} v_{Y}$.
\STATE Compute the test statistic $\hat{\eta}_{\lambda,l} = \frac{A}{n(n-1)} + \frac{B}{m(m-1)} - \frac{2C}{nm}$.
\RETURN $\hat{\eta}_{\lambda,l}$
\end{algorithmic}
\end{algorithm}

% Combining all these operations, the total \textbf{computational complexity of the RFF-based approximate spectral regularized MMD test statistic $\hat{\eta}_{\lambda,l}$ is} 
% \[
% O(l^{3} + (s+m+n) l^{2} + (s+m+n)ld).
% \]

The algorithm for computing the RFF-based test statistic is given in Algorithm \ref{alg:approx-test-statistic}. We calculate in detail the computational complexity of each step involved in Algorithm \ref{alg:approx-test-statistic} in Section \ref{Appendix: Computational complexity details of approximate test statistic} of the Appendix. Based on this calculation, the total computational complexity of the RFF-based approximate spectral regularized MMD test statistic $\hat{\eta}_{\lambda,l}$ is $
O(l^{3} + (s+m+n) l^{2} + (s+m+n)ld).$

For positive definite kernels on general domains other than Euclidean spaces, the corresponding feature maps differ from simple sine and cosine transformations of the inner product of spectral frequencies and data samples. This leads to different computational steps for calculating the appropriate feature matrices $\Phi(X)$, $\Phi(Y)$, and $\Phi(Z)$, in general. However, the worst-case computational complexity of the RFF-based approximate spectral regularized MMD test statistic $\hat{\eta}_{\lambda,l}$ is less than or equal to $O(l^{3} + (s+m+n) l^{2} + (s+m+n)ld)$. For instance, for kernels defined on the $d$-dimensional sphere $\mathbb{S}^{d-1}$, such as the Gaussian kernel, the Laplace kernel, the heat kernel, and the Poisson kernel, the feature maps are determined by the spherical harmonics. Consequently, the worst-case computational complexity of the feature matrices $\Phi(X)$, $\Phi(Y)$ and $\Phi(Z)$ is $O(l\log l + (s+m+n)l\sqrt{d})$. Therefore, for these kernels defined on the sphere, the total computational complexity of the RFF-based approximate spectral regularized MMD test statistic $\hat{\eta}_{\lambda,l}$ is
$O(l^{3} + (s+m+n) l^{2} + (s+m+n)l\sqrt{d}).$

\subsection{Comparison of computational complexity of exact and approximate spectral regularized MMD test statistics}

Since the total computational complexity of the exact test statistic is $O(s^3+ns^2+ms^2+s^2d+n^2d+m^2d+nsd+msd+mnd)$, under the setting where $s\asymp (N+M)$, the approximate RFF-based test statistic is computationally as efficient as the exact test statistic when the number of random features $l=O(N+M)$. If $l=c(N+M)^{a}$ for some $0\leq a <1$ and some constant $c>0$, then the computational complexity of the approximate RFF-based test statistic is $O((N+M)^{1+2a} + (N+M)^{1+a}d)$ which is of smaller order than the computational complexity of the exact test statistic, i.e., the approximate RFF-based test statistic is strictly more computationally efficient than the exact test statistic. 

% For both the exact and RFF-based permutation tests, we choose $s\asymp N+M$ in our numerical experiments and each permutation test is essentially implemented in a parallel manner.

Under polynomial decay of the eigenvalues of $\Sigma_{PQ}$ at rate $\beta>1$, the computational complexity of the RFF-based test statistic is
$$
\begin{cases}
O\left((N+M)^{1+\frac{4(\beta+1)}{1+4\theta\beta}}\right),&  \theta > \frac{1}{2} + \frac{1}{4\beta}\\
O\left((N+M)^{\frac{6(\beta+1)}{1+4\theta\beta}}\right),&  \frac{1}{2} - \frac{1}{4\beta}<\theta\le\frac{1}{2} + \frac{1}{4\beta}\\
O\left((N+M)^{\frac{3(\beta+1)}{\beta}}\right),& \theta \leq \frac{1}{2} - \frac{1}{4\beta}
\end{cases},
$$
% \[
% \begin{cases}
% O\left(\right)
% \end{cases}
% \]
% \[
% \begin{cases}
%     O\left(\left(N+M\right)^{\frac{6(\beta+1)}{1+4\beta\theta}} + \left(N+M\right)^{1+\frac{4(\beta+1)}{1+4\beta\theta}} +\left(N+M\right)^{1+\frac{2(\beta+1)}{1+4\beta\theta}}d\right)&    \textrm{ if }\theta > \frac{1}{2} - \frac{1}{4\beta}\textrm{ and }1 < \beta \leq 2\\
%     &\textrm{, or }\frac{1}{2} - \frac{1}{4\beta}< \theta \leq \frac{2\beta^2 + \beta + 2}{4\beta(\beta - 2)}\textrm{ and }\beta>2\\
%     O\left(\left(N+M\right)^{\frac{6(2\beta-1)}{1+2\beta+4\beta\theta}} + \left(N+M\right)^{1+\frac{4(2\beta-1)}{1+2\beta+4\beta\theta}} +\left(N+M\right)^{1+\frac{2(2\beta-1)}{1+2\beta+4\beta\theta}}d\right)&   \textrm{ if }\beta>2\textrm{ and }\theta > \frac{2\beta^2 + \beta + 2}{4\beta(\beta - 2)},\\
%     O\left(\left(N+M\right)^{\frac{3(\beta+1)}{\beta}} + \left(N+M\right)^{1+\frac{2(\beta+1)}{\beta}} +\left(N+M\right)^{1+\frac{\beta+1}{\beta}}d\right)&   \textrm{ if }\theta \leq \frac{1}{2} - \frac{1}{4\beta}\textrm{ and }1 < \beta < \frac{1+\sqrt{3}}{2}\\
%     &\textrm{, or }\frac{2\beta^2 - 2\beta - 1}{4\beta(\beta + 1)} \leq \theta \leq \frac{1}{2} - \frac{1}{4\beta}\textrm{ and }\beta \geq \frac{1+\sqrt{3}}{2}\\
%     O\left(\left(N+M\right)^{\frac{3(2\beta-1)}{1+4\beta\theta}} + \left(N+M\right)^{1+{\frac{2(2\beta-1)}{1+4\beta\theta}}} +\left(N+M\right)^{1+{\frac{2\beta-1}{1+4\beta\theta}}}d\right) &   \textrm{ if }\beta \geq \frac{1+\sqrt{3}}{2}\textrm{ and }\theta < \frac{2\beta^2 - 2\beta - 1}{4\beta(\beta + 1)}
% \end{cases}
% \]
%,
while the computational complexity of the exact test statistic is $O(\left(N+M\right)^{3})$. %$O(\left(N+M\right)^{3}+\left(N+M\right)^{2}d)$. 
Comparing these complexities, it can be noted that the RFF test is computationally efficient and statistically optimal in the regime $\theta>\frac{1}{2}+\frac{1}{4\beta}$. In contrast, the RFF test is statistically optimal and not computationally efficient (w.r.t.~the exact test) in the regime $\frac{1}{2}-\frac{1}{4\beta}<\theta\le\frac{1}{2}+\frac{1}{4\beta}$. Of course, the RFF test is neither computationally efficient and possibly not statistically optimal (as the statistically optimality of the Oracle Test in the regime $\theta\le \frac{1}{2}-\frac{1}{4\beta}$ is not known) if $\theta\le \frac{1}{2}-\frac{1}{4\beta}$. 
% So, the RFF-based test statistic $\hat{\eta}_{\lambda,l}$ is strictly more computationally efficient compared to the exact test statistic $\hat{\eta}_{\lambda}$ if the smoothness index $\theta$ is large enough. More specifically, the scenarios where the RFF-based test statistic $\hat{\eta}_{\lambda,l}$ is strictly more computationally efficient compared to the exact test statistic $\hat{\eta}_{\lambda}$, together with the computational complexity of the RFF-based test statistic is 
% % \[
% % \begin{cases}
% %     O\left(\left(N+M\right)^{1+\frac{4(\beta+1)}{1+4\beta\theta}} +\left(N+M\right)^{1+\frac{2(\beta+1)}{1+4\beta\theta}}d\right)&    \textrm{ if }\theta > \frac{1}{2} + \frac{1}{4\beta}\textrm{ and }1 < \beta \leq 2\\
% %     O\left( \left(N+M\right)^{1+\frac{4(2\beta-1)}{1+2\beta+4\beta\theta}} +\left(N+M\right)^{1+\frac{2(2\beta-1)}{1+2\beta+4\beta\theta}}d\right)&   \textrm{ if }\beta>2\textrm{ and }\theta > \frac{2\beta^2 + \beta + 2}{4\beta(\beta - 2)}
% % \end{cases}.
% % \]
% \[
% \begin{cases}
%     O\left(\left(N+M\right)^{1+\frac{2(\beta+1)}{1+4\beta\theta}}d\right)&    \textrm{ if }\theta > \frac{1}{2} + \frac{1}{4\beta}\textrm{ and }1 < \beta \leq 2\\
%     O\left( \left(N+M\right)^{1+\frac{2(2\beta-1)}{1+2\beta+4\beta\theta}}d\right)&   \textrm{ if }\beta>2\textrm{ and }\theta > \frac{2\beta^2 + \beta + 2}{4\beta(\beta - 2)}
% \end{cases}.
% \]
% % $\theta \geq \frac{1}{2} + \frac{1}{4\beta}$, with a computational complexity of \[ O\left(
% % %\left(N+M\right)^{\frac{6(\beta+1)}{1+4\beta\theta}}+
% % \left(N+M\right)^{1+\frac{2(\beta+1)}{1+4\beta\theta}}d\right).\]
Moreover, the approximate test scales sub-quadratic in $N+M$ if $\theta>1+\frac{3}{4\beta}$ and tends to scale linearly in $N+M$ as $\theta\rightarrow\infty$.  

Under exponential decay of the eigenvalues of $\Sigma_{PQ}$, the computational complexity of the RFF-based test statistic is
\[
\begin{cases}
    O\left( \left(N+M\right)^{1+\frac{1}{\theta}}\log(N+M)^{2} \right)%+\left(N+M\right)^{1+\frac{1}{2\theta}}\log(N+M)d\right)
    ,& \theta > \frac{1}{2}\\
    O\left(\left(N+M\right)^{3} %+ \left(N+M\right)^{2} d
    \right),& \theta \leq \frac{1}{2}
\end{cases},
\]
while the computational complexity of the exact test statistic is $O(\left(N+M\right)^{3})$.  %+\left(N+M
%\right)^{2}d)$. 
So, the RFF test is both computationally efficient compared to the exact test  and statistically minimax if %compared to the exact test statistic $\hat{\eta}_{\lambda}$ if 
$\theta > \frac{1}{2}$. 
% with a computational complexity of \[O\left( \left(N+M\right)^{1+\frac{1}{\theta}}\log\left[(N+M)^{2}\right] +d\left(N+M\right)^{1+\frac{1}{2\theta}}\log(N+M)\right).\]
%Therefore, we can guarantee a gain in computational efficiency when 
% the \say{smoothness} of the likelihood ratio deviation $u=\frac{dP}{dR}-1$ is sufficiently large i.e. when $u \in \operatorname{Ran}(\mathcal{T}^{\theta})$ for
%$\theta>\frac{1}{2}$. 
Moreover, the RFF test scales sub-quadratic in $N+M$ if $\theta>1$ and tends to scale as linear in $N+M$ as $\theta\rightarrow\infty$.

\section{Numerical experiments}
\label{sec:expts}
In this section, we evaluate the empirical performance of the RFF-based Kernel Adaptive Test by comparing it to its most natural competitor, the \say{exact} Adaptive Test proposed in Theorem 4.10 of \cite{SpectralTwoSampleTest}. We consider both simulated and real-life benchmark datasets in our numerical experiments.

We demonstrate the empirical performance of the tests under consideration by choosing the regularizer/spectral function $g_{\lambda}$ to be the popularly used Showalter regularizer, i.e., $g_{\lambda}(x) = \frac{1-e^{-x / \lambda}}{x} \mathbf{1}_{\{x \neq 0\}}+\frac{1}{\lambda} \mathbf{1}_{\{x=0\}}$. Using the Tikhonov regularizer, i.e., $g_{\lambda}(x) = \frac{1}{x+\lambda}$ yields qualitatively similar results. Type-I errors are controlled at $\alpha=0.05$. For both the \say{exact} Adaptive Test and RFF-based Kernel Adaptive test, we choose $s\asymp N+M$ in our numerical experiments, and each permutation test is essentially implemented in a parallel manner. To ensure Type-I errors are controlled, the number of permutations $B$ is chosen to be large enough for both the \say{exact} Adaptive Test and the RFF-based Kernel Adaptive Test. However, it is empirically observed that the number of permutations $B$ required for achieving the specified Type-I error control is a bit higher for the RFF-based Kernel Adaptive Test compared to the \say{exact} Adaptive Test. Despite this fact, the computational efficiency gain achieved by the RFF-based Kernel Adaptive Test over the \say{exact} Adaptive Test is substantial in most scenarios. We average all reported results over 3 replications over the sampling of random Fourier features for any given choice of the number of Fourier features $l$. In addition, for simulated datasets, we average all reported results over 100 random simulations. Any tests with \say{0} random Fourier features are basically the \say{exact} Adaptive test.

We consider the Gaussian kernel $K_{RBF}(x,y) = \exp(-\norm{x-y}_{2}^{2}/2h)$ and the Laplace kernel $K_{Lap}(x,y) = \exp(-\norm{x-y}_{1}/h)$ in our experiments, with the bandwidth parameter $h$ being chosen from a set of bandwidth choices, denoted by $W$. For any given experiment, let us denote the set of choices for the number of random Fourier feature samples as $F$. Then, to perform the RFF-based Kernel Adaptive Test, we first choose some $l \in F$. For a given combination of $l \in F$, $\lambda \in \Lambda$ and $h\in W$, let us denote the RFF-based Kernel Adaptive Test statistic as $\hat{\eta}_{\lambda,l}^{(h)}$. Similarly, for a given combination of number of permutations $B$, $l \in F$, $\lambda \in \Lambda$ and $h\in W$, let us denote the critical threshold for the RFF-based Kernel Adaptive Test (as defined in \eqref{B Permutation quantile for RFF based test}) to be $\hat{q}_{1-\alpha,h}^{B,\lambda,l}$. For the corresponding \say{exact} Adaptive Test with a potentially different number of permutations $B^{\prime}$, we just drop the superscripts corresponding to $l$ and analogously define the \say{exact} Adaptive test statistic $\hat{\eta}_{\lambda}^{(h)}$ together with its corresponding critical threshold $\hat{q}_{1-\alpha,h}^{B^\prime,\lambda}$. When performing the RFF-based Kernel Adaptive Test, we reject the null hypothesis $H_{0}:P=Q$ if and only if $\hat{\eta}_{\lambda,l}^{(h)} \geq \hat{q}_{1-\frac{\alpha}{|\Lambda||W|},h}^{B,\lambda,l}$ for some $(l,\lambda,h) \in F\times \Lambda \times W$. Similarly, when performing \say{exact} Adaptive Test, we reject the null hypothesis $H_{0}:P=Q$ if and only if $\hat{\eta}_{\lambda}^{(h)} \geq \hat{q}_{1-\frac{\alpha}{|\Lambda||W|},h}^{B^{\prime},\lambda}$ for some $(\lambda,h) \in \Lambda \times W$.

\subsection{Gaussian mean shift }
In the first set of simulation-based experiments, we consider $P=N(0,I_{d})$ and $Q=N(\mu,I_{d})$, where $N(\mu,C)$ denotes the Gaussian distribution in $R^{d}$ with mean $\mu$ and covariance matrix $C$. That is, we consider the class of mean-shifted alternatives, and we use the choices $\mu\in\{0,0.05,0.1,0.3,0.5,0.7,1\}$ as the value of the mean shift for our experiments.

We consider the sample size to be $N=M=200$ and data dimensions to be $d=1,10,20,50,100$. We choose $s=20$. All experiments are performed using the Gaussian kernel. Collection of bandwidths that we adapt over is $W = \left\{10^{-2 + 0.5 i}: i=0,1,\dots,8\right\}$, while the set of values of the regularization parameter that we adapt over is given by $\Lambda = \left\{10^{- 6 + 0.75 i}:  \textrm{ for }i=0,1,\dots,9\right\}$. For the RFF-based Kernel Adaptive Test, we consider $F=\left\{1,3,5,7,9\right\}$. The number of permutations for the RFF-based Kernel Adaptive Test and the \say{exact} Adaptive Test are chosen to be $B=600$ and $B^{\prime}=250$, respectively. From Figure \ref{fig: Empirical power for Gaussian mean shift experiments}, we can observe that a relatively small number of random Fourier features (around 7 or 9) is sufficient to ensure that the power of the RFF-based Kernel Adaptive Test is nearly as high as the \say{exact} Adaptive Test. Most importantly, based on Figure \ref{fig: Time comparison for Gaussian mean shift experiments} and Table \ref{tab: Table for comparison of computation times for Gaussian mean shift experiments}, the RFF-based Kernel Adaptive Test compensates more than adequately for the slight loss in power by taking around 33-44$\%$ of the computation time required by the \say{exact} Adaptive Test. Therefore, a very favorable trade-off between test power and computational efficiency is achieved by the RFF-based Kernel Adaptive Test, as demonstrated through these experiments.
\begin{figure}[t]
	\centering
\includegraphics[width=0.7\linewidth]{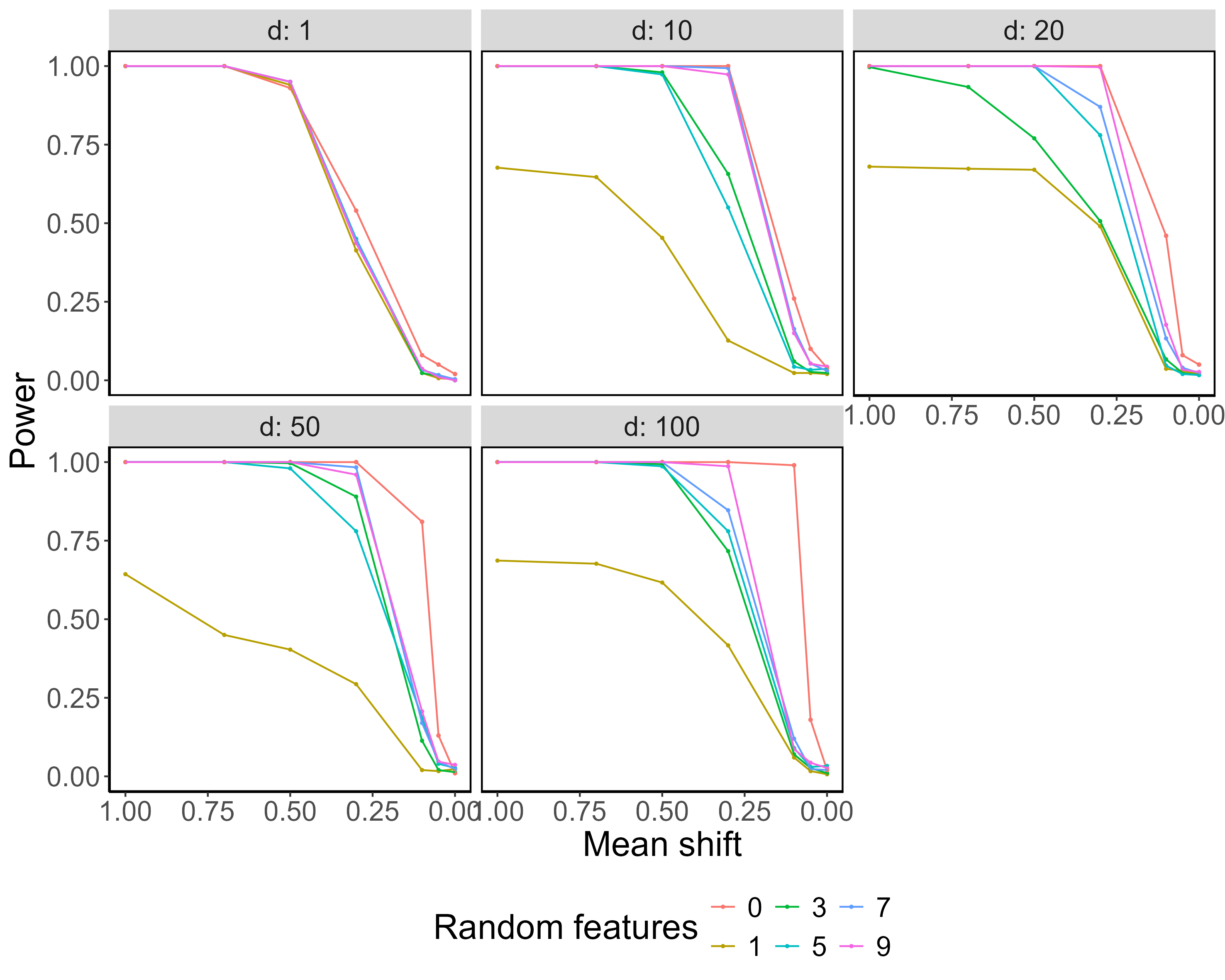}
	\caption{Empirical power for Gaussian mean shift experiments.}
	\label{fig: Empirical power for Gaussian mean shift experiments}
\end{figure}
\begin{figure}[!htbp]
	\centering
	\includegraphics[width=0.7\linewidth]{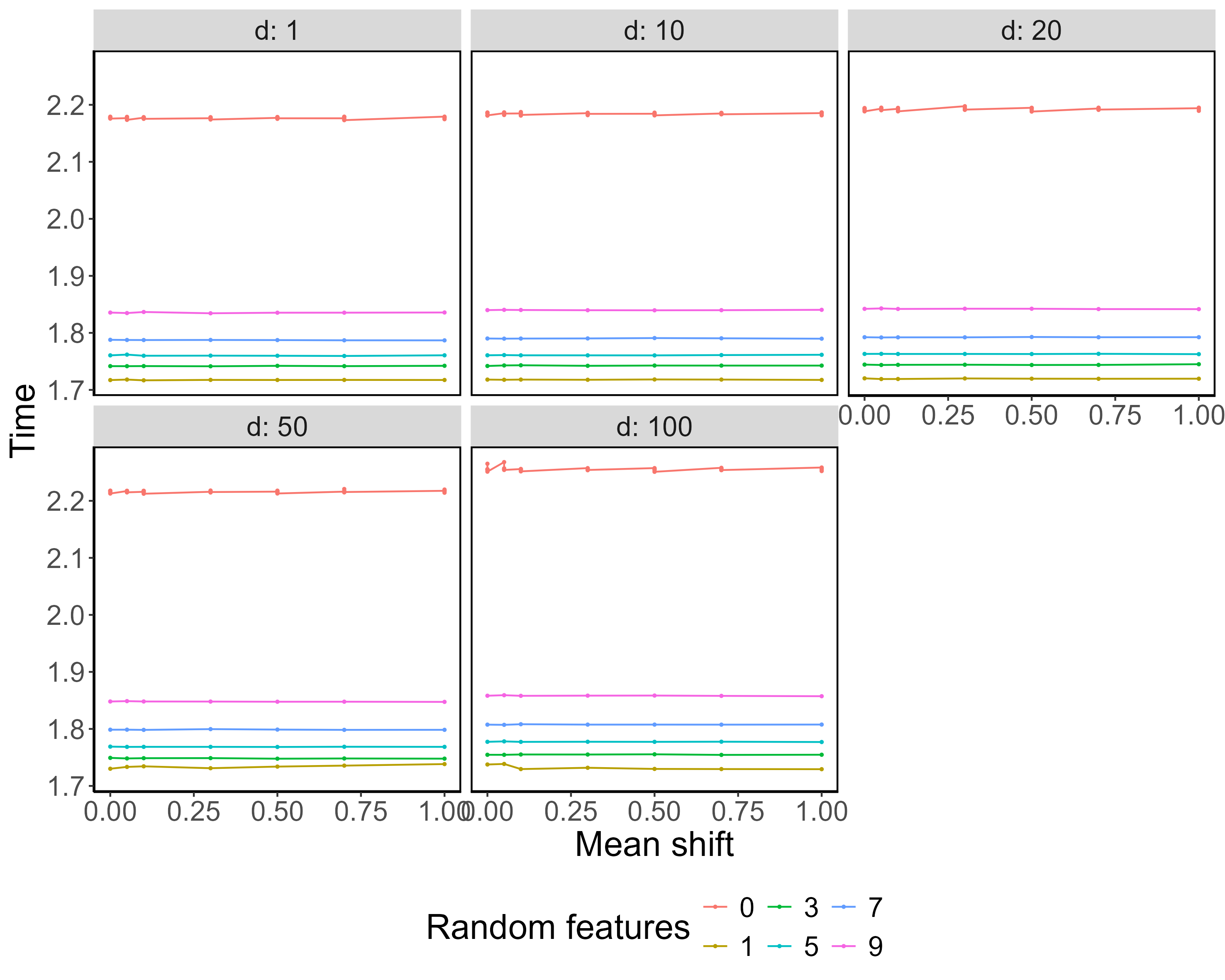}\vspace{-2mm}
	\caption{Comparison of computation time (in log seconds) for Gaussian mean shift experiments.}
	\label{fig: Time comparison for Gaussian mean shift experiments}
    \vspace{1mm}
\end{figure}

\begin{table}[t]
\centering
\begin{tabular}{|c|c|}
\hline
No. of random features ($l$) & Ratio of time taken by RFF-based test \\
\hline
\begin{filecontents*}{"Results/Gaussmeanshift/time_comparisons.CSV"}
No. of random features (l),Ratio of compute time b/w efficient and original test
1,0.328602856253882
3,0.345744890607588
5,0.364280476967538
7,0.387232530816974
9,0.437589612449758
\end{filecontents*}
\csvreader[head to column names, 
           late after line=\\\hline]
{"Results/Gaussmeanshift/time_comparisons.CSV"}{1=\one,2=\two}
{
           \one & \num[round-precision=2,round-mode=places]\two
}
\hline
\end{tabular}
\caption{Table for comparison of computation times for Gaussian mean shift experiments.}
\label{tab: Table for comparison of computation times for Gaussian mean shift experiments}
\end{table}

\FloatBarrier

\subsection{Gaussian scale shift}
In the second set of simulation-based experiments, we consider $P=N(0,I_{d})$ and $Q=N(0,\sigma^{2}I_{d})$ where $N(\mu,C)$ denotes the Gaussian distribution in $R^{d}$ with mean $\mu$ and covariance matrix $C$. Here, we consider the class of scale-shifted alternatives and we use the choices $$\sigma^{2} \in \left\{10^{i}: i=0, 0.05, 0.10, 0.20, 0.30, 0.40, 0.50\right\}$$ as the value of the scale shift for our experiments.

We consider the sample size to be $N=M=200$ and data dimensions to be $d=1,10,20,50,100$. We choose $s=20$. All experiments are performed using the Gaussian kernel. Collection of bandwidths that we adapt over is $W = \left\{10^{-2 + 0.5  i}: i=0,1,\dots,8\right\}$, while the set of values of the regularization parameter that we adapt over is given by $\Lambda = \left\{10^{- 6 + 0.75 i}: i=0,1,\dots,9\right\}$. For the RFF-based Kernel Adaptive Test, we consider $F=\left\{1,3,5,7,9\right\}$. The number of permutations for RFF-based Kernel Adaptive Test and the \say{exact} Adaptive Test are chosen to be $B=550$ and $B^{\prime}=250$, respectively.

\begin{figure}[!htbp]
	\centering
	\includegraphics[width=0.7\linewidth]{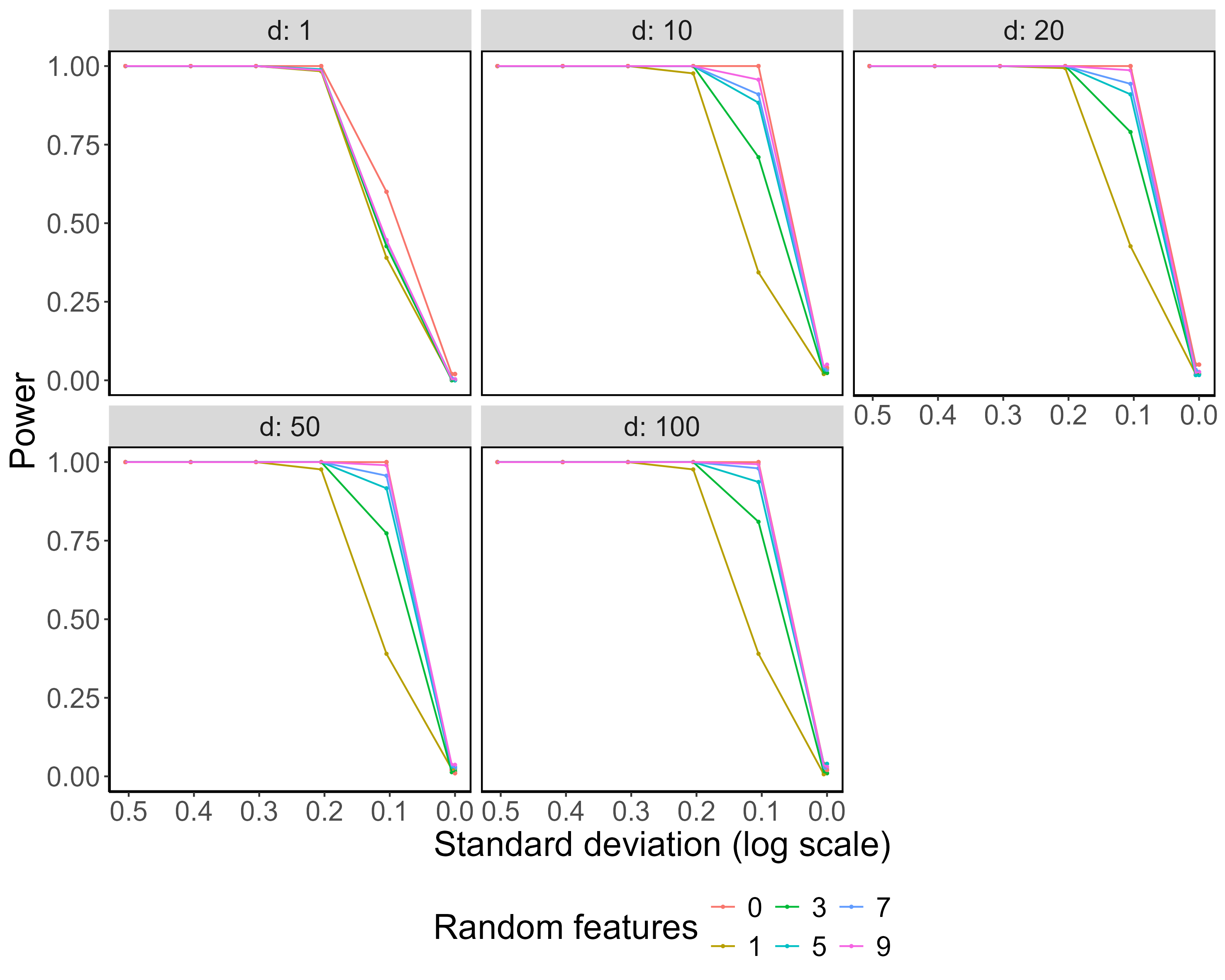}
	\caption{Empirical power for Gaussian scale shift experiments.}
	\label{fig: Empirical power for Gaussian scale shift experiments}
\end{figure}

From Figure \ref{fig: Empirical power for Gaussian scale shift experiments}, we can observe that a relatively small number of random Fourier features (more than or equal to 5) is sufficient to ensure that the power of the RFF-based Kernel Adaptive Test is nearly as high as the \say{exact} Adaptive Test. Most importantly, based on Figure \ref{fig: Time comparison for Gaussian scale shift experiments} and Table \ref{tab: Table for comparison of computation times for Gaussian scale shift experiments}, the RFF-based Kernel Adaptive Test compensates more than adequately for the slight loss in power by taking around 30-40$\%$ of the computation time required by the \say{exact} Adaptive test. Therefore, a very favorable trade-off between test power and computational efficiency is achieved by the RFF-based Kernel Adaptive Test, as demonstrated through these experiments.

\begin{figure}[!htbp]
	\centering
	\includegraphics[width=0.7\linewidth]{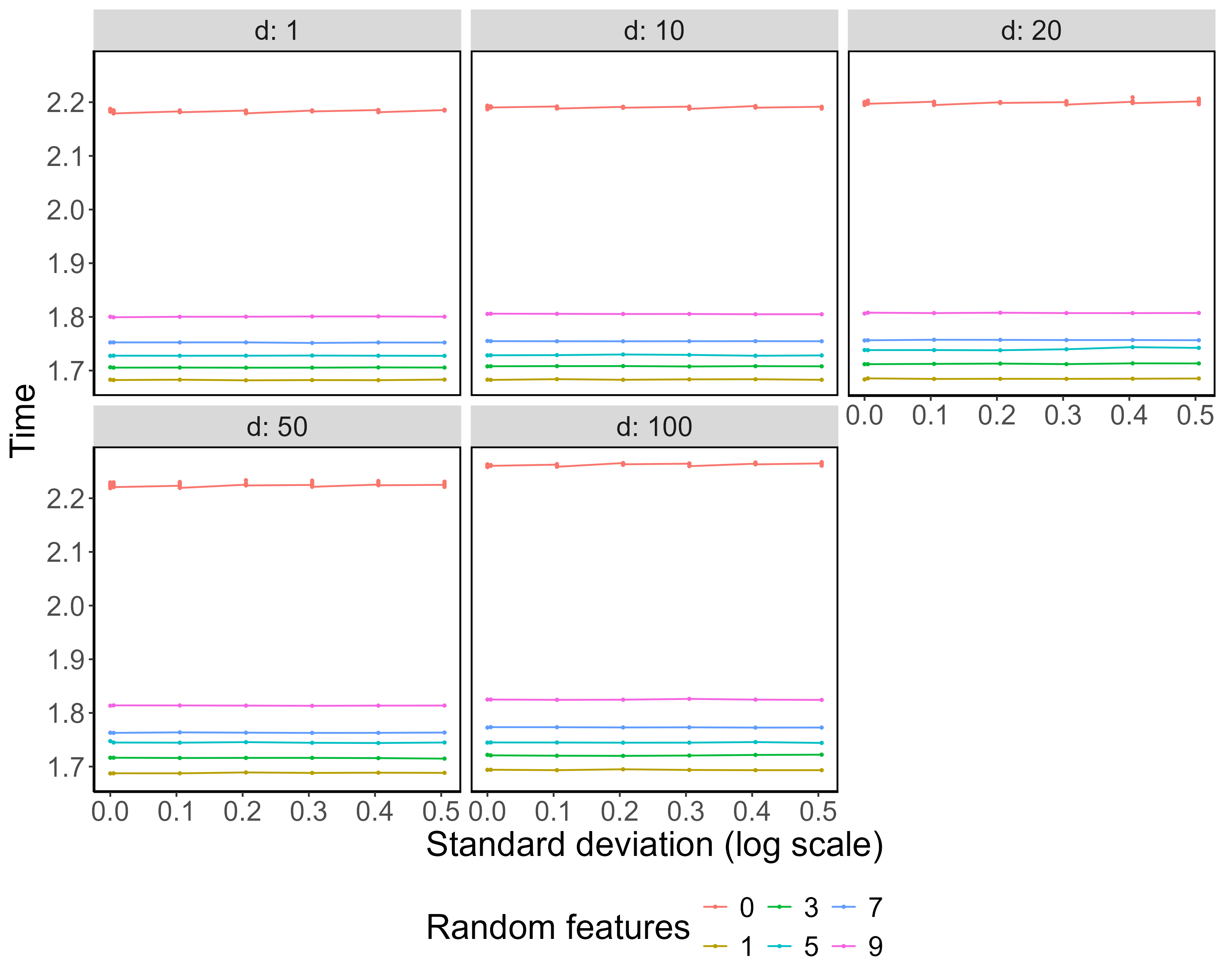}
	\caption{Comparison of computation time (in log seconds) for Gaussian scale shift experiments.}
	\label{fig: Time comparison for Gaussian scale shift experiments}
\end{figure}

\begin{table}[!htbp]
\centering
\begin{tabular}{|c|c|}
\hline
No. of random features ($l$) & Ratio of time taken by RFF-based test \\
\hline
\begin{filecontents*}{"Results/Gaussscaleshift/time_comparisons.CSV"}
No. of random features (l),Ratio of compute time b/w efficient and original test
1,0.296731459029116
3,0.314812932189351
5,0.332686928167132
7,0.352229974447427
9,0.39783293346122
\end{filecontents*}
\csvreader[head to column names, 
           late after line=\\\hline]
{"Results/Gaussscaleshift/time_comparisons.CSV"}{1=\one,2=\two}
{
           \one & \num[round-precision=2,round-mode=places]\two
}
\hline
\end{tabular}
\caption{Table for comparison of computation times for Gaussian scale shift experiments.}
\label{tab: Table for comparison of computation times for Gaussian scale shift experiments}
\end{table}

%\FloatBarrier

\subsection{Cauchy median shift}

In the third set of simulation-based experiments, we consider $P$ as a Cauchy distribution with median $0$ and identity scale, while $Q$ is considered to be a Cauchy distribution with median $\mu$ and identity scale. Here, we consider the class of median-shifted alternatives, and we use the choices $\mu\in\{ 0,0.05,0.1,0.3,0.5,0.7,1\}$ as the value of the median shift for our experiments.

We consider the sample size to be $N=M=500$ and data dimensions to be $d=1,10,20,50,100$. We choose $s=50$. All experiments are performed using the Gaussian kernel. Collection of bandwidths that we adapt over is $W = \left\{10^{-2 + 0.5 i }: i=0,1,\dots,8\right\}$, while the set of values of the regularization parameter that we adapt over is given by $\Lambda = \left\{10^{- 6 + 0.75 i}: i=0,1,\dots,9\right\}$. For the RFF-based Kernel Adaptive Test, we consider $F=\left\{1,3,5,7,9\right\}$. The number of permutations for the RFF-based Kernel Adaptive Test is chosen to be $B=800$, while that for the \say{exact} Adaptive Test is chosen to be $B^{\prime}=450$.

\begin{figure}[!tbp]
	\centering
	\includegraphics[width=0.7\linewidth]{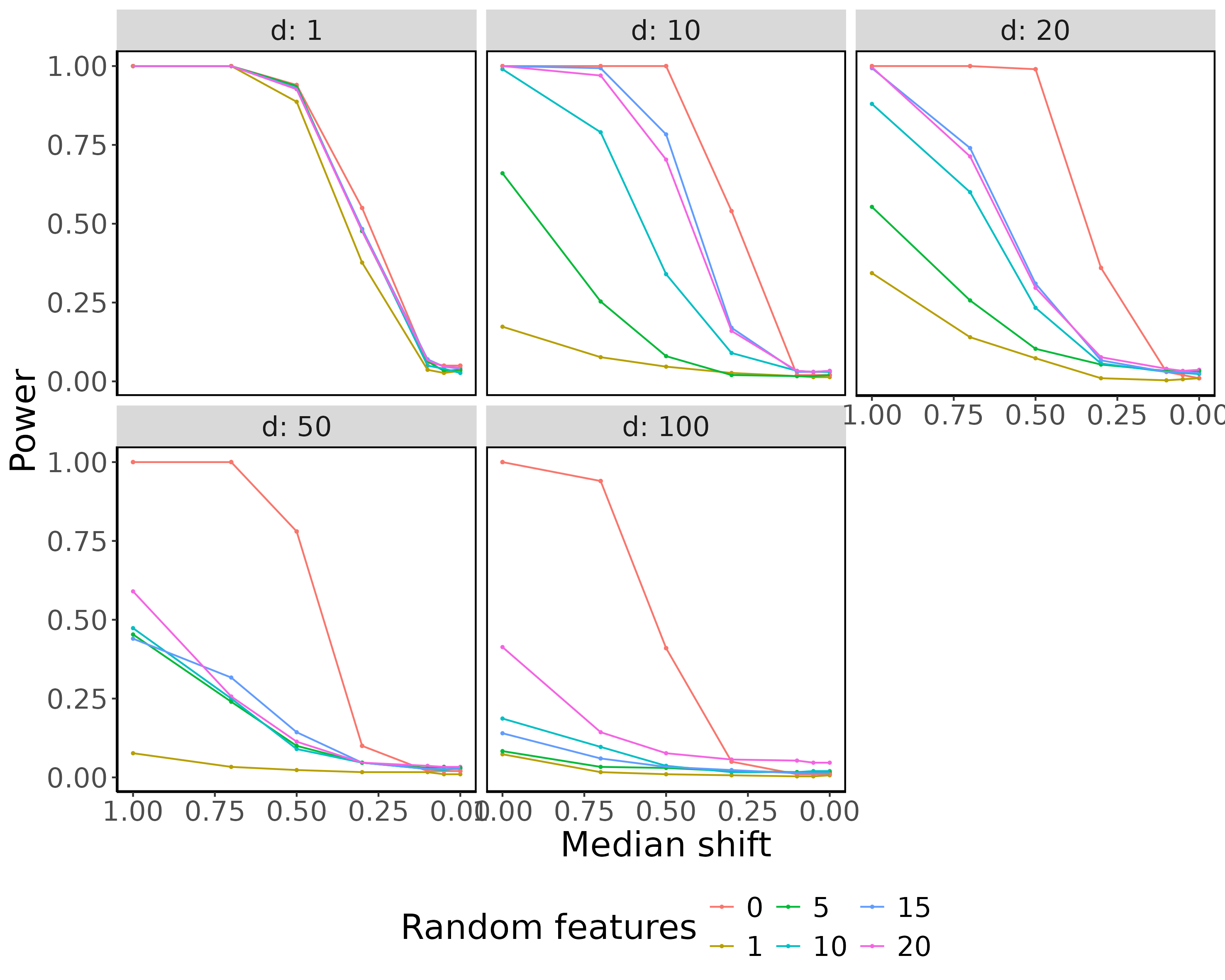}
	\caption{Empirical power for Cauchy median shift experiments.}
	\label{fig: Empirical power for Cauchy median shift experiments}
\end{figure}

From Figure \ref{fig: Empirical power for Cauchy median shift experiments}, we can observe that a relatively larger number of random Fourier features (more than or equal to 10) is required to ensure that the power of the RFF-based Kernel Adaptive Test is close to that of the \say{exact} Adaptive Test, especially when the data dimension is high. On the other hand, based on Figure \ref{fig: Time comparison for Cauchy median shift experiments} and Table \ref{tab: Table for comparison of computation times for Cauchy median shift experiments}, the RFF-based Kernel Adaptive Test extensively only takes 5-6$\%$ of the computation time required by the \say{exact} Adaptive Test. It is possible that a larger number of random Fourier features (more than 30 or so) may lead the RFF-based Kernel Adaptive Test to a more favorable trade-off between test power and computational efficiency for the current experimental setting.

\begin{figure}[!htbp]
	\centering
	\includegraphics[width=0.7\linewidth]{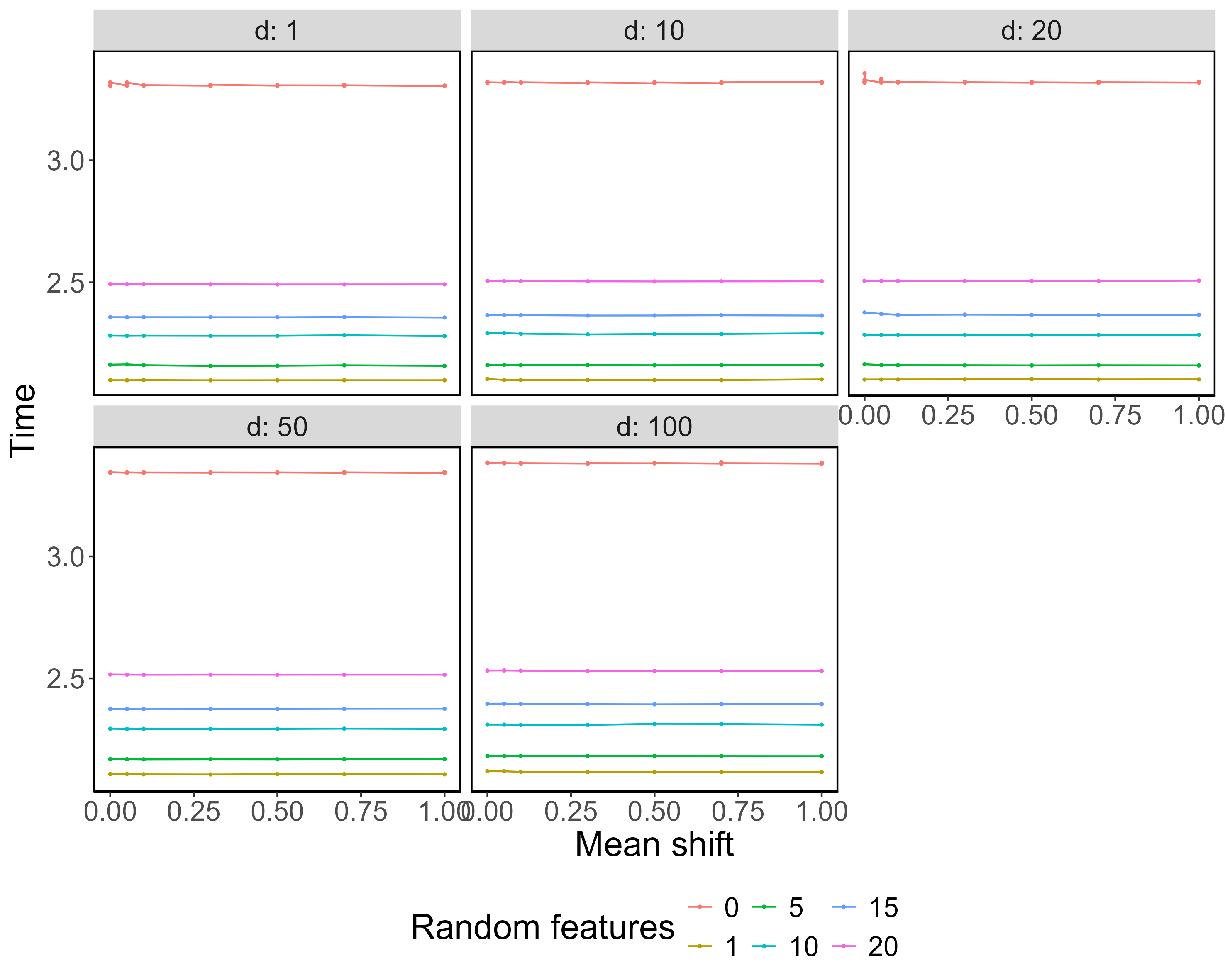}
	\caption{Comparison of computation time (in log seconds) for Cauchy median shift experiments.}
	\label{fig: Time comparison for Cauchy median shift experiments}
\end{figure}

\begin{table}[!htbp]
\centering
\begin{tabular}{|c|c|}
\hline
No. of random features ($l$) & Ratio of time taken by RFF-based test \\
\hline
\begin{filecontents*}{"Results/Cauchymedianshift/time_comparisons.CSV"}
No. of random features (l),Ratio of compute time b/w efficient and original test
1,0.0620506636716876
5,0.0606219934596438
10,0.0543900185690148
15,0.0608158404966831
20,0.0581904947067245
\end{filecontents*}
\csvreader[head to column names, 
           late after line=\\\hline]
{"Results/Cauchymedianshift/time_comparisons.CSV"}{1=\one,2=\two}
{
           \one & \num[round-precision=2,round-mode=places]\two
}
\hline
\end{tabular}
\caption{Table for comparison of computation times for Cauchy median shift experiments.}
\label{tab: Table for comparison of computation times for Cauchy median shift experiments}
\end{table}

\FloatBarrier

\subsection{MNIST dataset}

The MNIST dataset (\cite{Mnist}) is a collection of black-and-white handwritten digits from $0$ to $9$, which is one of the most popular datasets in Machine Learning. Analogous to the experimental setup considered in \cite{SpectralTwoSampleTest} and \cite{schrab2023mmd}, the images were downsampled to $7 \times 7$ pixels, leading to each image being embedded in $\R^{d}$ for $d=49$. We define the set $P$ to be the distribution of images of the digits
$$
P: 0,1,2,3,4,5,6,7,8,9
$$
and $Q_i$ for $i=1,2,3,4,5$ are defined as distributions over different subsets of digits from $0$ to $9$, given as follows:
$$
\begin{gathered}
Q_1: 1,3,5,7,9, \quad Q_2: 0,1,3,5,7,9, \quad Q_3: 0,1,2,3,5,7,9, \\
Q_4: 0,1,2,3,4,5,7,9, \quad Q_5: 0,1,2,3,4,5,6,7,9.
\end{gathered}
$$
Clearly, $Q_{i}$ becomes harder to distinguish from $P$ as $i$ increases from $1$ to $5$. We consider $N=M=500$ samples drawn with replacement from $P$ while testing against $Q_i$ for $i=1,2,3,4,5$. We choose $s=50$. Collection of bandwidths that we adapt over is $W = \left\{10^{-2 + 0.5 i}: i=0,1,\dots,8\right\}$, while the set of $\lambda$ that we adapt over is given by \\$\Lambda = \left\{10^{- 6 + 0.75 i}: i=0,1,\dots,9\right\}$. For the RFF-based Kernel Adaptive Test, we consider $F=\left\{1,3,5,7,9\right\}$. The number of permutations for the RFF-based Kernel Adaptive Test is chosen to be $B=550$, while the number of permutations for the \say{exact} Adaptive Test is chosen to be $B^{\prime}=350$.

We consider two sets of experiments: one using the Gaussian kernel and the other using the Laplace kernel.

\subsubsection{Results using Gaussian kernel}

\begin{figure}[!htbp]
	\centering
	\includegraphics[width=0.5\linewidth]{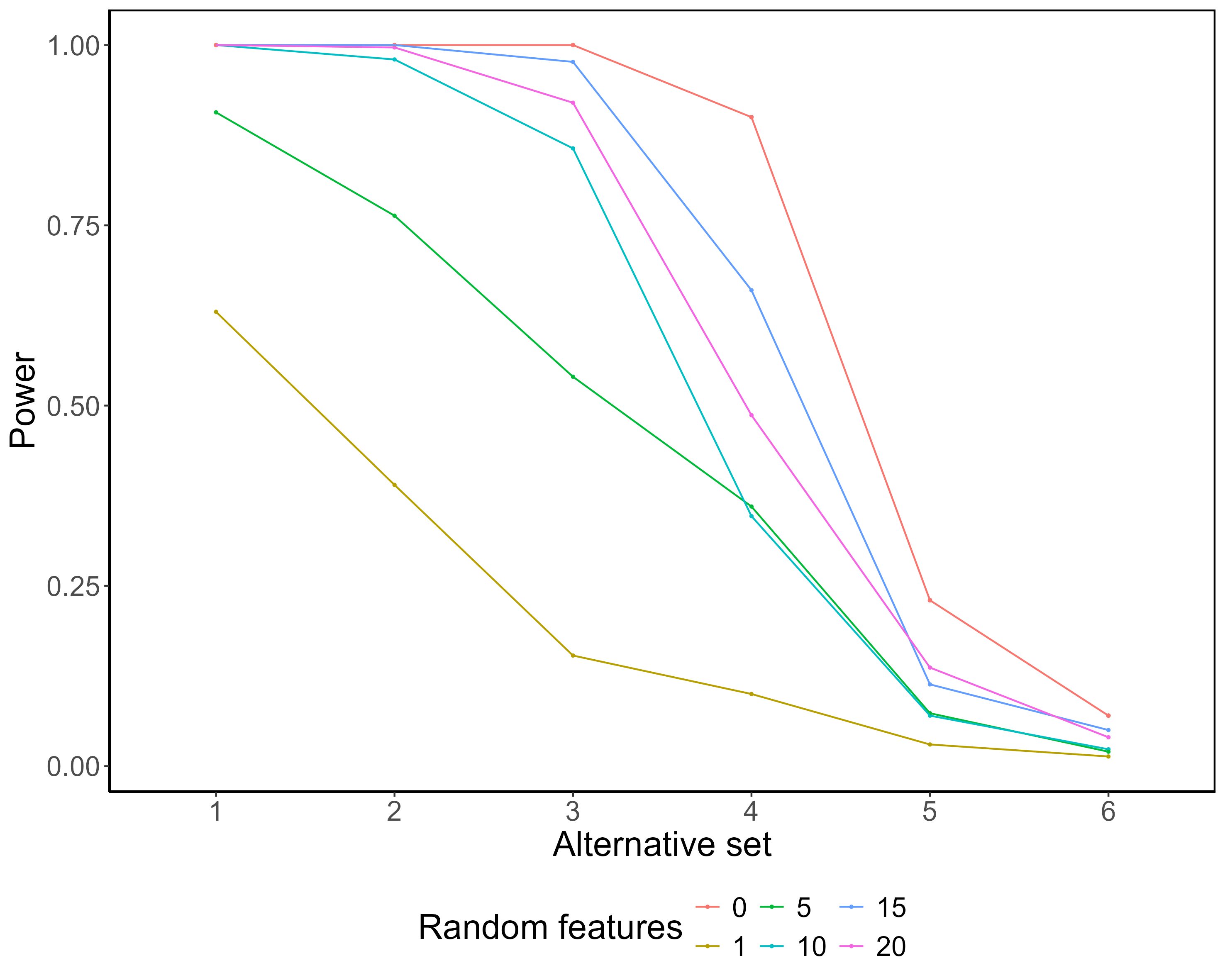}
	\caption{Empirical power for MNIST experiments using a Gaussian kernel.}
	\label{fig: Empirical power for MNIST experiments using Gaussian kernel}
\end{figure}

From Figure \ref{fig: Empirical power for MNIST experiments using Gaussian kernel}, we can observe that a moderately large number of random Fourier features (more than or equal to 15) is required to ensure that the power of the RFF-based Kernel Adaptive Test is close to that of the \say{exact} Adaptive test. Most importantly, based on Figure \ref{fig: Time comparison for MNIST experiments using Gaussian kernel} and Table \ref{tab: Table for comparison of computation times for MNIST experiments using Gaussian kernel}, the RFF-based Kernel Adaptive Test compensates more than adequately for the slight loss in power by taking around only 5-15$\%$ of the computation time required by the \say{exact} Adaptive Test. Therefore, a very favorable trade-off between test power and computational efficiency is achieved by the RFF-based Kernel Adaptive Test, as demonstrated through these experiments.

\begin{figure}[!htbp]
	\centering
	\includegraphics[width=0.5\linewidth]{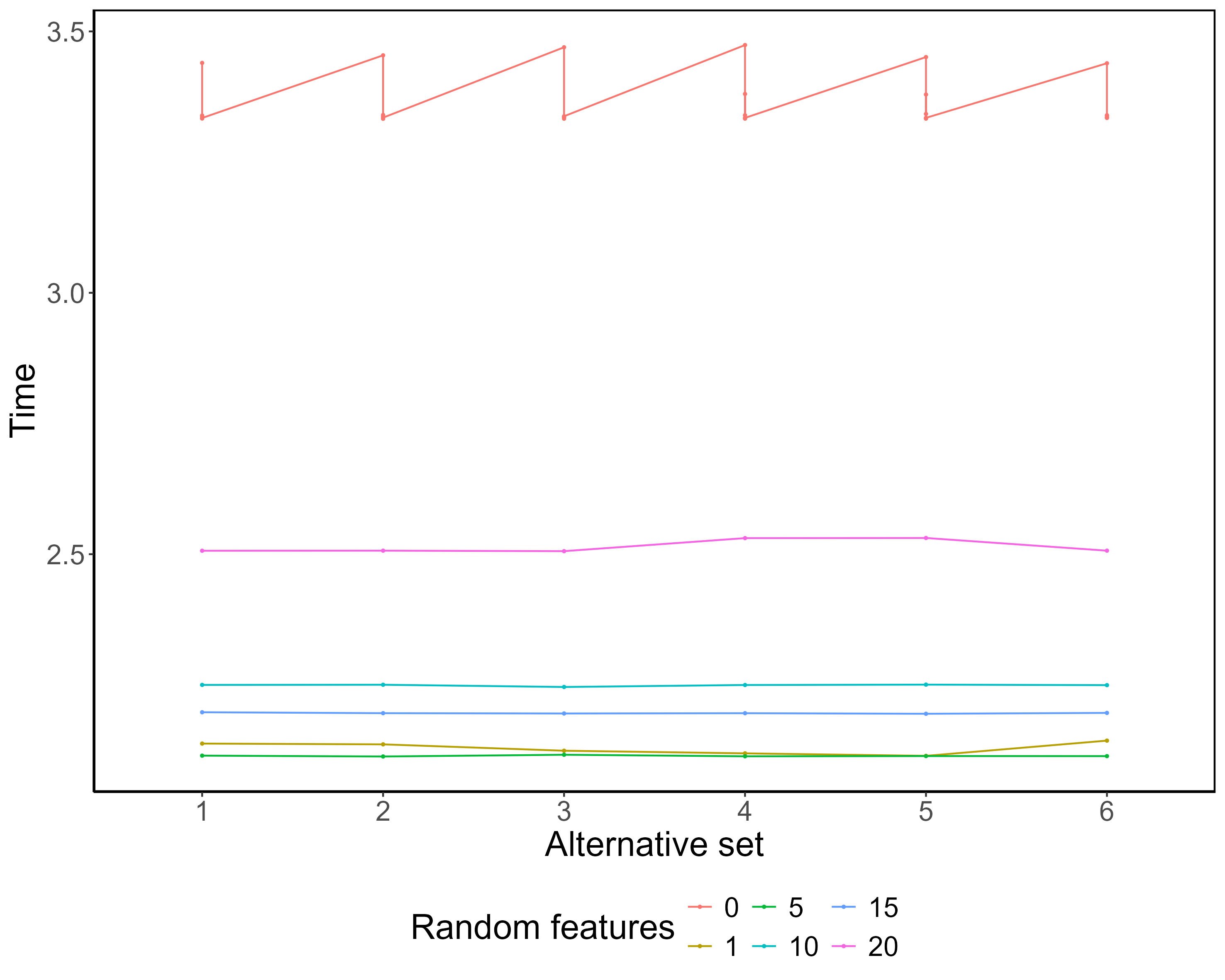}
	\caption{Comparison of computation time (in log seconds) for MNIST experiments using a Gaussian kernel.}
	\label{fig: Time comparison for MNIST experiments using Gaussian kernel}
\end{figure}

\begin{table}[!htbp]
\centering
\begin{tabular}{|c|c|}
\hline
No. of random features (l) & Ratio of time taken by RFF-based test \\
\hline
\begin{filecontents*}{"Results/MNISTGausskernel/time_comparisons.CSV"}
No. of random features (l),Ratio of compute time b/w efficient and original test
1,0.0472565630509046
5,0.0814409016122347
10,0.0728683526220269
15,0.145168864185975
20,0.060122664497626
\end{filecontents*}
\csvreader[head to column names, 
           late after line=\\\hline]
{"Results/MNISTGausskernel/time_comparisons.CSV"}{1=\one,2=\two}
{
           \one & \num[round-precision=2,round-mode=places]\two
}
\hline
\end{tabular}
\caption{Table for comparison of computation times for MNIST experiments using Gaussian kernel}
\label{tab: Table for comparison of computation times for MNIST experiments using Gaussian kernel}
\end{table}

\subsubsection{Results using Laplace kernel}

\begin{figure}[!htbp]
	\centering
	\includegraphics[width=0.5\linewidth]{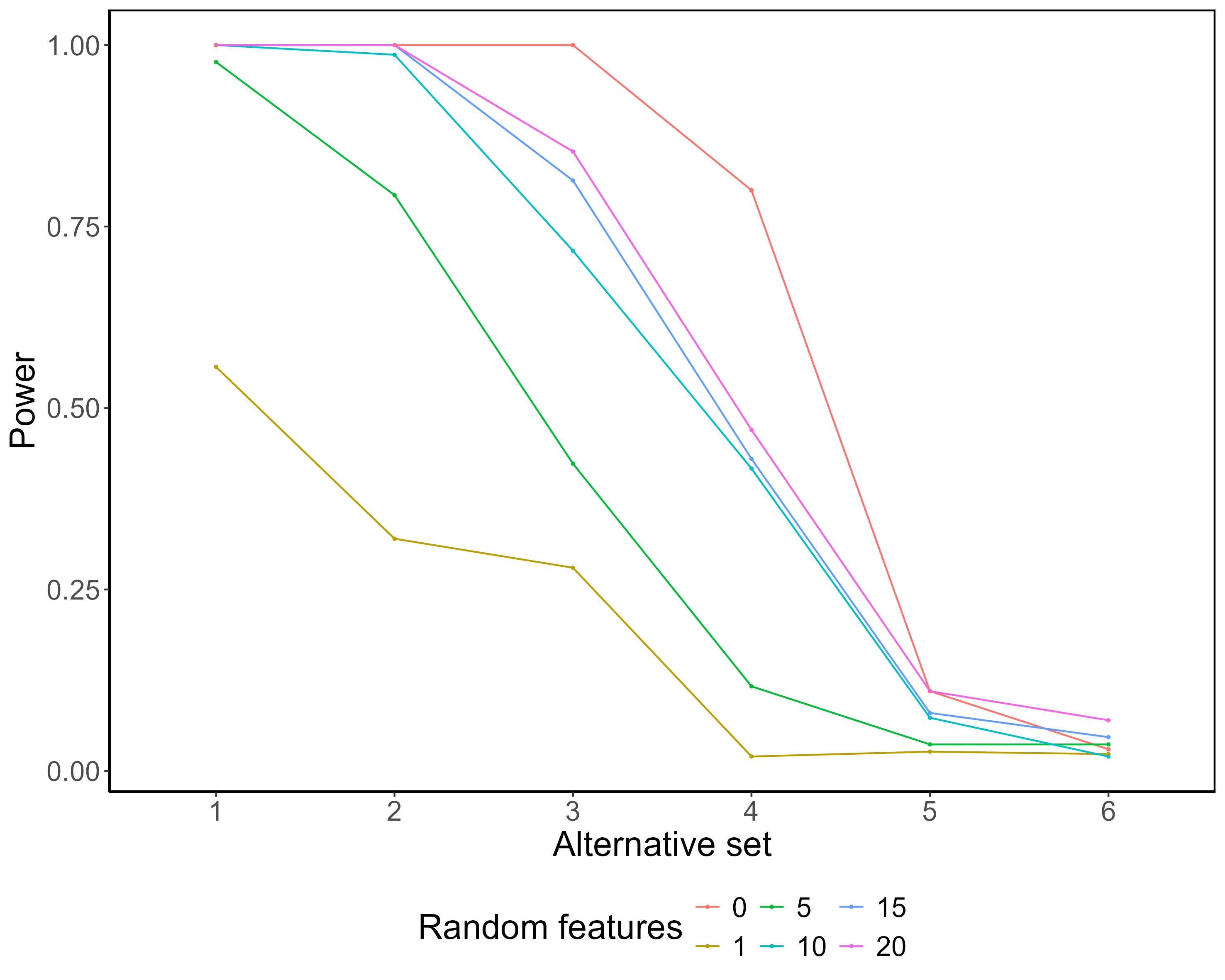}
	\caption{Empirical power for MNIST experiments using Laplace kernel.}
	\label{fig: Empirical power for MNIST experiments using Laplace kernel}
\end{figure}

From Figure \ref{fig: Empirical power for MNIST experiments using Laplace kernel}, we can observe that a moderately large number of random Fourier features (more than or equal to 15) is required to ensure that the power of the RFF-based Kernel Adaptive Test is close to that of the \say{exact} Adaptive Test. Most importantly, based on Figure \ref{fig: Time comparison for MNIST experiments using Laplace kernel} and Table \ref{tab: Table for comparison of computation times for MNIST experiments using Laplace kernel}, the RFF-based Kernel Adaptive Test compensates more than adequately for the slight loss in power by taking around 7-15$\%$ of the computation time required by the \say{exact} Adaptive Test. Therefore, a very favorable trade-off between test power and computational efficiency is achieved by the RFF-based Kernel Adaptive Test, as demonstrated through these experiments.

\begin{figure}[!htbp]
	\centering
	\includegraphics[width=0.5\linewidth]{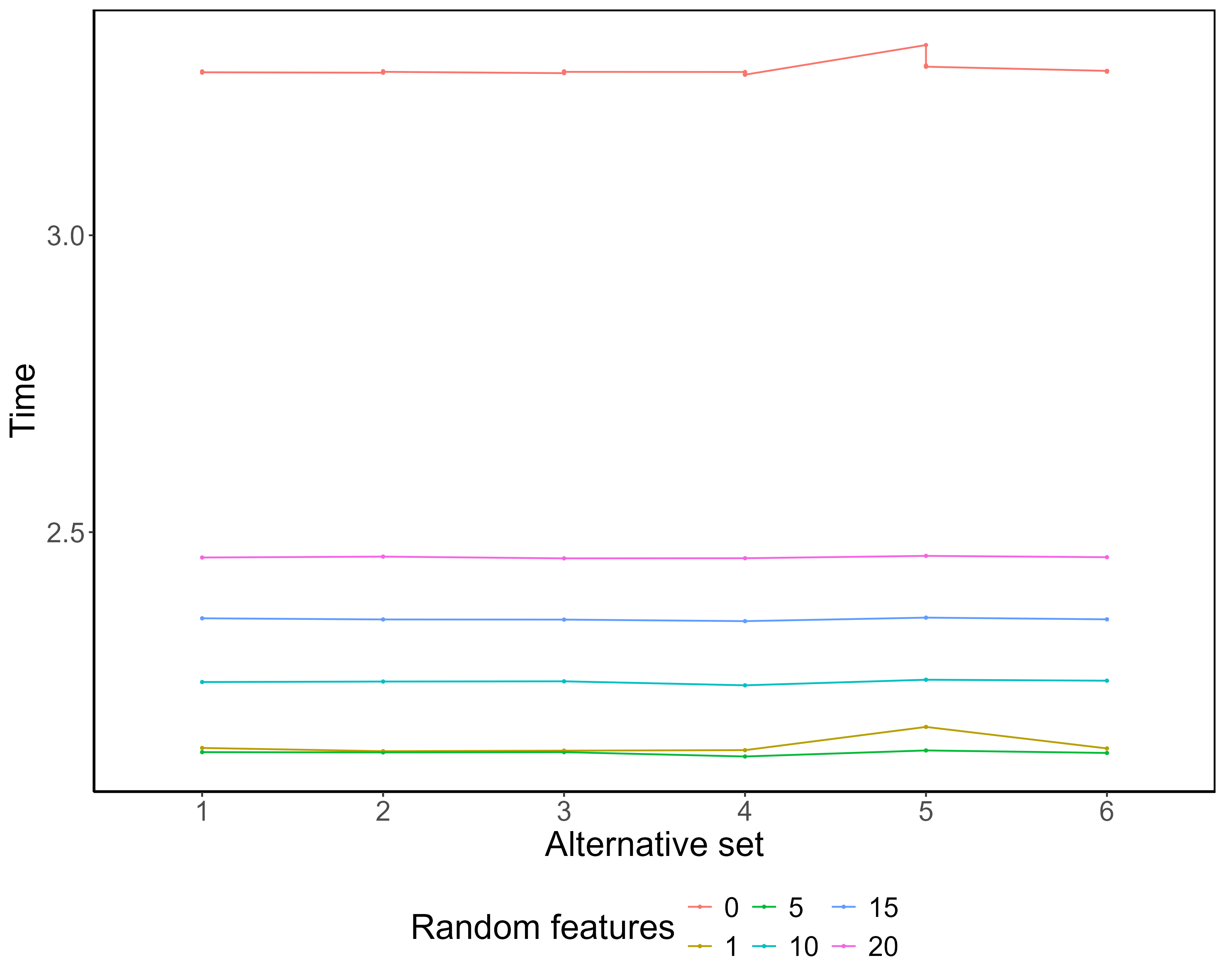}
	\caption{Comparison of computation time (in log seconds) for MNIST experiments using Laplace kernel.}
	\label{fig: Time comparison for MNIST experiments using Laplace kernel}
\end{figure}

\begin{table}[!htbp]
\centering
\begin{tabular}{|c|c|}
\hline
No. of random features ($l$) & Ratio of time taken by RFF-based test \\
\hline
\begin{filecontents*}{"Results/MNISTLapkernel/time_comparisons.CSV"}
No. of random features (l),Ratio of compute time b/w efficient and original test
1,0.072069150757817
5,0.093684364419171
10,0.119403179034101
15,0.151380968200208
20,0.0712188089483738
\end{filecontents*}
\csvreader[head to column names, 
           late after line=\\\hline]
{"Results/MNISTLapkernel/time_comparisons.CSV"}{1=\one,2=\two}
{
           \one & \num[round-precision=2,round-mode=places]\two
}
\hline
\end{tabular}
\caption{Table for comparison of computation times for MNIST experiments using Laplace kernel.}
\label{tab: Table for comparison of computation times for MNIST experiments using Laplace kernel}
\end{table}

% \subsection{Perturbed uniform}

\section{Conclusion}

In this work, we introduced a two-sample test employing a spectral regularization framework with random Fourier feature (RFF) approximation. We analyzed the trade-offs between statistical optimality and computational cost. We showed that the test achieves minimax optimality provided the RFF approximation order - governed by the smoothness of the likelihood ratio deviation and the decay of the integral operator’s eigenvalues - is sufficiently large. We then proposed a practical permutation-based implementation that adaptively selects the regularization parameter. Finally, through experiments on both simulated and benchmark datasets, we illustrated that the RFF-based test is computationally efficient and achieves performance comparable to the exact test, with only a minor reduction in power in many scenarios. In addition to Random Fourier Features, alternative approximation techniques such as the Nystr\"{o}m method could be investigated for similar purposes. Exploring these methods and their computational versus statistical tradeoffs remains an intriguing avenue for future research. Another interesting direction for future research would be to incorporate the idea of cheap permutation testing proposed by \cite{Cheappermutationtesting} in place of the vanilla permutation test currently used in the paper. This allows for additional computational speedup of the test. The goal, then, is to investigate the computational-statistical trade-off behavior of the resultant test that is based on cheap permutation and random Fourier features/Nystr\"{o}m method.

\section{Proofs}\label{Proofs of main theorems, corollaries and propositions}
%\subsection{Proofs of main theorems, corollaries and propositions}\label{Proofs of main theorems, corollaries and propositions}
In this section, we provide the proofs of the main theorems and corollaries. %In addition, we also provide the statements and proofs of the main propositions in this section.

\subsection{Proof of Theorem \ref{Type-I error bound of Oracle Test in terms of N2}}\label{subsec:thm1}
% \begin{proof}
Let us define $\gamma_{1,l}:=\frac{2 \sqrt{6}(C_{1}+C_{2}) \mathcal{N}_{2,l}(\lambda)}{\sqrt{\delta}}\left(\frac{1}{n}+\frac{1}{m}\right)$ and set $\delta=\frac{\alpha}{2}$. Then, we have
\[
\begin{aligned}
P_{H_{0}}\left\{\hat{\eta}_{\lambda,l} \leq \gamma\right\} & \geq P_{H_{0}}\left\{\left\{\hat{\eta}_{\lambda,l} \leq \gamma_{1,l}\right\} \cap\left\{\gamma_{1,l} \leq \gamma\right\}\right\}  \geq 1-P_{H_{0}}\left\{\hat{\eta}_{\lambda,l} \geq \gamma_{1,l}\right\}-P_{H_{0}}\left\{\gamma_{1,l} \geq \gamma\right\}\\
& \stackrel{(a)}{\geq} 1-2\delta=1-\alpha,
\end{aligned}
\]
where $(a)$ follows using Proposition \ref{Proposition: Type-I error bound with random threshold} and Lemma \ref{Lemma 13}.

\subsection{Proof of Theorem \ref{Power analysis of Oracle test}}\label{subsec:thm2}
% \begin{proof}
Let us define $\zeta_{l}=\mathbb{E}_{P^{n} \times Q^{m}}\left(\hat{\eta}_{\lambda,l} \mid\mathbb{Z}^{1:s},\theta^{1:l}\right)=\left\|g_{\lambda}^{1/2}(\hat{\Sigma}_{P Q,l})\left(\mu_{P,l}-\mu_{Q,l}\right)\right\|_{\mathcal{H}_{l}}^{2}$. Further, define $N_{2}^{*}(\kappa,\lambda,\alpha,l) = \frac{4\sqrt{2\kappa \mathcal{N}_{1}(\lambda)\log \frac{8}{\alpha}}}{\sqrt{\lambda l}} + \frac{16\kappa \log \frac{8}{\alpha}}{\lambda l}  + 2\sqrt{2}\mathcal{N}_{2}(\lambda)$, 
%$N_{2}^{*}(\kappa,\lambda,\delta,l) = \frac{4\sqrt{2\kappa \textcolor{blue}{\mathcal{N}_{1}(\lambda)} \log \frac{4}{\delta}}}{\sqrt{\lambda l}} + \frac{16\kappa \log \frac{4}{\delta}}{\lambda l}  + 2\sqrt{2}N_{2}(\lambda)$, 
$T_{1}=\zeta_{l} - \sqrt{\frac{\operatorname{Var}\left(\hat{\eta}_{\lambda,l} \mid \mathbb{Z}^{1:s},\theta^{1:l}\right)}{\delta}}$ and $D^{\prime} = \frac{D - d_{2}}{1-d_{2}}$. Then, clearly we have, $\gamma=\frac{ 4\sqrt{3}(C_{1}+C_{2}) N_{2}^{*}(\kappa,\lambda,\alpha,l)}{\sqrt{\alpha}}\left(\frac{1}{n}+\frac{1}{m}\right)$. Provided
\begin{equation}
\label{First threshold for gamma}
P_{H_{1}}\left\{\gamma > T_{1}\right\} \leq 3\delta
\end{equation}
holds for any $(P, Q) \in \mathcal{P}$ (i.e., under the condition when $H_{1}$ is true and the pair of distribution $(P,Q)$ belongs to the collection of $\Delta_{N,M}$-separated alternatives as defined in \eqref{Class of alternatives}) under the conditions stated in Theorem \ref{Power analysis of Oracle test}, we obtain $P_{H_{1}}\left\{\hat{\eta}_{\lambda,l} \geq \gamma\right\} \geq 1-4 \delta$ through the application of Lemma \ref{Lemma 1}. Taking the infimum over $(P, Q) \in \mathcal{P}$, the result stated in Theorem \ref{Power analysis of Oracle test} is obtained. 
Therefore, to complete the proof, it remains to verify that \eqref{First threshold for gamma} holds under the conditions of this theorem, which we do below. 

Let us define the quantity $\mathcal{M}_{l}=\hat{\Sigma}_{P Q, \lambda, l}^{-1 / 2} \Sigma_{P Q, \lambda, l}^{1/2}$ and the events $E_{1}=\left\{\mathcal{N}_{2,l}(\lambda) \leq N_{2}^{*}(\kappa,\lambda,2\delta,l)\right\}$, $E_{2}=\left\{\zeta_{l} \geq c_2\left\|\mathcal{M}_{l}^{-1}\right\|_{\mathcal{L}^{\infty}(\mathcal{H}_{l})}^{-2}\|u\|_{L^2(R)}^2\right\}$ and $E_{3} = \left\{\sqrt{\frac{2}{3}}\leq\|\mathcal{M}_{l}\|_{\mathcal{L}^{\infty}(\mathcal{H}_{l})} \leq \sqrt{2}\right\}$, where $c_{2}:=\frac{C_{4}^{2}}{2(C_{1}+C_{2})}$. Under $\boldsymbol{(\RFFAssumptionone)}$  and $\boldsymbol{(\RFFAssumptionfour)}$ and using Lemma \ref{Lemma 13}, we have that, if $\frac{86 \kappa}{l}\log \frac{32 \kappa l}{\delta} \leq \lambda \leq \|\Sigma_{PQ}\|_{\mathcal{L}^{\infty}(\mathcal{H})}$, then 
\begin{equation}
\label{Probability of E1}
P_{H_{1}}(E_{1}^{c}) = P(E_{1}^{c}) \leq \delta.
\end{equation}
For $(P,Q) \in \mathcal{P}$, we have  $u = \frac{dP}{dR} - 1 \in \operatorname{Ran}(\mathcal{T}_{PQ}^{\theta})$. Further, under $\boldsymbol{(\SpectralAssumptionone)}$, $\boldsymbol{(\SpectralAssumptiontwo)}$ and $\boldsymbol{(\SpectralAssumptionthree)}$ along with the conditions $\|u\|_{L^{2}(R)}^{2} \geq 16 \lambda^{2 \theta}\|\mathcal{T}_{PQ}^{-\theta} u\|_{L^{2}(R)}^{2}$ and $l \geq \max{\left(160,3200 \mathcal{N}_{1}(\lambda)\right)}\frac{ \kappa \log \frac{2}{\delta}}{\lambda }$, by employing Proposition~\ref{Proposition: Upper and lower bound of eta} and Lemma~\ref{Lemma 10}, we obtain 
\begin{equation}
\label{Probability of E2}
P_{H_{1}}(E_{2}^{c}) \leq \delta.
\end{equation}
% If \textcolor{blue}{$u \in \operatorname{Ran}\left(\mathcal{T}^{\theta}\right)$,$\|u\|_{L^{2}(R)}^{2} \geq 16 \lambda^{2 \theta}\left\|\mathcal{T}^{-\theta} u\right\|_{L^{2}(R)}^{2},\lambda \gtrsim l^{-\frac{\beta}{\beta +1}}$} (under polynomial decay) and \textcolor{blue}{$u \in \operatorname{Ran}\left(\mathcal{T}^{\theta}\right)$,$\|u\|_{L^{2}(R)}^{2} \geq 16 \lambda^{2 \theta}\left\|\mathcal{T}^{-\theta} u\right\|_{L^{2}(R)}^{2},\lambda > \frac{\log l}{l}$} (under exponential decay) [The general conditions required are $u \in \operatorname{Ran}\left(\mathcal{T}^{\theta}\right)$,$\|u\|_{L^{2}(R)}^{2} \geq 16 \lambda^{2 \theta}\left\|\mathcal{T}^{-\theta} u\right\|_{L^{2}(R)}^{2}$
% , $l > \frac{160\sqrt{2} \kappa \log (\frac{2}{\delta})}{\lambda }$ and $l > \frac{6400 \kappa N_{1}(\lambda) \log(\frac{2}{\delta})}{\lambda}$ ], using Proposition \eqref{Proposition: Upper and lower bound of eta} and Lemma \eqref{Lemma 10}, we have that \begin{equation}P_{H_{1}}(E_{2}^{c}) \leq \delta.\end{equation}
Following the proof of Proposition \ref{Proposition: Type-I error bound with random threshold}, specifically the proof of \eqref{Bound on probability of Ml given E}, we have that, if $n,m \geq 2$,
$\frac{140 \kappa}{s} \log \frac{32 \kappa s}{1-\sqrt{1-\delta}} \leq \lambda \leq \frac{1}{2}\left\|\Sigma_{P Q}\right\|_{\mathcal{L}^{\infty}(\mathcal{H})}
$
and $l\geq \max\left\{2\log\frac{2}{1-\sqrt{1-\delta}},\frac{128\kappa^{2}\log\frac{2}{1-\sqrt{1-\delta}}}{\left\|\Sigma_{P Q}\right\|_{\mathcal{L}^{\infty}(\mathcal{H})}^{2}}\right\}$, then,
\begin{equation}
\label{Probability of E3}
P_{H_{1}}(E_{3}^{c}) = P(E_{3}^{c}) \leq \delta.
\end{equation}

Let us define the event $E^{*} = \left\{\gamma \geq T_{1}\right\}$. Provided that the occurrence of the events $E_{1}$, $E_{2}$ and $E_{3}$ imply that event $E^{*}$ cannot occur under the conditions of Theorem \ref{Power analysis of Oracle test}, i.e., $E_1\cap E_2\cap E_3\subset (E^*)^c$, and using \eqref{Probability of E1}, \eqref{Probability of E2} and \eqref{Probability of E3}, we have that 
\begin{equation*}
\label{Probability of Estar}
P_{H_{1}}(E^{*}) \leq P_{H_{1}}(E_{1}^{c} \cup E_{2}^{c} \cup E_{3}^{c} ) \leq P(E_{1}^{c})+P_{H_{1}}(E_{2}^{c}) +P(E_{3}^{c}) \leq 3\delta.
\end{equation*}
Therefore, to complete the proof of this theorem, we only need to prove that the simultaneous occurrence of the events $E_{1}$, $E_{2}$ and $E_{3}$ precludes the occurrence of the event $E^{*}$ under the conditions specified in this theorem, i.e., the event $(E^{*})^{c}= \left\{\gamma < T_{1}\right\}$ occurs, or equivalently
\begin{equation}
\label{Events E1, E2 and E3 imply gamma less than T1}
E_{1} \cap E_{2} \cap E_{3}\subset (E^{*})^{c}.
\end{equation}
Note that, provided the event $E_{1}$ occurs, under $\boldsymbol{(\RFFAssumptionone)}$ and $\boldsymbol{(\RFFAssumptionfour)}$, $$N_{2}^{\prime}(\kappa,\lambda,\delta,l)\coloneq\frac{2N_{2}^{*}(\kappa,\lambda,\delta,l)\kappa}{\lambda}$$ is an upper bound on $C_{\lambda,l}=\frac{2\mathcal{N}_{2,l}(\lambda)}{\lambda} \sup _{x}\|K_{l}(\cdot, x)\|_{\mathcal{H}_{l}}^{2}$ as defined in Lemma \ref{Lemma 9}. Let us define $$\gamma_{l}:=\frac{1}{\sqrt{\delta}}\left(\frac{\sqrt{N_{2}^{\prime}(\kappa,\lambda,\delta,l)}\|u\|_{L^2(R)}+N_{2}^{*}(\kappa,\lambda,\delta,l)}{n+m}+\frac{N_{2}^{\prime}(\kappa,\lambda,\delta,l)^{1 / 4}\|u\|_{L^2(R)}^{3 / 2}+\|u\|_{L^2(R)}}{\sqrt{n+m}}\right)$$ and $T_{2}: = \zeta_{l} - \tilde{C}^{1/2}\|\mathcal{M}_{l}\|_{\mathcal{L}^{\infty}(\mathcal{H}_{l})}^2 \gamma_{l}$ where $\tilde{C}$ is a constant defined in Lemma \ref{Lemma 11} that depends only on $C_{1}, C_{2}$ and $D^{\prime}$. Further, let us define the event $E^{\prime} = \left\{\gamma > T_{2}\right\}$.

Now, under $\boldsymbol{(\Samplesizeassumption)}$ and the choice of the sample splitting size $s=d_{1}N=d_{2}M$ for estimating the covariance operator $\Sigma_{PQ,l}$ as stated in Theorem \ref{Power analysis of Oracle test}, we have that $m \leq n \leq D^{\prime} m$ where $D^{\prime} = \frac{D-d_{2}}{1-d_{2}}\geq 1$ is a constant. Therefore, using Lemma~\ref{Lemma 11} under  $\boldsymbol{(\SpectralAssumptionone)}$ and $\boldsymbol{(\SpectralAssumptiontwo)}$ and provided the events $E_{1}$ and $E_{2}$ occur simultaneously, we observe that $T_{2} \leq T_{1}$ and consequently, the occurrence of the event $(E^{\prime})^{c} = \left\{\gamma \leq T_{2}\right\}$ implies the occurrence of the event $(E^{*})^{c} = \left\{\gamma \leq T_{1}\right\}$. Therefore, it is sufficient to show that the simultaneous occurrence of the events $E_{1}$, $E_{2}$ and $E_{3}$ precludes the occurrence of the event $E^{\prime}$ under the conditions specified in this theorem, i.e. the event $(E^{\prime})^{c}= \left\{\gamma \leq T_{2}\right\}$ occurs, or equivalently
\begin{equation}
\label{Events E1, E2 and E3 imply gamma less than T2}
E_{1} \cap E_{2} \cap E_{3} \subset (E^{\prime})^{c}.
\end{equation}
When the event $E_{3}$ occurs, using the fact that $\norm{\mathcal{M}_{l}^{-1}}_{\mathcal{L}^{\infty}(\mathcal{H}_{l})}\geq \frac{1}{\norm{\mathcal{M}_{l}}_{\mathcal{L}^{\infty}(\mathcal{H}_{l})}}$, we obtain $\norm{\mathcal{M}_{l}}_{\mathcal{L}^{\infty}(\mathcal{H}_{l})}^{2} \leq 2$, $\norm{\mathcal{M}_{l}^{-1}}_{\mathcal{L}^{\infty}(\mathcal{H}_{l})}^{2} \leq \frac{3}{2}$, and consequently, we must have \begin{equation}
\label{Bound on sum of function of Ml and ML inverse under event E3}
\frac{\left\|\mathcal{M}_{l}^{-1}\right\|_{\mathcal{L}^{\infty}(\mathcal{H}_{l})}^2}{3}+\frac{\left\|\mathcal{M}_{l}^{-1}\right\|_{\mathcal{L}^{\infty}(\mathcal{H}_{l})}^2\|\mathcal{M}_{l}\|_{\mathcal{L}^{\infty}(\mathcal{H}_{l})}^2}{6} \leq 1.
\end{equation}
Suppose we assume \begin{equation}\label{Sufficient condition 1 on norm of u to get function of Ml1 and ML inverse bounded above}
\|u\|_{L^2(R)}^2 \geq \frac{3 \gamma}{c_2} = \frac{ 12\sqrt{3}(C_{1}+C_{2}) N_{2}^{*}(\kappa,\lambda,\alpha,l)}{c_{2}\sqrt{\alpha}}\left(\frac{1}{n}+\frac{1}{m}\right),
\end{equation} 
and 
\begin{equation}\label{Sufficient condition 2 on norm of u to get function of Ml1 and ML inverse bounded above}
\begin{aligned}
&\|u\|_{L^2(R)}^2 \geq \frac{6 \tilde{C}^{1/2}\gamma_{l}}{c_{2}}\\
&= \frac{6\tilde{C}^{1/2}}{c_{2}\sqrt{\delta}}\left(\frac{\sqrt{N_{2}^{\prime}(\kappa,\lambda,\delta,l)}\|u\|_{L^2(R)}+N_{2}^{*}(\kappa,\lambda,\delta,l)}{n+m}+\frac{N_{2}^{\prime}(\kappa,\lambda,\delta,l)^{1 / 4}\|u\|_{L^2(R)}^{3 / 2}+\|u\|_{L^2(R)}}{\sqrt{n+m}}\right),
\end{aligned}
\end{equation}
which imply $\frac{\left\|\mathcal{M}_{l}^{-1}\right\|_{\mathcal{L}^{\infty}(\mathcal{H}_{l})}^2 \gamma}{c_2\|u\|_{L^2(R)}^2} \leq \frac{\left\|\mathcal{M}_{l}^{-1}\right\|_{\mathcal{L}^{\infty}(\mathcal{H}_{l})}^2}{3}$ 
and
$$\frac{\tilde{C}^{1/2} \gamma_{l}\left\|\mathcal{M}_{l}^{-1}\right\|_{\mathcal{L}^{\infty}(\mathcal{H}_{l})}^2\|\mathcal{M}_{l}\|_{\mathcal{L}^{\infty}(\mathcal{H}_{l})}^2}{c_2\|u\|_{L^2(R)}^2} \leq \frac{\left\|\mathcal{M}_{l}^{-1}\right\|_{\mathcal{L}^{\infty}(\mathcal{H}_{l})}^2\|\mathcal{M}_{l}\|_{\mathcal{L}^{\infty}(\mathcal{H}_{l})}^2}{6},$$ respectively. Therefore, it follows from \eqref{Bound on sum of function of Ml and ML inverse under event E3}, \eqref{Sufficient condition 1 on norm of u to get function of Ml1 and ML inverse bounded above}, and \eqref{Sufficient condition 2 on norm of u to get function of Ml1 and ML inverse bounded above},
we have
\begin{equation*}
\label{gamma less than T2 type bound version 1}
\frac{\left\|\mathcal{M}_{l}^{-1}\right\|_{\mathcal{L}^{\infty}(\mathcal{H}_{l})}^2 \gamma+\tilde{C}^{1/2} \gamma_{l}\left\|\mathcal{M}_{l}^{-1}\right\|_{\mathcal{L}^{\infty}(\mathcal{H}_{l})}^2\|\mathcal{M}_{l}\|_{\mathcal{L}^{\infty}(\mathcal{H}_{l})}^2}{c_2\|u\|_{L^2(R)}^2} \leq 1,
\end{equation*} which is equivalent to 
\begin{equation}
\label{gamma less than T2 type bound version 2}
\gamma \leq c_2\left\|\mathcal{M}_{l}^{-1}\right\|_{\mathcal{L}^{\infty}(\mathcal{H}_{l})}^{-2}\|u\|_{L^2(R)}^2-\tilde{C}^{1/2}\|\mathcal{M}_{l}\|_{\mathcal{L}^{\infty}(\mathcal{H}_{l})}^2 \gamma_{l}.
\end{equation}
Provided the event $E_{2}$ occurs, it follows from \eqref{gamma less than T2 type bound version 2} that $\gamma \leq T_{2}$. Therefore, \eqref{Events E1, E2 and E3 imply gamma less than T2} and consequently \eqref{Events E1, E2 and E3 imply gamma less than T1} is proved, if \eqref{Sufficient condition 1 on norm of u to get function of Ml1 and ML inverse bounded above} and \eqref{Sufficient condition 2 on norm of u to get function of Ml1 and ML inverse bounded above} are true. In the following, we show that the sufficient conditions mentioned in the statement of
Theorem \ref{Power analysis of Oracle test} are sufficient for \eqref{Probability of E1}, \eqref{Probability of E2}, \eqref{Probability of E3}, \eqref{Sufficient condition 1 on norm of u to get function of Ml1 and ML inverse bounded above} and \eqref{Sufficient condition 2 on norm of u to get function of Ml1 and ML inverse bounded above} to hold.
% Let us define $c_{1,\theta} = \underset{(P,Q) \in \mathcal{P}}{\sup} \norm{\mathcal{T}_{PQ}^{-\theta}u}_{L^{2}(R)}$

Let us define $c_{1} = \underset{\theta > 0}{\sup} \underset{(P,Q) \in \mathcal{P}}{\sup} \norm{\mathcal{T}_{PQ}^{-\theta}u}_{L^{2}(R)}$ which is assumed to be finite and $d_{\theta} = \left(\frac{1}{16c_{1}^{2}}\right)^{\frac{1}{2\theta}}$. Since $(P,Q) \in \mathcal{P}$ under $H_{1}$, we have that $\norm{u}^2_{L^{2}(R)} \geq \Delta_{N,M}$. Consequently, the choice \begin{equation}
\label{Choice of lambda in general case of power analysis of oracle test}
\lambda=d_{\theta}\Delta_{N,M}^{\frac{1}{2\theta}}\end{equation} implies that $\norm{u}^2_{L^{2}(R)} \geq 16 \lambda^{2 \theta} \norm{\mathcal{T}_{PQ}^{-\theta}u}_{L^{2}(R)}^{2}$ holds. Under
% $l > \max\left\{160\sqrt{2},6400 N_{1}\left(d_{\theta}\Delta_{N,M}^{\frac{1}{2\theta}}\right)\right\}\times \frac{ \kappa \log (\frac{2}{\delta})}{d_{\theta}\Delta_{N,M}^{\frac{1}{2\theta}}}$
% $\Delta_{N,M}^{\frac{1}{2\theta}} > \max\left\{d_{\theta}^{-1}l^{-1}160\sqrt{2} \kappa \log (\frac{2}{\delta}) ,d_{\theta}^{-1}l^{-1} 6400 \kappa \log (\frac{2}{\delta}) N_{1}\left(\left(d_{\theta}\Delta_{N,M}\right)^{1/2\theta}\right)\right\}$
this choice of $\lambda$ as given in \eqref{Choice of lambda in general case of power analysis of oracle test}, the conditions $$\Delta_{N,M}^{\frac{1}{2\theta}} \geq \frac{d_{\theta}^{-1}160\kappa \log \frac{2}{\delta}}{l}\quad\text{and}\quad\frac{\Delta_{N,M}^{\frac{1}{2\theta}} }{\mathcal{N}_{1}\left(d_{\theta}\Delta_{N,M}^{1/2\theta}\right)} \geq  \frac{d_{\theta}^{-1} 3200 \kappa \log (\frac{2}{\delta})}{l}$$ are sufficient to ensure that \eqref{Probability of E2} holds. The conditions $$\lambda = d_{\theta}\Delta_{N,M}^{\frac{1}{2\theta}} \leq \frac{1}{2}\|\Sigma_{PQ}\|_{\mathcal{L}^{\infty}(\mathcal{H})}\quad \text{and}\quad\Delta_{N,M} \geq \left(d_{\theta}^{-1}\frac{86 \kappa}{l}\log \frac{32 \kappa l}{\delta}\right)^{2\theta}$$ are sufficient to ensure that \eqref{Probability of E1} holds. 
Note that, since $\frac{d_{1}(N+M)}{2} \leq s \leq \frac{d_{2}(N+M)}{2} $ and $(N+M) \geq \frac{32\kappa d_{2}}{\delta}$, \eqref{Probability of E3} holds if $$l\geq \max\left\{2\log\frac{2}{1-\sqrt{1-\delta}},\frac{128\kappa^{2}\log\frac{2}{1-\sqrt{1-\delta}}}{\left\|\Sigma_{P Q}\right\|_{\mathcal{L}^{\infty}(\mathcal{H})}^{2}}\right\},\,d_{\theta}\Delta_{N,M}^{\frac{1}{2\theta}} \geq \frac{560\kappa\log(N+M)}{d_{1}(N+M)},$$ $\frac{1}{2}\|\Sigma_{PQ}\|_{\mathcal{L}^{\infty}(\mathcal{H})} \geq d_{\theta}\Delta_{N,M}^{\frac{1}{2\theta}}$ and $n,m \geq 2$. 

Note that, $(n+m)=(1-d_1)N+(1-d_2)M \geq (1-d_2)(N+M)$, where $1\geq d_2\geq d_1\geq 0$ and using Lemma A.13 in \citep{SpectralTwoSampleTest}, we have that $\frac{1}{n} + \frac{1}{m} \leq \frac{2(D^{\prime}+1)}{(1-d_{2})(N+M)}$. 
Therefore, \eqref{Sufficient condition 1 on norm of u to get function of Ml1 and ML inverse bounded above} and
\eqref{Sufficient condition 2 on norm of u to get function of Ml1 and ML inverse bounded above} hold if $\mathcal{N}_{2}(\lambda)=\mathcal{N}_{2}(d_{\theta}\Delta_{N,M}^{\frac{1}{2\theta}}) \geq 1$, and the conditions \emph{3.--8.} in the statement of Theorem~\ref{Power analysis of Oracle test} hold.

\subsection{Proof of Corollary \ref{Power Analysis of Oracle test polynomial decay}}\label{subsec:cor3}
% \begin{proof}
%We assume the setting of Theorem \ref{Power analysis of Oracle test} is valid. 
Under the polynomial decay of the eigenvalues of $\Sigma_{PQ}$ i.e. $\lambda_{i} \asymp i^{-\beta}$ for $\beta>1$, using Lemma \ref{Lemma 12}(i) and Lemma C.9 from \cite{ApproximateKernelPCARandomFeaturesStergeSriperumbudur}, we have that,
\begin{equation}
\label{N2 simplified under polynomial decay}
    \mathcal{N}_{2}(\lambda) \asymp \lambda^{-\frac{1}{2\beta}}
\end{equation}
and
\begin{equation}
\label{N1 simplified under polynomial decay}
    \mathcal{N}_{1}(\lambda) \asymp \lambda^{-\frac{1}{\beta}}.
\end{equation}
Using \eqref{N2 simplified under polynomial decay} and \eqref{N1 simplified under polynomial decay},
%and provided $N+M \geq k(\alpha,\delta,\theta,\beta)$ for some constant $k(\alpha,\delta,\theta,\beta) \in \mathbb{N}$ depending on  $\alpha,\delta,\theta$ and $\beta$, 
the conditions 
%\ref{Theorem Power Analysis Oracle Test Condition 1} to \ref{Theorem Power Analysis Oracle Test Condition 10} as specified 
in Theorem~\ref{Power analysis of Oracle test} reduce to 
\begin{equation}
\label{Oracle test Power analysis conditions simplified under polynomial decay 1}
    \lambda = d_{\theta}\Delta_{N,M}^{\frac{1}{2\theta}} \gtrsim \max\left\{\circled{\emph{\small{1}}},\circled{\emph{\small{2}}},\circled{\emph{\small{3}}},\circled{\emph{\small{4}}},\circled{\emph{\small{5}}},\circled{\emph{\small{6}}},\circled{\emph{\small{7}}},\circled{\emph{\small{8}}},\circled{\emph{\small{9}}},\circled{\emph{\small{10}}}\right\},
\end{equation}
where the constant depending on $d_\theta,\,\alpha,\,\delta,\,\beta,\,\kappa$ is absorbed in $\gtrsim$ and
$$\circled{\emph{\small{1}}} = \frac{\log(N+M)}{(N+M)},\,  \circled{\emph{\small{2}}} = \frac{1}{l},\, %\frac{\log (\frac{2}{\delta})}{l}, \, 
\circled{\emph{\small{3}}} =\left[\frac{1}{l}\right]^{\frac{\beta}{\beta+1}},\, %\left[\frac{\log (\frac{2}{\delta})}{l}\right]^{\frac{\beta}{\beta+1}}, \,
\circled{\emph{\small{4}}} = \frac{\log l}{l},\,\circled{\emph{\small{5}}}=\left[ \frac{1}{\sqrt{l}(N+M)}\right]^{\frac{2}{1+4\theta}},
%\frac{\log (\frac{l}{\delta})}{l},
$$ 
$$\circled{\emph{\small{6}}} = \left[\frac{1}{l(N+M)}\right]^{\frac{1}{1+2\theta}},\,\circled{\emph{\small{7}}}= \left[\frac{1}{N+M}\right]^{\frac{2\beta}{1+4\beta\theta}},\,\circled{\emph{\small{8}}} = \left[\frac{1}{\sqrt{l}(N+M)^{2}}\right]^{\frac{2}{3+4\theta}},
$$
$$\circled{\emph{\small{9}}} = \left[\frac{1}{l(N+M)^{2}}\right]^{\frac{1}{2(1+\theta)}},\quad\text{and}\quad
\circled{\emph{\small{10}}} = \left[\frac{1}{(N+M)^2}\right]^{\frac{2\beta}{1+2\beta+4\beta\theta}}.
$$
%$$\circled{\emph{\small{5}}}=\max\left\{d_{\theta}^{\frac{4\theta}{1+4\theta}},d_{\theta}^{\frac{4\theta-3}{1+4\theta}}\right\}\left[\left(\sqrt{\frac{\log(\frac{8}{\alpha})}{\alpha}} + \frac{\sqrt{\log(\frac{4}{\delta})}}{\delta^{2}}\right) \frac{1}{\sqrt{l}(N+M)}\right]^{\frac{2}{1+4\theta}},$$
%$\circled{\emph{\small{6}}} = d_{\theta}^{\frac{2\theta}{1+2\theta}}\left[\left(\frac{\log(\frac{8}{\alpha})}{\sqrt{\alpha}} + \frac{\log(\frac{4}{\delta})}{\delta^{2}}\right)\frac{1}{l(N+M)}\right]^{\frac{1}{1+2\theta}}$,
%$\circled{\emph{\small{7}}} = d_{\theta}^{\frac{4\beta\theta}{1+4\beta\theta}}\left[\frac{\alpha^{-1/2} + \delta^{-2}}{(N+M)}\right]^{\frac{2\beta}{1+4\beta\theta}}$,\\
%$\circled{\emph{\small{8}}} = d_{\theta}^{\frac{4\theta}{3+4\theta}}\left[\frac{\sqrt{\log(\frac{4}{\delta})}}{\delta}\times\frac{1}{\sqrt{l}(N+M)^{2}}\right]^{\frac{2}{3+4\theta}}$,
%$\circled{\emph{\small{9}}} = d_{\theta}^{\frac{\theta}{1+\theta}}\left[\frac{\log(\frac{4}{\delta})}{\delta}\times\frac{1}{l(N+M)^{2}}\right]^{\frac{1}{2(1+\theta)}}$ \\and
%$\circled{\emph{\small{10}}} = d_{\theta}^{\frac{4\beta\theta}{1+2\beta + 4\beta\theta}}\left[\frac{1}{\delta(N+M)^2}\right]^{\frac{2\beta}{1+2\beta+4\beta\theta}}$.
Note that $\circled{\emph{\small{3}}} \gtrsim \circled{\emph{\small{4}}} \gtrsim \circled{\emph{\small{2}}}$. Moreover, it can be verified that
\begin{equation*}
\circled{\emph{\small{7}}}\gtrsim
\begin{cases}
\circled{\emph{\small{3}}}, & l\gtrsim (N+M)^{\frac{2(\beta+1)}{1+4\beta\theta}}\\
\circled{\emph{\small{5}}}, & l\gtrsim (N+M)^{\frac{2(\beta-1)}{1+4\beta\theta}}\\
\circled{\emph{\small{6}}}, & l\gtrsim (N+M)^{\frac{2\beta-1}{1+4\beta\theta}}\\
\circled{\emph{\small{8}}}, & l\gtrsim (N+M)^{\frac{2(3\beta-4\theta\beta-2)}{1+4\beta\theta}}\\
\circled{\emph{\small{9}}}, & l\gtrsim (N+M)^{\frac{2(2\beta-2\theta\beta-1)}{1+4\beta\theta}}
\end{cases},
\end{equation*}
$\circled{\emph{\small{7}}}\gtrsim \circled{\emph{\small{1}}}$, $\circled{\emph{\small{7}}}\gtrsim \circled{\emph{\small{10}}}$ if $\theta>\frac{1}{2}-\frac{1}{4\beta}$. Therefore, if $\theta>\frac{1}{2}-\frac{1}{4\beta}$ and $l\gtrsim(N+M)^{\frac{2(\beta+1)}{1+4\theta\beta}}$, then  $\circled{\emph{\small{7}}}$ dominates. 
Similarly, it can be verified that
\begin{equation*}
\circled{\emph{\small{1}}}\gtrsim
\begin{cases}
\circled{\emph{\small{3}}}, & l\gtrsim \left[\frac{N+M}{\log (N+M)}\right]^{\frac{\beta+1}{\beta}}\\
\circled{\emph{\small{5}}}, & l\gtrsim (N+M)^{4\theta-1}\left[\log(N+M)\right]^{-(1+4\theta)}\\
\circled{\emph{\small{6}}}, & l\gtrsim (N+M)^{2\theta}\left[\log(N+M)\right]^{-(1+2\theta)}\\
\circled{\emph{\small{8}}}, & l\gtrsim (N+M)^{4\theta-1}\left[\log(N+M)\right]^{-(3+4\theta)}\\
\circled{\emph{\small{9}}}, & l\gtrsim (N+M)^{2\theta}\left[\log(N+M)\right]^{-2(1+\theta)}
\end{cases},
\end{equation*}
$\circled{\emph{\small{1}}}\gtrsim \circled{\emph{\small{7}}}$, $\circled{\emph{\small{1}}}\gtrsim \circled{\emph{\small{10}}}$ if $\theta\le \frac{1}{2}-\frac{1}{4\beta}$ and $N+M$ is large enough. Therefore, for large enough $N+M$, if  $\theta\le\frac{1}{2}-\frac{1}{4\beta}$ and $l\gtrsim \left[\frac{N+M}{\log (N+M)}\right]^{\frac{\beta+1}{\beta}}$, then  $\circled{\emph{\small{1}}}$ dominates, and the result follows.

\subsection{Proof of Corollary \ref{Power Analysis of Oracle test exponential decay}}\label{subsec:cor4}
% \begin{proof}
%We assume the setting of Theorem \ref{Power analysis of Oracle test} is valid. 
Under exponential decay of the eigenvalues of $\Sigma_{PQ}$ i.e. $\lambda_{i} \asymp e^{-\tau i}$ for $\tau>0$, using Lemma \ref{Lemma 12}(ii) and Lemma C.9 from \cite{ApproximateKernelPCARandomFeaturesStergeSriperumbudur}, we have that,
\begin{equation}
\label{N2 simplified under exponential decay}
    \mathcal{N}_{2}(\lambda) \asymp\sqrt{\log\frac{1}{\lambda}}
\end{equation}
and
\begin{equation}
\label{N1 simplified under exponential decay}
    \mathcal{N}_{1}(\lambda) \asymp \log\frac{1}{\lambda}.
\end{equation}
Note that, based on Remark \ref{Remark: Conditions for N2l to behave as N2}, if \begin{equation}e^{-1} \geq \lambda \gtrsim \frac{\left[\log\frac{4}{\delta}+\log\frac{8}{\alpha}\right]\log l}{l},\label{eq:cor-2-lambda}
\end{equation} and 
\begin{equation}l \geq 2\bar{C}e\left[\log\frac{4}{\delta}+\log\frac{8}{\alpha}\right] \left[\log\left(\log\frac{4}{\delta}+\log\frac{8}{\alpha}\right)+1\right] \label{eq:cor-2-l}
\end{equation}
for some universal constant $\bar{C} \geq 1$, conditions \ref{Theorem Power Analysis Oracle Test Condition 5} and \ref{Theorem Power Analysis Oracle Test Condition 6} of Theorem \ref{Power analysis of Oracle test} are automatically satisfied if condition \ref{Theorem Power Analysis Oracle Test Condition 7} is satisfied, while conditions \ref{Theorem Power Analysis Oracle Test Condition 8} and \ref{Theorem Power Analysis Oracle Test Condition 9} are automatically satisfied if condition \ref{Theorem Power Analysis Oracle Test Condition 10} is satisfied. 
Using \eqref{N2 simplified under exponential decay} and \eqref{N1 simplified under exponential decay} and, provided $N+M \geq k(\alpha,\delta,\theta)$ for some constant $k(\alpha,\delta,\theta) \in \mathbb{N}$ depending on  $\alpha$, $\delta$ and $\theta$, the conditions \ref{Theorem Power Analysis Oracle Test Condition 3}, \ref{Theorem Power Analysis Oracle Test Condition 7} and \ref{Theorem Power Analysis Oracle Test Condition 10} reduce to the following after taking into account the constant factors (independent of $N,\,M$ or $l$, but are functions of $\alpha$, $\delta$ and $\theta$):
\begin{equation}
\label{Power Analysis Oracle test Condition 3 simplified under exponential decay}
\frac{\Delta_{N,M}^{\frac{1}{2\theta}}}{\log(d_{\theta}^{-1}\Delta_{N,M}^{-\frac{1}{2\theta}})} \gtrsim %d_{\theta}^{-1}
\frac{\log (\frac{2}{\delta})}{l},
\end{equation}

\begin{equation}
\label{Power Analysis Oracle test Condition 7 simplified under exponential decay}
\frac{\Delta_{N,M}}{\sqrt{\log(d_{\theta}^{-1}\Delta_{N,M}^{-\frac{1}{2\theta}})}} \gtrsim %\left(\frac{\log(\frac{8}{\alpha})}{\sqrt{\alpha}} + \frac{\log(\frac{4}{\delta})}{\delta^{2}}\right)
\frac{1}{N+M},
\end{equation}
and
\begin{equation}
\label{Power Analysis Oracle test Condition 10 simplified under exponential decay}
\frac{\Delta_{N,M}^{\frac{1+2\theta}{2\theta}}}{\sqrt{\log(d_{\theta}^{-1}\Delta_{N,M}^{-\frac{1}{2\theta}})}} \gtrsim %%d_{\theta}^{-1} \times
%\frac{1}{\delta(N+M)^2}.
\frac{1}{(N+M)^2}.
\end{equation} 
\eqref{Power Analysis Oracle test Condition 3 simplified under exponential decay} holds if \eqref{eq:cor-2-lambda} and \eqref{eq:cor-2-l} hold. 
%$e^{-1} \geq \lambda \gtrsim \frac{\left[\log(\frac{4}{\delta})+\log(\frac{8}{\alpha})\right]\log l}{l}$ and\\ $l \geq 2\bar{C}\left[\log(\frac{4}{\delta})+\log(\frac{8}{\alpha})\right]e \left[\log\left[\log(\frac{4}{\delta})+\log(\frac{8}{\alpha})\right]+1\right]$ for some universal constant $\bar{C} \geq 1$. 
% Furthermore, the conditions \ref{Theorem Power Analysis Oracle Test Condition 1}, \ref{Theorem Power Analysis Oracle Test Condition 2} and \ref{Theorem Power Analysis Oracle Test Condition 4} reduce to 
Furthermore, Condition \ref{Theorem Power Analysis Oracle Test Condition 1} reduces to 

\begin{equation}
\label{Oracle test Power analysis conditions 1, 2 and 4 simplified under exponential decay}
    \lambda = d_{\theta}\Delta_{N,M}^{\frac{1}{2\theta}} \gtrsim \max\left\{\circled{\emph{\small{1}}},\circled{\emph{\small{2}}},\circled{\emph{\small{4}}}\right\},
\end{equation}
% where $\circled{\emph{\small{1}}} = \frac{\log(N+M)}{(N+M)}$,  $\circled{\emph{\small{2}}} = \frac{\log (\frac{2}{\delta})}{l}$ and $\circled{\emph{\small{4}}} = \frac{\log (\frac{l}{\delta})}{l}$. Now, $\circled{\emph{\small{4}}} \gtrsim \circled{\emph{\small{2}}}$ and $\log(l) \gtrsim \log(\frac{l}{\delta})$ if $l \geq 2$. Further, if $\Delta_{N,M} \gtrsim \max\left\{\frac{1}{\sqrt{2\theta}},1\right\}\times \left(\frac{\log(\frac{8}{\alpha})}{\sqrt{\alpha}} + \frac{\log(\frac{4}{\delta})}{\delta^{2}}\right)\times\frac{\sqrt{\log(N+M)}}{N+M}$, then condition \ref{Theorem Power Analysis Oracle Test Condition 7} is satisfied, while condition \ref{Theorem Power Analysis Oracle Test Condition 10} is satisfied if $\Delta_{N,M} \gtrsim \max\left\{\frac{1}{(\frac{1}{2}+\theta)^{\frac{\theta}{(1+2\theta)}}},1\right\}\times d_{\theta}^{-\frac{2\theta}{1+2\theta}} \times\frac{\left[\log(N+M)\right]^{\frac{\theta}{1+2\theta}}}{\delta^{\frac{2\theta}{1+2\theta}}\left[N+M\right]^{\frac{4\theta}{1+2\theta}}}$. 
where $\circled{\emph{\small{1}}} = \frac{\log(N+M)}{(N+M)}$,  $\circled{\emph{\small{2}}} = \frac{1}{l}$ and $\circled{\emph{\small{4}}} = \frac{\log l}{l}$. Clearly, $\circled{\emph{\small{4}}} \gtrsim \circled{\emph{\small{2}}}$. Further, if $\Delta_{N,M} \gtrsim \frac{\sqrt{\log(N+M)}}{N+M}$, then condition \ref{Theorem Power Analysis Oracle Test Condition 7} is satisfied, while condition \ref{Theorem Power Analysis Oracle Test Condition 10} is satisfied if $\Delta_{N,M} \gtrsim\frac{\left[\log(N+M)\right]^{\frac{\theta}{1+2\theta}}}{\left[N+M\right]^{\frac{4\theta}{1+2\theta}}}$. Therefore, \eqref{Power Analysis Oracle test Condition 3 simplified under exponential decay}, \eqref{Power Analysis Oracle test Condition 7 simplified under exponential decay}, \eqref{Power Analysis Oracle test Condition 10 simplified under exponential decay} and \eqref{Oracle test Power analysis conditions 1, 2 and 4 simplified under exponential decay} are satisfied if the following condition holds:
% \begin{equation}
% \begin{aligned}
% \label{Oracle test Power analysis conditions simplified under exponential decay}
%     \Delta_{N,M} \gtrsim& \max\left\{\frac{1}{\sqrt{2\theta}},\frac{1}{(\frac{1}{2}+\theta)^{\frac{\theta}{(1+2\theta)}}},1\right\} \times \left(\frac{\log(\frac{8}{\alpha})}{\sqrt{\alpha}} + \frac{\log(\frac{4}{\delta})}{\delta^{2}}\right)^{\max\left(2\theta,1\right)}\times\max\left\{d_{\theta}^{-2\theta},d_{\theta}^{-\frac{2\theta}{1+2\theta}}\right\}\times\\
%     &\max\left\{\left[\frac{\log(N+M)}{N+M}\right]^{2\theta},\left[\frac{\log(l)}{l}\right]^{2\theta},\frac{\sqrt{\log(N+M)}}{N+M},\frac{\left[\log(N+M)\right]^{\frac{\theta}{1+2\theta}}}{\left[N+M\right]^{\frac{4\theta}{1+2\theta}}}\right\}.
% \end{aligned}
% \end{equation}
\begin{equation}
\begin{aligned}
\label{Oracle test Power analysis conditions simplified under exponential decay}
    \Delta_{N,M} \gtrsim& \max\left\{\left[\frac{\log(N+M)}{N+M}\right]^{2\theta},\left[\frac{\log l}{l}\right]^{2\theta},\frac{\sqrt{\log(N+M)}}{N+M},\frac{\left[\log(N+M)\right]^{\frac{\theta}{1+2\theta}}}{\left[N+M\right]^{\frac{4\theta}{1+2\theta}}}\right\}.
\end{aligned}
\end{equation}
Now, we consider two scenarios based on the values of the smoothness index $\theta$.\vspace{1mm}

\textbf{Case I: } Suppose $\theta > \frac{1}{2}$. 

Then, we have that $\frac{\sqrt{\log(N+M)}}{N+M} \gtrsim \max\left\{\left[\frac{\log(N+M)}{N+M}\right]^{2\theta},\frac{\left[\log(N+M)\right]^{\frac{\theta}{1+2\theta}}}{\left[N+M\right]^{\frac{4\theta}{1+2\theta}}}\right\}$. Further, $\frac{\sqrt{\log(N+M)}}{N+M}\gtrsim  \left[\frac{\log l}{l}\right]^{2\theta}$ if $l \gtrsim \left(N+M\right)^{\frac{1}{2\theta}}\log(N+M)^{1-\frac{1}{4\theta}}$. Therefore, provided $l \gtrsim \left(N+M\right)^{\frac{1}{2\theta}}\log(N+M)^{1-\frac{1}{4\theta}}$, \eqref{Oracle test Power analysis conditions simplified under exponential decay} reduces to 
\begin{equation}
\label{Oracle test Power analysis conditions simplified under exponential decay 1}
\Delta_{N,M} = c(\alpha,\delta,\theta)\frac{\sqrt{\log(N+M)}}{N+M},
\end{equation}
% ,where 
% \begin{equation}\label{Definition of constant in separation boundary, type-2 RFF Oracle exp decay}
% c(\alpha,\delta,\theta) =  \max\left\{\frac{1}{\sqrt{2\theta}},\frac{1}{(\frac{1}{2}+\theta)^{\frac{\theta}{(1+2\theta)}}},1\right\} \times \left(\frac{\log(\frac{8}{\alpha})}{\sqrt{\alpha}} + \frac{\log(\frac{4}{\delta})}{\delta^{2}}\right)^{\max\left(2\theta,1\right)}\times\max\left\{d_{\theta}^{-2\theta},d_{\theta}^{-\frac{2\theta}{1+2\theta}}\right\}
% \end{equation}
where $c(\alpha,\delta,\theta)$ is a positive constant that depends on $\alpha$, $\delta$ and $\theta$. 

% Further, Condition \ref{Theorem Power Analysis Oracle Test Condition 11} in Theorem \ref{Power analysis of Oracle test} is also satisfied.

\textbf{Case II: } Suppose $\theta \leq \frac{1}{2}$.

Then, we have that $ \left[\frac{\log(N+M)}{N+M}\right]^{2\theta} \gtrsim \max\left\{\frac{\sqrt{\log(N+M)}}{N+M},\frac{\left[\log(N+M)\right]^{\frac{\theta}{1+2\theta}}}{\left[N+M\right]^{\frac{4\theta}{1+2\theta}}}\right\}$. Further, \\$\left[\frac{\log(N+M)}{N+M}\right]^{2\theta} \gtrsim  \left[\frac{\log l}{l}\right]^{2\theta}$ if $l \gtrsim N+M$. Therefore, provided $l \gtrsim N+M$, \eqref{Oracle test Power analysis conditions simplified under exponential decay} reduces to 
\begin{equation*}
\label{Oracle test Power analysis conditions simplified under exponential decay 2}
\Delta_{N,M} = c(\alpha,\delta,\theta)\left[\frac{\log(N+M)}{N+M}\right]^{2\theta},
\end{equation*}
% , where $c(\alpha,\delta,\theta)$ is a positive constant that depends on $\alpha$, $\delta$ and $\theta$ as defined in \eqref{Definition of constant in separation boundary, type-2 RFF Oracle exp decay}.
where $c(\alpha,\delta,\theta)$ is the positive constant that depends on $\alpha$, $\delta$ and $\theta$ as used in \eqref{Oracle test Power analysis conditions simplified under exponential decay 1}.

% Further, Condition \ref{Theorem Power Analysis Oracle Test Condition 11} in Theorem \ref{Power analysis of Oracle test} is also satisfied. 

% \end{proof}

% \begin{theorem} (Type-I error for Permutation Test)
% \label{Type I-error of RFF-based permutation test} 

% Let $0<\alpha\leq 1$ be the chosen level of significance and $w,\tilde{w}$ be such that $0 < \tilde{w} < w  < \frac{1}{2}$. Define the RFF-based test statistic $\hat{\eta}_{\lambda,l}$ and the $(1-w\alpha)$-th quantile corresponding to the empirical distribution function for the RFF-based test statistic using $B$ randomly selected permutations of $(U_{i})_{i=1}^{n+m}$ as defined in Section \ref{subsection: Permutation test}. Provided the number of permutations $B \geq \frac{1}{2\tilde{w}^{2}\alpha^{2}}\log(\frac{2}{\alpha(1-w-\tilde{w})})$, the level-$\alpha$ critical region for testing $H_0 :P=Q$ vs $H_1:P\neq Q$ is given by $\left\{\hat{\eta}_{\lambda,l} \geq q_{1-w\alpha}^{B,\lambda,l}\right\}$ i.e, 
% \[P_{H_{0}}\left\{\hat{\eta}_{\lambda,l} \geq q_{1-w\alpha}^{B,\lambda,l}\right\} \leq \alpha\]
% \end{theorem}

\subsection{Proof of Theorem \ref{Type I-error of RFF-based permutation test}}\label{subsec:thm5}
% \begin{proof}
Conditional on $\theta^{1:l}$, the kernel $K_{l}$ and its corresponding RKHS $\mathcal{H}_{l}$ are fixed. From this point onwards, the proof is similar to that of Theorem 4.6 in \citep{SpectralTwoSampleTest} upon replacing all test statistics based on the kernel $K$ with test statistics based on the kernel $K_{l}$ and probabilities with conditional probabilities and using Lemma \ref{Lemma 15}. This leads us to obtain
\[
P_{H_{0}}\left\{\hat{\eta}_{\lambda,l} \geq q_{1-w\alpha}^{B,\lambda,l} \mid \theta^{1:l}\right\} \geq 1 - w\alpha - \alpha^{\prime} - (1-w-\tilde{w})\alpha .
\]
Using the monotonicity and the tower property of conditional expectations, we can remove the conditioning on the random features $\theta^{1:l}$ and obtain the desired result upon choosing $\alpha^{\prime}=\tilde{w}\alpha$.

\subsection{Proof of Theorem \ref{Power Analysis for RFF based test based on permutations}}\label{subsec:thm6}

Let us define $\zeta_{l}=\mathbb{E}_{P^{n} \times Q^{m}}\left(\hat{\eta}_{\lambda,l} \mid\mathbb{Z}^{1:s},\theta^{1:l}\right)=\left\|g_{\lambda}^{1/2}(\hat{\Sigma}_{P Q,l})\left(\mu_{P,l}-\mu_{Q,l}\right)\right\|_{\mathcal{H}_{l}}^{2}$ and $\tilde{\alpha} = (w-\tilde{w})\alpha$. Further, define $N_{2}^{*}(\kappa,\lambda,\tilde{\alpha},l) \coloneq \frac{4\sqrt{2\kappa \mathcal{N}_{1}(\lambda) \log \frac{8}{\tilde{\alpha}}}}{\sqrt{\lambda l}} + \frac{16\kappa \log \frac{8}{\tilde{\alpha}}}{\lambda l}  + 2\sqrt{2}\mathcal{N}_{2}(\lambda)$, $N_{2}^{*}(\kappa,\lambda,\delta,l) \coloneq \frac{4\sqrt{2\kappa \mathcal{N}_{1}(\lambda) \log \frac{4}{\delta}}}{\sqrt{\lambda l}} + \frac{16\kappa \log \frac{4}{\delta}}{\lambda l}  + 2\sqrt{2}N_{2}(\lambda)$, $N_{2}^{\prime}(\kappa,\lambda,\delta,l)\coloneq\frac{2N_{2}^{*}(\kappa,\lambda,\delta,l)\kappa}{\lambda}$, $T_{1}=\zeta_{l} - \sqrt{\frac{\operatorname{Var}\left(\hat{\eta}_{\lambda,l} \mid \mathbb{Z}^{1:s},\theta^{1:l}\right)}{\delta}}$ and $D^{\prime} = \frac{D - d_{2}}{1-d_{2}}$, $$\gamma_{1,l} =\frac{1}{\sqrt{\delta}}\left(\frac{\sqrt{N_{2}^{\prime}(\kappa,\lambda,\delta,l)}\|u\|_{L^2(R)}+N_{2}^{*}(\kappa,\lambda,\delta,l)}{n+m}+\frac{(N_{2}^{\prime}(\kappa,\lambda,\delta,l))^{1 / 4}\|u\|_{L^2(R)}^{3 / 2}+\|u\|_{L^2(R)}}{\sqrt{n+m}}\right),$$ $$\gamma_{2,l} = \frac{\log\frac{1}{\tilde{\alpha}}}{\sqrt{\delta}(n+m)}\left(\sqrt{N_{2}^{\prime}(\kappa,\lambda,\delta,l)}\norm{u}_{L^{2}(R)} + N_{2}^{*}(\kappa,\lambda,\delta,l) + (N_{2}^{\prime}(\kappa,\lambda,\delta,l))^{\frac{1}{4}}\norm{u}_{L^{2}(R)}^{\frac{3}{2}} + \norm{u}_{L^{2}(R)}\right),$$ and $$%\begin{aligned}$
\gamma_{3,l} =%&= 
\norm{\mathcal{M}_{l}}_{\mathcal{L}^{\infty}(\mathcal{H}_{l})}^{2}\gamma_{2,l} + \frac{\zeta_{l}\log\frac{1}{\tilde{\alpha}}}{\sqrt{\delta}(n+m)}.$$
% \\
% &= \frac{\norm{\mathcal{M}_{l}}_{\mathcal{L}^{\infty}(\mathcal{H}_{l})}^{2}\log(\frac{1}{\tilde{\alpha}})}{\sqrt{\delta}(n+m)}\left(\sqrt{N_{2}^{\prime}(\kappa,\lambda,\delta,l)}\norm{u}_{L^{2}(R)} + N_{2}^{*}(\kappa,\lambda,\delta,l)+ (N_{2}^{\prime}(\kappa,\lambda,\delta,l))^{\frac{1}{4}}\norm{u}_{L^{2}(R)}^{\frac{3}{2}} \right.\\
% &\left.+ \norm{u}_{L^{2}(R)}\right)+ \frac{\zeta_{l}\log(\frac{1}{\tilde{\alpha}})}{\sqrt{\delta}(n+m)}.
%\end{aligned}$
Provided
\begin{equation}
\label{First threshold for permutation quantile}
P_{H_{1}}\left\{q_{1-\tilde{\alpha}}^{\lambda,l} \geq T_{1}\right\} \leq 5\delta
\end{equation}
holds true for any $(P, Q) \in \mathcal{P}$ (i.e., under the condition when $H_{1}$ is true and the pair of distribution $(P,Q)$ belongs to the collection of $\Delta_{N,M}$-separated alternatives as defined in \eqref{Class of alternatives}) under the conditions stated in Theorem \ref{Power Analysis for RFF based test based on permutations}, we obtain 
\begin{equation}
\label{Second threshold for permutation quantile}
P_{H_{1}}\left\{\hat{\eta}_{\lambda,l} \geq q_{1-\tilde{\alpha}}^{\lambda,l}\right\} \geq 1-6 \delta
\end{equation} through the application of Lemma \ref{Lemma 1}. Setting $\alpha^{\prime}= \tilde{w}\alpha$ and using Lemma \ref{Lemma 15} and \eqref{Second threshold for permutation quantile}, provided the number of randomly selected permutations $B \geq \frac{\log\frac{2}{\delta}}{2\tilde{\alpha}^{2}}$, we have that
\[
\begin{aligned}
P_{H_{1}}(\hat{\eta}_{\lambda,l} \geq q_{1-\tilde{\alpha}}^{B,\lambda,l}) &\geq P_{H_{1}}\left\{ \left\{\hat{\eta}_{\lambda,l} \geq q_{1-w\alpha+\alpha^{\prime}}^{\lambda,l}\right\} \cap \left\{ q_{1-\alpha}^{B,\lambda,l} \leq q_{1-w\alpha+\alpha^{\prime}}^{\lambda,l}\right\} \right\}\\
& \geq 1- P_{H_{1}}(\hat{\eta}_{\lambda,l} < q_{1-w\alpha+\alpha^{\prime}}^{\lambda,l}) - P_{H_{1}}(q_{1-\alpha}^{B,\lambda,l} > q_{1-w\alpha+\alpha^{\prime}}^{\lambda,l})\\
&\geq 1- 6\delta - \delta\\
&= 1-7\delta.
\end{aligned}
\]
Taking the infimum over $(P, Q) \in \mathcal{P}$, the result stated in Theorem \ref{Power Analysis for RFF based test based on permutations} is obtained. Therefore, to complete the proof, it remains to verify that \eqref{First threshold for permutation quantile} holds under the conditions of this theorem, which we do below. 

Let us define the quantity $\mathcal{M}_{l}=\hat{\Sigma}_{P Q, \lambda, l}^{-1 / 2} \Sigma_{P Q, \lambda, l}^{1/2}$ and the events $E_{1}=\left\{\mathcal{N}_{2,l}(\lambda) \leq N_{2}^{*}(\kappa,\lambda,\delta,l)\right\}$, $E_{2}=\left\{\zeta_{l} \geq c_2\left\|\mathcal{M}_{l}^{-1}\right\|_{\mathcal{L}^{\infty}(\mathcal{H}_{l})}^{-2}\|u\|_{L^2(R)}^2\right\}$, $E_{3} = \left\{\sqrt{\frac{2}{3}}\leq\|\mathcal{M}_{l}\|_{\mathcal{L}^{\infty}(\mathcal{H}_{l})} \leq \sqrt{2}\right\}$, where $c_{2}=\frac{C_{4}^{2}}{2(C_{1}+C_{2})}$. Also define $E_{4} = \left\{q_{1-\tilde{\alpha}}^{\lambda,l} < C^{*} \gamma_{3,l}\right\}$. Under $\boldsymbol{(\RFFAssumptionone)}$ and $\boldsymbol{(\RFFAssumptionfour)}$, and using Lemma \ref{Lemma 13}, we have that, if $\frac{86 \kappa}{l}\log \frac{32 \kappa l}{\delta} \leq \lambda \leq \|\Sigma_{PQ}\|_{\mathcal{L}^{\infty}(\mathcal{H})}$, then 
\begin{equation}
\label{Probability of E1 RFF based permutation}
P_{H_{1}}(E_{1}^{c}) = P(E_{1}^{c}) \leq \delta.
\end{equation}
For $(P,Q) \in \mathcal{P}$, we have that $u = \frac{dP}{dR} - 1 \in \operatorname{Ran}(\mathcal{T}_{PQ}^{\theta})$. Further, under $\boldsymbol{(\SpectralAssumptionone)}$,$\boldsymbol{(\SpectralAssumptiontwo)}$ and $\boldsymbol{(\SpectralAssumptionfour)}$ along with the conditions $\|u\|_{L^{2}(R)}^{2} \geq 16 \lambda^{2 \theta}\|\mathcal{T}_{PQ}^{-\theta} u\|_{L^{2}(R)}^{2}$ and $l \geq \max{\left(160,3200 N_{1}(\lambda)\right)}\times \frac{ \kappa \log (\frac{2}{\delta})}{\lambda }$, by employing Proposition~\ref{Proposition: Upper and lower bound of eta} and Lemma~\ref{Lemma 10}, we have that 
\begin{equation}
\label{Probability of E2 RFF based permutation}
P_{H_{1}}(E_{2}^{c}) \leq \delta.
\end{equation}

% If \textcolor{blue}{$u \in \operatorname{Ran}\left(\mathcal{T}^{\theta}\right)$,$\|u\|_{L^{2}(R)}^{2} \geq 16 \lambda^{2 \theta}\left\|\mathcal{T}^{-\theta} u\right\|_{L^{2}(R)}^{2},\lambda \gtrsim l^{-\frac{\beta}{\beta +1}}$} (under polynomial decay) and \textcolor{blue}{$u \in \operatorname{Ran}\left(\mathcal{T}^{\theta}\right)$,$\|u\|_{L^{2}(R)}^{2} \geq 16 \lambda^{2 \theta}\left\|\mathcal{T}^{-\theta} u\right\|_{L^{2}(R)}^{2},\lambda > \frac{\log l}{l}$} (under exponential decay) [The general conditions required are $u \in \operatorname{Ran}\left(\mathcal{T}^{\theta}\right)$,$\|u\|_{L^{2}(R)}^{2} \geq 16 \lambda^{2 \theta}\left\|\mathcal{T}^{-\theta} u\right\|_{L^{2}(R)}^{2}$
% , $l > \frac{160\sqrt{2} \kappa \log (\frac{2}{\delta})}{\lambda }$ and $l > \frac{6400 \kappa N_{1}(\lambda) \log(\frac{2}{\delta})}{\lambda}$ ], using Proposition \eqref{Proposition: Upper and lower bound of eta} and Lemma \eqref{Lemma 10}, we have that \begin{equation}P_{H_{1}}(E_{2}^{c}) \leq \delta.\end{equation}

Following the proof of Proposition \ref{Proposition: Type-I error bound with random threshold}, specifically the proof of \eqref{Bound on probability of Ml given E}, we have that, if $n,m \geq 2$,
$\frac{140 \kappa}{s} \log \frac{32 \kappa s}{1-\sqrt{1-\delta}} \leq \lambda \leq \frac{1}{2}\left\|\Sigma_{P Q}\right\|_{\mathcal{L}^{\infty}(\mathcal{H})}
$
and $l\geq \max\left\{2\log\frac{2}{1-\sqrt{1-\delta}},\frac{128\kappa^{2}\log\frac{2}{1-\sqrt{1-\delta}}}{\left\|\Sigma_{P Q}\right\|_{\mathcal{L}^{\infty}(\mathcal{H})}^{2}}\right\}$, then,
\begin{equation}
\label{Probability of E3 RFF based permutation}
P_{H_{1}}(E_{3}^{c}) = P(E_{3}^{c}) \leq \delta.
\end{equation}

Now, under $\boldsymbol{(\Samplesizeassumption)}$ and the choice of the sample splitting size $s=d_{1}N=d_{2}M$ for estimating the covariance operator $\Sigma_{PQ,l}$ as stated in Theorem \ref{Power Analysis for RFF based test based on permutations}, we have that $m \leq n \leq D^{\prime} m$ for some constant $D^{\prime} \geq 1$. Therefore, using Lemma \ref{Lemma 16}, we have that
% \[
% P_{H_{1}}(E_{4}^{c} \mid \theta^{1:l}) = P(E_{4}^{c} \mid \theta^{1:l})\leq \delta.
% \]
\begin{equation}
\label{Probability of E4 given E1 RFF based permutation}
P_{H_{1}}(E_{4}^{c} \mid E_{1})\leq \delta.
\end{equation}

% Consequently, using the monotonicity and the tower property of expectations, we have that 
% \begin{equation}
% \label{Probability of E4 RFF based permutation}
% P_{H_{1}}(E_{4}^{c}) = P(E_{4}^{c}) \leq \delta.
% \end{equation}

Let us define the event $E^{*} = \left\{q_{1-\tilde{\alpha}}^{\lambda,l} \geq T_{1}\right\}$. Provided that the occurrence of the events $E_{1}$, $E_{2}$, $E_{3}$ and $E_{4}$ imply that event $E^{*}$ cannot occur under the conditions of Theorem \ref{Power Analysis for RFF based test based on permutations}, and using \eqref{Probability of E1 RFF based permutation}, \eqref{Probability of E2 RFF based permutation}, \eqref{Probability of E3 RFF based permutation} and \eqref{Probability of E4 given E1 RFF based permutation}, we have that 
\begin{equation*}
\label{Probability of Estar RFF based permutation}
\begin{aligned}
P_{H_{1}}(E^{*}) &\leq P_{H_{1}}(E_{1}^{c} \cup E_{2}^{c} \cup E_{3}^{c} \cup E_{4}^{c}) \\
&\leq P(E_{1}^{c})+P_{H_{1}}(E_{2}^{c}) +P(E_{3}^{c}) +P_{H_{1}}(E_{4}^{c})\\
&\leq 2P(E_{1}^{c})+P_{H_{1}}(E_{2}^{c}) +P(E_{3}^{c}) +P_{H_{1}}(E_{4}^{c} \mid E_{1})\\
&\leq 5\delta.
\end{aligned}
\end{equation*}
Therefore, to complete the proof of this theorem, we only need to prove that the simultaneous occurrence of the events $E_{1}$, $E_{2}$, $E_{3}$ and $E_{4}$ precludes the occurrence of the event $E^{*}$ under the conditions specified in this theorem, i.e., the event $(E^{*})^{c}= \left\{q_{1-\tilde{\alpha}}^{\lambda,l} < T_{1}\right\}$ occurs, or equivalently
\begin{equation}
\label{Events E1, E2 ,E3 and E4 imply permutation quantile than T1}
\left\{E_{1} \cap E_{2} \cap E_{3} \cap E_{4}\right\} \subset (E^{*})^{c}.
\end{equation}
Note that, provided the event $E_{1}$ occurs, under $\boldsymbol{(\RFFAssumptionone)}$ and $\boldsymbol{(\RFFAssumptionfour)}$, $$N_{2}^{\prime}(\kappa,\lambda,\delta,l)\coloneq\frac{2N_{2}^{*}(\kappa,\lambda,\delta,l)\kappa}{\lambda}$$ is an upper bound on $C_{\lambda,l}=\frac{2\mathcal{N}_{2,l}(\lambda)}{\lambda} \sup _{x}\|K_{l}(\cdot, x)\|_{\mathcal{H}_{l}}^{2}$ as defined in Lemma \ref{Lemma 9}. Let us define 
%$$\gamma_{1,l}=\frac{1}{\sqrt{\delta}}\left(\frac{\sqrt{N_{2}^{\prime}(\kappa,\lambda,\delta,l)}\|u\|_{L^2(R)}+N_{2}^{*}(\kappa,\lambda,\delta,l)}{n+m}+\frac{N_{2}^{\prime}(\kappa,\lambda,\delta,l)^{1 / 4}\|u\|_{L^2(R)}^{3 / 2}+\|u\|_{L^2(R)}}{\sqrt{n+m}}\right)$$ and 
$$T_{2} = \zeta_{l} - \tilde{C}^{1/2}\|\mathcal{M}_{l}\|_{\mathcal{L}^{\infty}(\mathcal{H}_{l})}^2 \gamma_{1,l},$$ where $\tilde{C}$ is a constant defined in Lemma \ref{Lemma 11} that depends only on $C_{1}, C_{2}$ and $D^{\prime}$. Further, let us define the event $E^{\prime} = \left\{q_{1-\tilde{\alpha}}^{\lambda,l} \geq T_{2}\right\}$. 

% Further, let us define $T_{3} = C^{*}\gamma_{3,l}$ and the event $E^{\prime \prime} = \left\{q_{1-\tilde{\alpha}}^{\lambda,l} \geq T_{3}\right\}$ where $C^{*}$ is a positive constant as defined in Lemma \ref{Lemma 16}.

Now, under $\boldsymbol{(\Samplesizeassumption)}$ and the choice of the sample splitting size $s=d_{1}N=d_{2}M$ for estimating the covariance operator $\Sigma_{PQ,l}$ as stated in Theorem \ref{Power analysis of Oracle test}, we have that $m \leq n \leq D^{\prime} m$ where $D^{\prime} = \frac{D-d_{2}}{1-d_{2}}\geq 1$ is a constant. Therefore, using Lemma~\ref{Lemma 11} under $\boldsymbol{(\SpectralAssumptionone)}$ and $\boldsymbol{(\SpectralAssumptiontwo)}$ and provided the events $E_{1}$ and $E_{2}$ occur simultaneously, we observe that $T_{2} \leq T_{1}$ and consequently, the occurrence of the event $(E^{\prime})^{c} = \left\{q_{1-\tilde{\alpha}}^{\lambda,l} < T_{2}\right\}$ implies the occurrence of the event $(E^{*})^{c} = \left\{q_{1-\tilde{\alpha}}^{\lambda,l} < T_{1}\right\}$. Therefore, it is sufficient to show that the simultaneous occurrence of the events $E_{1}$, $E_{2}$, $E_{3}$ and $E_{4}$ precludes the occurrence of the event $E^{\prime}$ under the conditions specified in this theorem, i.e., the event $(E^{\prime})^{c}= \{q_{1-\tilde{\alpha}}^{\lambda,l} < T_{2}\}$ occurs, or equivalently
\begin{equation}
\label{Events E1, E2, E3 and E4 imply permutation quantile less than T2}
E_{1} \cap E_{2} \cap E_{3} \cap E_{4} \subset (E^{\prime})^{c}.
\end{equation}
When the event $E_{3}$ occurs, using the fact that $\norm{\mathcal{M}_{l}^{-1}}_{\mathcal{L}^{\infty}(\mathcal{H}_{l})}\geq \frac{1}{\norm{\mathcal{M}_{l}}_{\mathcal{L}^{\infty}(\mathcal{H}_{l})}}$, we obtain  $\norm{\mathcal{M}_{l}}_{\mathcal{L}^{\infty}(\mathcal{H}_{l})}^{2} \leq 2$, $\norm{\mathcal{M}_{l}^{-1}}_{\mathcal{L}^{\infty}(\mathcal{H}_{l})}^{2} \leq \frac{3}{2}$, and consequently, we must have \begin{equation}
\label{Bound on sum of function of Ml and ML inverse under event E3 RFF based permutation}
\left\|\mathcal{M}_{l}^{-1}\right\|_{\mathcal{L}^{\infty}(\mathcal{H}_{l})}^2\|\mathcal{M}_{l}\|_{\mathcal{L}^{\infty}(\mathcal{H}_{l})}^2 \leq 3.
\end{equation}
Observe that, if
\begin{equation}\label{Sufficient condition 1 on norm of u to get function of Ml1 and ML inverse bounded above RFF based permutation}
\|u\|_{L^2(R)}^2 > \frac{6(\tilde{C}^{1/2}\gamma_{1,l} + C^{*}\gamma_{2,l})}{c_{2}}
\end{equation} holds, using \eqref{Bound on sum of function of Ml and ML inverse under event E3 RFF based permutation}, then
\begin{equation*}
\label{permutation quantile less than T2 type bound version 1 RFF based permutation}
\frac{\left\|\mathcal{M}_{l}^{-1}\right\|_{\mathcal{L}^{\infty}(\mathcal{H}_{l})}^2\|\mathcal{M}_{l}\|_{\mathcal{L}^{\infty}(\mathcal{H}_{l})}^2(\tilde{C}^{1/2}\gamma_{1,l} + C^{*}\gamma_{2,l})}{c_{2}\norm{u}_{L^{2}(R)}^{2}} < \frac{1}{2}.
\end{equation*}
Under the condition $(n+m) \geq \frac{2C^{*}\log(\frac{2}{\tilde{\alpha}})}{\sqrt{\delta}}$ and provided the event $E_{2}$ occurs, we have that 
\begin{equation}
\label{permutation quantile less than T2 type bound version 2 RFF based permutation}
\|\mathcal{M}_{l}\|_{\mathcal{L}^{\infty}(\mathcal{H}_{l})}^2(\tilde{C}^{1/2}\gamma_{1,l} + C^{*}\gamma_{2,l}) < \zeta_{l}\left(1-\frac{C^{*}\log(\frac{1}{\tilde{\alpha}})}{\sqrt{\delta}(n+m)}\right)
\end{equation}
which is equivalent to
\begin{equation}
\label{permutation quantile less than T2 type bound version 3 RFF based permutation}
C^{*}\gamma_{3,l} < T_{2}.
\end{equation}
Provided the event $E_{4}$ occurs, from \eqref{permutation quantile less than T2 type bound version 3 RFF based permutation} we have that $q_{1-\tilde{\alpha}}^{\lambda,l} < T_{2}$. Therefore, \eqref{Events E1, E2, E3 and E4 imply permutation quantile less than T2} and consequently \eqref{Events E1, E2 ,E3 and E4 imply permutation quantile than T1} is proved, which, barring the verification of the sufficiency of the conditions stated in this theorem, completes the proof of the theorem.

We now consolidate the assumptions and sufficient conditions under which the theorem is true. We proceed to show the conditions specified in the statement of Theorem \ref{Power Analysis for RFF based test based on permutations} are sufficient for \eqref{Probability of E1 RFF based permutation}, \eqref{Probability of E2 RFF based permutation}, \eqref{Probability of E3 RFF based permutation}, \eqref{Probability of E4 given E1 RFF based permutation}, \eqref{Sufficient condition 1 on norm of u to get function of Ml1 and ML inverse bounded above RFF based permutation} and \eqref{permutation quantile less than T2 type bound version 2 RFF based permutation} to hold.

% Replaced c_{1,\theta} by c_{1} using sup over theta
Let us define $c_{1} = \underset{\theta>0}{\sup}\underset{(P,Q) \in \mathcal{P}}{\sup} \norm{\mathcal{T}_{PQ}^{-\theta}u}_{L^{2}(R)}$ which is assumed to be finite and $d_{\theta} = \left(\frac{1}{16c_{1}^{2}}\right)^{\frac{1}{2\theta}}$. Since $(P,Q) \in \mathcal{P}$ under $H_{1}$, we have that $\norm{u}^2_{L^{2}(R)} \geq \Delta_{N,M}$. Consequently, the choice \begin{equation}
\label{Choice of lambda in general case of power analysis of RFF based permutation test}
\lambda=d_{\theta}\Delta_{N,M}^{\frac{1}{2\theta}}\end{equation} implies that $\norm{u}^2_{L^{2}(R)} \geq 16 \lambda^{2 \theta} \norm{\mathcal{T}_{PQ}^{-\theta}u}_{L^{2}(R)}^{2}$ holds. Using the choice of $\lambda$ as given in \eqref{Choice of lambda in general case of power analysis of RFF based permutation test}, the conditions  $$\Delta_{N,M}^{\frac{1}{2\theta}} \geq \frac{d_{\theta}^{-1}160\kappa \log (\frac{2}{\delta})}{l},\quad\text{and}\quad \frac{\Delta_{N,M}^{\frac{1}{2\theta}} }{\mathcal{N}_{1}\left(\left(d_{\theta}\Delta_{N,M}\right)^{1/2\theta}\right)} \geq  \frac{d_{\theta}^{-1} 3200 \kappa \log (\frac{2}{\delta})}{l}$$ are sufficient to ensure that \eqref{Probability of E2 RFF based permutation} holds. The conditions $$\lambda = d_{\theta}\Delta_{N,M}^{\frac{1}{2\theta}} \leq \frac{1}{2}\|\Sigma_{PQ}\|_{\mathcal{L}^{\infty}(\mathcal{H})},\quad\text{and}\quad \Delta_{N,M} \geq \left(d_{\theta}^{-1}\frac{86 \kappa}{l}\log \frac{32 \kappa l}{\delta}\right)^{2\theta}$$ are sufficient to ensure that \eqref{Probability of E1 RFF based permutation} holds. Note that, since $s \geq \frac{d_{1}(N+M)}{2}$ and $(N+M) \geq \frac{32\kappa d_{1}}{\delta}$, \eqref{Probability of E3 RFF based permutation} holds if $$l\geq \max\left\{2\log\frac{2}{1-\sqrt{1-\delta}},\frac{128\kappa^{2}\log\frac{2}{1-\sqrt{1-\delta}}}{\left\|\Sigma_{P Q}\right\|_{\mathcal{L}^{\infty}(\mathcal{H})}^{2}}\right\},$$ $d_{\theta}\Delta_{N,M}^{\frac{1}{2\theta}} \geq \frac{560\kappa\log(N+M)}{d_{1}(N+M)}$, $\frac{1}{2}\|\Sigma_{PQ}\|_{\mathcal{L}^{\infty}(\mathcal{H})} \geq d_{\theta}\Delta_{N,M}^{\frac{1}{2\theta}}$ and $n,m \geq 2$. Under $\boldsymbol{(\Samplesizeassumption)}$ and the choice of the sample splitting size $s=d_{1}N=d_{2}M$ for estimating the covariance operator $\Sigma_{PQ,l}$ as stated in Theorem \ref{Power Analysis for RFF based test based on permutations}, \eqref{Probability of E4 given E1 RFF based permutation} holds true.

Note that, $(n+m)=(1-d_1)N+(1-d_2)M \geq (1-d_2)(N+M)$, where $1\geq d_2\geq d_1\geq 0$. Hence, \eqref{permutation quantile less than T2 type bound version 2 RFF based permutation} holds if $N+M \geq \frac{2C^{*}\log(\frac{2}{\tilde{\alpha}})}{(1-d_{2})\sqrt{\delta}}$. Further, \eqref{Sufficient condition 1 on norm of u to get function of Ml1 and ML inverse bounded above RFF based permutation} holds if $\mathcal{N}_{2}(\lambda)=\mathcal{N}_{2}(d_{\theta}\Delta_{N,M}^{\frac{1}{2\theta}}) \geq 1$ and the conditions \emph{3.--8.} in the statement of Theorem~\ref{Power Analysis for RFF based test based on permutations} hold.

\subsection{Proof of Corollary \ref{Power Analysis of RFF based permutation test polynomial decay}}\label{subsec:cor7}
% \begin{proof}
% The proof is almost similar to that of Corollary \ref{Power Analysis of Oracle test polynomial decay}. We assume the setting of Theorem \ref{Power Analysis for RFF based test based on permutations} is valid. Under polynomial decay of the eigenvalues of $\Sigma_{PQ}$ i.e. $\lambda_{i} \asymp i^{-\beta}$ for $\beta>1$, using Lemma \ref{Lemma 12}(i) and Lemma B.9 from \cite{ApproximateKernelPCARandomFeaturesStergeSriperumbudur}, we have that,
% \begin{equation}
% \label{N2 simplified under polynomial decay again}
%     N_{2}(\lambda) \asymp \lambda^{-\frac{1}{2\beta}}
% \end{equation}
% and
% \begin{equation}
% \label{N1 simplified under polynomial decay again}
%     N_{1}(\lambda) \asymp \lambda^{-\frac{1}{\beta}}.
% \end{equation}

The proof is almost similar to that of Corollary \ref{Power Analysis of Oracle test polynomial decay}. Under polynomial decay of the eigenvalues of $\Sigma_{PQ}$ i.e. $\lambda_{i} \asymp i^{-\beta}$ for $\beta>1$, using Lemma \ref{Lemma 12}(i) and Lemma C.9 from \cite{ApproximateKernelPCARandomFeaturesStergeSriperumbudur}, we have that \eqref{N2 simplified under polynomial decay} and \eqref{N1 simplified under polynomial decay} hold. 
% \begin{equation}
% \label{N2 simplified under polynomial decay again}
%     \mathcal{N}_{2}(\lambda) \asymp \lambda^{-\frac{1}{2\beta}}
% \end{equation}
% and
% \begin{equation}
% \label{N1 simplified under polynomial decay again}
%     \mathcal{N}_{1}(\lambda) \asymp \lambda^{-\frac{1}{\beta}}.
% \end{equation}
Using \eqref{N2 simplified under polynomial decay} and \eqref{N1 simplified under polynomial decay} and provided $N+M \geq k(\tilde{\alpha},\delta,\theta,\beta)$ for some constant $k(\tilde{\alpha},\delta,\theta,\beta) \in \mathbb{N}$ depending on  $\tilde{\alpha}$, $\delta$, $\theta$ and $\beta$, and $B \geq \frac{\log(\frac{2}{\min\left\{{\delta},\alpha(1-w-\tilde{w})\right\}})}{2\tilde{w}^{2}\alpha^{2}}$, the conditions \ref{Theorem Power Analysis RFF based permutation test Condition 1} to \ref{Theorem Power Analysis RFF based permutation test Condition 10} as specified in Theorem \ref{Power Analysis for RFF based test based on permutations} reduce to \eqref{Oracle test Power analysis conditions simplified under polynomial decay 1}.
The rest of the proof matches that of Corollary~\ref{Power Analysis of Oracle test polynomial decay}.

\subsection{Proof of Corollary \ref{Power Analysis of RFF based permutation test exponential decay}}\label{subsec:cor8}
% \begin{proof}
% The proof is almost similar to that of Corollary \ref{Power Analysis of Oracle test exponential decay}. We assume the setting of Theorem \ref{Power Analysis for RFF based test based on permutations} is valid. Under exponential decay of the eigenvalues of $\Sigma_{PQ}$ i.e. $\lambda_{i} \asymp e^{-\tau i}$ for $\tau>0$, using Lemma \ref{Lemma 12}(ii) and Lemma B.9 from \cite{ApproximateKernelPCARandomFeaturesStergeSriperumbudur}, we have that,
% \begin{equation}
% \label{N2 simplified under exponential decay again}
%     N_{2}(\lambda) \asymp\sqrt{\log(\frac{1}{\lambda})}
% \end{equation}
% and
% \begin{equation}
% \label{N1 simplified under exponential decay again}
%     N_{1}(\lambda) \asymp \log(\frac{1}{\lambda}).
% \end{equation}
The proof is almost similar to that of Corollary \ref{Power Analysis of Oracle test exponential decay}. Under exponential decay of the eigenvalues of $\Sigma_{PQ}$ i.e. $\lambda_{i} \asymp e^{-\tau i}$ for $\tau>0$, using Lemma \ref{Lemma 12}(ii) and Lemma C.9 from \cite{ApproximateKernelPCARandomFeaturesStergeSriperumbudur}, we have that \eqref{N2 simplified under exponential decay} and \eqref{N1 simplified under exponential decay} hold.
% \begin{equation}
% \label{N2 simplified under exponential decay again}
%     \mathcal{N}_{2}(\lambda) \asymp\sqrt{\log\frac{1}{\lambda}}
% \end{equation}
% and
% \begin{equation}
% \label{N1 simplified under exponential decay again}
%     \mathcal{N}_{1}(\lambda) \asymp \log\frac{1}{\lambda}.
% \end{equation}
Note that, based on Remark \ref{Remark: Conditions for N2l to behave as N2}, if 
\eqref{eq:cor-2-lambda} and \eqref{eq:cor-2-l} hold 
% \begin{equation}\label{eq:cor-4-lambda}e^{-1} \geq \lambda \gtrsim \frac{\left[\log\frac{4}{\delta}+\log\frac{8}{\alpha}\right]\log l}{l}
% \end{equation}
% and
% \begin{equation}\label{eq:cor-4-l}
% l \geq 2\bar{C}e\left[\log\frac{4}{\delta}+\log\frac{8}{\alpha}\right] \left[\log\left(\log\frac{4}{\delta}+\log\frac{8}{\alpha}\right)+1\right]
% \end{equation}
for some universal constant $\bar{C} \geq 1$, conditions \ref{Theorem Power Analysis RFF based permutation test Condition 5} and \ref{Theorem Power Analysis RFF based permutation test Condition 6} of Theorem \ref{Power Analysis for RFF based test based on permutations} are automatically satisfied if condition \ref{Theorem Power Analysis RFF based permutation test Condition 7} is satisfied, while conditions \ref{Theorem Power Analysis RFF based permutation test Condition 8} and \ref{Theorem Power Analysis RFF based permutation test Condition 9} are automatically satisfied if condition \ref{Theorem Power Analysis RFF based permutation test Condition 10} is satisfied. Using \eqref{N2 simplified under exponential decay} and \eqref{N1 simplified under exponential decay} and provided $N+M \geq k(\tilde{\alpha},\delta,\theta)$ for some constant $k(\tilde{\alpha},\delta,\theta) \in \mathbb{N}$ depending on  $\alpha$, $\delta$ and $\theta$ and $B \geq \frac{\log(\frac{2}{\min\left\{{\delta},\alpha(1-w-\tilde{w})\right\}})}{2\tilde{w}^{2}\alpha^{2}}$, the conditions \ref{Theorem Power Analysis RFF based permutation test Condition 3}, \ref{Theorem Power Analysis RFF based permutation test Condition 7} and \ref{Theorem Power Analysis RFF based permutation test Condition 10} reduce to \eqref{Power Analysis Oracle test Condition 3 simplified under exponential decay}, \eqref{Power Analysis Oracle test Condition 7 simplified under exponential decay}, and \eqref{Power Analysis Oracle test Condition 10 simplified under exponential decay}, respectively.
Therefore, the rest of the proof matches that of Corollary~\ref{Power Analysis of Oracle test exponential decay}.

\subsection{Proof of Theorem \ref{Type I-error of RFF-based permutation test with adaptation over lambda}}\label{subsec:thm9}

% \begin{proof}

Choosing $\alpha$ as $\frac{\alpha}{|\Lambda|}$ and under the condition that the number of randomly selected permutations $B \geq \frac{|\Lambda|^{2}}{2\tilde{w}^{2}\alpha^{2}}\log(\frac{2|\Lambda|}{\alpha(1-w-\tilde{w})})$, using Theorem \ref{Type I-error of RFF-based permutation test} we have that, for any $\lambda \in \Lambda$,
\[
P_{H_{0}}\left\{\hat{\eta}_{\lambda,l} \geq q_{1-\frac{w\alpha}{|\Lambda|}}^{B,\lambda,l}\right\} \leq \frac{\alpha}{|\Lambda|}.
\]
Consequently, the proof is complete using Lemma \ref{Lemma 17}.

\subsection{Proof of Theorem \ref{Power Analysis for RFF based test based on permutations and adaptation}}\label{subsec:thm10}
% \begin{proof}
The basic structure of the proof follows the same steps as in the proof of Theorem \ref{Power Analysis for RFF based test based on permutations}. The primary difference is the use of $\hat{q}_{1-\frac{w\alpha}{|\Lambda|}}^{B,\lambda,l}$, instead of $\hat{q}_{1-w\alpha}^{B,\lambda,l}$, as the threshold for each choice of $\lambda \in \Lambda$, which is equivalent to choosing the significance level as $\frac{\alpha}{|\Lambda|}$ for each choice of $\lambda \in \Lambda$. This results in the emergence of an extra factor of $\log|\Lambda|$ in the expression of $\gamma_{2,l}$ and $\gamma_{3,l}$, i.e., $\log\frac{1}{\tilde{\alpha}}$ is replaced with $\log\frac{|\Lambda|}{\tilde{\alpha}}$. This ultimately results in an extra factor of $\log|\Lambda|$ in the expression of the separation boundary. Further, note that, under the conditions of Theorem \ref{Power Analysis for RFF based test based on permutations and adaptation}, we can ensure that $|\Lambda| = 1 + \log_{2}\frac{\lambda_{U}}{\lambda_{L}} \lesssim \log(N+M)$ under both polynomial and exponential decay of eigenvalues.

Before proceeding to the proofs of parts (i) and (ii), let us define $c_{1} \coloneq \underset{\theta>0}{\sup}\underset{(P,Q) \in \mathcal{P}}{\sup} \norm{\mathcal{T}_{PQ}^{-\theta}u}_{L^{2}(R)}$ and $d_{\theta} \coloneq \left(16c_{1}^{2}\right)^{-\frac{1}{2\theta}}$.

\paragraph{Proof of (i):} 
%Under polynomial decay of the eigenvalues of $\Sigma_{PQ}$ i.e. $\lambda_{i} \asymp i^{-\beta}$ for $\beta>1$, using Lemma \ref{Lemma 12}(i) and Lemma B.9 from \cite{ApproximateKernelPCARandomFeaturesStergeSriperumbudur}, we have that, for any $\lambda>0$,
% \begin{equation}
% \label{N2 simplified under polynomial decay repeated}
%     N_{2}(\lambda) \asymp \lambda^{-\frac{1}{2\beta}}
% \end{equation}
% and
% \begin{equation}
% \label{N1 simplified under polynomial decay repeated}
%     N_{1}(\lambda) \asymp \lambda^{-\frac{1}{\beta}}.
% \end{equation}
Let us define  $\tilde{\alpha}\coloneqq(w-\tilde{w})\alpha$ and $A_{M,N}\coloneqq\log\left(\frac{\log(N+M)}{\tilde{\alpha}}\right)$. Using \eqref{N2 simplified under polynomial decay} and \eqref{N1 simplified under polynomial decay}, and provided $$N+M \geq \max\left\{\frac{32\kappa d_{1}}{\delta},\frac{2 C^{*}\log\left(\frac{2\log(N+M)}{(w-\tilde{w})\alpha}\right)}{(1-d_{2})\sqrt{\delta}},e^{e}\right\}$$ and $$B \gtrsim \frac{\left[\log(N+M)\right]^{2}}{2\tilde{w}^{2}\alpha^{2}}\max\left\{\log\frac{2\log(N+M)}{\alpha(1-w-\tilde{w})},\log\frac{2}{\delta}\right\},$$ the conditions \ref{Theorem Power Analysis RFF based permutation test Condition 1} to \ref{Theorem Power Analysis RFF based permutation test Condition 10} as specified in Theorem \ref{Power Analysis for RFF based test based on permutations} reduce to 
\begin{equation}
\label{RFF based permutation and adaptation test Power analysis conditions simplified under polynomial decay 1}
    % \Delta_{N,M}^{\frac{1}{2\theta}} \gtrsim \max\left\{\circled{\emph{\small{1}}},\circled{\emph{\small{2}}},\circled{\emph{\small{3}}},\circled{\emph{\small{4}}},\circled{\emph{\small{5}}},\circled{\emph{\small{6}}},\circled{\emph{\small{7}}},\circled{\emph{\small{8}}},\circled{\emph{\small{9}}},\circled{\emph{\small{10}}}\right\}
    d_{\theta}\Delta_{N,M}^{\frac{1}{2\theta}} \gtrsim \max\left\{\circled{\emph{\small{1}}},\circled{\emph{\small{2}}},\circled{\emph{\small{3}}},\circled{\emph{\small{4}}},\circled{\emph{\small{5}}},\circled{\emph{\small{6}}},\circled{\emph{\small{7}}},\circled{\emph{\small{8}}},\circled{\emph{\small{9}}},\circled{\emph{\small{10}}}\right\},
\end{equation}
where $$\circled{\emph{\small{1}}} = \frac{\log(N+M)}{(N+M)},\,\,  \circled{\emph{\small{2}}} = \frac{1}{l},\,\,  \circled{\emph{\small{3}}} = \left[\frac{1}{l}\right]^{\frac{\beta}{\beta+1}},\,\, \circled{\emph{\small{4}}} = \frac{\log (l)}{l},$$ $$\circled{\emph{\small{5}}} = \left[\frac{A^{4}_{M,N}}{\sqrt{l}(N+M)}\right]^{\frac{2}{1+4\theta}},\,\,
\circled{\emph{\small{6}}} = \left[\frac{A^{4}_{M,N}}{l(N+M)}\right]^{\frac{1}{1+2\theta}},\,\,
\circled{\emph{\small{7}}} = \left[\frac{A^{4}_{M,N}}{(N+M)}\right]^{\frac{2\beta}{1+4\beta\theta}},$$
$$\circled{\emph{\small{8}}} = \left[\frac{A^{2}_{M,N}}{\sqrt{l}(N+M)^{2}}\right]^{\frac{2}{3+4\theta}},\,\,
\circled{\emph{\small{9}}} = \left[\frac{A^{2}_{M,N}}{l(N+M)^{2}}\right]^{\frac{1}{2(1+\theta)}},\,\,\text{and}\,\,
\circled{\emph{\small{10}}} = \left[\frac{A^{2}_{M,N}}{(N+M)^{2}}\right]^{\frac{2\beta}{1+2\beta+4\beta\theta}}.$$

We now use the fact that $A_{M,N} \leq \log(\frac{1}{\tilde{\alpha}}) \log \log (N+M)$ along with the fact that $[\log \log (N+M)]^{-a} \leq 1$ for any $a>0$ if $N+M \geq k$ where $k \in \mathbb{N}$ is sufficiently large, to simplify the condition given in \eqref{RFF based permutation and adaptation test Power analysis conditions simplified under polynomial decay 1}. Note that, while the true value of the smoothness index $\theta$ is unknown, we assume $\theta \geq \theta^*$ for some known $\theta^* \in\left(0, \frac{1}{4}\right]$. It is straightforward to verify that $\circled{\emph{\small{3}}} \gtrsim \circled{\emph{\small{4}}} \gtrsim \circled{\emph{\small{2}}}$. Moreover, for large enough $N+M$, it can be verified that
\begin{equation*}
\circled{\emph{\small{7}}}\gtrsim
\begin{cases}
\circled{\emph{\small{3}}}, & l\gtrsim (N+M)^{\frac{2(\beta+1)}{1+4\beta\theta}}\\
\circled{\emph{\small{5}}}, & l\gtrsim (N+M)^{\frac{2(\beta-1)}{1+4\beta\theta}}\\
\circled{\emph{\small{6}}}, & l\gtrsim (N+M)^{\frac{2\beta-1}{1+4\beta\theta}}\\
\circled{\emph{\small{8}}}, & l\gtrsim (N+M)^{\frac{2(3\beta-4\theta\beta-2)}{1+4\beta\theta}}\\
\circled{\emph{\small{9}}}, & l\gtrsim (N+M)^{\frac{2(2\beta-2\theta\beta-1)}{1+4\beta\theta}}
\end{cases},
\end{equation*}
$\circled{\emph{\small{7}}}\gtrsim \circled{\emph{\small{1}}}$, $\circled{\emph{\small{7}}}\gtrsim \circled{\emph{\small{10}}}$ if $\theta>\frac{1}{2}-\frac{1}{4\beta}$. Since $\theta^*\le\frac{1}{4},$ we have that $(N+M)^{\frac{1}{2 \theta^*}} \geq (N+M)^{\frac{2(\beta+1)}{1+4 \beta \theta^*}}$. Therefore $l\gtrsim (N+M)^{\frac{1}{2 \theta^*}}$ implies $l\gtrsim (N+M)^{\frac{2(\beta+1)}{1+4 \beta \theta^*}}$, which further implies $l\gtrsim (N+M)^{\frac{2(\beta+1)}{1+4 \beta \theta}}$ if $\theta \geq \theta^*$. 
%Further, if $\theta \geq \max \left\{\frac{1}{2}-\frac{1}{4 \beta}, \theta^*\right\}$, $l \gtrsim (N+M)^{\frac{1}{2 \theta^*}} \geq (N+M)^{\frac{2(\beta+1)}{1+4 \beta \theta^*}} \geq (N+M)^{\frac{2(\beta+1)}{1+4 \beta \theta}} $. 
Therefore, if $\theta \geq \max \left\{\frac{1}{2}-\frac{1}{4 \beta}, \theta^*\right\}$ and $l \gtrsim (N+M)^{\frac{1}{2 \theta^*}}$, then  $\circled{\emph{\small{7}}}$ dominates. Similarly, for large enough $N+M$, it can be verified that
\begin{equation*}
\circled{\emph{\small{1}}}\gtrsim
\begin{cases}
\circled{\emph{\small{3}}}, & l\gtrsim (N+M)^{\frac{\beta+1}{\beta}}\\
\circled{\emph{\small{5}}}, & l\gtrsim (N+M)^{4\theta-1}\\
\circled{\emph{\small{6}}}, & l\gtrsim (N+M)^{2\theta}\\
\circled{\emph{\small{8}}}, & l\gtrsim (N+M)^{4\theta-1}\\
\circled{\emph{\small{9}}}, & l\gtrsim (N+M)^{2\theta}
\end{cases},
\end{equation*}
$\circled{\emph{\small{1}}}\gtrsim \circled{\emph{\small{7}}}$, $\circled{\emph{\small{1}}}\gtrsim \circled{\emph{\small{10}}}$ if $\theta\le \frac{1}{2}-\frac{1}{4\beta}$. Further, $l \gtrsim (N+M)^{2}$ is sufficient to ensure $ l \gtrsim (N+M)^{\frac{\beta + 1}{\beta}}$. Therefore, for large enough $N+M$, if $\theta^{*} \leq \theta\le\frac{1}{2}-\frac{1}{4\beta}$ and $l\gtrsim (N+M)^{2}$, then  $\circled{\emph{\small{1}}}$ dominates. 

Combining, we have that, for any $\theta \geq \theta^{*}$ with $\theta^{*} \in (0,\frac{1}{4}]$ and $l \gtrsim \max\left\{(N+M)^{\frac{1}{2\theta^{*}}},(N+M)^{2}\right\} = (N+M)^{\frac{1}{2\theta^{*}}}$, the separation boundary satisfies
\[
    \Delta_{N,M} = 
    c(\tilde{\alpha},\delta)\max\left\{\left[\frac{\log(N+M)}{N+M}\right]^{2\theta},\left[\frac{\log\log (N+M)}{N+M}\right]^{\frac{4\beta\theta}{1+4\beta\theta}} \right\}.
\]
We observe that 
\[
\lambda^{*} = d_{\theta}\Delta_{NM}^{\frac{1}{2\theta}} = \left[\frac{c(\tilde{\alpha},\delta)}{16c_{1}^{2}}\right]^{\frac{1}{2\theta}} \max\left\{\frac{\log(N+M)}{N+M},\left[\frac{\log\log (N+M)}{N+M}\right]^{\frac{2\beta}{1+4\beta\theta}} \right\}
\] is a rate-optimal choice of the regularization parameter $\lambda$, which depends on the unknown $\theta$. Define $r:=\frac{c(\tilde{\alpha},\delta)}{16c_{1}^{2}}$.
%can be expressed as $(r)^{\frac{1}{2\theta}}$ for some $r>0$ that depends only on $\tilde{\alpha}$ and $\delta$. 
If $r \leq 1$,  then $r^{\frac{1}{2\theta^{*}}} \leq r^{\frac{1}{2\theta}} \leq 1$ for $\theta \geq \theta^{*}$. On the other hand, if $r>1$, then $1 \leq r^{\frac{1}{2\theta}} \leq r^{\frac{1}{2\theta^{*}}}$ for $\theta \geq \theta^{*}$. Also $\frac{\log(N+M)}{N+M} \leq \max\left\{\frac{\log(N+M)}{N+M},\left[\frac{\log\log (N+M)}{N+M}\right]^{\frac{2\beta}{1+4\beta\theta}} \right\} \leq 1$. Therefore, for any $\theta \geq \theta^{*}$ and $\beta > 1$, $\lambda^{*}$ can be bounded as $\lambda_{L} = r_{1}\frac{\log(N+M)}{N+M} \leq \lambda^{*} \leq \min\left\{r_{2},\frac{1}{2}\norm{\Sigma_{PQ}}_{\mathcal{L}^{\infty}(\mathcal{H})})\right\}=\lambda_{U}$ for some constants $r_{1},r_{2}>0$ that depend on $\tilde{\alpha}$, $\delta$ and $\theta^{*}$.

Let us define $s^{*} = \underset{\lambda \in \Lambda : \lambda \leq \lambda^{*}}{\sup} \lambda$. Then, one can easily deduce from the definition of $\Lambda$ that $\frac{\lambda^{*}}{2} \leq s^{*} \leq \lambda^{*}$ and consequently, $s^{*} \asymp \lambda^{*}$. Further, $\lambda_{L} \leq \lambda^{*} \leq \lambda_{U}$. Therefore, $s^{*}$ is also a rate-optimal choice of the regularization parameter $\lambda$ that will lead to the same conditions on the separation boundary $\Delta_{NM}$ up to constants. Hence, using Lemma \ref{Lemma 17}, for any $\theta \geq \theta^{*}$ and any $(P,Q) \in \mathcal{P}_{\theta,\Delta_{NM}}$, we have 
\[
P_{H_{1}}\left(\underset{\lambda \in \Lambda}{\bigcup}\left\{\hat{\eta}_{\lambda,l} \geq q_{1-\frac{w\alpha}{|\Lambda|}}^{B,\lambda,l}\right\}\right) \geq 1-7\delta.
\]
Taking infimum over $(P,Q) \in \mathcal{P}_{\theta,\Delta_{NM}}$ and  $\theta \geq \theta^{*}$, the proof is complete.

\paragraph{Proof of (ii):}  
% Under exponential decay of the eigenvalues of $\Sigma_{PQ}$ i.e. $\lambda_{i} \asymp e^{-\tau i}$ for $\tau>0$, using Lemma \ref{Lemma 12}(ii) and Lemma B.9 from \cite{ApproximateKernelPCARandomFeaturesStergeSriperumbudur}, we have that, for any $\lambda>0$,
% \begin{equation}
% \label{N2 simplified under exponential decay repeated}
%     N_{2}(\lambda) \asymp\sqrt{\log(\frac{1}{\lambda})}
% \end{equation}
% and
% \begin{equation}
% \label{N1 simplified under exponential decay repeated}
%     N_{1}(\lambda) \asymp \log(\frac{1}{\lambda}).
% \end{equation} 

% Consider $\lambda^{*} = d_{\theta}\Delta_{NM}^{\frac{1}{2\theta}}$ as the choice of the regularization parameter $\lambda$. 
 % Note that, based on Remark \ref{Remark: Conditions for N2l to behave as N2}, if 
 % \begin{equation}\label{thm-6-expdecay-lambda}
 % e^{-1} \geq \lambda \gtrsim \frac{\left[\log(\frac{4}{\delta})+\log(\frac{8}{\alpha})\right]\log l}{l}
 % \end{equation} and \begin{equation}\label{thm-6-expdecay-l}
 % l \geq 2\bar{C}\left[\log(\frac{4}{\delta})+\log(\frac{8}{\alpha})\right]e \left[\log\left[\log(\frac{4}{\delta})+\log(\frac{8}{\alpha})\right]+1\right]
 % \end{equation}for some universal constant $\bar{C} \geq 1$, 
 Under \eqref{eq:cor-2-lambda} and \eqref{eq:cor-2-l} and for $\lambda=d_{\theta}\Delta_{NM}^{\frac{1}{2\theta}}$, conditions \ref{Theorem Power Analysis RFF based permutation test Condition 5} and \ref{Theorem Power Analysis RFF based permutation test Condition 6} of Theorem \ref{Power Analysis for RFF based test based on permutations} are automatically satisfied for %$\lambda=d_{\theta}\Delta_{NM}^{\frac{1}{2\theta}}$ 
 if condition \ref{Theorem Power Analysis RFF based permutation test Condition 7} is satisfied, while conditions \ref{Theorem Power Analysis RFF based permutation test Condition 8} and \ref{Theorem Power Analysis RFF based permutation test Condition 9} are automatically satisfied 
 %for $\lambda=d_{\theta}\Delta_{NM}^{\frac{1}{2\theta}}$ 
 if condition \ref{Theorem Power Analysis RFF based permutation test Condition 10} is satisfied. Using \eqref{N2 simplified under exponential decay} and \eqref{N1 simplified under exponential decay} along with the conditions $N+M \geq \max\left\{\frac{32\kappa d_{1}}{\delta},\frac{2 C^{*}\log\left(\frac{2\log(N+M)}{(w-\tilde{w})\alpha}\right)}{(1-d_{2})\sqrt{\delta}},e^{e}\right\}$ and $B \gtrsim \frac{\left[\log(N+M)\right]^{2}}{2\tilde{w}^{2}\alpha^{2}}\max\left\{\log\frac{2\log(N+M)}{\alpha(1-w-\tilde{w})},\log\frac{2}{\delta}\right\}$, the conditions \ref{Theorem Power Analysis RFF based permutation test Condition 3}, \ref{Theorem Power Analysis RFF based permutation test Condition 7} and \ref{Theorem Power Analysis RFF based permutation test Condition 10} reduce to the following:
\begin{equation}
\label{Power Analysis RFF based permutation and adaptation test  Condition 3 simplified under exponential decay}
\frac{\Delta_{N,M}^{\frac{1}{2\theta}}}{\log(d_{\theta}^{-1}\Delta_{N,M}^{-\frac{1}{2\theta}})} \gtrsim d_{\theta}^{-1}\frac{\log (\frac{2}{\delta})}{l},
\end{equation}

\begin{equation}
\label{Power Analysis RFF based permutation and adaptation test  Condition 7 simplified under exponential decay}
% \frac{\Delta_{N,M}}{\sqrt{\log(\Delta_{N,M}^{-\frac{1}{2\theta}})}} \gtrsim \frac{\log(\frac{\log(N+M)}{\tilde{\alpha}})}{\sqrt{\delta}}\times\frac{1}{N+M}
\frac{\Delta_{N,M}}{\sqrt{\log( d_{\theta}^{-1}\Delta_{N,M}^{-\frac{1}{2\theta}})}} \gtrsim \frac{\left[\log\left(\frac{\log(N+M)}{\tilde{\alpha}}\right)\right]^{4}}{\delta^{2}(N+M)},%\frac{1}{(N+M)}
\end{equation}
and
\begin{equation}
\label{Power Analysis RFF based permutation and adaptation test  Condition 10 simplified under exponential decay}
% \frac{\Delta_{N,M}^{\frac{1+2\theta}{2\theta}}}{\sqrt{\log(\Delta_{N,M}^{-\frac{1}{2\theta}})}} \gtrsim \frac{\left(\log(\frac{\log(N+M)}{\tilde{\alpha}}\right)^{2}}{\delta^{2}(N+M)^{2}}.
\frac{\Delta_{N,M}^{\frac{1+2\theta}{2\theta}}}{\sqrt{\log( d_{\theta}^{-1}\Delta_{N,M}^{-\frac{1}{2\theta}})}} \gtrsim d_{\theta}^{-1}\frac{\left[\log\left(\frac{\log(N+M)}{\tilde{\alpha}}\right)\right]^{2}}{\delta(N+M)^2}.
\end{equation}
% \eqref{Power Analysis RFF based permutation and adaptation test  Condition 3 simplified under exponential decay} holds true if $e^{-1} \geq \lambda \gtrsim \frac{\left[\log(\frac{4}{\delta})+\log(\frac{8}{\alpha})\right]\log l}{l}$ and\\ $l \geq 2\bar{C}\left[\log(\frac{4}{\delta})+\log(\frac{8}{\alpha})\right]e \left[\log\left[\log(\frac{4}{\delta})+\log(\frac{8}{\alpha})\right]+1\right]$ for some universal constant $\bar{C} \geq 1$.
\eqref{Power Analysis RFF based permutation and adaptation test  Condition 3 simplified under exponential decay} holds if \eqref{eq:cor-2-lambda} and \eqref{eq:cor-2-l} hold. % Further, the conditions \ref{Theorem Power Analysis RFF based permutation test Condition 1}, \ref{Theorem Power Analysis RFF based permutation test Condition 2} and \ref{Theorem Power Analysis RFF based permutation test Condition 4} reduce to 
Further, Condition \ref{Theorem Power Analysis RFF based permutation test Condition 1} reduces to 
\begin{equation}
\label{RFF based permutation and adaptation test Power analysis conditions 1, 2 and 4 simplified under exponential decay}
    d_{\theta}\Delta_{N,M}^{\frac{1}{2\theta}} \gtrsim \max\left\{\circled{\emph{\small{1}}},\circled{\emph{\small{2}}},\circled{\emph{\small{4}}}\right\},
\end{equation}
where $$\circled{\emph{\small{1}}} = \frac{\log(N+M)}{(N+M)},\,\,  \circled{\emph{\small{2}}} = \frac{\log (\frac{2}{\delta})}{l},\,\, \text{and}\,\, \circled{\emph{\small{4}}} = \frac{\log (\frac{l}{\delta})}{l}.$$ Note that $\circled{\emph{\small{4}}} \gtrsim \circled{\emph{\small{2}}}$ and $\log l \gtrsim \log(\frac{l}{\delta})$ if $l \geq 2$. Further, if $$\Delta_{N,M} \gtrsim \max\left\{\frac{1}{\sqrt{2\theta}},1\right\}\frac{\left[\log(\frac{1}{\tilde{\alpha}})\right]^{4}}{\delta^{2}}\frac{\sqrt{\log(N+M)}\log \log(N+M)}{N+M},$$ then condition \ref{Theorem Power Analysis RFF based permutation test Condition 7} is satisfied, while condition \ref{Theorem Power Analysis RFF based permutation test Condition 10} is satisfied if $$\Delta_{N,M} \gtrsim  \max\left\{\frac{1}{(\frac{1}{2}+\theta)^{\frac{\theta}{(1+2\theta)}}},1\right\}d_{\theta}^{-\frac{2\theta}{1+2\theta}}  \frac{\left[\log(\frac{1}{\tilde{\alpha}})\right]^{\frac{4\theta}{1+2\theta}}}{\delta^{\frac{2\theta}{1+2\theta}}} \frac{\left[\log(N+M)\right]^{\frac{\theta}{1+2\theta}} \left[\log \log(N+M)\right]^{\frac{4\theta}{1+2\theta}}}{\left[N+M\right]^{\frac{4\theta}{1+2\theta}}}.$$ 
Therefore, \eqref{Power Analysis RFF based permutation and adaptation test Condition 3 simplified under exponential decay}, \eqref{Power Analysis RFF based permutation and adaptation test Condition 7 simplified under exponential decay}, \eqref{Power Analysis RFF based permutation and adaptation test Condition 10 simplified under exponential decay} and \eqref{RFF based permutation and adaptation test Power analysis conditions 1, 2 and 4 simplified under exponential decay} reduce to the following condition:
\begin{equation}
\label{RFF based permutation and adaptation test Power analysis conditions simplified under exponential decay}
\begin{aligned}
    \Delta_{N,M} \gtrsim &\max\left\{\frac{1}{\sqrt{2\theta}},\frac{1}{(\frac{1}{2}+\theta)^{\frac{\theta}{(1+2\theta)}}},1\right\} \times \frac{\left[\log(\frac{8}{\tilde{\alpha}})+\log(\frac{4}{\delta})\right]^{\max\left(2\theta,4\right)}}{\delta^{\frac{2\theta}{1+2\theta}}}\times\max\left\{d_{\theta}^{-2\theta},d_{\theta}^{-\frac{2\theta}{1+2\theta}}\right\}\\&\times \max\left\{\left[\frac{\log(N+M)}{N+M}\right]^{2\theta},\left[\frac{\log l}{l}\right]^{2\theta},\frac{\sqrt{\log(N+M)}\left[\log \log(N+M)\right]}{N+M},\right.\\
    &\qquad\qquad\qquad\left.\frac{\left[\log(N+M)\right]^{\frac{\theta}{1+2\theta}} \left[\log \log(N+M)\right]^{\frac{4\theta}{1+2\theta}}}{\left[N+M\right]^{\frac{4\theta}{1+2\theta}}}\right\}.
\end{aligned}
\end{equation}
Since the true value of $\theta$ is unknown, 
%we only consider that we have a known positive lower bound on $\theta$, say $\theta^{*}$ i.e., 
we assume $\theta \geq \theta^{*}$ for some $\theta^{*} \in (0,\frac{1}{4}]$. Now, we consider two scenarios based on the values of the smoothness index $\theta$.\vspace{1mm}

\textbf{Case I: } Suppose $\theta \geq \max\left\{\frac{1}{2},\theta^{*}\right\}=\frac{1}{2}$. Then, we have that 
\begin{align*}&\frac{\sqrt{\log(N+M)}\left[\log \log(N+M)\right]}{N+M}\\& \gtrsim \max\left\{\left[\frac{\log(N+M)}{N+M}\right]^{2\theta},\frac{\left[\log(N+M)\right]^{\frac{\theta}{1+2\theta}} \left[\log \log(N+M)\right]^{\frac{4\theta}{1+2\theta}}}{\left[N+M\right]^{\frac{4\theta}{1+2\theta}}}\right\}.\end{align*} Further, $\frac{\sqrt{\log(N+M)}\left[\log \log(N+M)\right]}{N+M}\gtrsim  \left[\frac{\log l}{l}\right]^{2\theta}$ if $l \gtrsim \left(N+M\right)^{\frac{1}{2\theta^{*}}}\log(N+M) \gtrsim \left(N+M\right)^{\frac{1}{2\theta}}\log(N+M)^{1-\frac{1}{4\theta}}$. Therefore, provided $l \gtrsim \left(N+M\right)^{\frac{1}{2\theta^{*}}}\log(N+M)$, \eqref{RFF based permutation and adaptation test Power analysis conditions simplified under exponential decay} reduces to 
\begin{equation*}
\label{RFF based permutation and adaptation test Power analysis conditions simplified under exponential decay 1}
\Delta_{N,M} = c^{*}(\tilde{\alpha},\delta,\theta)\frac{\sqrt{\log(N+M)}\left[\log \log(N+M)\right]}{N+M},
\end{equation*}
where $c^{*}(\tilde{\alpha},\delta,\theta)$ 
% \begin{equation}
% \label{Definition of constant in separation boundary, type-2 RFF Permutation and adaptation test exp decay}
% c^{*}(\tilde{\alpha},\delta,\theta) = \max\left\{\frac{1}{\sqrt{2\theta}},\frac{1}{(\frac{1}{2}+\theta)^{\frac{\theta}{(1+2\theta)}}},1\right\} \times \frac{\left[\log(\frac{8}{\tilde{\alpha}})+\log(\frac{4}{\delta})\right]^{\max\left(2\theta,4\right)}}{\delta^{\frac{2\theta}{1+2\theta}}}\times\max\left\{d_{\theta}^{-2\theta},d_{\theta}^{-\frac{2\theta}{1+2\theta}}\right\}
% \end{equation} 
is a positive constant that depends on $\tilde{\alpha}$, $\delta$ and $\theta$.\vspace{1mm}

% Further, Condition \ref{Theorem Power Analysis RFF based permutation test Condition 11} in Theorem \ref{Power Analysis for RFF based test based on permutations} is also satisfied. 

\textbf{Case II: } Suppose $\theta^{*} \leq \theta < \frac{1}{2}$. Then, we have that \begin{align*} &\left[\frac{\log(N+M)}{N+M}\right]^{2\theta}\\& \gtrsim \max\left\{\frac{\sqrt{\log(N+M)}\left[\log \log(N+M)\right]}{N+M},\frac{\left[\log(N+M)\right]^{\frac{\theta}{1+2\theta}} \left[\log \log(N+M)\right]^{\frac{4\theta}{1+2\theta}}}{\left[N+M\right]^{\frac{4\theta}{1+2\theta}}}\right\}.
\end{align*} Further, $\left[\frac{\log(N+M)}{N+M}\right]^{2\theta} \gtrsim \left[\frac{\log l}{l}\right]^{2\theta}$ if $l \gtrsim N+M$. Therefore, provided $l \gtrsim N+M$, \eqref{RFF based permutation and adaptation test Power analysis conditions simplified under exponential decay} reduces to 
\begin{equation*}
\label{RFF based permutation and adaptation test Power analysis conditions simplified under exponential decay 2}
\Delta_{N,M} = c^{*}(\tilde{\alpha},\delta,\theta)\left[\frac{\log(N+M)}{N+M}\right]^{2\theta},
\end{equation*}
where $c^{*}(\tilde{\alpha},\delta,\theta)$ is a positive constant that depends on $\tilde{\alpha}$, $\delta$ and $\theta$. %as defined in \eqref{Definition of constant in separation boundary, type-2 RFF Permutation and adaptation test exp decay}.

% Further, Condition \ref{Theorem Power Analysis RFF based permutation test Condition 11} in Theorem \ref{Power Analysis for RFF based test based on permutations} is also satisfied.

Combining cases I and II, we have that, for any $\theta \geq \theta^{*}$ and  $$l \gtrsim \max\left\{\left(N+M\right)^{\frac{1}{2\theta^{*}}}\log(N+M),N+M\right\},$$ the separation boundary satisfies
\[
    \Delta_{N,M} = 
    c(\tilde{\alpha},\delta,\theta)\max\left\{\left[\frac{\log(N+M)}{N+M}\right]^{2\theta},\frac{\sqrt{\log(N+M)}\left[\log \log(N+M)\right]}{N+M} \right\},
\]
where $c(\tilde{\alpha},\delta,\theta) = \max\left\{(\frac{1}{16c_{1}^{2}})^{2\theta-1},16c_{1}^{2}\right\} \times \left[\max\left\{(\frac{1}{2\theta^{*}})^{\frac{1}{4\theta^{*}}},\sqrt{2}\right\}\right]^{2\theta}\times \frac{\left[\log(\frac{8}{\tilde{\alpha}})+\log(\frac{4}{\delta})\right]^{\frac{4\theta}{\theta^{*}}}}{\delta^{2\theta}}$ is a constant that depends on $\tilde{\alpha}$, $\delta$ and $\theta$, but the dependence on $\theta$ is only through the exponent.  Therefore, 
\[
\lambda^{*} = d_{\theta}\Delta_{NM}^{\frac{1}{2\theta}} = \left[\frac{c^{*}(\tilde{\alpha},\delta,\theta)}{16c_{1}^{2}}\right]^{\frac{1}{2\theta}} \max\left\{\frac{\log(N+M)}{N+M},\frac{\left[\log(N+M)\right]^{\frac{1}{4\theta}}\left[\log \log(N+M)\right]^{\frac{1}{2\theta}}}{\left(N+M\right)^{\frac{1}{2\theta}}} \right\}
\] is a rate-optimal choice of the regularization parameter $\lambda$, which depends on the unknown $\theta$.

Note that the constant $\left[\frac{c(\tilde{\alpha},\delta,\theta)}{16c_{1}^{2}}\right]^{\frac{1}{2\theta}}$ can be expressed as $$ \max\left\{\frac{1}{16c_{1}^{2}},1\right\} \times \max\left\{\left(\frac{1}{2\theta^{*}}\right)^{\frac{1}{4\theta^{*}}},\sqrt{2}\right\}\times \frac{\left[\log(\frac{8}{\tilde{\alpha}})+\log(\frac{4}{\delta})\right]^{\frac{2}{\theta^{*}}}}{\delta}$$ and therefore, it is a constant that depends only $\tilde{\alpha}$ ,$\delta$ and $\theta^{*}$, since $\theta \geq \theta^{*}$. Now, $\frac{\log(N+M)}{N+M} \leq \max\left\{\frac{\log(N+M)}{N+M},\frac{\left[\log(N+M)\right]^{\frac{1}{4\theta}}\left[\log \log(N+M)\right]^{\frac{1}{2\theta}}}{\left(N+M\right)^{\frac{1}{2\theta}}} \right\} \leq 1$. Therefore, for any $\theta \geq \theta^{*}$ and $\tau> 0$, $\lambda^{*}$ can be bounded as $\lambda_{L} = r_{3}\frac{\log(N+M)}{N+M} \leq \lambda^{*} \leq \min\left\{r_{4},e^{-1},\frac{1}{2}\norm{\Sigma_{PQ}}_{\mathcal{L}^{\infty}(\mathcal{H})})\right\}=\lambda_{U}$ for some constants $r_{3},r_{4}>0$ that depend on $\tilde{\alpha}$, $\delta$ and $\theta^{*}$.

Let us define $s^{*} = \underset{\lambda \in \Lambda : \lambda \leq \lambda^{*}}{\sup} \lambda$. Then, one can easily deduce from the definition of $\Lambda$ that $\frac{\lambda^{*}}{2} \leq s^{*} \leq \lambda^{*}$. Further, $\lambda_{L} \leq \lambda^{*} \leq \lambda_{U}$. Therefore, $s^{*}$ is also a rate-optimal choice of the regularization parameter $\lambda$ that will lead to the same conditions on the separation boundary $\Delta_{NM}$ up to constants. Hence, using Lemma \ref{Lemma 17}, for any $\theta \geq \theta^{*}$ and any $(P,Q) \in \mathcal{P}_{\theta,\Delta_{NM}}$, we have 
\[
P_{H_{1}}\left(\underset{\lambda \in \Lambda}{\bigcup}\left\{\hat{\eta}_{\lambda,l} \geq q_{1-\frac{w\alpha}{|\Lambda|}}^{B,\lambda,l}\right\}\right) \geq 1-7\delta.
\]
Taking infimum over $(P,Q) \in \mathcal{P}_{\theta,\Delta_{NM}}$ and  $\theta \geq \theta^{*}$, the proof is complete.

% \end{proof}

\subsection{Proof of Theorem \ref{Type I-error of RFF-based permutation test with adaptation over kernel and lambda}}\label{subsec:thm11}

% \begin{proof}

Choosing $\alpha$ as $\frac{\alpha}{|\Lambda||\mathcal{K}|}$ and under the condition that the number of randomly selected permutations $B \geq \frac{|\Lambda|^{2}|\mathcal{K}|^{2}}{2\tilde{w}^{2}\alpha^{2}}\log(\frac{2|\Lambda||\mathcal{K}|}{\alpha(1-w-\tilde{w})})$, using Theorem \ref{Type I-error of RFF-based permutation test} we have that, for any $(\lambda,K) \in \Lambda \times \mathcal{K}$,
\[
P_{H_{0}}\left\{\hat{\eta}_{\lambda,l,K} \geq q_{1-\frac{w\alpha}{|\Lambda||\mathcal{K}|}}^{B,\lambda,l,K}\right\} \leq \frac{\alpha}{|\Lambda||\mathcal{K}|}.
\]
Consequently, the proof is complete using Lemma \ref{Lemma 17}.

\section*{Acknowledgments}
SM and BKS are partially supported by the National Science Foundation (NSF) CAREER award DMS-1945396.
% \begin{ack}
% Use unnumbered first level headings for the acknowledgments. All acknowledgments
% More information about this disclosure can be found at: \url{https://neurips.cc/Conferences/2024/PaperInformation/FundingDisclosure}.

% Do {\bf not} include this section in the anonymized submission, only in the final paper. You can use the \texttt{ack} environment provided in the style file to automatically hide this section in the anonymized submission.
% \end{ack}

%\section*{References}
%\bibliographystyle{plain}
\bibliographystyle{plainnat}
\bibliography{ref}

\begin{thebibliography}{16}
\providecommand{\natexlab}[1]{#1}
\providecommand{\url}[1]{\texttt{#1}}
\expandafter\ifx\csname urlstyle\endcsname\relax
  \providecommand{\doi}[1]{doi: #1}\else
  \providecommand{\doi}{doi: \begingroup \urlstyle{rm}\Url}\fi

\bibitem[Choi and Kim(2024)]{choi2024computational}
Ikjun Choi and Ilmun Kim.
\newblock {C}omputational-statistical trade-off in kernel two-sample testing with random {F}ourier features.
\newblock \emph{arXiv:2407.08976}, 2024.

\bibitem[Domingo-Enrich et~al.(2025)Domingo-Enrich, Dwivedi, and Mackey]{Cheappermutationtesting}
Carles Domingo-Enrich, Raaz Dwivedi, and Lester Mackey.
\newblock Cheap permutation testing.
\newblock \emph{arXiv:2502.07672}, 2025.

\bibitem[Gretton et~al.(2006)Gretton, Borgwardt, Rasch, Sch\"{o}lkopf, and Smola]{gretton2006kernel}
Arthur Gretton, Karsten Borgwardt, Malte Rasch, Bernhard Sch\"{o}lkopf, and Alex Smola.
\newblock A kernel method for the two-sample problem.
\newblock In B.~Sch\"{o}lkopf, J.~Platt, and T.~Hoffman, editors, \emph{Advances in Neural Information Processing Systems}, volume~19. MIT Press, 2006.

\bibitem[Gretton et~al.(2012)Gretton, Borgwardt, Rasch, Sch{{\"o}}lkopf, and Smola]{gretton2012kernel}
Arthur Gretton, Karsten~M. Borgwardt, Malte~J. Rasch, Bernhard Sch{{\"o}}lkopf, and Alexander Smola.
\newblock A kernel two-sample test.
\newblock \emph{Journal of Machine Learning Research}, 13\penalty0 (25):\penalty0 723--773, 2012.

\bibitem[Hagrass et~al.(2024)Hagrass, Sriperumbudur, and Li]{SpectralTwoSampleTest}
Omar Hagrass, Bharath Sriperumbudur, and Bing Li.
\newblock Spectral regularized kernel two-sample tests.
\newblock \emph{The Annals of Statistics}, 52\penalty0 (3):\penalty0 1076--1101, 2024.

\bibitem[Kim et~al.(2022)Kim, Balakrishnan, and Wasserman]{kim2022minimax}
Ilmun Kim, Sivaraman Balakrishnan, and Larry Wasserman.
\newblock Minimax optimality of permutation tests.
\newblock \emph{The Annals of Statistics}, 50\penalty0 (1):\penalty0 225--251, 2022.

\bibitem[LeCun et~al.(2010)LeCun, Cortes, and Burges]{Mnist}
Y.~LeCun, C.~Cortes, and C.~Burges.
\newblock {MNIST} handwritten digit database. {AT} \&{T} {L}abs, 2010.

\bibitem[Lehmann and Romano(2005)]{lehmann}
Erich~Leo Lehmann and Joseph~P Romano.
\newblock \emph{Testing {S}tatistical {H}ypotheses}.
\newblock Springer, 2005.

\bibitem[Li and Yuan(2024)]{li2019optimality}
Tong Li and Ming Yuan.
\newblock On the optimality of {G}aussian kernel based nonparametric tests against smooth alternatives.
\newblock \emph{Journal of Machine Learning Research}, 25\penalty0 (334):\penalty0 1--62, 2024.

\bibitem[Rahimi and Recht(2007)]{Rahimi-08a}
Ali Rahimi and Benjamin Recht.
\newblock Random features for large-scale kernel machines.
\newblock In J.~Platt, D.~Koller, Y.~Singer, and S.~Roweis, editors, \emph{Advances in Neural Information Processing Systems}, volume~20. Curran Associates, Inc., 2007.

\bibitem[Reed and Simon(1980)]{reed1980methods}
M.~Reed and B.~Simon.
\newblock \emph{Methods of Modern Mathematical Physics: Functional Analysis I}.
\newblock Academic Press, New York, 1980.

\bibitem[Schrab et~al.(2023)Schrab, Kim, Albert, Laurent, Guedj, and Gretton]{schrab2023mmd}
Antonin Schrab, Ilmun Kim, M{\'e}lisande Albert, B{\'e}atrice Laurent, Benjamin Guedj, and Arthur Gretton.
\newblock {MMD} aggregated two-sample test.
\newblock \emph{Journal of Machine Learning Research}, 24\penalty0 (194):\penalty0 1--81, 2023.

\bibitem[Sriperumbudur and Sterge(2022)]{ApproximateKernelPCARandomFeaturesStergeSriperumbudur}
Bharath~K. Sriperumbudur and Nicholas Sterge.
\newblock {Approximate kernel {PCA}: Computational versus statistical trade-off}.
\newblock \emph{The Annals of Statistics}, 50\penalty0 (5):\penalty0 2713 -- 2736, 2022.

\bibitem[Wendland(2004)]{wendland2004scattered}
Holger Wendland.
\newblock \emph{Scattered Data Approximation}, volume~17.
\newblock Cambridge University Press, 2004.

\bibitem[Yurinsky(1995)]{Yurinsky_1995}
Vadim Yurinsky.
\newblock \emph{Sums and Gaussian Vectors}.
\newblock Springer, 1995.

\bibitem[Zhao and Meng(2015)]{zhao2015fastmmd}
Ji~Zhao and Deyu Meng.
\newblock {FastMMD: Ensemble of circular discrepancy for efficient two-sample test}.
\newblock \emph{Neural Computation}, 27\penalty0 (6):\penalty0 1345--1372, 06 2015.

\end{thebibliography}
%\nocite{*}

%%%%%%%%%%%%%%%%%%%%%%%%%%%%%%%%%%%%%%%%%%%%%%%%%%%%%%%%%%%%
%\newpage

% \appendix

% \section{Appendix}

%\subsection{Background}

% (Theorem D.1 in \citep{ApproximateKernelPCARandomFeaturesStergeSriperumbudur})

%\subsection{Auxiliary results}
\setcounter{section}{0}
\renewcommand\thesection{\Alph{section}}
\numberwithin{equation}{section}
\newtheorem{thmm}{Theorem}[section]
\newtheorem{lem}[thmm]{Lemma}
\newtheorem{pro}[thmm]{Proposition}
\newtheorem{rem}{Remark}[section]
\section{Technical Results}
In this section, we provide some auxiliary results required to prove the main results of the paper. Unless otherwise stated, the notations introduced in the main paper carry over to this section as well.
\begin{pro}\label{Proposition: Upper and lower bound of eta}
Let $u=\frac{d P}{d R}-1 \in L^{2}(R)$ and $\eta_{\lambda,l}=\left\|g_{\lambda}^{1/2}(\Sigma_{P Q,l})\left(\mu_{Q,l}-\mu_{P,l}\right)\right\|_{\mathcal{H}_{l}}^{2}$, where the regularizer $g_{\lambda}$ satisfies $\boldsymbol{(\SpectralAssumptionone)}$, $\boldsymbol{(\SpectralAssumptiontwo)}$ and $\boldsymbol{(\SpectralAssumptionthree)}$. Then
\[
\eta_{\lambda,l} \leq 4 C_{1} \|u\|_{L^{2}(R)}^{2}.
\]
% where $C_{1}=1$ (Using assumption A1 in report).
Furthermore, suppose $u \in \operatorname{Ran}(\mathcal{T}_{PQ}^{\theta})$, $\theta>0$, 
\[
\|u\|_{L^{2}(R)}^{2} \geq 16 \lambda^{2 \theta}\left\|\mathcal{T}_{PQ}^{-\theta} u\right\|_{L^{2}(R)}^{2},\,\,\text{and}\,\,
%Let the number of random features $l$ and the regularization parameter $\lambda$ satisfy
l > \max{\left(160,3200 \mathcal{N}_{1}(\lambda)\right)} \frac{ \kappa \log \frac{2}{\delta}}{\lambda }.
\] Then, for any $0 < \delta < 1$, with probability at least $1-\delta$, we have 
\[
\eta_{\lambda,l} \geq \frac{C_{4}}{2}\|u\|_{L^{2}(R)}^{2} .
\]
\end{pro}
% the Bernstein's inequality is almost similar to the one used in the proof of part (iv) of Lemma B.3 in Approximate Kernel PCA Sterge paper arxiv version page 53, but I derived it on my own with slightly different constants but exactly same terms in section 0.2 of extra.tex file in rough work tex

\begin{proof}
Note that
    \begin{equation*}
\begin{split}
    \eta_{\lambda,l} & =\left\|g_{\lambda}^{1/2}(\Sigma_{P Q,l})\left(\mu_{Q,l}-\mu_{P,l}\right)\right\|_{\mathcal{H}_{l}}^{2}
    = \left\langle g_{\lambda}(\Sigma_{P Q,l})(\mu_{P,l}-\mu_{Q,l}),\mu_{P,l}-\mu_{Q,l}\right\rangle_{\mathcal{H}_{l}} \\
    & =4 \left\langle g_{\lambda}(\Sigma_{PQ,l})\mathfrak{A}_{l}^{*}u,\mathfrak{A}_{l}^{*}u\right\rangle_{\mathcal{H}_{l}} 
    =4 \left\langle \mathfrak{A}_{l}g_{\lambda}(\Sigma_{PQ,l})\mathfrak{A}_{l}^{*}u,u\right\rangle_{L^{2}(R)}\\
    &\stackrel{(a)}{=}4 \left\langle \mathcal{T}_{PQ,l} g_{\lambda}(\mathcal{T}_{PQ,l})u,u\right\rangle_{L^{2}(R)}
    \leq 4 \left\|\mathcal{T}_{PQ,l} g_{\lambda}(\mathcal{T}_{PQ,l})\right\|_{\mathcal{L}^{\infty}(L^{2}(R))} \left\|u\right\|_{L^{2}(R)}^{2}\\
    &\stackrel{(b)}{\leq} 4 C_{1}\left\|u\right\|_{L^{2}(R)}^{2},
\end{split}
\end{equation*}
where $(a)$ follows from Lemma~\ref{Lemma 8}\emph{(i)} and $(b)$ follows from  $\boldsymbol{(\SpectralAssumptionone)}$, which provides the required upper bound.

We now proceed to prove the lower bound. First, we observe that
\[
\begin{aligned}
\left\|\Sigma_{P Q, \lambda,l}^{-1 / 2}\left(\mu_{Q,l}-\mu_{P,l}\right)\right\|_{\mathcal{H}_{l}}^{2} &= \left\|\Sigma_{P Q, \lambda,l}^{-1 / 2}g_{\lambda}^{-1/2}(\Sigma_{PQ,l})g_{\lambda}^{1/2}(\Sigma_{PQ,l})\left(\mu_{Q,l}-\mu_{P,l}\right)\right\|_{\mathcal{H}_{l}}^{2}\\
&\leq \norm{\Sigma_{P Q, \lambda,l}^{-1 / 2}g_{\lambda}^{-1/2}(\Sigma_{PQ,l})}_{\mathcal{L}^{\infty}(\mathcal{H}_{l})}^{2} \norm{g_{\lambda}^{1/2}(\Sigma_{PQ,l})\left(\mu_{Q,l}-\mu_{P,l}\right)}_{\mathcal{H}_{l}}^{2}\\
&\stackrel{(c)}{\leq} C_{4}^{-1}  \norm{g_{\lambda}^{1/2}(\Sigma_{PQ,l})\left(\mu_{Q,l}-\mu_{P,l}\right)}_{\mathcal{H}_{l}}^{2},
\end{aligned}
\]
where $(c)$ follows from Lemma \ref{Lemma 8}(iv). Therefore, we have that,
\begin{equation}
\label{General lower bound in terms of Tikhonov lower bound in Proposition eta upper lower bound}
\norm{g_{\lambda}^{1/2}(\Sigma_{PQ,l})\left(\mu_{Q,l}-\mu_{P,l}\right)}_{\mathcal{H}_{l}}^{2} \geq C_{4} \left\|\Sigma_{P Q, \lambda,l}^{-1 / 2}\left(\mu_{Q,l}-\mu_{P,l}\right)\right\|_{\mathcal{H}_{l}}^{2}.
\end{equation}
Further, we note that, 
\begin{equation*}
\begin{aligned}
\label{Decomposition of lower bound in Proposition eta upper lower bound into 3 terms}
    &\left\|\Sigma_{P Q, \lambda,l}^{-1 / 2}\left(\mu_{Q,l}-\mu_{P,l}\right)\right\|_{\mathcal{H}_{l}}^{2}
    = \left\langle\Sigma_{P Q, \lambda,l}^{-1}(\mu_{P,l}-\mu_{Q,l}),\mu_{P,l}-\mu_{Q,l}\right\rangle_{\mathcal{H}_{l}} 
     =4 \left\langle \Sigma_{PQ,\lambda,l}^{-1}\mathfrak{A}_{l}^{*}u,\mathfrak{A}_{l}^{*}u\right\rangle_{\mathcal{H}_{l}}\\ 
    &=4 \left\langle \mathfrak{A}_{l}\Sigma_{PQ,\lambda,l}^{-1}\mathfrak{A}_{l}^{*}u,u\right\rangle_{L^{2}(R)}
    \stackrel{(*)}{=}4 \left\langle  (\mathcal{T}_{PQ,l}+\lambda I)^{-1}\mathcal{T}_{PQ,l}u,u\right\rangle_{L^{2}(R)}\\
&=2\underbrace{\norm{(\mathcal{T}_{PQ,l}+\lambda I)^{-1}\mathcal{T}_{PQ,l}u}_{L^{2}(R)}^{2}}_{T_{1}} + 2\norm{u}_{L^{2}(R)}^{2} - 2\underbrace{\norm{(\mathcal{T}_{PQ,l}+\lambda I)^{-1}\mathcal{T}_{PQ,l}u - u}_{L^{2}(R)}^{2}}_{T_{2}},
\end{aligned}
\end{equation*}
where we used Lemma~\ref{Lemma 8}\emph{(i)} in $(*)$. We now proceed to provide a lower bound on $T_{1}$ and an upper bound on $T_{2}$ which hold with high probability in order to provide a lower bound on $\left\|\Sigma_{P Q, \lambda,l}^{-1 / 2}\left(\mu_{Q,l}-\mu_{P,l}\right)\right\|_{\mathcal{H}_{l}}^{2}$ and consequently to $\eta_{\lambda,l}$ (using \eqref{General lower bound in terms of Tikhonov lower bound in Proposition eta upper lower bound}) which holds with high probability.

Define $M_1:=(\mathcal{T}_{PQ}+\lambda I)^{-1}\mathcal{T}_{PQ}$, $M_2:=(\mathcal{T}_{PQ}+\lambda I)^{-1}(\mathcal{T}_{PQ}-\mathcal{T}_{PQ,l})$, and $t=\norm{M_{2}}_{\mathcal{L}^{2}(L^{2}(R))}$. Observe that $\norm{M_{1}}_{\mathcal{L}^{\infty}(L^{2}(R))} \leq 1$ and $\norm{M_{2}}_{\mathcal{L}^{\infty}(L^{2}(R))} \leq \norm{M_{2}}_{\mathcal{L}^{2}(L^{2}(R))} = t$. Since $L^{2}(R)$ is a separable Hilbert space under  $\boldsymbol{(\RFFAssumptionone)}$, so is $\mathcal{L}^{2}(L^2(R))$. Under the conditions on $l$ and $\lambda$ as stated in the statement of Proposition~\ref{Proposition: Upper and lower bound of eta}, we have that $\frac{4 \kappa \log \frac{2}{\delta}}{\lambda l} \leq \frac{1}{40}$ and $\frac{\sqrt{2\kappa \mathcal{N}_{1}(\lambda)\log \frac{2}{\delta}}}{\sqrt{\lambda l}} \leq \frac{1}{40}$. Therefore, using Bernstein's inequality in $\mathcal{L}^{2}(L^2(R))$ (see Theorem~\ref{Bernstein ineq for Hilbert spaces}), we have that, for any $0<\delta<1$,
\[\label{Bernstein}
\begin{aligned}
P\left\{\theta^{1:l} : t = \norm{M_{2}}_{\mathcal{L}^{2}(L^{2}(R))}\leq \frac{4 \kappa \log \frac{2}{\delta}}{\lambda l} + \frac{\sqrt{2\kappa \mathcal{N}_{1}(\lambda)\log \frac{2}{\delta}}}{\sqrt{\lambda l}} \leq \frac{1}{20}\right\} \geq 1 - \delta,
\end{aligned}
\]
% Proof of above bernstein's inequality is in section 0.2 of extra.tex in rough work project
i.e., we have, with probability at least $1-\delta$, \begin{equation}\label{Bound on t less than 1}
\begin{aligned}
\norm{(\mathcal{T}_{PQ} + \lambda I)^{-1}(\mathcal{T}_{PQ,l} - \mathcal{T}_{PQ})}_{\mathcal{L}^{\infty}(L^{2}(R))} &= \norm{M_{2}}_{\mathcal{L}^{\infty}(L^{2}(R))} \\
&\leq \norm{M_{2}}_{\mathcal{L}^{2}(L^{2}(R))} =t\leq\frac{1}{20}<1.
\end{aligned}
\end{equation}
Note that,
\begin{equation}
\label{T2 upper bound auxiliary result}
\begin{aligned}
\norm{(\mathcal{T}_{PQ}+\lambda I)^{-1}\mathcal{T}_{PQ}u - u}_{L^{2}(R)}^{2} &= \norm{(\mathcal{T}_{PQ}+\lambda I)^{-1}\left[\mathcal{T}_{PQ} - (\mathcal{T}_{PQ} + \lambda)\right]u}_{L^{2}(R)}^{2} \\
&= \lambda^{2} \norm{(\mathcal{T}_{PQ}+\lambda I)^{-1}u }_{L^{2}(R)}^{2}.
\end{aligned}
\end{equation}
Using \eqref{Bound on t less than 1}, the following upper bound on $T_{2} = \norm{(\mathcal{T}_{PQ,l}+\lambda I)^{-1}\mathcal{T}_{PQ,l}u - u}_{L^{2}(R)}^{2}$ holds with probability at least $1-\delta$ :
\begin{equation}
\label{Upper bound on T1}
\begin{aligned}
&\norm{(\mathcal{T}_{PQ,l}+\lambda I)^{-1}\mathcal{T}_{PQ,l}u - u}_{L^{2}(R)}^{2}\\ &= \norm{(\mathcal{T}_{PQ,l}+\lambda I)^{-1}\left[\mathcal{T}_{PQ,l} - (\mathcal{T}_{PQ,l} + \lambda I)\right]u}_{L^{2}(R)}^{2} \\
&= \lambda^{2} \norm{(\mathcal{T}_{PQ,l}+\lambda I)^{-1}(\mathcal{T}_{PQ}+\lambda I)(\mathcal{T}_{PQ}+\lambda I)^{-1}u }_{L^{2}(R)}^{2} \\
&\leq \lambda^{2} \norm{(\mathcal{T}_{PQ,l}+\lambda I)^{-1}(\mathcal{T}_{PQ}+\lambda I)}_{\mathcal{L}^{\infty}(L^{2}(R))}^{2} 
\norm{(\mathcal{T}_{PQ}+\lambda I)^{-1}u }_{L^{2}(R)}^{2} \\
&\stackrel{(d)}{=}\norm{(\mathcal{T}_{PQ,l}+\lambda I)^{-1}(\mathcal{T}_{PQ}+\lambda I)}_{\mathcal{L}^{\infty}(L^{2}(R))}^{2} 
\norm{(\mathcal{T}_{PQ}+\lambda I)^{-1}\mathcal{T}_{PQ}u - u}_{L^{2}(R)}^{2} \\
&=\norm{(\mathcal{T}_{PQ,l} - \mathcal{T}_{PQ} + \mathcal{T}_{PQ} +\lambda I)^{-1}(\mathcal{T}_{PQ}+\lambda I)}_{\mathcal{L}^{\infty}(L^{2}(R))}^{2} \norm{(\mathcal{T}_{PQ}+\lambda I)^{-1}\mathcal{T}_{PQ}u - u}_{L^{2}(R)}^{2} \\
% &=\norm{(\mathcal{T}_{PQ,l} - \mathcal{T}_{PQ} + \mathcal{T}_{PQ} +\lambda I_{l})^{-1}(\mathcal{T}_{PQ}+\lambda I)}_{\mathcal{L}^{\infty}(L^{2}(R))}^{2} \\
% &\times  \norm{(\mathcal{T}_{PQ}+\lambda I)^{-1}\mathcal{T}_{PQ}u - u}_{L^{2}(R)}^{2} \\
&= \norm{\left[(\mathcal{T}_{PQ} + \lambda I)^{-1}(\mathcal{T}_{PQ,l} - \mathcal{T}_{PQ}) + I\right]^{-1}}_{\mathcal{L}^{\infty}(L^{2}(R))}^{2} \norm{(\mathcal{T}_{PQ}+\lambda I)^{-1}\mathcal{T}_{PQ}u - u}_{L^{2}(R)}^{2} \\
&\leq \frac{1}{1- \norm{(\mathcal{T}_{PQ} + \lambda I)^{-1}(\mathcal{T}_{PQ,l} - \mathcal{T}_{PQ})}_{\mathcal{L}^{\infty}(L^{2}(R))}^{2}} \norm{(\mathcal{T}_{PQ}+\lambda I)^{-1}\mathcal{T}_{PQ}u - u}_{L^{2}(R)}^{2} \\
&\leq \frac{1}{1- \norm{(\mathcal{T}_{PQ} + \lambda I)^{-1}(\mathcal{T}_{PQ,l} - \mathcal{T}_{PQ})}_{\mathcal{L}^{2}(L^{2}(R))}^{2}} \norm{(\mathcal{T}_{PQ}+\lambda I)^{-1}\mathcal{T}_{PQ}u - u}_{L^{2}(R)}^{2}\\
& = \frac{1}{1-t^{2}}  \norm{M_{1}u - u}_{L^{2}(R)}^{2},
\end{aligned}
\end{equation}
where $(d)$ follows from \eqref{T2 upper bound auxiliary result}. 

Note that,
\begin{equation}
\label{T1 lower bound auxiliary result}
\begin{aligned}
\norm{(\mathcal{T}_{PQ}  +\lambda I)^{-1}\mathcal{T}_{PQ,l}u}_{L^{2}(R)}^{2} &= \norm{(\mathcal{T}_{PQ}  +\lambda I)^{-1}(\mathcal{T}_{PQ,l}  +\lambda I)(\mathcal{T}_{PQ,l}  +\lambda I)^{-1}\mathcal{T}_{PQ,l}u}_{L^{2}(R)}^{2}\\
&\leq \norm{(\mathcal{T}_{PQ} +\lambda I)^{-1}(\mathcal{T}_{PQ,l} +\lambda I)}_{\mathcal{L}^{\infty}(L^{2}(R))}\\
&\qquad\qquad \times\norm{(\mathcal{T}_{PQ,l}  +\lambda I)^{-1}\mathcal{T}_{PQ,l}u}_{L^{2}(R)}^{2}.
\end{aligned}
\end{equation}

Using \eqref{T1 lower bound auxiliary result} and under the same event for which \eqref{Upper bound on T1} holds, the following lower bound hold for $T_{1} = \norm{(\mathcal{T}_{PQ,l}  +\lambda I)^{-1}\mathcal{T}_{PQ,l}u}_{L^{2}(R)}^{2}$ with probability at least $1-\delta$:
\begin{equation*}
%\label{Lower bound on T2}
\begin{aligned}
& \norm{(\mathcal{T}_{PQ,l}  +\lambda I)^{-1}\mathcal{T}_{PQ,l}u}_{L^{2}(R)}^{2} \\
&\geq \frac{\norm{(\mathcal{T}_{PQ}  +\lambda I)^{-1}\mathcal{T}_{PQ,l}u}_{L^{2}(R)}^{2}}{\norm{(\mathcal{T}_{PQ} +\lambda I)^{-1}(\mathcal{T}_{PQ,l} +\lambda I)}_{\mathcal{L}^{\infty}(L^{2}(R))}^{2} }\\
 &= \frac{\norm{(\mathcal{T}_{PQ}  +\lambda I)^{-1}\mathcal{T}_{PQ,l}u}_{L^{2}(R)}^{2}}{\norm{(\mathcal{T}_{PQ} +\lambda I)^{-1}(\mathcal{T}_{PQ,l} -\mathcal{T}_{PQ}) + I}_{\mathcal{L}^{\infty}(L^{2}(R))}^{2} }\\
 &\geq \frac{\norm{(\mathcal{T}_{PQ}  +\lambda I)^{-1}\mathcal{T}_{PQ,l}u}_{L^{2}(R)}^{2}}{2\norm{(\mathcal{T}_{PQ} +\lambda I)^{-1}(\mathcal{T}_{PQ} -\mathcal{T}_{PQ,l})}_{\mathcal{L}^{\infty}(L^{2}(R))}^{2} + 2}\\
 &\geq \frac{\left(\norm{(\mathcal{T}_{PQ}  +\lambda I)^{-1}\mathcal{T}_{PQ}u}_{L^{2}(R)} - \norm{(\mathcal{T}_{PQ}  +\lambda I)^{-1}(\mathcal{T}_{PQ}-\mathcal{T}_{PQ,l})u}_{L^{2}(R)}\right)^{2}}{2\norm{(\mathcal{T}_{PQ} +\lambda I)^{-1}(\mathcal{T}_{PQ} -\mathcal{T}_{PQ,l})}_{\mathcal{L}^{\infty}(L^{2}(R))}^{2} + 2}\nonumber
 \end{aligned}
\end{equation*}
\begin{equation}
    \label{Lower bound on T2}
\begin{aligned}
 &\geq \frac{\left(\norm{(\mathcal{T}_{PQ}  +\lambda I)^{-1}\mathcal{T}_{PQ}u}_{L^{2}(R)} - \norm{(\mathcal{T}_{PQ}  +\lambda I)^{-1}(\mathcal{T}_{PQ}-\mathcal{T}_{PQ,l})u}_{L^{2}(R)}\right)^{2}}{2\norm{(\mathcal{T}_{PQ} +\lambda I)^{-1}(\mathcal{T}_{PQ} -\mathcal{T}_{PQ,l})}_{\mathcal{L}^{2}(L^{2}(R))}^{2} + 2}\\
 &= \frac{\left(\norm{M_{1}u}_{L^{2}(R)} - \norm{M_{2}u}_{L^{2}(R)}\right)^{2}}{2t^{2} + 2} \\
 &= \frac{\norm{M_{1}u}_{L^{2}(R)}^{2} + \norm{M_{2}u}_{L^{2}(R)}^{2} - 2\norm{M_{1}u}_{L^{2}(R)} \norm{M_{2}u}_{L^{2}(R)}  }{2t^{2} + 2} \\
 &\geq \frac{\norm{M_{1}u}_{L^{2}(R)}^{2} + \norm{M_{2}u}_{L^{2}(R)}^{2} - 2t\norm{u}_{L^{2}(R)}^{2}  }{2t^{2} + 2}\\
 &\geq \frac{\norm{M_{1}u}_{L^{2}(R)}^{2} - 2t\norm{u}_{L^{2}(R)}^{2}   }{2t^{2} + 2}.
\end{aligned}
\end{equation}

Hence, using \eqref{Bound on t less than 1}, \eqref{Upper bound on T1} and \eqref{Lower bound on T2}  we have that, with probability at least $1-\delta$,
\begin{equation*}
%\label{Lower bound on Tikhonov lower bound}
\begin{aligned}
    & \left\|\Sigma_{P Q, \lambda,l}^{-1 / 2}\left(\mu_{Q,l}-\mu_{P,l}\right)\right\|_{\mathcal{H}_{l}}^{2}\\
&=2\norm{(\mathcal{T}_{PQ,l}+\lambda I)^{-1}\mathcal{T}_{PQ,l}u}_{L^{2}(R)}^{2} + 2\norm{u}_{L^{2}(R)}^{2} - 2\norm{(\mathcal{T}_{PQ,l}+\lambda I)^{-1}\mathcal{T}_{PQ,l}u - u}_{L^{2}(R)}^{2} \\
    &\geq \frac{2\norm{M_{1}u}_{L^{2}(R)}^{2} - 4t\norm{u}_{L^{2}(R)}^{2}   }{2t^{2} + 2} + 2\norm{u}_{L^{2}(R)}^{2} - \frac{2}{1-t^{2}} \norm{M_{1}u - u}_{L^{2}(R)}^{2} \\
    &= \frac{1}{1+t^{2}} \norm{M_{1}u}_{L^{2}(R)}^{2} - \frac{2t}{1+t^{2}}\norm{u}_{L^{2}(R)}^{2}  + 2\norm{u}_{L^{2}(R)}^{2} - \frac{2}{1-t^{2}} \norm{M_{1}u - u}_{L^{2}(R)}^{2} \\
    &= \frac{1}{1+t^{2}} \left[\norm{M_{1}u}_{L^{2}(R)}^{2} - \norm{M_{1}u - u}_{L^{2}(R)}^{2} \right] + \left(2 - \frac{2t}{1+t^{2}}\right) \norm{u}_{L^{2}(R)}^{2}  \\
    &\qquad\qquad+\left( \frac{1}{1+t^{2}} - \frac{2}{1-t^{2}} \right)  \norm{M_{1}u - u}_{L^{2}(R)}^{2} \\
    &= \frac{1}{1+t^{2}} \left[\norm{M_{1}u}_{L^{2}(R)}^{2} - \norm{M_{1}u - u}_{L^{2}(R)}^{2} \right] + \left(2 - \frac{2t}{1+t^{2}}\right) \norm{u}_{L^{2}(R)}^{2} \\
    &\qquad\qquad- \frac{1 + 3t^{2}}{1-t^{4}}  \norm{M_{1}u - u}_{L^{2}(R)}^{2} \\
    &= \frac{1}{1+t^{2}} \left[\norm{M_{1}u}_{L^{2}(R)}^{2} - \norm{M_{1}u - u}_{L^{2}(R)}^{2} \right] + \left(2 - \frac{2t}{1+t^{2}}\right) \norm{u}_{L^{2}(R)}^{2} \\
    &\qquad\qquad+ \frac{1 + 3t^{2}}{1-t^{4}} \left[\norm{M_{1}u}_{L^{2}(R)}^{2}-\norm{M_{1}u - u}_{L^{2}(R)}^{2}\right]- \frac{1 + 3t^{2}}{1-t^{4}}  \norm{M_{1}u}_{L^{2}(R)}^{2} \\
    &\stackrel{(*)}{\geq} \left( \frac{1}{1+t^{2}} + \frac{1 + 3t^{2}}{1-t^{4}}\right)\left[\norm{M_{1}u}_{L^{2}(R)}^{2}-\norm{M_{1}u - u}_{L^{2}(R)}^{2}\right]\\
    &\qquad\qquad+ \left(2 - \frac{2t}{1+t^{2}} - \frac{1 + 3t^{2}}{1-t^{4}}\right)  \norm{u}_{L^{2}(R)}^{2} \\
    &= \frac{2(1+t^{2})}{1-t^{4}} \left[\norm{M_{1}u}_{L^{2}(R)}^{2}-\norm{M_{1}u - u}_{L^{2}(R)}^{2}\right] + \frac{1-2t^{4}+2t^{3}-3t^{2}-2t}{1-t^{4}}  \norm{u}_{L^{2}(R)}^{2} \\
    &\stackrel{(**)}{\geq} \frac{1+t^2}{1-t^{4}} \left\{2\left[\norm{M_{1}u}_{L^{2}(R)}^{2}-\norm{M_{1}u - u}_{L^{2}(R)}^{2}\right] +(1- 5t)\norm{u}_{L^{2}(R)}^{2} \right\}\\
        &\stackrel{(\dagger)}{\geq}2\left[\norm{M_{1}u}_{L^{2}(R)}^{2}-\norm{M_{1}u - u}_{L^{2}(R)}^{2}\right] +(1- 5t)\norm{u}_{L^{2}(R)}^{2}
    \end{aligned}
    \end{equation*}
    \begin{equation}
    \label{Lower bound on Tikhonov lower bound}
    \begin{aligned}
    &\stackrel{(e)}{\geq} -\frac{1}{4 } \norm{u}_{L^{2}(R)}^{2} + \frac{3}{4 }\norm{u}_{L^{2}(R)}^{2} 
    = \frac{1}{2 } \norm{u}_{L^{2}(R)}^{2},
\end{aligned}
\end{equation}
where $(e)$ follows from Lemma~\ref{Lemma 7} under the conditions $u=\frac{d P}{d R}-1 \in \operatorname{Ran}(\mathcal{T}_{PQ}^{\theta})$ and $\|u\|_{L^{2}(R)}^{2} \geq 16 \lambda^{2 \theta}\|\mathcal{T}_{PQ}^{-\theta} u\|_{L^{2}(R)}^{2}$. $(*)$ holds since $t<1$. $(**)$ holds since $1-2t^4+2t^3-3t^2-2t\ge (1+t^2)(1-5t)$ for all $0\le t\le \frac{1}{20},$ which indeed is true because simplifying the above inequality, we obtain $3+7t^2\ge 4t+2t^3$ and clearly $\inf_{0\le t\le \frac{1}{20}} 3+7t^2\ge \sup_{0\le t\le\frac{1}{20}} 4t+2t^3$. $(\dagger)$ holds since $1+t^2\ge 1-t^4$.
%\textcolor{red}{ I do not follow how $(**)$ follows in the first term, particularly how 2 is multiplied for one term and 2 is not multiplied for the other term. From $(**)$ to $(\dagger)$, how did you remove $(1-t^4)$? I think if $t\le 1/20$, then $2(1+t^2)\ge 2(1-t^4)$ and also $1-2t^4+2t^3-3t^2-2t\ge (1-t^4)(1-5t)$ and this condition of $t\le 1/20$ is anyway needed in going from $(\dagger)$ to $(e)$ for the second term that $(1-5t)\ge 3/4$.}

Finally, using \eqref{Lower bound on Tikhonov lower bound} and \eqref{General lower bound in terms of Tikhonov lower bound in Proposition eta upper lower bound}, we have that, if $u \in \operatorname{Ran}(\mathcal{T}_{PQ}^{\theta})$, $\|u\|_{L^{2}(R)}^{2} \geq 16 \lambda^{2 \theta}\|\mathcal{T}_{PQ}^{-\theta} u\|_{L^{2}(R)}^{2}$, and $l$ and $\lambda$ satisfy $l > \max{\left(160,3200 \mathcal{N}_{1}(\lambda)\right)} \frac{ \kappa \log (\frac{2}{\delta})}{\lambda }$, then, for any $0 < \delta < 1$, we have 
\[
\eta_{\lambda,l} \geq \frac{C_{4}}{2}\|u\|_{L^{2}(R)}^{2}
\]
with probability at least $1-\delta$.
\end{proof}

\begin{rem}
Under $\boldsymbol {(\RFFAssumptionone)}$, the covariance operator $\Sigma_{PQ}$ and the integral operator 
$\mathcal{T}_{PQ}$ corresponding to the kernel $K$ and distribution $R=\frac{P+Q}{2}$ are trace-class with the same eigenvalues $(\lambda_{i})_{i \in I}$ as defined in \eqref{Spectral representation of SigmaPQ}. All eigenvalues are non-negative, and if the number of eigenvalues is countable, we must have $\lambda_{i} \to 0$ as $i \to \infty$. Two common scenarios arise concerning the rate of decay of $\lambda_{i}$'s: polynomial rate of decay where $\lambda_{i} \asymp i^{-\beta}$ for $\beta > 1$ and exponential rate of decay $\lambda_{i} \asymp e^{-\tau i}$ for $\tau >0$. Under polynomial decay of eigenvalues, the condition on $\lambda$ and $l$ in Proposition~\ref{Proposition: Upper and lower bound of eta} reduces to $\lambda \gtrsim l^{-\frac{\beta}{\beta + 1}}$, while under exponential decay, it reduces to $\lambda \gtrsim \frac{\log l}{l}$. 
%is a sufficient condition on $\lambda$ and $l$ to ensure that Proposition~\ref{Proposition: Upper and lower bound of eta} holds.
\end{rem}

\begin{pro}\label{Proposition: Type-I error bound with random threshold}
Suppose $n,m \geq 2$ and let $\hat{\eta}_{\lambda,l}$ be the test statistic as defined in~\eqref{Approximate Kernel Test statistic}. Further, given any level of significance $\alpha>0$ and any  $0<f<1$, define $$L(\alpha,f) \coloneq \max\left\{2\log\frac{2}{1-\sqrt{1-\frac{\alpha}{2}}},\frac{32\kappa^{2}\log\frac{2}{1-\sqrt{1-\frac{\alpha}{2}}}}{(1-f)^{2}\left\|\Sigma_{PQ}\right\|_{\mathcal{L}^{\infty}(\mathcal{H})}^{2}}\right\}.$$ If $l \geq L(\alpha,f)$ and $\frac{140 \kappa}{s} \log \frac{32 \kappa s}{1-\sqrt{1-\frac{\alpha}{2}}} \leq \lambda \leq f\left\|\Sigma_{P Q}\right\|_{\mathcal{L}^{\infty}(\mathcal{H})}$,  then, under the null hypothesis $H_{0}: P = Q$, we have that,
\[
P_{H_{0}}\left\{\hat{\eta}_{\lambda,l} \geq \gamma_{1,l}\right\} \leq \alpha,
\]
where $\gamma_{1,l}\coloneq\frac{4 \sqrt{3}(C_{1}+C_{2}) \mathcal{N}_{2,l}(\lambda)}{\sqrt{\alpha}}\left(\frac{1}{n}+\frac{1}{m}\right)$ is the (random) critical threshold.
\end{pro}

% Requires Assumptions $\RFFAssumptionone$ and $\RFFAssumptionfour$

\begin{proof}
Let us set $\delta = \frac{\alpha}{2}$ and define the quantities,  $\gamma_{1,l}\coloneq\frac{2 \sqrt{6}(C_{1}+C_{2}) \mathcal{N}_{2,l}(\lambda)}{\sqrt{\delta}}\left(\frac{1}{n}+\frac{1}{m}\right)$ and $\gamma_{2,l}\coloneq\frac{\sqrt{6}(C_{1}+C_{2})\|\mathcal{M}_{l}\|_{\mathcal{L}^{\infty}(\mathcal{H}_{l})}^{2} \mathcal{N}_{2,l}(\lambda)}{\sqrt{\delta}}\left(\frac{1}{n}+\frac{1}{m}\right)$. It is easy to observe that, conditional on $\mathbb{Z}^{1:s}$ and $\theta^{1:l}$,  the conditional expectation of the random feature approximation of the test statistic is zero under the null hypothesis $H_{0}:P = Q$ i.e. $\mathbb{E}_{H_{0}}\left(\hat{\eta}_{\lambda,l} \mid\mathbb{Z}^{1:s},\theta^{1:l}\right)=0$. Therefore, using Lemma \ref{Lemma 11} and Chebychev's inequality, we have that,
\[
P_{H_{0}}\left\{\left|\hat{\eta}_{\lambda,l}\right| \geq \gamma_{2.l} \mid \mathbb{Z}^{1:s},\theta^{1:l}\right\} \leq \delta,
\]
and consequently, we obtain
\begin{equation}
\label{RFF test statistic less than gamm2l with high probability}
\begin{aligned}
P_{H_{0}}\left\{\hat{\eta}_{\lambda,l} \geq \gamma_{2,l}\right\} &\leq P_{H_{0}}\left\{\left|\hat{\eta}_{\lambda,l}\right| \geq \gamma_{2,l}\right\}
=\mathbb{E}_{R^{s} \times \Xi^{l}}\left[P_{H_{0}}\left\{\left|\hat{\eta}_{\lambda,l}\right| \geq \gamma_{2,l} \mid\mathbb{Z}^{1:s},\theta^{1:l}\right\}\right]
\leq \delta.
\end{aligned}
\end{equation}
Finally, using \eqref{RFF test statistic less than gamm2l with high probability} and provided $P_{H_{0}}\left\{\gamma_{2,l} \geq \gamma_{1,l}\right\} \leq \delta $, we have that
\[
\begin{aligned}
P_{H_{0}}\left\{\hat{\eta}_{\lambda,l} \leq \gamma_{1,l}\right\} & \geq P_{H_{0}}\left\{\left\{\hat{\eta}_{\lambda,l} \leq \gamma_{2,l}\right\} \cap\left\{\gamma_{2,l} \leq \gamma_{1,l}\right\}\right\} \\
& \geq 1-P_{H_{0}}\left\{\hat{\eta}_{\lambda,l} \geq \gamma_{2,l}\right\}-P\left\{\gamma_{2,l} \geq \gamma_{1,l}\right\}
 \geq 1-2\delta
 =1 -\alpha.
\end{aligned}
\]
To complete the proof, it only remains to verify that
\begin{equation}
\label{gamma2,l is less than gamma1,l with high probability}
P\left\{\gamma_{2,l} \geq \gamma_{1,l}\right\}=P\left\{\|\mathcal{M}_{l}\|_{\mathcal{L}^{\infty}(\mathcal{H}_l)}^{2} \geq 2\right\} \leq \delta,
\end{equation}
which we do below.

Let us define the event $E = \left\{\theta^{1:l}: \left\|\Sigma_{PQ,l}\right\|_{\mathcal{L}^{\infty}(\mathcal{H}_l)}\geq \lambda\right\}$. We will first prove that under the conditions stated in Proposition \ref{Proposition: Type-I error bound with random threshold}, $P(E) \geq \sqrt{1-\delta}$. Note that $\left\|\Sigma_{PQ,l}\right\|_{\mathcal{L}^{\infty}(\mathcal{H}_l)}=\left\|\mathfrak{A_{l}^{*}A_{l}}\right\|_{\mathcal{L}^{\infty}(\mathcal{H}_l)} = \left\|\mathfrak{A_{l}A_{l}}^{*}\right\|_{\mathcal{L}^{\infty}(L^{2}(R))}$ where $R=\frac{P+Q}{2}$. Using reverse triangle inequality, we have that
\begin{equation}\label{Reverse triangle inequality}
\left|\left\|\mathfrak{II}^{*}\right\|_{\mathcal{L}^{\infty}(L^{2}(R))} - \left\|\mathfrak{A_{l}A_{l}}^{*}-\mathfrak{II}^{*}\right\|_{\mathcal{L}^{\infty}(L^{2}(R))}\right| \leq \left\|\mathfrak{A_{l}A_{l}}^{*}\right\|_{\mathcal{L}^{\infty}(L^{2}(R))}.
\end{equation}
Further, the fact that the operator norm does not exceed the Hilbert-Schmidt norm, we obtain
\begin{equation}
\label{Operator/HS norm iinequality}
\left\|\mathfrak{A_{l}A_{l}}^{*}-\mathfrak{II}^{*}\right\|_{\mathcal{L}^{\infty}(L^{2}(R))} \leq \left\|\mathfrak{A_{l}A_{l}}^{*}-\mathfrak{II}^{*}\right\|_{\mathcal{L}^{2}(L^{2}(R))}.\end{equation}
Now, the lower bound on the number of random features, as assumed under the condition $$l\geq \max\left\{2\log\frac{2}{1-\sqrt{1-\frac{\alpha}{2}}},\frac{32\kappa^{2}\log\frac{2}{1-\sqrt{1-\frac{\alpha}{2}}}}{(1-f)^{2}\left\|\Sigma_{PQ}\right\|_{\mathcal{L}^{\infty}(\mathcal{H})}^{2}}\right\}= L(\alpha,f)$$ implies that $l \geq 2 \log \frac{2}{1-\sqrt{1-\delta}}$ and $4\kappa\sqrt{\frac{2\log\frac{2}{1-\sqrt{1-\delta}}}{l}} \leq (1-f)\left\|\Sigma_{PQ}\right\|_{\mathcal{L}^{\infty}(\mathcal{H})}$. Therefore, using Lemma C.4 from \cite{ApproximateKernelPCARandomFeaturesStergeSriperumbudur} and \eqref{Operator/HS norm iinequality}, we have that,
\[
\begin{aligned}
&P\left\{\theta^{1:l}:\left\|\mathfrak{A_{l}A_{l}}^{*}-\mathfrak{II}^{*}\right\|_{\mathcal{L}^{\infty}(L^{2}(R))}\leq \left\|\mathfrak{A_{l}A_{l}}^{*}-\mathfrak{II}^{*}\right\|_{\mathcal{L}^{2}(L^{2}(R))}\right.\\
&\qquad\qquad\left.\leq 4\kappa\sqrt{\frac{2\log\frac{2}{1-\sqrt{1-\delta}}}{l}} \leq (1-f)\left\|\Sigma_{PQ}\right\|_{\mathcal{L}^{\infty}(\mathcal{H})}\right\}\geq \sqrt{1 - \delta},
\end{aligned}
\]
which implies that
\begin{equation}\label{Application of Lemma B.4 to diff of AAstar and IIstar}
\begin{aligned}
P\left\{\theta^{1:l}: \left\|\mathfrak{II}^{*}\right\|_{\mathcal{L}^{\infty}(L^{2}(R))} - \left\|\mathfrak{A_{l}A_{l}}^{*}-\mathfrak{II}^{*}\right\|_{\mathcal{L}^{\infty}(L^{2}(R))} \geq f\left\|\Sigma_{PQ}\right\|_{\mathcal{L}^{\infty}(\mathcal{H})}\right\} \geq \sqrt{1 - \delta}.
\end{aligned}
\end{equation}
Thus, using \eqref{Application of Lemma B.4 to diff of AAstar and IIstar}, \eqref{Reverse triangle inequality} and \eqref{Operator/HS norm iinequality}, we have,
\begin{equation*}\label{Bound on probability of E part 1}
P\left\{\theta^{1:l}:\left\|\mathfrak{A_{l}A_{l}}^{*}\right\|_{\mathcal{L}^{\infty}(L^{2}(R))}\geq f\left\|\Sigma_{PQ}\right\|_{\mathcal{L}^{\infty}(\mathcal{H})}\right\} \geq \sqrt{1 - \delta},
\end{equation*}
which implies, under the condition $\frac{140 \kappa}{s} \log \frac{32 \kappa s}{1-\sqrt{1-\frac{\alpha}{2}}} \leq \lambda \leq f\left\|\Sigma_{P Q}\right\|_{\mathcal{L}^{\infty}(\mathcal{H})}$, 
\begin{equation}
\label{Bound on probability of E part 2}
P(E) \geq \sqrt{1 - \delta}.
\end{equation}
Let us define $\left\|\mathcal{M}_{l}\right\|_{\mathcal{L}^{\infty}(\mathcal{H}_{l})}\coloneq\left\|\hat{\Sigma}_{P Q, \lambda, l}^{-1 / 2} \Sigma_{P Q, \lambda, l}^{1 / 2}\right\|_{\mathcal{L}^{\infty}(\mathcal{H}_{l})} =\left\| \Sigma_{P Q, \lambda, l}^{1 / 2}\hat{\Sigma}_{P Q, \lambda, l}^{-1 / 2}\right\|_{\mathcal{L}^{\infty}(\mathcal{H}_{l})}$. Under event $E$, the condition $\frac{140 \kappa}{s} \log \frac{32 \kappa s}{1-\sqrt{1-\delta}} \leq \lambda \leq f\left\|\Sigma_{P Q}\right\|_{\mathcal{L}^{\infty}(\mathcal{H})}$ 
implies
\[
\frac{140 \kappa}{s} \log \frac{32 \kappa s}{1-\sqrt{1-\delta}} \leq \lambda \leq \left\|\mathfrak{A_{l}A_{l}}^{*}\right\|_{\mathcal{L}^{\infty}(L^{2}(R))}.
\]
Therefore, using Lemma C.2(ii) of \cite{ApproximateKernelPCARandomFeaturesStergeSriperumbudur}, we have that, conditional on the occurrence of the event $E$,
\begin{equation}\label{Bound on probability of Ml given E}
P\left\{\mathbb{Z}^{1:s}:\sqrt{\frac{2}{3}}\leq\|\mathcal{M}_{l}\|_{\mathcal{L}^{\infty}(\mathcal{H}_{l})} \leq \sqrt{2} \Big| E\right\} \geq \sqrt{1-\delta}.
\end{equation}
Using \eqref{Bound on probability of E part 2} and $\eqref{Bound on probability of Ml given E}$, we have, by the law of total probability,
\begin{equation}\label{Bound of probability of Ml}
\begin{aligned}
& P\left\{\sqrt{\frac{2}{3}}\leq\|\mathcal{M}_{l}\|_{\mathcal{L}^{\infty}(\mathcal{H}_{l})} \leq \sqrt{2}\right\}\\
& = P\left\{\sqrt{\frac{2}{3}}\leq\|\mathcal{M}_{l}\|_{\mathcal{L}^{\infty}(\mathcal{H}_{l})} \leq \sqrt{2}\Big| E\right\} P(E) + P\left\{\sqrt{\frac{2}{3}}\leq\|\mathcal{M}_{l}\|_{\mathcal{L}^{\infty}(\mathcal{H}_{l})} \leq \sqrt{2}\Big| E^{c}\right\} P(E^{c})\\
& \geq P\left\{\sqrt{\frac{2}{3}}\leq\|\mathcal{M}_{l}\|_{\mathcal{L}^{\infty}(\mathcal{H}_{l})} \leq \sqrt{2}\Big| E\right\} P(E)\\
& \geq 1-\delta.
\end{aligned}
\end{equation}
Therefore, using \eqref{Bound of probability of Ml}, we obtain
\[
\begin{aligned}
P\left\{\|\mathcal{M}_{l}\|_{\mathcal{L}^{\infty}(\mathcal{H}_{l})}^{2} \geq 2\right\}
& \leq P\left\{\|\mathcal{M}_{l}\|_{\mathcal{L}^{\infty}(\mathcal{H}_{l})}^{2} \geq 2 \cup \|\mathcal{M}_{l}\|_{\mathcal{L}^{\infty}(\mathcal{H}_{l})}^{2} \leq \frac{2}{3}\right\}\\
& = 1- P\left\{\sqrt{\frac{2}{3}}\leq\|\mathcal{M}_{l}\|_{\mathcal{L}^{\infty}(\mathcal{H}_{l})} \leq \sqrt{2}\right\}\\
& \leq \delta,
\end{aligned}
\]
and this completes the verification of \eqref{gamma2,l is less than gamma1,l with high probability}.
\end{proof}

\begin{lem}\label{Lemma 1}
Let $X$, $Y$, and $Z$ be random variables. Define $\zeta=\mathbb{E}[X \mid Y,Z]$ and let $\gamma$ be any function of  $Y$ and $Z$. Suppose, for any $\delta_{1},\delta_{2}>0$,
\[P\left\{\zeta > \gamma(Y,Z)+\sqrt{\frac{\operatorname{Var}(X \mid Y,Z)}{\delta_{1}}}\right\} \geq 1-\delta_{2}.\] Then, we have
\[
P\{X \geq \gamma(Y,Z)\} \geq 1-\delta_{1} - \delta_{2} .
\]
\end{lem}

\begin{proof}
The proof is similar to that of Lemma A.1 in \citep{SpectralTwoSampleTest}.
\end{proof}

% \begin{proof}
% Define $\gamma_{1}=\sqrt{\frac{\operatorname{Var}(X \mid Y,Z)}{\delta}}$. Consider
% \[
% \begin{aligned}
% P\{X \geq \gamma(Y,Z)\} & \geq P\left\{\left\{X \geq \zeta-\gamma_{1}\right\} \cap\left\{\gamma(Y,Z) \leq \zeta-\gamma_{1}\right\}\right\} \\
% & \geq 1-P\left\{X \leq \zeta-\gamma_{1}\right\}-P\left\{\gamma(Y,Z) \geq \zeta-\gamma_{1}\right\} \\
% & \geq 1-P\left\{|X-\zeta| \geq \gamma_{1}\right\}-P\left\{\gamma(Y,Z) \geq \zeta-\gamma_{1}\right\} \\
% & \geq 1-4 \delta,
% \end{aligned}
% \]
% where in the last step we apply Chebyshev's inequality: $P\left\{|X-\zeta| \geq \gamma_{1}\right\} \leq \delta$.
% \end{proof}

\begin{lem}\label{Lemma 2}
Let $\mu \in \mathcal{H}_{l}$ be any function and define $a(x)=g_{\lambda}^{1/2}(\hat{\Sigma}_{P Q,l})(K_{l}(\cdot, x)-\mu)$. Then $\hat{\eta}_{\lambda,l}$, as defined in \eqref{Approximate Kernel Test statistic}, can be expressed as
\[
\begin{aligned}
\hat{\eta}_{\lambda,l}=\frac{1}{n(n-1)} & \sum_{i \neq j}\left\langle a\left(X_{i}\right), a\left(X_{j}\right)\right\rangle_{\mathcal{H}_{l}}+\frac{1}{m(m-1)} \sum_{i \neq j}\left\langle a\left(Y_{i}\right), a\left(Y_{j}\right)\right\rangle_{\mathcal{H}_{l}} \\
&-\frac{2}{n m} \sum_{i, j}\left\langle a\left(X_{i}\right), a\left(Y_{j}\right)\right\rangle_{\mathcal{H}_{l}} .
\end{aligned}
\]
\end{lem}

\begin{proof}
The proof is similar to that of   \citep[Lemma A.2]{SpectralTwoSampleTest} and is obtained by replacing $\hat{\Sigma}_{P Q}$, $K$ and $\mathcal{H}$ by $\hat{\Sigma}_{P Q,l}$, $K_{l}$ and $\mathcal{H}_{l}$ respectively.
\end{proof}

% \begin{proof}
% The proof follows by using $a(x)+g_{\lambda}^{1/2}(\hat{\Sigma}_{P Q,l})\mu$ for $g_{\lambda}^{1/2}(\hat{\Sigma}_{P Q,l})K_{l}(\cdot, x)$ in $\hat{\eta}_{\lambda,l}$ as shown below:
% \[
% \begin{aligned}
% \hat{\eta}_{\lambda,l}=& \frac{1}{n(n-1)} \sum_{i \neq j}\left\langle a\left(X_{i}\right)+g_{\lambda}^{1/2}(\hat{\Sigma}_{P Q,l}) \mu, a\left(X_{j}\right)+g_{\lambda}^{1/2}(\hat{\Sigma}_{P Q,l}) \mu\right\rangle_{\mathcal{H}_{l}} \\
% &+\frac{1}{m(m-1)} \sum_{i \neq j}\left\langle a\left(Y_{i}\right)+g_{\lambda}^{1/2}(\hat{\Sigma}_{P Q,l}) \mu, a\left(Y_{j}\right)+g_{\lambda}^{1/2}(\hat{\Sigma}_{P Q,l}) \mu\right\rangle_{\mathcal{H}_{l}} \\
% & \quad-\frac{2}{n m} \sum_{i, j}\left\langle a\left(X_{i}\right)+g_{\lambda}^{1/2}(\hat{\Sigma}_{P Q,l}) \mu, a\left(Y_{j}\right)+g_{\lambda}^{1/2}(\hat{\Sigma}_{P Q,l}) \mu\right\rangle_{\mathcal{H}_{l}}
% \end{aligned}
% \]
% and noting that all the terms in the expansion of the inner product cancel except for the terms of the form $\langle a(\cdot), a(\cdot)\rangle_{\mathcal{H}_{l}}$.
% \end{proof}

For the remaining lemmas in this section, let $\theta^{1:l} \coloneq (\theta_{i})_{i=1}^{l}$ be an i.i.d sample from the spectral distribution $\Xi$ corresponding to the kernel $K$. Let the approximate kernel $K_{l}$ be defined as in \eqref{Random feature approximation of kernel} and $\mathcal{H}_{l}$ be the corresponding RKHS.

\begin{lem}\label{Lemma 3} Conditioned on $\theta^{1:l}$, let
$\left(G_{i}\right)_{i=1}^{n}$ and $\left(F_{i}\right)_{i=1}^{m}$ be conditionally independent sequences taking values in $\mathcal{H}_{l}$ such that $\mathbb{E}(G_{i} | \theta^{1:l})=\mathbb{E}(F_{j} | \theta^{1:l})=0$ for all $i=1,\dots,n$ and $j=1,\dots,m$. Let $f \in \mathcal{H}_{l}$ be an arbitrary function. Then, we have the following statements:\\

(i) $\mathbb{E}\left[\left(\sum_{i, j}\left\langle G_{i}, F_{j}\right\rangle_{\mathcal{H}_{l}}\right)^{2}\mid \theta^{1:l}\right]=\sum_{i, j} \mathbb{E}\left[\left\langle G_{i}, F_{j}\right\rangle_{\mathcal{H}_{l}}^{2}\mid \theta^{1:l}\right],\,\,\Xi-a.s. ;$\vspace{1mm}\\

(ii) $\mathbb{E}\left[\left(\sum_{i \neq j}\left\langle G_{i}, G_{j}\right\rangle_{\mathcal{H}_{l}}\right)^{2}\mid \theta^{1:l}\right]=2 \sum_{i \neq j} \mathbb{E}\left[\left\langle G_{i}, G_{j}\right\rangle_{\mathcal{H}_{l}}^{2}\mid \theta^{1:l}\right],\,\,\Xi-a.s. ;$
\vspace{1mm}\\

(iii) $\mathbb{E}\left[\left(\sum_{i}\left\langle G_{i}, f\right\rangle_{\mathcal{H}_{l}}\right)^{2}\mid \theta^{1:l}\right]=\sum_{i} \mathbb{E}\left[\left\langle G_{i}, f\right\rangle_{\mathcal{H}_{l}}^{2}\mid \theta^{1:l}\right],\,\,\Xi-a.s$.
\end{lem}
\begin{proof}
Conditioned on $\theta^{1:l}$, the kernel $K_{l}$ and its corresponding RKHS $\mathcal{H}_{l}$ are non-random. Therefore, following the same steps as in the proof of Lemma A.3 in \citep{SpectralTwoSampleTest} by replacing $\mathcal{H}$ by $\mathcal{H}_{l}$ and expectations with conditional expectations, the above result is proved.
% From this point onwards, the proof is similar to that of Lemma A.3 in \citep{SpectralTwoSampleTest} upon replacing $\mathcal{H}$ by $\mathcal{H}_{l}$ and expectations by conditional expectations.
\end{proof}

\begin{lem}\label{Lemma 4}
Let $\mu_{Q,l}=\int_{\mathcal{X}} K_{l}(\cdot, x) d Q(x)$ be the mean element of $Q$ with respect to the kernel $K_{l}$, $\mathcal{B}: \mathcal{H}_{l} \rightarrow \mathcal{H}_{l}$ be a bounded operator and $\left(X_{i}\right)_{i=1}^{n} \stackrel{i . i . d}{\sim} Q$ with $n \geq 2$. Define
\[
I=\frac{1}{n(n-1)} \sum_{i \neq j}\left\langle a\left(X_{i}\right), a\left(X_{j}\right)\right\rangle_{\mathcal{H}_{l}},
\]
where $a(x)=\mathcal{B} \Sigma_{P Q, \lambda, l}^{-1 / 2}(K_{l}(\cdot, x)-\mu_{Q,l})$. Then, we have the following statements:\\

(i) $\mathbb{E}\left[\left\langle a\left(X_{i}\right), a\left(X_{j}\right)\right\rangle_{\mathcal{H}_{l}}^{2}\mid \theta^{1:l}\right] \leq\|\mathcal{B}\|_{\mathcal{L}^{\infty}(\mathcal{H}_{l})}^{4}\left\|\Sigma_{P Q, \lambda,l}^{-1 / 2} \Sigma_{Q,l} \Sigma_{P Q, \lambda,l}^{-1 / 2}\right\|_{\mathcal{L}^{2}(\mathcal{H}_{l})}^{2},\,\,\Xi-a.s. ;$\vspace{1mm}\\

(ii) $\mathbb{E}\left[I^{2}\mid \theta^{1:l}\right] \leq \frac{4}{n^{2}}\|\mathcal{B}\|_{\mathcal{L}^{\infty}(\mathcal{H}_{l})}^{4}\left\|\Sigma_{P Q, \lambda,l}^{-1 / 2} \Sigma_{Q,l} \Sigma_{P Q, \lambda,l}^{-1 / 2}\right\|_{\mathcal{L}^{2}(\mathcal{H}_{l})}^{2},\,\,\Xi-a.s.$
\end{lem}

\begin{proof}
Conditioned on $\theta^{1:l}$, the kernel $K_{l}$ and its corresponding RKHS $\mathcal{H}_{l}$ are non-random, and so, the proof is similar to that of Lemma A.4 in \citep{SpectralTwoSampleTest} upon replacing $\mathcal{H}$ by $\mathcal{H}_{l}$, expectations by conditional expectations and covariance operators corresponding to the kernel $K$ by covariance operators corresponding to the kernel $K_{l}$. Lemma~\ref{Lemma 3}(ii) is also required for the derivation.
\end{proof}

\begin{lem}\label{Lemma 5}
Let $\mathcal{B}: \mathcal{H}_{l} \rightarrow \mathcal{H}_{l}$ be a bounded operator, $G \in \mathcal{H}_{l}$ be an arbitrary function and $\left(X_{i}\right)_{i=1}^{n} \stackrel{i . i . d}{\sim} Q$. Define
\[
I=\frac{2}{n} \sum_{i=1}^{n}\left\langle a\left(X_{i}\right), \mathcal{B} \Sigma_{P Q, \lambda,l}^{-1 / 2}(G-\mu_{Q,l})\right\rangle_{\mathcal{H}_{l}}
\]
where $a(x)=\mathcal{B} \Sigma_{P Q, \lambda,l}^{-1 / 2}(K_{l}(\cdot, x)-\mu_{Q,l})$ and $\mu_{Q,l}=\int_{\mathcal{X}} K_{l}(\cdot, x) d Q(x)$ is the mean element of $Q$ with respect to the kernel $K_{l}$. Then, we have the following statements:\\

(i) $\begin{aligned} &\mathbb{E}\left[\left\langle a\left(X_{i}\right), \mathcal{B} \Sigma_{P Q, \lambda,l}^{-1 / 2}(G-\mu_{Q,l})\right\rangle_{\mathcal{H}_{l}}^{2}\mid \theta^{1:l}\right]\\ & \leq \|\mathcal{B}\|_{\mathcal{L} \infty(\mathcal{H}_{l})}^{4}\left\|\Sigma_{P Q, \lambda,l}^{-1 / 2} \Sigma_{Q,l} \Sigma_{P Q, \lambda,l}^{-1 / 2}\right\|_{\mathcal{L}^{\infty}(\mathcal{H}_{l})}
\left\|\Sigma_{P Q, \lambda,l}^{-1 / 2}(G-\mu_{Q,l})\right\|_{\mathcal{H}_{l}}^{2},\,\,\Xi-a.s. ;
\end{aligned}$\vspace{1mm}\\

(ii) $\mathbb{E}\left[I^{2}\mid \theta^{1:l}\right] \leq \frac{4}{n}\|\mathcal{B}\|_{\mathcal{L}^{\infty}(\mathcal{H}_{l})}^{4}\left\|\Sigma_{P Q, \lambda}^{-1 / 2} \Sigma_{Q,l} \Sigma_{P Q, \lambda,l}^{-1 / 2}\right\|_{\mathcal{L}^{\infty}(\mathcal{H}_{l})}\left\|\Sigma_{P Q, \lambda,l}^{-1 / 2}(G-\mu_{Q,l})\right\|_{\mathcal{H}_{l}}^{2},\,\,\Xi-a.s.$
\end{lem}

\begin{proof}
Conditioned on $\theta^{1:l}$, the kernel $K_{l}$ and its corresponding RKHS $\mathcal{H}_{l}$ are non-random. Therefore, the proof is similar to that of Lemma A.5 in \citep{SpectralTwoSampleTest} by replacing $\mathcal{H}$ by $\mathcal{H}_{l}$, expectations by conditional expectations, and mean elements and covariance operators corresponding to the kernel $K$ by mean elements and covariance operators corresponding to the kernel $K_{l}$. Lemma~\ref{Lemma 3}(iii) is also required for the derivation.
\end{proof}

\begin{lem}\label{Lemma 6}
Let $\mathcal{B}: \mathcal{H}_{l} \rightarrow \mathcal{H}_{l}$ be a bounded operator, $\left(X_{i}\right)_{i=1}^{n} \stackrel{i . i . d}{\sim} Q$ and $\left(Y_{i}\right)_{i=1}^{m} \stackrel{i . i . d}{\sim} P$. Define
\[
I=\frac{2}{n m} \sum_{i, j}\left\langle a\left(X_{i}\right), b\left(Y_{j}\right)\right\rangle_{\mathcal{H}_{l}},
\]
where $a(x)=\mathcal{B} \Sigma_{P Q, \lambda,l}^{-1 / 2}\left(K_{l}(\cdot, x)-\mu_{Q,l}\right)$, and $b(x)=\mathcal{B} \Sigma_{P Q, \lambda,l}^{-1 / 2}\left(K_{l}(\cdot, x)-\mu_{P,l}\right)$ with \\$\mu_{P,l}=\int_{\mathcal{X}} K_{l}(\cdot, y) d P(y)$ and $\mu_{Q,l}=\int_{\mathcal{X}} K_{l}(\cdot, x) d Q(x)$ being the mean elements of $P$ and $Q$ with respect to the kernel $K_{l}$. Then, we have the following statements:\\

(i) $\mathbb{E}\left[\left\langle a\left(X_{i}\right), b\left(Y_{j}\right)\right\rangle_{\mathcal{H}_{l}}^{2}\mid \theta^{1:l}\right] \leq\|\mathcal{B}\|_{\mathcal{L}^{\infty}(\mathcal{H}_{l})}^{4}\left\|\Sigma_{P Q, \lambda,l}^{-1 / 2} \Sigma_{P Q,l} \Sigma_{P Q, \lambda,l}^{-1 / 2}\right\|_{\mathcal{L}^{2}(\mathcal{H}_{l})}^{2},\,\,\Xi-a.s.;$\vspace{1mm}\\

(ii) $\mathbb{E}\left[I^{2}\mid \theta^{1:l}\right] \leq \frac{4}{n m}\|\mathcal{B}\|_{\mathcal{L}^{\infty}(\mathcal{H}_{l})}^{4}\left\|\Sigma_{P Q, \lambda,l}^{-1 / 2} \Sigma_{P Q,l} \Sigma_{P Q, \lambda,l}^{-1 / 2}\right\|_{\mathcal{L}^{2}(\mathcal{H}_{l})}^{2},\,\,\Xi-a.s.$
\end{lem}

\begin{proof}
Conditioned on $\theta^{1:l}$, the kernel $K_{l}$ and its corresponding RKHS $\mathcal{H}_{l}$ are non-random, and the proof is similar to that of Lemma A.6 in \citep{SpectralTwoSampleTest} upon replacing $\mathcal{H}$ by $\mathcal{H}_{l}$, expectations by conditional expectations, and mean elements and covariance operators corresponding to the kernel $K$ by mean elements and covariance operators corresponding to the kernel $K_{l}$. Lemma~\ref{Lemma 3}(i) is also required for the derivation.
\end{proof}

\begin{lem}\label{Lemma 7}
Let $u=\frac{d P}{d R}-1 \in \operatorname{Ran}(\mathcal{T}_{PQ}^{\theta}) \subset L^{2}(R)$ with $\mathcal{T}_{PQ}$ being the integral operator defined over $L^{2}(R)$ corresponding to the kernel $K$. Further, let $\lambda>0$ be such that $\|u\|_{L^{2}(R)}^{2} \geq 16 \lambda^{2 \theta}\left\|\mathcal{T}_{PQ}^{-\theta} u\right\|_{L^{2}(R)}^{2}$. Define $M_{1} \coloneq (\mathcal{T}_{PQ}+\lambda I)^{-1}\mathcal{T}_{PQ}$. Then, we have  
\[
\norm{M_{1}u}_{L^{2}(R)}^{2} -\norm{M_{1}u-u}_{L^{2}(R)}^{2} \geq -\frac{1}{8 }\norm{u}_{L^{2}(R)}^{2}.
\]
\end{lem}
\begin{proof}
Since $u \in \operatorname{Ran}(\mathcal{T}_{PQ}^{\theta})$, there exists $f\in L^{2}(R)$ such that $u= \mathcal{T}_{PQ}^{\theta}f$. Therefore, we have 
\[
\begin{aligned}
\norm{\mathcal{T}_{PQ}(\mathcal{T}_{PQ}+\lambda I)^{-1} u}_{L^{2}(R)}^2 &\stackrel{(*)}{=} \norm{(\mathcal{T}_{PQ}+\lambda I)^{-1}\mathcal{T}_{PQ} u}_{L^{2}(R)}^2 
= \norm{M_{1}u}_{L^{2}(R)}^{2}\\
&= \sum_{i} \lambda_i^{2\theta+2}\left(\frac{1}{\lambda_{i}+\lambda}\right)^{2}\langle f,\Tilde{\phi}_i\rangle_{L^{2}(R)}^2,\end{aligned}
\]
where $(*)$ follows from Lemma A.8(i) in \cite{SpectralTwoSampleTest} (similar to Lemma~\ref{Lemma 8}(i)). Similarly,
\[\begin{aligned}
\norm{\mathcal{T}_{PQ}(\mathcal{T}_{PQ} + \lambda I)^{-1} u-u}_{L^{2}(R)}^2 &= \norm{(\mathcal{T}_{PQ} + \lambda I)^{-1}\mathcal{T}_{PQ} u-u}_{L^{2}(R)}^2
= \norm{M_{1}u-u}_{L^{2}(R)}^{2} \\
&= \sum_{i} \lambda_i^{2\theta}\left(\frac{\lambda_i}{\lambda_{i}+\lambda}-1\right)^2 \langle f,\Tilde{\phi}_i\rangle_{L^{2}(R)}^2,
\end{aligned}
\]
where $(\lambda_i,\tilde{\phi}_i)_i$ are the eigenvalues and eigenfunctions of $\mathcal{T}_{PQ}$. Using these expressions, we have 
\[
\begin{aligned}
\norm{M_{1}u}_{L^{2}(R)}^{2} -\norm{M_{1}u-u}_{L^{2}(R)}^{2} &= 2 \sum_{i} \lambda_{i}^{2\theta} \left(\frac{\lambda_{i}}{\lambda_{i} + \lambda} - \frac{1}{2}\right)\langle f,\Tilde{\phi}_i\rangle_{L^{2}(R)}^2 \\
&\geq 2 \sum_{\{i:\frac{\lambda_{i}}{\lambda_{i} + \lambda}< \frac{1}{2}\}} \lambda_{i}^{2\theta} \left(\frac{\lambda_{i}}{\lambda_{i} + \lambda} - \frac{1}{2}\right)\langle f,\Tilde{\phi}_i\rangle_{L^{2}(R)}^2 .
\end{aligned}
\]
It is easy to verify that
\begin{equation*}
    \sup_{\{i:\frac{\lambda_{i}}{\lambda{i}+\lambda}< \frac{1}{2}\}} \lambda_i^{2\theta}\left(\frac{1}{2}-\frac{\lambda_{i}}{\lambda_{i}+\lambda}\right) \leq \lambda^{2\theta}. 
\end{equation*}
Therefore, we have that,
\[
\norm{M_{1}u}_{L^{2}(R)}^{2} -\norm{M_{1}u-u}_{L^{2}(R)}^{2} \geq -2\lambda^{2\theta} \norm{\mathcal{T}_{PQ}^{-\theta}u}_{L^{2}(R)}^{2} \geq -\frac{1}{8 }\norm{u}_{L^{2}(R)}^{2}
\]
since $\norm{u}_{L^{2}(R)}^{2} \geq 16 \lambda^{2\theta} \norm{\mathcal{T}_{PQ}^{-\theta}u}_{L^{2}(R)}^{2}$.
\end{proof}

\begin{lem}\label{Lemma 8}
Let $g_{\lambda}$ satisfy $\boldsymbol{(\SpectralAssumptionone)}$, $\boldsymbol{(\SpectralAssumptiontwo)}$ and $\boldsymbol{(\SpectralAssumptionfour)}$. Then, we have the following statements:\\

(i) $\mathfrak{A}_{l}g_{\lambda}(\Sigma_{P Q,l}) \mathfrak{A}_{l}^{*}=\mathcal{T}_{PQ,l} g_{\lambda}(\mathcal{T}_{PQ,l})=g_{\lambda}(\mathcal{T}_{PQ,l}) \mathcal{T}_{PQ,l},\,\,\Xi-a.s. ;$\vspace{1mm}\\

(ii) $\left\|g_{\lambda}^{1/2}(\Sigma_{P Q,l}) \Sigma_{P Q, \lambda,l}^{1 / 2}\right\|_{\mathcal{L}^ \infty(\mathcal{H}_{l})} \leq\left(C_{1}+C_{2}\right)^{1 / 2},\,\,\Xi-a.s. ;$
%We have, in fact, $\left\|\left(\Sigma_{P Q,\lambda,l}\right)^{-1/2} \Sigma_{P Q, %\lambda,l}^{1 / 2}\right\|_{\mathcal{L} \infty(\mathcal{H})} =1$
\vspace{1mm}\\

(iii) $\left\|g_{\lambda}^{1/2}(\hat{\Sigma}_{P Q,l}) \hat{\Sigma}_{P Q, \lambda,l}^{1 / 2}\right\|_{\mathcal{L}^ \infty(\mathcal{H}_{l})} \leq \left(C_{1}+C_{2}\right)^{1 / 2},\,\,\Xi-a.s. ;$
%We have, in fact, $\left\|\left(\hat{\Sigma}_{P Q,\lambda,l}\right)^{-1/2} \hat{\Sigma}_{P Q, \lambda,l}^{1 / 2}\right\|_{\mathcal{L} \infty(\mathcal{H})}=1$.
\vspace{1mm}\\

(iv) $\left\|\Sigma_{P Q, \lambda,l}^{-1 / 2} g_{\lambda}^{-1/2}(\Sigma_{P Q,l}) \right\|_{\mathcal{L}^ \infty(\mathcal{H}_{l})} \leq \left(C_{4}\right)^{-1 / 2} ,\,\,\Xi-a.s.;$\vspace{1mm}\\

(v) $\left\|\hat{\Sigma}_{P Q, \lambda,l}^{-1 / 2} g_{\lambda}^{-1/2}(\hat{\Sigma}_{P Q,l}) \right\|_{\mathcal{L} ^\infty(\mathcal{H}_{l})} \leq \left(C_{4}\right)^{-1 / 2},\,\,\Xi-a.s$.
\end{lem}
\begin{proof}
Conditioned on $\theta^{1:l}$, the kernel $K_{l}$ and its corresponding RKHS $\mathcal{H}_{l}$ are non-random, and the proof therefore is similar to that of Lemma A.8 in \citep{SpectralTwoSampleTest} upon replacing $\mathcal{H}$ by $\mathcal{H}_{l}$, inclusion operator corresponding to the kernel $K$ by the approximation operator corresponding to the kernel $K_{l}$, and covariance and integral operators corresponding to the kernel $K$ by covariance and integral operators corresponding to the kernel $K_{l}$.
\end{proof}

\begin{lem}\label{Lemma 9}
Let $u = \frac{d P}{d R}-1 \in L^{2}(R)$ where $R=\frac{P+Q}{2}$. Further, define \\
$\mathcal{N}_{1,l}(\lambda):=\operatorname{Tr}\left(\Sigma_{P Q, \lambda,l}^{-1 / 2} \Sigma_{P Q,l} \Sigma_{P Q, \lambda,l}^{-1 / 2}\right)$ and $\mathcal{N}_{2,l}(\lambda):=\left\|\Sigma_{P Q, \lambda,l}^{-1 / 2} \Sigma_{P Q,l} \Sigma_{P Q, \lambda,l}^{-1 / 2}\right\|_{\mathcal{L}^{2}(\mathcal{H}_{l})}$. Then, we have the following statements:\\

 (i) $\left\|\Sigma_{P Q, \lambda, l}^{-1 / 2} \Sigma_{A, l} \Sigma_{P Q, \lambda, l}^{-1 / 2}\right\|_{\mathcal{L}^{2}(\mathcal{H}_{l})}^{2} \leq 4 C_{\lambda,l}\|u\|_{L^{2}(R)}^{2}+2 \mathcal{N}_{2,l}^{2}(\lambda),\,\,\Xi-a.s.$\vspace{1mm}\\

(ii) $\left\|\Sigma_{P Q, \lambda, l}^{-1 / 2} \Sigma_{A,l} \Sigma_{P Q, \lambda,l}^{-1 / 2}\right\|_{\mathcal{L}^{\infty}(\mathcal{H}_{l})} \leq 2 \sqrt{C_{\lambda,l}}\|u\|_{L^{2}(R)}+1,\,\,\Xi-a.s,$\vspace{2mm}\\
where $A$ can be either $P$ or $Q$ with $\Sigma_{P,l}$ and $\Sigma_{Q,l}$ being the covariance operators corresponding to the kernel $K_{l}$ and distributions $P$ and $Q$ respectively, and $C_{\lambda,l} = \frac{2\mathcal{N}_{2,l}(\lambda)}{\lambda} \sup _{x}\|K_{l}(\cdot, x)\|_{\mathcal{H}_{l}}^{2}$. 
% \[
% C_{\lambda,l}=\left\{\begin{array}{ll}
% \mathcal{N}_{1,l}(\lambda) \sup _{i}\left\|\phi_{i,l}\right\|_{\infty}^{2}, & \sup _{i}\left\|\phi_{i,l}\right\|_{\infty}^{2}<\infty \\
% \frac{2\mathcal{N}_{2,l}(\lambda)}{\lambda} \sup _{x}\|K_{l}(\cdot, x)\|_{\mathcal{H}_{l}}^{2}, & \text { otherwise }
% \end{array} \right.
% \]
%where $(\lambda_{l,i},\phi_{l,i})_{i \in I}$ be the eigenvalues and eigenfunctions of $\Sigma_{PQ,l}$.
\end{lem}

\begin{proof}
Conditioned on $\theta^{1:l}$, the kernel $K_{l}$ and its corresponding RKHS $\mathcal{H}_{l}$ are non-random, and the proof therefore follows from Lemma A.9 of \citep{SpectralTwoSampleTest} upon replacing $\mathcal{H}$ by $\mathcal{H}_{l}$, inclusion operator corresponding to the kernel $K$ by the approximation operator corresponding to kernel $K_{l}$, and covariance and integral operators corresponding to the kernel $K$ by covariance and integral operators corresponding to the kernel $K_{l}$.
\end{proof}

\begin{lem}\label{Lemma 10}
Let $\zeta_{l} = \left\|g_{\lambda}^{1/2}(\hat{\Sigma}_{PQ,l})(\mu_{Q,l}-\mu_{P,l})\right\|_{\mathcal{H}_{l}}^{2}$, $\eta_{\lambda,l}= \left\|g_{\lambda}^{1/2}(\Sigma_{PQ,l})(\mu_{Q,l}-\mu_{P,l})\right\|_{\mathcal{H}_{l}}^{2}$ and $\mathcal{M}_{l} = \hat{\Sigma}_{PQ,\lambda,l}^{-1/2}\Sigma_{PQ,\lambda,l}^{1/2}$. Then, we have
\[\zeta_{l} \geq C_4 (C_1+C_2)^{-1}\norm{\mathcal{M}_{l}^{-1}}^{-2}_{\mathcal{L}^{\infty}(\mathcal{H}_{l})} \eta_{\lambda,l},\,\,\Xi-a.s.\]
\end{lem}

\begin{proof}
Conditioned on $\theta^{1:l}$, the kernel $K_{l}$ and its corresponding RKHS $\mathcal{H}_{l}$ are non-random, and the proof is therefore similar to that of Lemma A.11 in \citep{SpectralTwoSampleTest} upon replacing $\mathcal{H}$ by $\mathcal{H}_{l}$ and covariance operators corresponding to the kernel $K$ by covariance operators corresponding to the kernel $K_{l}$.
\end{proof}

\begin{lem}\label{Lemma 11}
Define $\zeta_{l}=\left\|g_{\lambda}^{1/2}(\hat{\Sigma}_{P Q,l})\left(\mu_{P,l}-\mu_{Q,l}\right)\right\|_{\mathcal{H}_{l}}^{2}$, and $ \mathcal{M}_{l}=\hat{\Sigma}_{P Q, \lambda, l}^{-1 / 2} \Sigma_{P Q, \lambda, l}^{1 / 2}$. Further, let $m \leq n \leq D^{\prime} m$ for some constant $D^{\prime} \geq 1$. Then, the following statement holds:
\[
\begin{aligned}
&\mathbb{E}\left[\left(\hat{\eta}_{\lambda,l}-\zeta_{l}\right)^{2} \mid \mathbb{Z}^{1:s},\theta^{1:l}\right] \\
&\quad \leq \tilde{C}\|\mathcal{M}_{l}\|_{\mathcal{L}^{\infty}(\mathcal{H}_{l})}^{4}\left(\frac{C_{\lambda,l}\|u\|_{L^{2}(R)}^{2}+\mathcal{N}_{2,l}^{2}(\lambda)}{(n+m)^{2}}+\frac{\sqrt{C_{\lambda,l}}\|u\|_{L^{2}(R)}^{3}+\|u\|_{L^{2}(R)}^{2}}{n+m}\right),\,\,R\times\Xi-a.s.
\end{aligned}
\]
where $C_{\lambda,l}$ is defined in Lemma \ref{Lemma 9} and $\tilde{C}$ is a constant that depends only on $C_{1}, C_{2}$ and $D^{\prime}$. In addition, if $P=Q$, then the following statement holds:
\[
\mathbb{E}\left[\hat{\eta}_{\lambda,l}^{2} \mid \mathbb{Z}^{1:s},\theta^{1:l}\right] \leq 6(C_{1}+C_{2})^{2}\|\mathcal{M}_{l}\|_{\mathcal{L}^{\infty}(\mathcal{H}_{l})}^{4} \mathcal{N}_{2,l}^{2}(\lambda)\left(\frac{1}{n^{2}}+\frac{1}{m^{2}}\right),\,\,R\times\Xi-a.s.
\]
\end{lem}

\begin{proof}
Conditioned on $\theta^{1:l}$, the kernel $K_{l}$ and its corresponding RKHS $\mathcal{H}_{l}$ are non-random. Therefore the proof is similar to that of Lemma A.11 in \citep{SpectralTwoSampleTest} upon replacing $\mathcal{H}$ by $\mathcal{H}_{l}$, conditional expectations given only $\mathbb{Z}^{1:s}$ by conditional expectations given both $\mathbb{Z}^{1:s}$ and $\theta^{1:l}$, and mean elements and covariance operators corresponding to the kernel $K$ by mean elements and covariance operators corresponding to the kernel $K_{l}$.
\end{proof}

\begin{lem}
\label{Lemma 12}
Let $H$ be a Hilbert space and $\Sigma:H\rightarrow H$ be trace class self-adjoint operator with eigenvalues $(\lambda_i(\Sigma))_i$. Then, the following hold:\vspace{1mm}\\

(i) Suppose $ai^{-\alpha}\le \lambda_i(\Sigma)\le Ai^{-\alpha}$ for $\alpha>1$ and $a,A\in (0,\infty)$. Then 
$$N_{2}(\lambda) = \norm{(\Sigma+\lambda I)^{-1/2}\Sigma (\Sigma+\lambda I)^{-1/2}}_{\mathcal{L}^{2}(H)}\asymp \lambda^{-\frac{1}{2\alpha}},\,\,\text{i.e.},\,\,\lambda^{-\frac{1}{2\alpha}} \lesssim N_{2}(\lambda) \lesssim \lambda^{-\frac{1}{2\alpha}}.$$

%(i) Consider that the eigenvalues of $\Sigma$ decay towards 0 as $i \to \infty$ at a polynomial rate i.e. for some $a,A>0$, and $\alpha>1$, $ai^{-\alpha} \leq \lambda_{i} \leq Ai^{-\alpha}$. Then, $\lambda^{-\frac{1}{2\alpha}} \lesssim N_{2}(\lambda) \lesssim \lambda^{-\frac{1}{2\alpha}}$ i.e. $N_{2}(\lambda) \asymp \lambda^{-\frac{1}{2\alpha}}$.\\
(ii) Suppose $be^{-\tau i} \leq \lambda_{i}(\Sigma) \leq Be^{-\tau i}$ for $\tau>0$, and $b,B\in(0,\infty)$. Then $$\log\frac{1}{\lambda} \lesssim N_{2}(\lambda) \lesssim \log\frac{1}{\lambda},\,\,\text{i.e.},\,\, N_{2}(\lambda) \asymp \log\frac{1}{\lambda}.$$
%Consider that the eigenvalues of $\Sigma$ decay towards 0 as $i \to \infty$ at an exponential rate i.e. for some $b,B>0$, and $\alpha>1$, $be^{-\tau i} \leq \lambda_{i} \leq Be^{-\tau i}$. Then, $\log(\frac{1}{\lambda}) \lesssim N_{2}(\lambda) \lesssim \log(\frac{1}{\lambda})$ i.e. $N_{2}(\lambda) \asymp \log(\frac{1}{\lambda})$.
\end{lem}

\begin{proof}
Note that 
\[
\begin{aligned}
N_{2}^{2}(\lambda) &=  \norm{(\Sigma+\lambda I)^{-1/2}\Sigma (\Sigma+\lambda I)^{-1/2}}_{\mathcal{L}^{2}(H)} =\operatorname{Tr}((\Sigma+\lambda I)^{-1/2}\Sigma (\Sigma+\lambda I)^{-1}\Sigma (\Sigma+\lambda I)^{-1/2})\\
&=\operatorname{Tr}\left((\Sigma+\lambda I)^{-2}\Sigma^{2}\right)
=\sum_{i \geq 1} \left(\frac{\lambda_{i}}{\lambda_{i}+t}\right)^{2}.
\end{aligned}
\]
(i) Under the polynomial decay of eigenvalues of $\Sigma$, we have
\[
\begin{aligned}
N_{2}^{2}(\lambda) &= \sum_{i \geq 1} \left(\frac{\lambda_{i}}{\lambda_{i}+\lambda}\right)^{2}
\leq \sum_{i \geq 1} \frac{A^{2}i^{-2\alpha}}{\left(ai^{-\alpha}+\lambda\right)^{2}}=\frac{A^{2}}{a^{2}}\sum_{i \geq 1} \left(\frac{i^{-\alpha}}{i^{-\alpha}+\frac{\lambda}{a}}\right)^{2}\\
&\leq \frac{A^{2}}{a^{2}}\int_{0}^{\infty} \left(\frac{x^{-\alpha}}{x^{-\alpha}+\frac{\lambda}{a}}\right)^{2} dx
= \frac{A^{2}}{a^{2}} \left(\frac{a}{\lambda}\right)^{\frac{1}{\alpha}} \int_{0}^{\infty} \left(\frac{1}{1+x^{\alpha}}\right)^{2} dx.
\end{aligned}
\]
Due to the finiteness of the integral when $\alpha>1$, we thus obtain that $N_{2}^{2}(\lambda) \lesssim \lambda^{-\frac{1}{\alpha}}$ or equivalently, $N_{2}(\lambda) \lesssim \lambda^{-\frac{1}{2\alpha}}$. The proof of the lower bound is similar.\vspace{1mm}\\
(ii) Under the exponential decay of eigenvalues of $\Sigma$, we have
\[
\begin{aligned}
N_{2}^{2}(\lambda) &= \sum_{i \geq 1} \left(\frac{\lambda_{i}}{\lambda_{i}+t}\right)^{2}
\leq \sum_{i \geq 1} \frac{B^{2}e^{-2\tau i}}{\left(be^{-\tau i}+\lambda\right)^{2}}=\frac{B^{2}}{b^{2}}\sum_{i \geq 1} \left(\frac{e^{-\tau i}}{e^{-\tau i}+\frac{\lambda}{b}}\right)^{2}\\
&\leq \frac{B^{2}}{b^{2}}\int_{0}^{\infty} \left(\frac{1}{1+\frac{\lambda}{b}e^{\tau x}}\right)^{2} dx\leq \frac{B^{2}}{b^{2}} \frac{1}{\tau}\log\left(1+\frac{b}{\lambda}\right)
\lesssim \log\frac{1}{\lambda}.
\end{aligned}
\]
The proof of the lower bound is similar.
\end{proof}

\begin{lem}\label{Lemma 13}
Let $\boldsymbol{(\RFFAssumptionone)}$  and $\boldsymbol{(\RFFAssumptionfour)}$ hold. Then, for any $0<\delta<1$, and $\frac{86 \kappa}{l}\log \frac{32 \kappa l}{\delta} \leq \lambda \leq \|\Sigma_{PQ}\|_{\mathcal{L}^{\infty}(\mathcal{H})}$, we have
\[
P\left\{\mathcal{N}_{2,l}^{2}(\lambda) \leq \frac{32\kappa  \mathcal{N}_{1}(\lambda) \log \frac{4}{\delta}}{\lambda l} + \frac{256\kappa^{2}\left(\log \frac{4}{\delta}\right)^{2}}{\lambda^{2}l^2} + 8 \mathcal{N}_{2}^{2}(\lambda)\right\} \geq 1-\delta.
\]
As a corollary, we have that, for any $0<\delta<1$, and $\frac{86 \kappa}{l}\log \frac{32 \kappa l}{\delta} \leq \lambda \leq \|\Sigma_{PQ}\|_{\mathcal{L}^{\infty}(\mathcal{H})}$, 
\[
P\left\{\mathcal{N}_{2,l}(\lambda) \leq  \frac{4\sqrt{2\kappa  \mathcal{N}_{1}(\lambda)\log \frac{4}{\delta}}}{\sqrt{\lambda l}} + \frac{16\kappa \log \frac{4}{\delta}}{\lambda l}  + 2\sqrt{2}\mathcal{N}_{2}(\lambda)\right\} \geq 1-\delta.
\]
\end{lem}

\begin{proof}
Define $A=\mathcal{T}_{PQ,l}$ , $B=\mathcal{T}_{PQ}$ , $A_{\lambda} = \mathcal{T}_{PQ,\lambda,l} = \mathcal{T}_{PQ,l}+ \lambda I_l = A + \lambda I_l$ and $B_{\lambda} = \mathcal{T}_{PQ,\lambda} =\mathcal{T}_{PQ}+ \lambda I = B + \lambda I$. Note that, since $(\mathfrak{I}\mathfrak{I}^{*} + \lambda I)\mathfrak{I} = \mathfrak{I}(\mathfrak{I}^{*}\mathfrak{I} + \lambda I)$, we have that
\begin{equation}
\label{Lemma 13 identity 1}
\mathfrak{I} (\mathfrak{I}^{*}\mathfrak{I} + \lambda I)^{-1} = (\mathfrak{I}\mathfrak{I}^{*} + \lambda I)^{-1}\mathfrak{I}
\end{equation}
and therefore,
\begin{equation*}
\label{N2^2lambda identity}
\begin{aligned}
\mathcal{N}_{2}^{2}(\lambda) &=\left\|\Sigma_{P Q, \lambda}^{-1 / 2} \Sigma_{P Q} \Sigma_{P Q, \lambda}^{-1 / 2}\right\| _{\mathcal{L}^{2}(\mathcal{H})}^{2} 
= \operatorname{Tr}\left((\Sigma_{P Q, \lambda}^{-1 / 2} \Sigma_{P Q} \Sigma_{P Q, \lambda}^{-1 / 2})^{*}\Sigma_{P Q, \lambda}^{-1 / 2} \Sigma_{P Q} \Sigma_{P Q, \lambda}^{-1 / 2}\right)\\
&= \operatorname{Tr}\left(\Sigma_{P Q, \lambda}^{-1 / 2} \Sigma_{P Q} \Sigma_{P Q, \lambda}^{-1 } \Sigma_{P Q} \Sigma_{P Q, \lambda}^{-1 / 2}\right)
= \operatorname{Tr}\left(\Sigma_{P Q}\Sigma_{P Q, \lambda}^{-1 } \Sigma_{P Q} \Sigma_{P Q, \lambda}^{-1 }\right)\\
& = \operatorname{Tr}\left[\mathfrak{I}^{*}\mathfrak{I} (\mathfrak{I}^{*}\mathfrak{I} + \lambda I)^{-1}\mathfrak{I}^{*}\mathfrak{I} (\mathfrak{I}^{*}\mathfrak{I} + \lambda I)^{-1}\right]\\
&\stackrel{(a)}{=}\operatorname{Tr}\left[\mathfrak{I}^{*}(\mathfrak{I}\mathfrak{I}^{*} + \lambda I)^{-1}\mathfrak{I}\mathfrak{I}^{*}(\mathfrak{I}\mathfrak{I}^{*} + \lambda I)^{-1}\mathfrak{I}\right]\\
& = \operatorname{Tr}\left[(\mathfrak{I}\mathfrak{I}^{*} + \lambda I)^{-1}\mathfrak{I}\mathfrak{I}^{*}(\mathfrak{I}\mathfrak{I}^{*} + \lambda I)^{-1}\mathfrak{I}\mathfrak{I}^{*}\right]\\
& = \operatorname{Tr}\left[B_{\lambda}^{-1}BB_{\lambda}^{-1}B\right],
\end{aligned}
\end{equation*}
where $(a)$ follows from \eqref{Lemma 13 identity 1}. Similarly, we have
\begin{equation}
\label{N2l^2lambda identity}
\begin{aligned}
\mathcal{N}_{2,l}^{2}(\lambda) &=\left\|\Sigma_{P Q, \lambda,l}^{-1 / 2} \Sigma_{P Q,l} \Sigma_{P Q, \lambda,l}^{-1 / 2}\right\| _{\mathcal{L}^{2}(\mathcal{H}_{l})}^{2}
= \operatorname{Tr}\left(\Sigma_{P Q, \lambda,l}^{-1 / 2} \Sigma_{P Q,l} \Sigma_{P Q, \lambda,l}^{-1 } \Sigma_{P Q,l} \Sigma_{P Q, \lambda,l}^{-1 / 2}\right)\\
&= \operatorname{Tr}\left(\Sigma_{P Q,l}\Sigma_{P Q, \lambda,l}^{-1 } \Sigma_{P Q,l} \Sigma_{P Q, \lambda,l}^{-1 }\right)
%\quad \textrm{ (Using the cyclic property of trace operator) }
\\
& = \operatorname{Tr}\left[(\mathfrak{A}_{l}\mathfrak{A}_{l}^{*} + \lambda I)^{-1}\mathfrak{A}_{l}\mathfrak{A}_{l}^{*}(\mathfrak{A}_{l}\mathfrak{A}_{l}^{*} + \lambda I)^{-1}\mathfrak{A}_{l}\mathfrak{A}_{l}^{*}\right]
%\, \textrm{(Using the cyclic property of trace operator) }
\\
& = \operatorname{Tr}\left[A_{\lambda}^{-1}AA_{\lambda}^{-1}A\right].
\end{aligned}
\end{equation}
Note that,
\[
\begin{aligned}
A_{\lambda} &= A + \lambda I 
= B + \lambda I + A - B\\
& = \left(B+ \lambda I\right)^{1/2} \left[I + \left(B+ \lambda I\right)^{-1/2} (A-B) \left(B+ \lambda I\right)^{-1/2}\right] \left(B+ \lambda I\right)^{1/2}\\
&=B_{\lambda}^{1/2}\left[I + B_{\lambda}^{-1/2}(A -B)B_{\lambda}^{-1/2}\right] B_{\lambda}^{1/2}
\end{aligned}
\]
and hence,
\begin{equation}
\label{Alambda inverse identity}
A_{\lambda}^{-1} = B_{\lambda}^{-1/2}\left[I + B_{\lambda}^{-1/2}(A -B)B_{\lambda}^{-1/2}\right]^{-1} B_{\lambda}^{-1/2}. 
\end{equation}
Let us define $E_{l} = B_{\lambda}^{-1/2}(B-A)B_{\lambda}^{-1/2} = \left(\mathfrak{I}\mathfrak{I}^{*}+ \lambda I\right)^{-1/2}\left(\mathfrak{I}\mathfrak{I}^{*} - \mathfrak{A}_{l}\mathfrak{A}_{l}^{*}\right)\left(\mathfrak{I}\mathfrak{I}^{*}+ \lambda I\right)^{-1/2}$. Therefore, we have, using \eqref{N2l^2lambda identity} and \eqref{Alambda inverse identity},
\begin{equation*}
%\label{N2l^2lambda identity advanced}
\begin{aligned}
&\mathcal{N}_{2,l}^{2}(\lambda)\\
=& \operatorname{Tr}\left[B_{\lambda}^{-1/2}\left[I + B_{\lambda}^{-1/2}(A -B)B_{\lambda}^{-1/2}\right]^{-1} B_{\lambda}^{-1/2} A B_{\lambda}^{-1/2}\left[I + B_{\lambda}^{-1/2}(A -B)B_{\lambda}^{-1/2}\right]^{-1} B_{\lambda}^{-1/2} A\right]\\
=&\operatorname{Tr}\left[\left[I + B_{\lambda}^{-1/2}(A -B)B_{\lambda}^{-1/2}\right]^{-1} B_{\lambda}^{-1/2} A B_{\lambda}^{-1/2}\left[I + B_{\lambda}^{-1/2}(A -B)B_{\lambda}^{-1/2}\right]^{-1} B_{\lambda}^{-1/2} AB_{\lambda}^{-1/2}\right]\\
\leq & \left\|\left(I + B_{\lambda}^{-1/2}(A -B)B_{\lambda}^{-1/2}\right)^{-1}\right\|_{\mathcal{L}^{\infty}(\mathcal{H})}\\
&\times\operatorname{Tr}\left[B_{\lambda}^{-1/2} A B_{\lambda}^{-1/2}\left[I + B_{\lambda}^{-1/2}(A -B)B_{\lambda}^{-1/2}\right]^{-1} B_{\lambda}^{-1/2} AB_{\lambda}^{-1/2}\right]\\
=& \left\|\left(I + B_{\lambda}^{-1/2}(A -B)B_{\lambda}^{-1/2}\right)^{-1}\right\|_{\mathcal{L}^{\infty}(\mathcal{H})}\operatorname{Tr}\left[\left[I + B_{\lambda}^{-1/2}(A -B)B_{\lambda}^{-1/2}\right]^{-1} B_{\lambda}^{-1/2} AB_{\lambda}^{-1} A B_{\lambda}^{-1/2}\right]
\end{aligned}
\end{equation*}
\begin{equation}
\label{N2l^2lambda identity advanced}
\begin{aligned}
\leq & \left\|\left(I + B_{\lambda}^{-1/2}(A -B)B_{\lambda}^{-1/2}\right)^{-1}\right\|_{\mathcal{L}^{\infty}(\mathcal{H})}^{2} \operatorname{Tr}\left[ B_{\lambda}^{-1/2} AB_{\lambda}^{-1} A B_{\lambda}^{-1/2}\right]\\
=& \left\|\left(I + B_{\lambda}^{-1/2}(A -B)B_{\lambda}^{-1/2}\right)^{-1}\right\|_{\mathcal{L}^{\infty}(\mathcal{H})}^{2} \operatorname{Tr}\left[ B_{\lambda}^{-1} AB_{\lambda}^{-1} A\right]\\
=& \left\|\left(I - E_{l}\right)^{-1}\right\|_{\mathcal{L}^{\infty}(\mathcal{H})}^{2} \operatorname{Tr}\left[ B_{\lambda}^{-1} AB_{\lambda}^{-1} A\right]\\
\leq & \frac{1}{(1 -\|E_{l}\|_{\mathcal{L}^{\infty}(\mathcal{H})} )^{2}}\operatorname{Tr}\left[ B_{\lambda}^{-1} AB_{\lambda}^{-1} A\right].
\end{aligned}
\end{equation}
As part of the proof of Lemma C.3(i) in \cite{ApproximateKernelPCARandomFeaturesStergeSriperumbudur}, it is proved that, if $\delta>0$ and $\frac{86 \kappa}{l}\log \frac{32 \kappa l}{\delta} \leq \lambda \leq \|\Sigma_{PQ}\|_{\mathcal{L}^{\infty}(\mathcal{H})}$, then 
\begin{equation}
\label{Upper bound on El with high probability}
P\left(\theta^{1:l}:\|E_{l}\|_{\mathcal{L}^{\infty}(\mathcal{H})} \leq \frac{1}{2}\right) \geq 1- \frac{\delta}{2}.
\end{equation}
Hence, using \eqref{Upper bound on El with high probability} and \eqref{N2l^2lambda identity advanced}, we have, with probability at least $1-\frac{\delta}{2}$, 
\begin{equation}
\label{N2l^2lambda upper bound by constant times trace}
\mathcal{N}_{2,l}^{2}(\lambda) \leq 4 \operatorname{Tr}\left[ B_{\lambda}^{-1} AB_{\lambda}^{-1} A\right],\end{equation}
where
\begin{equation*}
\label{Upper bound on trace for upper bound on N2l^2lambda}
\begin{aligned}
\operatorname{Tr}\left[ B_{\lambda}^{-1} AB_{\lambda}^{-1} A\right]
&= \operatorname{Tr}\left[ B_{\lambda}^{-1/2} A B_{\lambda}^{-1/2} B_{\lambda}^{-1/2} A B_{\lambda}^{-1/2}\right]\\
&= \operatorname{Tr}\left[\left( B_{\lambda}^{-1/2} A B_{\lambda}^{-1/2}\right)^{*} B_{\lambda}^{-1/2} A B_{\lambda}^{-1/2}\right]
= \norm{B_{\lambda}^{-1/2} A B_{\lambda}^{-1/2}}_{\mathcal{L}^{2}(L^{2}(R))}^{2}\\
&\leq 2\norm{B_{\lambda}^{-1/2} (A-B) B_{\lambda}^{-1/2}}_{\mathcal{L}^{2}(L^{2}(R))}^{2} + 2\norm{B_{\lambda}^{-1/2} B B_{\lambda}^{-1/2}}_{\mathcal{L}^{2}(L^{2}(R))}^{2}\\\
&= 2\underbrace{\norm{B_{\lambda}^{-1/2} (A-B) B_{\lambda}^{-1/2}}_{\mathcal{L}^{2}(L^{2}(R))}^{2}}_{M} + 2\mathcal{N}_{2}^{2}(\lambda).
\end{aligned}
\end{equation*}
Provided that $P\left(M=\norm{B_{\lambda}^{-1/2} (A-B) B_{\lambda}^{-1/2}}_{\mathcal{L}^{2}(L^{2}(R))}^{2} \geq \frac{4\kappa  \mathcal{N}_{1}(\lambda)}{\lambda l} \log \frac{4}{\delta} + \frac{32\kappa^{2}\left(\log \frac{4}{\delta}\right)^{2}}{\lambda^{2}l^2}\right) \leq \frac{\delta}{2}$, we have, with probability at least $1-\frac{\delta}{2}$,
\begin{equation}
\label{Upper bound on trace for upper bound on N2l^2lambda advanced}
\operatorname{Tr}\left[ B_{\lambda}^{-1} AB_{\lambda}^{-1} A\right] \leq \frac{8\kappa  \mathcal{N}_{1}(\lambda)}{\lambda l} \log \frac{4}{\delta} + \frac{64\kappa^{2}\left(\log \frac{4}{\delta}\right)^{2}}{\lambda^{2}l^2} + 2\mathcal{N}_{2}^{2}(\lambda).
\end{equation}
Hence, using \eqref{N2l^2lambda upper bound by constant times trace}  and \eqref{Upper bound on trace for upper bound on N2l^2lambda advanced}, if $\delta>0$ and $\frac{86 \kappa}{l}\log \frac{32 \kappa l}{\delta} \leq \lambda \leq \|\Sigma_{PQ}\|_{\mathcal{L}^{\infty}(\mathcal{H})}$
, then with probability at least $1-\delta$,
\[
\begin{aligned}
\mathcal{N}_{2,l}^{2}(\lambda)\leq \frac{32\kappa  \mathcal{N}_{1}(\lambda)}{\lambda l} \log \frac{4}{\delta} + \frac{256\kappa^{2}\left(\log \frac{4}{\delta}\right)^{2}}{\lambda^{2}l^2} + 8 \mathcal{N}_{2}^{2}(\lambda).
\end{aligned}
\]
Using the fact that $\sqrt{\sum_{k} a_{k}} \leq \sum_{k} \sqrt{a_{k}}$, we obtain the corollary: if $\delta>0$ and $\frac{86 \kappa}{l}\log \frac{32 \kappa l}{\delta} \leq \lambda \leq \|\Sigma_{PQ}\|_{\mathcal{L}^{\infty}(\mathcal{H})}$
, we have, with probability at least $1-\delta$,
\[
\mathcal{N}_{2,l}(\lambda) \leq  \frac{4\sqrt{2\kappa  \mathcal{N}_{1}(\lambda)} \log \frac{4}{\delta}}{\sqrt{\lambda l}} + \frac{16\kappa \log \frac{4}{\delta}}{\lambda l}  + 2\sqrt{2}\mathcal{N}_{2}(\lambda).
\]

To complete the proof, it only remains to verify that $P\left(M \geq \frac{4\kappa  \mathcal{N}_{1}(\lambda) \log \frac{4}{\delta}}{\lambda l} + \frac{32\kappa^{2}\left(\log \frac{4}{\delta}\right)^{2}}{\lambda^{2}l^2}\right) \leq \frac{\delta}{2}$, which we do below. 
%We assume that Assumptions $(\RFFAssumptionone)$ and $(\RFFAssumptionfour)$ hold. 
Let us define
\[
\zeta_{i}=\left(\varphi\left(\cdot, \theta_{i}\right)-\left(1 \otimes_{L^{2}(R)} 1\right) \varphi\left(\cdot, \theta_{i}\right)\right) \otimes_{L^{2}(R)}\left(\varphi\left(\cdot, \theta_{i}\right)-\left(1 \otimes_{L^{2}(R)} 1\right) \varphi\left(\cdot, \theta_{i}\right)\right) = \tau_{i} \otimes_{L^{2}(R)} \tau_{i}
\]
where $\tau_{i}\coloneq\varphi\left(\cdot, \theta_{i}\right)-$ $\left(1 \otimes_{L^{2}(R)} 1\right) \varphi\left(\cdot, \theta_{i}\right)$. Then, it can be shown that $\mathbb{E}_{\Xi}\left[\zeta_{1}\right]=\mathfrak{I}\mathfrak{I}^{*} = \mathcal{T}_{PQ}$ and $\frac{1}{l} \sum_{i=1}^{l} \zeta_{i}=\mathfrak{A}_{l} \mathfrak{A}_{l}^{*} = \mathcal{T}_{PQ,l}$. 
Therefore, we have
\begin{equation*}
\label{Expressing M as empirical average norm squared}
\begin{aligned}
M&=\norm{B_{\lambda}^{-1/2} (A-B) B_{\lambda}^{-1/2}}_{\mathcal{L}^{2}(L^{2}(R))}^{2}\\
&= \left\|(\mathfrak{I}\mathfrak{I}^{*}+\lambda I)^{-1/2} (\mathfrak{A}_{l}\mathfrak{A}_{l}^{*}-\mathfrak{I}\mathfrak{I}^{*})(\mathfrak{I}\mathfrak{I}^{*}+\lambda I)^{-1/2}\right\|_{\mathcal{L}^{2}(L^{2}(R))}^{2}\\
&= \left\|\frac{1}{l}\sum_{i=1}^{l} (\mathfrak{I}\mathfrak{I}^{*}+\lambda I)^{-1/2}\left[\zeta_{i}-\mathbb{E}_{\Xi}(\zeta_{1})\right] (\mathfrak{I}\mathfrak{I}^{*}+\lambda I)^{-1/2}\right\|_{\mathcal{L}^{2}(L^{2}(R))}^{2}\\
&=\left\|\frac{1}{l}\sum_{i=1}^{l} \mu_{i}\right\|_{\mathcal{L}^{2}(L^{2}(R))}^{2},
\end{aligned}
\end{equation*}
where $\mu_{i} = (\mathfrak{I}\mathfrak{I}^{*}+\lambda I)^{-1/2}\left[\zeta_{i}-\mathbb{E}_{\Xi}(\zeta_{1})\right] (\mathfrak{I}\mathfrak{I}^{*}+\lambda I)^{-1/2}$. Next, we proceed to find bounds on the norm of $\mu_{i}$ and the second moment of the norm of $\mu_{i}$ to apply Bernstein's inequality. 
% and obtain a concentration inequality for the squared Hilbert-Schmidt norm of $M$. 
To this end, note that $$
\begin{aligned}
\mu_{i} &= (\mathfrak{I}\mathfrak{I}^{*}+\lambda I)^{-1/2}\left[\zeta_{i}-\mathbb{E}_{\Xi}(\zeta_{1})\right] (\mathfrak{I}\mathfrak{I}^{*}+\lambda I)^{-1/2}\\
&= Z_{i} \otimes_{L^{2}(R)} Z_{i} - (\mathfrak{I}\mathfrak{I}^{*}+\lambda I)^{-1/2}\mathfrak{II^{*}} (\mathfrak{I}\mathfrak{I}^{*}+\lambda I)^{-1/2}\\
&=U_{i} - (\mathfrak{I}\mathfrak{I}^{*}+\lambda I)^{-1/2}\mathfrak{II^{*}} (\mathfrak{I}\mathfrak{I}^{*}+\lambda I)^{-1/2},
\end{aligned}
$$
where $Z_{i} = (\mathfrak{I}\mathfrak{I}^{*}+\lambda I)^{-1/2} \tau_{i} = B_{\lambda}^{-1/2} \tau_{i}$ and $U_{i}=Z_{i} \otimes_{L^{2}(R)} Z_{i}$.
Further, we have that
\[\mathbb{E}_{\Xi}\left[U_{i}\right] = (\mathfrak{I}\mathfrak{I}^{*}+\lambda I)^{-1/2}\mathfrak{II^{*}} (\mathfrak{I}\mathfrak{I}^{*}+\lambda I)^{-1/2},\]$\mathbb{E}_{\Xi} \left[\mu_{i}\right] =0$ and $\mathcal{L}^{2}(L^{2}(R))$ is a separable Hilbert space. Moreover, 
\[
\begin{aligned}
\left\|U_{i}\right\|_{\mathcal{L}^{2}(L^{2}(R))} = \left\|Z_{i}\right\|_{L^{2}(R)}^{2}
\leq \left\|(\mathfrak{I}\mathfrak{I}^{*}+\lambda I)^{-1/2}\right\|_{\mathcal{L}^{\infty}(L^{2}(R))}^{2}  \left\| \tau_{i}\right\|_{L^{2}(R)}^{2}
\leq \frac{\kappa}{\lambda}. 
\end{aligned}
\]
Define $B_{\mu} \coloneq \frac{2\kappa}{\lambda}$ and ${\sigma_{\mu}}^{2} \coloneq \frac{\kappa \mathcal{N}_{1}(\lambda)}{\lambda} $. Clearly, $\left\|\mu_{i}\right\|_{\mathcal{L}^{2}(L^{2}(R))} \leq B_{\mu}$. Further, observe that, 
\[\begin{aligned}
\E_{\Xi}\norm{Z_i}_{L^{2}(R)}^{2} &= \mathbb{E}_{\Xi}\langle Z_{i}, Z_{i}\rangle_{L^{2}(R)} 
= \mathbb{E}_{\Xi}\langle \tau_i, B_{\lambda}^{-1}\tau_i\rangle_{\mathcal{L}^{2}(L^{2}(R))} \\
&= \mathbb{E}_{\Xi} \operatorname{Tr}\left[B_{\lambda}^{-1}(\tau_{i}\otimes_{L^{2}(R)}\tau_{i})\right]=\operatorname{Tr}\left[B_{\lambda}^{-1}B\right]=\mathcal{N}_{1}(\lambda).
\end{aligned}
\]
% Let us define ${\sigma_{\mu}}^{2} \coloneq \frac{\kappa N_{1}(\lambda)}{\lambda} $. Then, we obtain
% \[
% \begin{aligned}
% \mathbb{E}_{\lambda} \left\|\mu_{i}\right\|_{\mathcal{L}^{2}(L^{2}(R))}^{2}  &= \mathbb{E}_{\lambda} \left\|U_{i}\right\|_{\mathcal{L}^{2}(L^{2}(R))}^{2} - \left\|\mathbb{E}_{\lambda}U_{i}\right\|_{\mathcal{L}^{2}(L^{2}(R))}^{2} \\
% &\leq \mathbb{E}_{\lambda} \left\|U_{i}\right\|_{\mathcal{L}^{2}(L^{2}(R))}^{2} \\
% &= \mathbb{E}_{\lambda} \langle U_{i}, U_{i}\rangle_{\mathcal{L}^{2}(L^{2}(R))} \\
% &= \sup_{\theta_{i}} \left\| B_{\lambda}^{-1/2} \zeta_{i} B_{\lambda}^{-1/2} \right\|_{\mathcal{L}^{\infty}(L^{2}(R))} \times \mathbb{E}_{\lambda}\operatorname{Tr}(B_{\lambda}^{-1/2} \zeta_{i} B_{\lambda}^{-1/2})\\
% &\leq \left\|B_{\lambda}^{-1/2}\right\|_{L^{\infty}(L^{2}(R))} \times  \sup_{\theta_{i}} \left\|\zeta_{i}\right\|_{L^{\infty}(L^{2}(R))} \times \left\|B_{\lambda}^{-1/2}\right\|_{L^{\infty}(L^{2}(R))} \times \operatorname{Tr}(B_{\lambda}^{-1}B)\\
% &\leq \frac{1}{\sqrt{\lambda}} \times \kappa \times \frac{1}{\sqrt{\lambda}} \times N_{1}(\lambda)\\
% &= \frac{\kappa N_{1}(\lambda)}{\lambda}= {\sigma_{\mu}}^{2}.
% \end{aligned}
% \]
Consequently, we have that
%Let us define ${\sigma_{\mu}}^{2} \coloneq \frac{\kappa}{\lambda} $. 
%Then, we obtain
% \[
% \begin{aligned}
% \mathbb{E}_{\Xi} \left\|\mu_{i}\right\|_{\mathcal{L}^{2}(L^{2}(R))}^{2}  &= \mathbb{E}_{\Xi} \left\|U_{i}\right\|_{\mathcal{L}^{2}(L^{2}(R))}^{2} - \left\|\mathbb{E}_{\Xi}U_{i}\right\|_{\mathcal{L}^{2}(L^{2}(R))}^{2} 
% \leq \mathbb{E}_{\Xi} \left\|U_{i}\right\|_{\mathcal{L}^{2}(L^{2}(R))}^{2} 
% = \mathbb{E}_{\Xi} \langle U_{i}, U_{i}\rangle_{\mathcal{L}^{2}(L^{2}(R))} \\
% &= \mathbb{E}_{\Xi} \operatorname{Tr}\left[(Z_{i}\otimes_{L^{2}(R)}Z_{i})^{*}(Z_{i}\otimes_{L^{2}(R)}Z_{i})\right]
% = \mathbb{E}_{\Xi}\langle Z_{i}, Z_{i}\rangle_{\mathcal{L}^{2}(L^{2}(R))}^{2}\\
% &\leq \left\|(\mathfrak{I}\mathfrak{I}^{*}+\lambda I)^{-1/2}\right\|_{\mathcal{L}^{\infty}(L^{2}(R))}^{2}  \mathbb{E}_{\Xi} \left\| \tau_{i}\right\|_{L^{2}(R)}^{2}\\
% &\leq \frac{\kappa}{\lambda}. 
% % &= \sup_{\theta_{i}} \left\| B_{\lambda}^{-1/2} \zeta_{i} B_{\lambda}^{-1/2} \right\|_{\mathcal{L}^{\infty}(L^{2}(R))} \times \mathbb{E}_{\lambda}\operatorname{Tr}(B_{\lambda}^{-1/2} \zeta_{i} B_{\lambda}^{-1/2})\\
% % &\leq \left\|B_{\lambda}^{-1/2}\right\|_{L^{\infty}(L^{2}(R))} \times  \sup_{\theta_{i}} \left\|\zeta_{i}\right\|_{L^{\infty}(L^{2}(R))} \times \left\|B_{\lambda}^{-1/2}\right\|_{L^{\infty}(L^{2}(R))} \times \operatorname{Tr}(B_{\lambda}^{-1}B)\\
% % &\leq \frac{1}{\sqrt{\lambda}} \times \kappa \times \frac{1}{\sqrt{\lambda}} \times N_{1}(\lambda)\\
% % &= \frac{\kappa N_{1}(\lambda)}{\lambda}= {\sigma_{\mu}}^{2}.
% \end{aligned}
% \]
%\textcolor{blue}{
\[
\begin{aligned}
\mathbb{E}_{\Xi} \left\|\mu_{i}\right\|_{\mathcal{L}^{2}(L^{2}(R))}^{2}  &= \mathbb{E}_{\Xi} \left\|U_{i}\right\|_{\mathcal{L}^{2}(L^{2}(R))}^{2} - \left\|\mathbb{E}_{\Xi}U_{i}\right\|_{\mathcal{L}^{2}(L^{2}(R))}^{2} 
\leq \mathbb{E}_{\Xi} \left\|U_{i}\right\|_{\mathcal{L}^{2}(L^{2}(R))}^{2} 
= \mathbb{E}_{\Xi} \langle U_{i}, U_{i}\rangle_{\mathcal{L}^{2}(L^{2}(R))} \\
&= \mathbb{E}_{\Xi} \operatorname{Tr}\left[(Z_{i}\otimes_{L^{2}(R)}Z_{i})^{*}(Z_{i}\otimes_{L^{2}(R)}Z_{i})\right]
= \mathbb{E}_{\Xi}\langle Z_{i}, Z_{i}\rangle_{L^{2}(R)}^{2}= \mathbb{E}_{\Xi}\norm{Z_{i}}_{L^{2}(R)}^{4}\\
&\leq \frac{\kappa}{\lambda}. \E_{\Xi}\norm{Z_i}_{L^{2}(R)}^{2}
\leq \frac{\kappa N_{1}(\lambda)}{\lambda}=\sigma^2_\mu.
\end{aligned}
\]
% \textcolor{red}{[I think the above bound is wrong. Shouldn't it be $1/\lambda^2$ because we are computing $\mathbb{E}\Vert Z_i\Vert^4$? I think this bound can be improved to $N_1(\lambda)/\lambda$.]} \textcolor{blue}{[Soumya: Yes the bound should be $\frac{\kappa \mathcal{N}_1(\lambda)}{\lambda}$. It does not affect the ultimate dependence between $\lambda$ and $l$, as can be seen from Remark \ref{{Remark: Conditions for N2l to behave as N2}}. I have fixed the definitions of quantities that depend on this bound throughout this paper, so this issue is fixed now. My changes are marked in blue.]}

Using Bernstein's inequality in separable Hilbert spaces (see Theorem~\ref{Bernstein ineq for Hilbert spaces}), we have that, for any $0<\delta<1$,
\[
P\left(\left\|\frac{1}{l}\sum_{i=1}^{l}\mu_{i}\right\|_{\mathcal{L}^{2}(L^{2}(R))} \geq \frac{\sqrt{2\sigma^2_{\mu}\log \frac{4}{\delta}}}{\sqrt{l}} + \frac{2B_{\mu}\log \frac{4}{\delta}}{l}\right) \leq \frac{\delta}{2}.
\]
Using the fact $(a+b)^{2}\leq 2(a^2+b^2)$, we have that, for any $0<\delta<1$,
\[
\begin{aligned}
&P\left(\left\|\frac{1}{l}\sum_{i=1}^{l}\mu_{i}\right\|_{\mathcal{L}^{2}(L^{2}(R))}^{2} \geq \frac{4\sigma_{\mu}^{2} \log \frac{4}{\delta}}{l} + \frac{8B_{\mu}^{2}\left(\log \frac{4}{\delta}\right)^{2}}{l^2}\right) \leq \frac{\delta}{2}.
% \\
% \textrm{ i.e. }& P\left(\norm{M}_{\mathcal{L}^{2}(L^{2}(R))}^{2} \geq \frac{4\kappa \log \frac{4}{\delta}}{\lambda l} + \frac{32\kappa^{2}\left(\log \frac{4}{\delta}\right)^{2}}{\lambda^{2}l^2}\right) \leq \frac{\delta}{2}.
\end{aligned}
\]
This completes the verification of $P\left(\norm{M}_{\mathcal{L}^{2}(L^{2}(R))}^{2} \geq \frac{4\kappa \mathcal{N}_{1}(\lambda)} \log \frac{4}{\delta}{\lambda l} + \frac{32\kappa^{2}\left(\log \frac{4}{\delta}\right)^{2}}{\lambda^{2}l^2}\right) \leq \frac{\delta}{2}$ and therefore completes the proof.
\end{proof}

\begin{rem}
\label{Remark: Conditions for N2l to behave as N2}
% There are some conditions under which the bound on $\mathcal{N}^2_{2,l}(\lambda)$ behaves like that of $\mathcal{N}_{2}^{2}(\lambda)$. 
Suppose the eigenvalues of $\Sigma_{PQ}$ has a polynomial decay rate, i.e., $i^{-\beta}\lesssim\lambda_i(\Sigma_{PQ})\lesssim i^{-\beta}$, $\beta>1$. Then, it is easy to verify that $\mathcal{N}_{2}^{2}(\lambda) \gtrsim \frac{ \mathcal{N}_{1}(\lambda)}{\lambda l}$ if $\lambda \gtrsim l^{-1}$ and $\mathcal{N}_{2}^{2}(\lambda) \gtrsim \frac{1}{\lambda^{2}l^{2}}$ if $\lambda \gtrsim l ^{-\frac{\beta}{\beta - \frac{1}{2}}}$.
% \begin{enumerate}
% \item \textbf{1st term:} $N_{2}^{2}(\lambda) \gtrsim \frac{ \textcolor{blue}{\mathcal{N}_{1}(\lambda)}\log \frac{4}{\delta}}{\lambda l}$ holds if $\lambda \gtrsim l^{-1}.$
% % $\lambda \gtrsim l^{-\frac{\beta}{\beta-1}}$. 
% \item \textbf{2nd term:} $N_{2}^{2}(\lambda) \gtrsim \frac{\left(\log \frac{4}{\delta}\right)^{2}}{\lambda^{2}l^{2}}$ holds if $\lambda \gtrsim l ^{-\frac{\beta}{\beta - \frac{1}{2}}}$. 
% \end{enumerate}
$\lambda \gtrsim l^{-1}$ is a sufficient condition for both the above bounds to hold. However, the conditions imposed on $\lambda$ and $l$ in the statement of Lemma \ref{Lemma 13} imply that $\lambda \gtrsim \frac{\log l}{l}$, which is an even stronger condition. Therefore, $\mathcal{N}_{2,l}^{2}(\lambda) \lesssim \mathcal{N}_{2}^{2}(\lambda)$ if $\lambda \gtrsim \frac{\log l}{l}$.

On the other hand, if the eigenvalues of $\Sigma_{PQ}$ have an exponential decay rate, i.e., $e^{-i}\lesssim\lambda_i(\Sigma_{PQ})\lesssim e^{-i}$. Then $\mathcal{N}_{2}^{2}(\lambda) \gtrsim \frac{\mathcal{N}_{1}(\lambda)}{\lambda l}$ if $\lambda \gtrsim l^{-1}$ and $\mathcal{N}_{2}^{2}(\lambda) \gtrsim \frac{1}{\lambda^{2}l^{2}}$ if $e^{-1} \geq \lambda \gtrsim \frac{\log l}{l}$.
%and \\$l \geq 2\bar{C}\left[\log(\frac{4}{\delta})+\log(\frac{8}{\alpha})\right]e \left[\log\left[\log(\frac{4}{\delta})+\log(\frac{8}{\alpha})\right]+1\right]$ for some universal constant $\bar{C} \geq 1$.
%Next, we assume exponential decay of eigenvalues of $\Sigma_{PQ}$.
% \begin{enumerate}
% \item \textbf{1st term:} $N_{2}^{2}(\lambda) \gtrsim \frac{\textcolor{blue}{\mathcal{N}_{1}(\lambda)}\log \frac{4}{\delta}}{\lambda l}$ holds if $\lambda \gtrsim l^{-1}$.
% % $e^{-1} \geq \lambda \gtrsim \frac{\left[\log(\frac{4}{\delta})+\log(\frac{8}{\alpha})\right]\log l}{l}$ and \\ $l \geq 2\bar{C}\left[\log(\frac{4}{\delta})+\log(\frac{8}{\alpha})\right]e \left[\log\left[\log(\frac{4}{\delta})+\log(\frac{8}{\alpha})\right]+1\right]$ for some universal constant $\bar{C} \geq 1$. 
% \item \textbf{2nd term:} $N_{2}^{2}(\lambda) \gtrsim \frac{(\log \frac{4}{\delta})^{2}}{\lambda^{2}l^{2}}$ holds if $e^{-1} \geq \lambda \gtrsim \frac{\left[\log(\frac{4}{\delta})+\log(\frac{8}{\alpha})\right]\log l}{l}$ and \\
% $l \geq 2\bar{C}\left[\log(\frac{4}{\delta})+\log(\frac{8}{\alpha})\right]e \left[\log\left[\log(\frac{4}{\delta})+\log(\frac{8}{\alpha})\right]+1\right]$ for some universal constant $\bar{C} \geq 1$. 
% \end{enumerate}
Thus, $\mathcal{N}_{2,l}^{2}(\lambda) \lesssim \mathcal{N}_{2}^{2}(\lambda)$ if $\lambda \gtrsim \frac{\log l}{l}$. %and\\ $l \geq 2\bar{C}\left[\log(\frac{4}{\delta})+\log(\frac{8}{\alpha})\right]e \left[\log\left[\log(\frac{4}{\delta})+\log(\frac{8}{\alpha})\right]+1\right]$ for some universal constant $\bar{C} \geq 1$. 
% Note that, $\frac{\log l}{l} \leq e^{-1}$ for any $l\geq 1$, so the upper bound on $\lambda$ does not impose any restrictions on the number of random features $l$.
\end{rem}

%The lower bound on l is derived using technique similar to https://mathoverflow.net/questions/354587/simple-bound-on-logx-x

\begin{lem}
\label{Lemma 14}
Let us define $\mathcal{N}_{1}(\lambda)\coloneq\operatorname{Tr}\left(\Sigma_{P Q, \lambda}^{-1 / 2} \Sigma_{P Q} \Sigma_{P Q, \lambda}^{-1 / 2}\right)$,
$\mathcal{N}_{2}(\lambda)\coloneq\left\|\Sigma_{P Q, \lambda}^{-1 / 2} \Sigma_{P Q} \Sigma_{P Q, \lambda}^{-1 / 2}\right\| _{\mathcal{L}^{2}(\mathcal{H})}$, $\mathcal{N}_{1,l}(\lambda)\coloneq\operatorname{Tr}\left(\Sigma_{P Q, \lambda,l}^{-1 / 2} \Sigma_{P Q,l} \Sigma_{P Q, \lambda,l}^{-1 / 2}\right)$ and $\mathcal{N}_{2,l}(\lambda)\coloneq\left\|\Sigma_{P Q, \lambda,l}^{-1 / 2} \Sigma_{P Q,l} \Sigma_{P Q, \lambda,l}^{-1 / 2}\right\| _{\mathcal{L}^{2}(\mathcal{H}_{l})}$.
Then, we have \[
\mathcal{N}_{2}(\lambda) \leq \sqrt{\mathcal{N}_{1}(\lambda)},\,\,\text{and}\,\,\mathcal{N}_{2,l}(\lambda) \leq \sqrt{\mathcal{N}_{1,l}(\lambda)}.
\]
% and
% \[
% N_{2,l}(\lambda) \leq N_{1,l}^{1/2}(\lambda). 
% \]
\end{lem}

\begin{proof}
Note that $\mathcal{V}=\Sigma_{P Q, \lambda}^{-1 / 2} \Sigma_{P Q} \Sigma_{P Q, \lambda}^{-1 / 2}$ is a positive self-adjoint trace-class operator with operator norm $\norm{\mathcal{V}}_{L^{\infty}(\mathcal{H})}=\norm{\Sigma_{P Q, \lambda}^{-1 / 2} \Sigma_{P Q} \Sigma_{P Q, \lambda}^{-1/2}}_{\mathcal{L}^{\infty}(\mathcal{H})} \leq 1$. By definition, we have,
$$\mathcal{N}_{2}^{2}(\lambda) = \operatorname{Tr}(\Sigma_{P Q, \lambda}^{-1 / 2} \Sigma_{P Q} \Sigma_{P Q, \lambda}^{-1} \Sigma_{P Q} \Sigma_{P Q, \lambda}^{-1 / 2}) = \operatorname{Tr}(\mathcal{V}^{*}\mathcal{V}) =\operatorname{Tr}(\mathcal{V}^{2})$$ and $\mathcal{N}_{1}(\lambda)=\operatorname{Tr}(\mathcal{V})$. Hence, using H\"{o}lder's inequality, $\mathcal{N}_{2}^{2}(\lambda) \leq \norm{\mathcal{V}}_{\mathcal{L}^{\infty}(\mathcal{H})} \operatorname{Tr}(\mathcal{V}) \leq \mathcal{N}_{1}(\lambda)$. Thus, $\mathcal{N}_{2}(\lambda) \leq \mathcal{N}_{1}^{1/2}(\lambda)$. The proof of the other result is exactly similar upon replacing $\Sigma_{PQ}$ and $\Sigma_{PQ,\lambda}$ by $\Sigma_{PQ,l}$ and $\Sigma_{PQ,\lambda,l}$, respectively.
\end{proof}

% \begin{lemma}
% \label{Lemma 15}
% Suppose that $F_{\lambda,l}$ and $\hat{F}^{B}_{\lambda,l}$ be the permutation distribution function and the empirical permutation distribution function based on $B$ randomly selected permutations $(\pi^i)_{i=1}^{B}$ from $\Pi_{n+m}$ as defined in \eqref{All Permutation CDF for RFF based test}  and \eqref{B Permutation CDF for RFF based test}, respectively, with $q_{1-\alpha}^{\lambda,l}$ and $\hat{q}_{1-\alpha}^{B,\lambda,l}$ being the corresponding $(1-\alpha)$-th quantiles for any $0<\alpha\leq 1$, as defined in \eqref{All Permutation quantile for RFF based test} and \eqref{B Permutation quantile for RFF based test}. Then, for any $\alpha,\tilde{\alpha},\delta >0$, we have the following statements:

% (i) $P_{\pi,\Xi^{l}}(\hat{q}_{1-\alpha}^{B,\lambda,l} \geq q_{1-\alpha -\tilde{\alpha}}^{\lambda,l}) \geq 1 - \delta$ and\\
% (ii) $P_{\pi,\Xi^{l}}(\hat{q}_{1-\alpha}^{B,\lambda,l} \leq q_{1-\alpha +\tilde{\alpha}}^{\lambda,l}) \geq 1 - \delta$

% provided the number of randomly selected permutations $B \geq \frac{\log(\frac{2}{\delta})}{2\tilde{\alpha}^{2}}$.
% \end{lemma}

% \begin{proof}
% Conditioning on $\theta^{1:l}$, the kernel $K_{l}$ and its corresponding RKHS $\mathcal{H}_{l}$ is fixed. Then Lemma A.14 in \cite{SpectralTwoSampleTest} applies with the kernel $K$ replaced by $K_{l}$ and the RKHS $\mathcal{H}$ replaced by the RKHS $\mathcal{H}_{l}$. Finally, taking expectations over the random features $\theta^{1:l}$, the required results are obtained.

% \end{proof}

\begin{lem}
\label{Lemma 15}
Let $F_{\lambda,l}$ and $\hat{F}^{B}_{\lambda,l}$ be the permutation distribution function and the empirical permutation distribution function based on $B$ randomly selected permutations $(\pi^i)_{i=1}^{B}$ from $\Pi_{n+m}$ as defined in \eqref{All Permutation CDF for RFF based test}  and \eqref{B Permutation CDF for RFF based test}, respectively, with $q_{1-\alpha}^{\lambda,l}$ and $\hat{q}_{1-\alpha}^{B,\lambda,l}$ being the corresponding $(1-\alpha)$-th quantiles for any $0<\alpha\leq 1$, as defined in \eqref{All Permutation quantile for RFF based test} and \eqref{B Permutation quantile for RFF based test}. Then, for any $\alpha,\alpha^{\prime},\delta >0$, we have the following statements:\\

(i) $P_{\pi}(\hat{q}_{1-\alpha}^{B,\lambda,l} \geq q_{1-\alpha -\alpha^{\prime}}^{\lambda,l}) \geq 1 - \delta$;\vspace{1mm}\\

(ii) $P_{\pi}(\hat{q}_{1-\alpha}^{B,\lambda,l} \leq q_{1-\alpha +\alpha^{\prime}}^{\lambda,l} ) \geq 1 - \delta$,
provided $B \geq \frac{\log(\frac{2}{\delta})}{2\left(\alpha^{\prime}\right)^{2}}$.
\end{lem}
\begin{proof}
Conditioned on $\theta^{1:l}$, the kernel $K_{l}$ and its corresponding RKHS $\mathcal{H}_{l}$ are non-random, and the proof therefore follows from Lemma A.14 in \citep{SpectralTwoSampleTest} upon replacing all test statistics based on the kernel $K$ with test statistics based on the kernel $K_{l}$ and probabilities with conditional probabilities, yielding
\[
P_{\pi}(\hat{q}_{1-\alpha}^{B,\lambda,l} \geq q_{1-\alpha -\alpha^{\prime}}^{\lambda,l} \mid \theta^{1:l}) \geq 1 - \delta\] 
and
\[
P_{\pi}(\hat{q}_{1-\alpha}^{B,\lambda,l} \leq q_{1-\alpha +\alpha^{\prime}}^{\lambda,l}  \mid \theta^{1:l}) \geq 1 - \delta.
\]
Finally, taking expectations over the random features $\theta^{1:l}$, the required results are obtained.
\end{proof}

% \begin{lemma}
% \label{Lemma 16}

% Define $\zeta_{l}=\left\|g_{\lambda}^{1/2}(\hat{\Sigma}_{P Q,l})\left(\mu_{P,l}-\mu_{Q,l}\right)\right\|_{\mathcal{H}_{l}}^{2}$
% and\\
% $\gamma_{3,l} = \frac{\norm{\mathcal{M}_{l}}_{\mathcal{L}^{\infty}(\mathcal{H}_{l})}^{2}\log(\frac{1}{\alpha})}{\sqrt{\delta}(n+m)}\left(\sqrt{C_{\lambda,l}}\norm{u}_{L^{2}(R)} + N_{2,l}(\lambda) + C_{\lambda,l}^{\frac{1}{4}}\norm{u}_{L^{2}(R)}^{\frac{3}{2}} + \norm{u}_{L^{2}(R)}\right) + \frac{\zeta_{l}\log(\frac{1}{\alpha})}{\sqrt{\delta}(n+m)}$ where $C_{\lambda,l}$ is as defined in Lemma \ref{Lemma 9}.
%  Further, let $m \leq n \leq D m$ for some constant $D \geq 1$.
% Then, for any $0<\alpha\leq e^{-1}$ and $\delta>0$, we have that
% \[
% P_{H_{1}}(q_{1-\alpha}^{\lambda,l} >C^{*}\gamma_{3,l}\mid \theta^{1:l}) \leq \delta
% \]
% for some absolute positive constant $C^{*}$.

% \end{lemma}

% \begin{proof}
% Conditional on $\theta^{1:l}$, the kernel $K_{l}$ and its corresponding RKHS $\mathcal{H}_{l}$ is fixed. From this point onwards, the proof is similar to that of Lemma A.15 in \citep{SpectralTwoSampleTest} upon replacing the kernel $K$ with the RFF-based kernel $K_{l}$ and probabilities with conditional probabilities.
% \end{proof}

\begin{lem}
\label{Lemma 16}
Define $\zeta_{l}=\left\|g_{\lambda}^{1/2}(\hat{\Sigma}_{P Q,l})\left(\mu_{P,l}-\mu_{Q,l}\right)\right\|_{\mathcal{H}_{l}}^{2}$
and\\
\begin{equation*}\begin{aligned}\gamma_{3,l} =& \frac{\norm{\mathcal{M}_{l}}_{\mathcal{L}^{\infty}(\mathcal{H}_{l})}^{2}\log(\frac{1}{\alpha})}{\sqrt{\delta}(n+m)}\left(\sqrt{N_{2}^{\prime}(\kappa,\lambda,\delta,l)}\norm{u}_{L^{2}(R)} + N_{2}^{*}(\kappa,\lambda,\delta,l)\right.\\
&\left.+ (N_{2}^{\prime}(\kappa,\lambda,\delta,l))^{\frac{1}{4}}\norm{u}_{L^{2}(R)}^{\frac{3}{2}} + \norm{u}_{L^{2}(R)}\right)+ \frac{\zeta_{l}\log(\frac{1}{\alpha})}{\sqrt{\delta}(n+m)},\end{aligned}\end{equation*} where $\mathcal{M}_{l} = \hat{\Sigma}_{PQ,\lambda,l}^{-1/2}\Sigma_{PQ,\lambda,l}^{1/2}$, $N_{2}^{*}(\kappa,\lambda,\delta,l) \coloneq \frac{4\sqrt{2\mathcal{N}_1(\lambda)\kappa \log \frac{4}{\delta}}}{\sqrt{\lambda l}} + \frac{16\kappa \log \frac{4}{\delta}}{\lambda l}  + 2\sqrt{2}\mathcal{N}_{2}(\lambda)$ and $N_{2}^{\prime}(\kappa,\lambda,\delta,l)\coloneq\frac{2N_{2}^{*}(\kappa,\lambda,\delta,l)\kappa}{\lambda}$. 
 Further, let $m \leq n \leq D^{\prime} m$ for some constant $D^{\prime} \geq 1$.
Then, for any $0<\alpha\leq e^{-1}$ and $\delta>0$, we have that
\[
P_{H_{1}}(q_{1-\alpha}^{\lambda,l} >C^{*}\gamma_{3,l}\mid E) \leq \delta,
\]
where $E=\left\{\mathcal{N}_{2,l}(\lambda) \leq N_{2}^{*}(\kappa,\lambda,\delta,l)\right\}$ and $C^{*}$
is an absolute positive constant.
\end{lem}
\begin{proof}
Let us define $$\gamma_{4,l} = \frac{\norm{\mathcal{M}_{l}}_{\mathcal{L}^{\infty}(\mathcal{H}_{l})}^{2}\log(\frac{1}{\alpha})}{\sqrt{\delta}(n+m)}\left(\sqrt{C_{\lambda,l}}\norm{u}_{L^{2}(R)} + \mathcal{N}_{2,l}(\lambda) + (C_{\lambda,l})^{\frac{1}{4}}\norm{u}_{L^{2}(R)}^{\frac{3}{2}} + \norm{u}_{L^{2}(R)}\right)+ \frac{\zeta_{l}\log(\frac{1}{\alpha})}{\sqrt{\delta}(n+m)},$$ where $C_{\lambda,l}$ is as defined in Lemma \ref{Lemma 9}. Conditioned on $\theta^{1:l}$, the kernel $K_{l}$ and its corresponding RKHS $\mathcal{H}_{l}$ are non-random. Therefore, the proof follows from Lemma A.15 in \citep{SpectralTwoSampleTest} by replacing the kernel $K$ with the RFF-based kernel $K_{l}$ and probabilities with conditional probabilities, yielding
\[
P_{H_{1}}(q_{1-\alpha}^{\lambda,l} > C^{*}\gamma_{4,l} \mid \theta^{1:l}) \leq \delta.
\]
If the event $E$ occurs, we must have that $\gamma_{4,l} \leq \gamma_{3,l}$ and therefore, we obtain
$
P_{H_{1}}(q_{1-\alpha}^{\lambda,l} >C^{*}\gamma_{3,l}\mid E) \leq \delta.$
\end{proof}

\begin{lem}
\label{Lemma 17}
Let $X$ be a random variable, $\lambda$ be a deterministic parameter (taking values in a finite set $\Lambda$), $\nu$ be a random parameter, and $0\leq \alpha \leq 1$ be the level of significance. Further, suppose that f is a function of $X$, $\lambda$ and $\nu$, while $\gamma$ is a function of $\alpha$, $\lambda$ and $\nu$. Further, assume that $P\left\{f(X,\lambda,\nu) \geq \gamma(\alpha,\lambda,\nu)\right\} \leq \alpha$ for any $0\leq \alpha \leq 1$ and $\lambda \in \Lambda$, with the probability being computed with respect to the distributions of $X$ and $\nu$. Then, we have that,
\[
P\left\{\underset{\lambda \in \Lambda}{\bigcup}f(X,\lambda,\nu) \geq \gamma\left(\frac{\alpha}{|\Lambda|},\lambda,\nu\right)\right\} \leq \alpha. 
\]
In addition, if $P\left\{f(X,\lambda^{*},\nu) \geq \gamma(\alpha,\lambda^{*},\nu)\right\} \geq \delta$ for some $0\leq \delta \leq 1$ and $\lambda^{*} \in  \Lambda$, we have that,
\[
P\left\{\underset{\lambda \in \Lambda}{\bigcup}f(X,\lambda,\nu) \geq \gamma\left(\alpha,\lambda,\nu\right)\right\} \geq \delta. 
\]
\end{lem}

\begin{proof}
Observe that
\[
P\left\{\underset{\lambda \in \Lambda}{\bigcup}f(X,\lambda,\nu) \geq \gamma\left(\frac{\alpha}{|\Lambda|},\lambda,\nu\right)\right\} \leq \underset{\lambda \in \Lambda}{\sum} P\left\{f(X,\lambda,\nu) \geq \gamma\left(\frac{\alpha}{|\Lambda|},\lambda,\nu\right)\right\} \leq |\Lambda| \times \frac{\alpha}{|\Lambda|} = \alpha
\]
and 
\[
P\left\{\underset{\lambda \in \Lambda}{\bigcup}f(X,\lambda,\nu) \geq \gamma\left(\alpha,\lambda,\nu\right)\right\} \geq P\left\{f(X,\lambda^{*},\nu) \geq \gamma\left(\alpha,\lambda^{*}\nu\right)\right\},
\]
and the results follow. 
\end{proof}
The following result is adapted from \citep[Theorem 3.3.4]{Yurinsky_1995}.
\begin{thmm}[Bernstein's inequality in separable Hilbert spaces]\label{Bernstein ineq for Hilbert spaces}
Let $(\Omega, \mathcal{A}, P)$ be a probability space, $H$ be a separable Hilbert space, $B>0$ and $v>0$. Furthermore, let $\xi_1, \ldots, \xi_n: \Omega \rightarrow H$ be zero mean i.i.d. random variables satisfying
\[
\mathbb{E}\left\|\xi_1\right\|_H^r \leq \frac{r!}{2} v^2 B^{r-2}, \forall r>2 .
\]
Then for any $0<\delta<1$,
\[
P^n\left\{\left(\xi_i\right)_{i=1}^n:\left\|\frac{1}{n} \sum_{i=1}^n \xi_i\right\|_H \geq \frac{2 B \log \frac{2}{\delta}}{n}+\sqrt{\frac{2 v^2 \log \frac{2}{\delta}}{n}}\right\} \leq \delta .
\]
\end{thmm}
\section{Calculation of computational complexity}
In this appendix, we provide computational complexity calculations for the exact and approximate test statistics.
\subsection{Computational complexity calculation of exact test statistic}\label{Appendix: Computational complexity details of exact test statistic}

The algorithm for computing the exact test statistic defined in \eqref{Spectral Regularized Kernel Test statistic}, along with the computational complexity of each step involved, is as follows:

\begin{enumerate}
    \item Constructing $Z_{i}$ : $\bm{O(sd)}$.
    \item Computation of pairwise $\ell_{1}$ or $\ell_{2}$ distance matrices required for computation of $K_{n}$, $K_{m}$, $K_{s}$, $K_{mn}$, $K_{ns}$ and $K_{ms}$ : $\bm{O(s^2d + n^2d + m^2d + nsd + msd + mnd)}$.

    \item Computation of  $K_{n}$, $K_{m}$, $K_{s}$, $K_{mn}$, $K_{ns}$ and $K_{ms}$ : $\bm{O(s^2 + n^2 + m^2 + ns + ms + mn)}$.

    \item Constructing $H_{s}$ and $\tilde{H}_{s}$ : $\bm{O(s^2)}$

    \item Constructing $H_{s}^{1/2}$ : $\bm{O(s^3)}$.

    \item Constructing $\frac{1}{s}H_{s}^{1/2}K_{s}H_{s}^{1/2}$ :  $\bm{O(s^3)}$.

    \item Computing eigen decomposition of $\frac{1}{s}H_{s}^{1/2}K_{s}H_{s}^{1/2}$ : $\bm{O(s^3)}$.

    \item Computing $\frac{g_{\lambda}(\hat{\lambda}_{i}) - g_{\lambda}(0)}{\hat{\lambda}_{i}}$ for $i=1\ldots,s$ : $\bm{O(s)}$.

    \item Computing $G$ : $\bm{O(s^3)}$.

    \item Computing $\circled{\emph{\small{1}}}$ : $\bm{O(n^2 + ns + s^2)}$.

    \item Computing $\circled{\emph{\small{2}}}$ : $\bm{O(ns^2 + s^3)}$.

    \item Computing $\circled{\emph{\small{3}}}$ : $\bm{O(m^2 + ms + s^2)}$.

    \item Computing $\circled{\emph{\small{4}}}$ : $\bm{O(ms^2 + s^3)}$.

    \item Computing $\circled{\emph{\small{5}}}$ : $\bm{O(mn + ms + ns + s^2)}$
\end{enumerate}

Based on this calculation, the total computational complexity of the \say{exact} spectral regularized MMD test statistic $\hat{\eta}_{\lambda}$ in terms of number of mathematical operations is \[O(s^3+ns^2+ms^2+s^2d+n^2d+m^2d+nsd+msd+mnd).\]

\subsection{Computational complexity calculation of RFF-based approximate test statistic}\label{Appendix: Computational complexity details of approximate test statistic}

The procedure for computing the RFF-based test statistic defined in \eqref{Approximate Kernel Test statistic}, along with the computational complexity of each step involved, is as follows:

\begin{enumerate}
    \item Construct the $d \times n$ matrix $X = \left[X_{1}  \dots X_{n}\right]$, the $d \times m$ matrix $Y = \left[Y_{1}  \dots Y_{m}\right]$ :  $\bm{O(nd+md)}$
    \item Sample $\mathbf{\theta_{i}} \in \mathbb{R}^d,i=1,2,\ldots,l$ in an i.i.d manner from the spectral distribution (inverse Fourier transform) $\Xi$ corresponding to the kernel $K$ and store it in an $l \times d$ matrix $\Theta$ : $\bm{O(ld)}$
    \item Compute $Z_i=\alpha_i X_i^1+\left(1-\alpha_i\right) Y_i^1$, for $1 \leq i \leq s$, and $\left(\alpha_i\right)_{i=1}^s \stackrel{i . i . d}{\sim} \operatorname{Bernoulli}(1 / 2)$. Construct the $d \times s$ matrix $Z = \left[Z_{1}  \dots Z_{s}\right]$ : $\bm{O(sd)}$
    \item Compute the $n \times l$ matrix $M_X = X^{T} \Theta ^{T} = (\Theta X)^T$, the $m \times l$ matrix $M_Y = Y^{T} \Theta ^{T} = (\Theta Y)^T$ and the $s \times l$ matrix $M_Z = Z^{T} \Theta ^{T} = (\Theta Z)^T$ : $\bm{O(nd+md+sd+ld+nld+mld+sld)}$
    \item Compute the $2l \times n$ matrix of random features corresponding to $X_{i}$'s ($i=1,\dots,n$) as $\Phi(X) = \frac{K(0,0)}{\sqrt{l}}P_{l}^{T}\begin{bmatrix} \cos(M_X) \,|\,\sin(M_X) \end{bmatrix}^{T}$, the $2l \times m$ matrix of random features corresponding to $Y_{j}$'s ($j=1,\dots,m$) as $\Phi(Y) = \frac{K(0,0)}{\sqrt{l}}P_{l}^{T}\begin{bmatrix} \cos(M_Y) \,|\, \sin(M_Y) \end{bmatrix}^{T}$ and the $2l \times s$ matrix of random features corresponding to $Z_{i}$'s as $\Phi(Z) = \frac{K(0,0)}{\sqrt{l}}P_{l}^{T}\begin{bmatrix} \cos(M_Z) \,|\, \sin(M_Z) \end{bmatrix}^{T}$. Here, the matrix $P_{l}$ is defined as the column interleaving permutation matrix 
    $$P_{l} = 
    \begin{bmatrix}
    e_{1,2l} \quad e_{l+1,2l} \quad e_{2,2l} \quad e_{l+2,2l}\;\cdots \;e_{l,2l}\quad e_{2l,2l}
    \end{bmatrix},$$
    where $e_{i,2l}$ is a unit column vector of length $2l$ with $1$ at the $i$-th position and $0$ elsewhere. To compute $\Phi(X)$, the cosine and sine functions are first applied elementwise to the \(n \times l\) matrix \(M_X\), yielding the matrices $\cos(M_X)$ and $\sin(M_X)$. These two matrices are then concatenated horizontally to form the $n \times 2l$ matrix $\begin{bmatrix} \cos(M_X) \,|\,\sin(M_X) \end{bmatrix}$. Next, the columns of this matrix are permuted so that the first column of $\cos(M_X)$ appears first, followed by the first column of $\sin(M_X)$, then the second column of $\cos(M_X)$, followed by the second column of $\sin(M_X)$, and so on. This interleaving of columns is achieved by post-multiplying \(\begin{bmatrix} \cos(M_X) \,|\, \sin(M_X) \end{bmatrix}\) with $P_l$. Finally, we compute the transpose of the permuted matrix $\begin{bmatrix} \cos(M_X) \,|\,\sin(M_X) \end{bmatrix}P_{l}$ and scale it by $\frac{K(0,0)}{\sqrt{l}}$to obtain $\Phi(X)$. The matrices $\Phi(Y)$ and $\Phi(Z)$ are computed in the same manner : 
    $\bm{O(nl+ml+sl)}$
    \item Compute the $2l \times 2l$ matrix $K_s = \Phi(Z)\Phi(Z)^{T} $ and $v_Z = \Phi(Z) \mathbf{1}_{s}$ : %.\\
    %\textbf{Computational complexity:- 
    $\bm{O(sl^2 + sl)}$%}.
    \item Compute the $2l \times 2l$ matrix $\hat{\Sigma}_{PQ,l} = \frac{1}{s(s-1)} (s K_{s} - v_{Z}v_{Z}^{T})$ : 
    %.\\\textbf{Computational complexity:- 
    $\bm{O(l^2)}$
    \item Compute the eigenvalue-eigenvector pairs $(\hat{\lambda}_{i},\hat{\alpha}_{i})$ corresponding to $\hat{\Sigma}_{PQ,l}$. Construct the diagonal $2l \times 2l$ matrix $D = \begin{bmatrix}
    \hat{\lambda}_{1} & & \\
    & \ddots & \\
    & & \hat{\lambda}_{2l}
  \end{bmatrix}$ and the $2l \times 2l$ matrix $V = \left[ \hat{\alpha}_{1} \dots \hat{\alpha}_{2l} \right]$ : %.\\
%  \textbf{Computational complexity:- 
$\bm{O(l^3)}$%}.
  \item Construct the $2l \times 2l$ matrix $G = V L^{1/2} V^{T}$, where $L^{1/2} = \begin{bmatrix}
    \sqrt{g_{\lambda}(\hat{\lambda}_{1})} & & \\
    & \ddots & \\
    & & \sqrt{g_{\lambda}(\hat{\lambda}_{2l})}
  \end{bmatrix}$ : %.\\
  %\textbf{Computational complexity:- 
  $\bm{O(l^3 + l^2)}$%}.
  \item Compute the $2l \times n$ matrix $\Psi(X) = G \Phi(X)$ and the $2l \times m$ matrix $\Psi(Y) = G \Phi(Y)$ : %.\\
  %\textbf{Computational complexity:- 
  $\bm{O(nl^2+ml^2)}$%}.
  \item Compute the vectors $v_{X,i} = \Psi(X) e_{i,n}$ for $i=1,\dots,n$ and $v_{Y,j} = \Psi(Y) e_{j.m}$ for $j=1,\dots,m$, where $e_{i,n}$ and $e_{j,m}$ are unit column vectors of lengths $n$ and $m$, respectively, each having a $1$ in its $i$-th or $j$-th position and 0 elsewhere : %\\
  %\textbf{Computational complexity:- 
  $\bm{O(nl+ml)}$%}.
  \item Compute $v_{X} = \sum_{i=1}^{n} v_{X,i}$ and $v_{Y} = \sum_{j=1}^{m} v_{Y,j}$ : %.\\
  %\textbf{Computational complexity:- 
  $\bm{O(nl+ml)}$%}.
  \item Compute $A = v_{X}^{T} v_{X} - \sum_{i=1}^{n}v_{X,i}^{T} v_{X,i}$ : %.\\
  %\textbf{Computational complexity:- 
  $\bm{O(nl)}$%}.
  \item Compute $B = v_{Y}^{T} v_{Y} - \sum_{j=1}^{m}v_{X,j}^{T} v_{X,j}$ : %.\\
  %\textbf{Computational complexity:- 
  $\bm{O(ml)}$%}.
  \item Compute $C = v_{X}^{T} v_{Y}$ : %.\\
  %\textbf{Computational complexity:- 
  $\bm{O(l)}$%}.
  \item Compute the test statistic $\hat{\eta}_{\lambda,l} = \frac{A}{n(n-1)} + \frac{B}{m(m-1)} - \frac{2C}{nm}$ : %.\\
  %\textbf{Computational complexity:- 
  $\bm{O(1)}$%}.
\end{enumerate}
Based on this calculation, the total computational complexity of the RFF-based approximate spectral regularized MMD test statistic $\hat{\eta}_{\lambda,l}$ is
\[
O(l^{3} + (s+m+n) l^{2} + (s+m+n)ld).
\]

%%%%%%%%%%%%%%%%%%%%%%%%%%%%%%%%%%%%%%%%%%%%%%%%%%%%%%%%%%%%

\end{document}